\documentclass[10pt]{article}

\newif\ifsiamart
\makeatletter%
\@ifclassloaded{siamart171218}%
  {\siamarttrue}%
  {\@ifclassloaded{siamart251216}%
    {\siamarttrue}%
    {\siamartfalse}%
  }%
\makeatother%

\ifsiamart\else
\usepackage{authblk}
\fi
\usepackage{silence}
\usepackage[T1]{fontenc}
\usepackage[utf8]{inputenc}
\usepackage{csquotes}
\usepackage[UKenglish]{babel}
\usepackage{paralist}
\usepackage[shortlabels]{enumitem}
\usepackage[margin=.8in]{geometry}
\usepackage{xcolor}
\usepackage{graphicx}
\usepackage{bbm}
\usepackage{bigints}
\usepackage{amsmath}
\usepackage{amssymb}
\usepackage{amsfonts}
\usepackage{stmaryrd}
\usepackage{mathrsfs}
\usepackage{mathtools}
\usepackage{epstopdf}
\usepackage{comment}
\usepackage{todonotes}
\usepackage{upgreek}

\usepackage{xspace}
\usepackage{float}
\usepackage{stackengine}

\ifsiamart\else
\usepackage{amsthm}
\usepackage{hyperref}
\hypersetup{%
    linktocpage=true,
    colorlinks=true,
    citecolor=red,
    linkcolor=blue,
    urlcolor=blue,
}
\usepackage[nameinlink,capitalise]{cleveref}
\newcommand{\email}[1]{\href{mailto:#1}{#1}}
\newenvironment{keywords}{\noindent \textbf{Keywords.}}{}
\newenvironment{AMS}{\noindent \textbf{AMS subject classification.}}{}
\fi

\ifsiamart\else
\usepackage{setspace}
\fi

\ifsiamart
\newtheorem{remark}[theorem]{Remark}
\newtheorem{example}[theorem]{Example}
\newtheorem{assumption}{Assumption}

\else
\theoremstyle{plain}

\newtheorem{theorem}{Theorem}
\newtheorem{lemma}[theorem]{Lemma}
\newtheorem{corollary}[theorem]{Corollary}
\newtheorem{proposition}[theorem]{Proposition}

\newtheorem{remark}[theorem]{Remark}

\newtheorem{example}[theorem]{Example}
\newtheorem{assumption}{Assumption}

\fi

\crefname{lemma}{Lemma}{Lemmas}
\crefname{remark}{Remark}{Remarks}
\crefname{assumption}{Assumption}{Assumptions}
\crefname{proposition}{Proposition}{Propositions}
\crefname{section}{Section}{Sections}
\crefname{subsection}{Subsection}{Subsections}
\crefname{example}{Example}{Examples}
\crefname{equation}{}{}
\Crefname{equation}{Equation}{Equations}
\newlist{lemmaenum}{enumerate}{3}
\setlist[lemmaenum]{label=(\alph*),ref=\,(\alph*)}
\crefname{lemmaenum}{Lemma}{Lemmas}
\newlist{assumpenum}{enumerate}{5}
\setlist[assumpenum]{label=(\alph*), font={\bfseries}}
\newlist{auxenum}{enumerate}{2}
\setlist[auxenum]{label=(\alph*),ref=(\alph*)}
\crefname{auxenumi}{Item}{Items}
\crefname{enumi}{}{}
\crefname{equation}{}{}
\crefname{assumpenumi}{}{}
\crefname{assumpenumii}{}{}
\Crefname{assumpenumi}{Assumption}{Assumptions}
\Crefname{assumpenumii}{Assumption}{Assumptions}
\Crefname{assumpenumii}{Assumption}{Assumptions}
\crefrangeformat{assumpenumi}{#3#1#4--#5#2#6}
\Crefname{lemmaenumi}{Part}{Parts}
\Crefname{figure}{Figure}{Figures}
\numberwithin{equation}{section}
\numberwithin{theorem}{section}

\usepackage{tikz}
\usetikzlibrary{shapes.misc}
\usetikzlibrary{positioning}

\ifsiamart\else
\usepackage[url=true,backend=biber,style=trad-abbrv,url=false,doi=false,isbn=false]{biblatex}
\DeclareFieldFormat{volume}{volume \textbf{#1}}
\DeclareFieldFormat{pages}{}
\DeclareFieldFormat[article]{volume}{\textbf{#1}}

\fi

\ifsiamart\else
\let\oldparagraph=\paragraph
\renewcommand\paragraph[1]{\oldparagraph{#1.}}
\fi

\AtBeginEnvironment{lemma}{\crefalias{theorem}{lemma}}
\AtBeginEnvironment{proposition}{\crefalias{theorem}{proposition}}
\AtBeginEnvironment{corollary}{\crefalias{theorem}{corollary}}
\AtBeginEnvironment{remark}{\crefalias{theorem}{remark}}
\AtBeginEnvironment{example}{\crefalias{theorem}{example}}

\DeclarePairedDelimiter\abs{\lvert}{\rvert}
\DeclarePairedDelimiter\norm{\lVert}{\rVert}
\DeclarePairedDelimiter\eucnorm{\lvert}{\rvert}

\DeclareMathOperator{\I}{I}

\DeclareMathOperator{\e}{e}

\DeclareMathOperator{\id}{id}

\newcommand{\N}{\mathbf{N}}
\newcommand{\R}{\mathbf R}
\newcommand{\Xin}[1]{\Xi}
\newcommand{\Z}{\mathbf{Z}}
\newcommand{\allxi}{\Xi^J}
\newcommand{\widetildeallxi}{\widetilde{\Xi}^J}

\newcommand{\bigo}{\mathcal O}

\newcommand{\cont}{\mathcal C}
\newcommand{\cov}{C^J}

\newcommand{\driftall}{F^{J}_{\sigma}}
\newcommand{\driftinf}{F^{\infty}_{\sigma}}

\newcommand{\driftlim}{F^{J}_{0}}
\newcommand{\driftone}{f_{\sigma}}
\newcommand{\driftonelim}{f_{0}}

\newcommand{\driftderiv}{\left.\partial_{\sigma} \driftall\right|_{\sigma=0}}

\newcommand{\dummy}{\mathord{\color{black!33}\bullet}}%

\newcommand{\expect}{\mathbf{E}}
\newcommand{\fastgaussianall}{\gamma_{dJ}}
\newcommand{\fastgaussianone}{\gamma_{d}}
\newcommand{\fastw}[1]{W^{#1}}

\newcommand{\fullinvariant}{\mu^J[\varepsilon,\sigma]}
\newcommand{\fullinvarianttheta}{\mu^J_\theta[\varepsilon,\sigma]}
\newcommand{\fullinvariantdiscrete}{\mu^J[\varepsilon,\sigma,h]}
\newcommand{\fullinvariantdiscretetheta}{\mu^J_\theta[\varepsilon,\sigma,h]}
\newcommand{\fullgenlimit}{\mathcal L^J[\varepsilon,\sigma=0]}
\newcommand{\fullgenparams}{\mathcal L^J[\varepsilon,\sigma]}
\newcommand{\fullgen}{\mathcal L}
\newcommand{\genfastone}{\mathcal L_{\xi}}
\newcommand{\genfast}{\mathcal L_{\xi}^J}
\newcommand{\genlimit}{\mathcal L_{\rm ovd}}
\newcommand{\genslow}{\mathcal L^J_{\theta}}

\newcommand{\grad}{\nabla}

\newcommand{\intd}{\mathrm{d}}

\newcommand{\laplacian}{\Delta}

\newcommand{\lipasym}{L}

\newcommand{\markovtransition}{P^J[\varepsilon,\sigma,h]}
\newcommand{\nat}{\N}

\newcommand{\normal}{\mathcal{N}}

\newcommand{\openball}[1]{\mathcal B_{#1}}

\newcommand{\placeholder}{\mathord{\color{black!33}\bullet}}%
\newcommand{\proba}{\mathbf P}
\newcommand{\range}[2]{\llbracket #1, #2 \rrbracket}
\newcommand{\real}{\R}

\newcommand{\sloww}{W}
\newcommand{\wslow}{W}
\newcommand{\wfast}[1]{W^{#1}}
\newcommand{\smoothcompact}{\mathcal C^{\infty}_{\rm c}}

\newcommand{\solall}{\psi^J_{\sigma}}

\newcommand{\solone}{\psi_{\sigma}}

\newcommand{\target}{\mu_{\Phi}}
\newcommand{\termi}[1]{\mathcal I^{(#1)}}
\newcommand{\torus}{\mathbf T}

\newcommand{\wn}[1]{W^{#1}}

\newcommand{\xin}[1]{\xi^{#1}}

\renewcommand{\d}{\mathrm d}
\renewcommand{\geq}{\geqslant}

\renewcommand{\leq}{\leqslant}
\renewcommand{\le}{\leqslant}

\newcommand{\sobolev}{\mathcal W}
\newcommand{\weight}[1]{\mathcal K_{#1}}
\newcommand{\weightall}[1]{\weight{#1}^J}
\renewcommand{\t}{\mathsf T}

\newcommand{\vvartheta}{\bar\vartheta}
\newcommand{\dvvartheta}[1]{\bar\vartheta_{#1}}
\newcommand{\dtheta}[1]{\uptheta_{#1}}
\newcommand{\dvartheta}[1]{\upvartheta_{#1}}
\newcommand{\dallxi}[1]{\Upxi^J_{#1}}
\newcommand{\dxin}[2]{\upxi^{#1}_{#2}}
\newcommand{\xslow}[1]{x_{#1}}
\newcommand{\xfast}[2]{x^{#1}_{#2}}

\newcommand{\sref}[1]{\cref{#1}}

\newcommand{\maintitle}{%
    Analysis of a zeroth-order multiscale sampler:
    strong convergence, 
    convergence of ergodic averages
    and weak convergence
}

\title{\maintitle}
\newif\ifdetails\detailsfalse

\author[1,2,3]{Antonin Della Noce\(^{a,}\)}
\author[1,2]{Urbain Vaes\(^{b,}\)}
\affil[ ]{\footnotesize
    \(^a\)\email{antonin.della-noce@centralesupelec.fr},
    \(^b\)\email{urbain.vaes@inria.fr}}
\affil[1]{\footnotesize MATHERIALS project-team, Inria Paris}
\affil[2]{\footnotesize CERMICS, \'Ecole nationale des ponts et chaussées, France}
\affil[3]{\footnotesize Département de Mathématiques, CentraleSupélec, France}
\date{\today}

\begin{document}
\maketitle

\begin{abstract}
    In many practical applications of Bayesian inference,
it is difficult or undesirable to compute the gradient of the negative log-posterior,
motivating the study of derivative-free samplers, also called zeroth-order samplers.
The aim of this paper is to generalize and rigorously analyze a multiscale
derivative-free sampler recently proposed in the literature~\cite{pavliotis21derivative}.
In the discrete-time setting, the multiscale sampler we study bears a strong
resemblance to other recently proposed methods, such as the so-called
zeroth-order perturbed stochastic gradient descent algorithm~\cite{JinEtal2018}
and the zeroth-order Langevin Monte Carlo algorithm~\cite{MR4411512}.
We present error estimates with respect to all the parameters of the problem,
both in the continuous-time and discrete-time settings.
The error estimates are in three different metrics:
strong convergence, convergence of ergodic averages,
and uniform-in-time weak convergence.

\end{abstract}

\ifsiamart

\fi

\newpage
\tableofcontents

\newpage
\section{Introduction}
\label{sec:introduction}

In this paper,
we mathematically analyze a zeroth-order method for generating
approximate samples from a probability measure on~\(\mathcal X\) of the form
\begin{equation}
    \label{eq:sampling_problem}
    \mu_\Phi(\intd u) =
    \frac{1}{Z}\exp\bigl(-\Phi(u)\bigr) \, \d u \, ,
\end{equation}
where \(\Phi \colon \mathcal X \to \real\) is the potential associated with the target distribution,
and \(Z\) denotes the normalization factor.
Here, the space~\(\mathcal X\) can be either~\(\real^d\) or~\(\torus^d \coloneq \real^d / \Z^d\).
Although the setting on~$\real^d$ is more relevant in practice,
the mathematical analysis is more straightforward on the torus,
and we will therefore focus on this setting in the main body of the paper.

Sampling from target distributions of the form~\eqref{eq:sampling_problem} is
fundamental in many areas of applied mathematics and statistics.
It is particularly relevant in fields such as
data assimilation~\cite{MR3363508,MR3242790},
molecular dynamics~\cite{MR2681239,MR3362507},
inverse problems~\cite{MR2102218,MR2652785,MR3839555},
and machine learning~\cite{neal2012bayesian,song2021scorebasedgenerativemodelingstochastic}.
In typical applications,
the dimension~\(d\) of the parameter space \(\mathcal X\) is large
and the normalization constant \(Z\) is intractable.
Furthermore,
the target distribution is often anisotropic or multimodal,
making it challenging to explore high probability regions of the parameter space \(\mathcal X\) efficiently.

A large number of numerical methods have been developed for this task.
These can be broadly categorized into two classes: variational methods and Markov chain Monte Carlo (MCMC) methods.
Methods of the latter type,
which date back to the seminal work of Metropolis et al.~\cite{metropolis1953equation}
and its subsequent generalization by Hastings~\cite{MR3363437},
rely on a Markov process which,
when initialized at a given point in~\(\mathcal X\),
converges in distribution to the target measure as time tends to infinity.
They have received considerable attention in the literature,
and many variants have been proposed to improve their efficiency and convergence properties.
We mention, for example,
gradient-free methods based on random walks~\cite{MR3363437,MR1428751},
as well as modifications of these methods that are robust to refinement of the discretization step~\cite{MR3135540}.
More efficient exploration of the target distribution can often be achieved
when the search follows directions informed by the gradient of~\(\Phi\).
Methods of this type include those based on Hamiltonian dynamics~\cite{neal2012mcmc},
Langevin dynamics~\cite{roberts2002langevin,Chewi26Book},
or Stein variational gradient descent~\cite{liu2017stein,MR4582478}.

In recent years,
there has been growing interest in sampling methods that are affine-invariant,
which are robust to anisotropy in the target distribution~\cite{MR2600822,MR3747563}.
As examples of such methods,
we mention the ensemble Kalman sampler (EKS)~\cite{MR4059375},
the Affine-invariant interacting Langevin dynamics (ALDI)~\cite{MR4123680},
consensus-based sampling~\cite{MR4429081},
Fokker--Planck particle systems~\cite{MR4248267},
and Metropolis-adjusted versions of some of these methods~\cite{MR4877495}.
While these methods are useful for sampling anisotropic distributions,
they are not designed to handle multimodal distributions,
which remain very challenging to sample and often require bespoke methods.
Multimodality manifests as metastability in the associated Markov process used for sampling:
the Markov process can become trapped in the basin of attraction of a local minimum of the potential for long periods of time,
before eventually escaping to explore other regions of the parameter space.
In molecular dynamics, this problem
gave rise to the development of \emph{accelerated dynamics} methods~\cite{voter1997method,voter1998parallel,so2000temperature},
which have recently received considerable attention from the mathematical community~\cite{MR2926765,lelievre2015accelerated}.

In many situations, the gradient of the potential~\(\nabla \Phi\) is not accessible or difficult to compute.
This is the case, for example, when the potential is available only through a black-box function,
or when evaluating this function requires running a physical experiment.
Relying on gradient-based methods may also be undesirable when the potential is not continuously differentiable,
or when only noisy approximations of the exact objective are available~\cite{JinEtal2018,DunbarEtal2022}.
This led to the development of zeroth-order methods,
both for optimization and sampling.
These include, for instance,
the methods inspired by the ensemble Kalman filter (EnKF)~\cite{MR3041539,MR4059375,MR3998631,2018arXiv180508034H},
as well as optimization and sampling methods based on consensus~\cite{MR3597012,MR4429081,MR4899391,MR5067178}.
% Recent gradient-free MCMC methods are also developed and compared in~\cite{MR4877495,baptista2026}.

Our work is a follow-up to~\cite{pavliotis21derivative} with a focus on theoretical analysis.
In that work, building on~\cite{2018arXiv180508034H},
the authors introduced a derivative-free multiscale method for approximating the overdamped Langevin dynamics associated with a potential \(\Phi\) arising from a Bayesian inverse problem.
Like methods based on the ensemble Kalman filter,
the method studied in~\cite{pavliotis21derivative} is derivative-free and based on an interacting particle system,
but unlike these methods, it can be refined to sample from the target distribution with arbitrary accuracy.
The main theoretical contributions of~\cite{pavliotis21derivative} were
strong error estimates between the multiscale dynamics and the overdamped Langevin dynamics associated with~\(\Phi\),
both for the continuous-time multiscale dynamics and for its time discretization.
That work also presented numerical experiments illustrating the performance of the method,
which we do not repeat here.
In the present paper, we generalize the method and extend the analysis of~\cite{pavliotis21derivative} in several directions.
Before presenting our contributions,
we briefly highlight works that are closely related to the method we study in this paper.

\paragraph{Comparison with other methods from the literature}%
The method we analyze builds on a sequence of related ideas.
We highlight here the works that are most closely related to our approach.

\begin{itemize}
    \item
        In 2018, Jin et coauthors~\cite{JinEtal2018} introduced a discrete-time zeroth-order optimization method based on finite-difference approximations of the gradient.
        The goal there was to approximate the minimizer of a population risk function in machine learning,
        by means of an algorithm that only has access to empirical risk evaluations.
        The main intuition for the method is the following Gaussian identity,
        sometimes known as Stein's lemma,
        which can be derived by integration by parts:
        \begin{equation}
            \label{eq:stein_identity}
            \expect\bigl[\nabla\Phi(u+\sigma\xi)\bigr]
            = \frac{1}{\sigma}\expect\bigl[\xi\bigl(\Phi(u+\sigma\xi)-\Phi(u)\bigr)\bigr] \, ,
            \qquad \xi\sim\normal(0, \I_d) \, ,
        \end{equation}
        where $\I_d$ denotes the $d\times d$ identity matrix and
        \(\sigma>0\) is a small parameter.
        The left-hand side is the gradient of the Gaussian-smoothed potential \(\Phi\star\normal(0,\sigma^2 \I_d)\),
        and the right-hand side is an expectation that can be approximated by Monte Carlo sampling.
        This approach leads to the \emph{Zeroth-order Perturbed Stochastic Gradient Descent} (ZPSGD) method that the authors proposed for minimizing~$\Phi$.
        For independent perturbations~\(\xin{j}_n\sim\normal(0,\I_d)\)
        and learning rate~\(\eta>0\),
        their method uses the update
        \begin{equation*}
            u_{n+1} = u_n - \frac{\eta}{J\sigma} \sum_{j=1}^J \xin{j}_n \Bigl(\Phi\bigl(u_n+\sigma\xin{j}_n\bigr)-\Phi(u_n)\Bigr) \, ,
        \end{equation*}

    \item
        Around the same time, Haber and coauthors~\cite{2018arXiv180508034H} independently introduced a related discrete-time,
        derivative-free optimization method specifically for potentials with a composite structure \(\Phi=\ell\circ G\).
        This method similarly perturbs the current iterate~$u_n$ by Gaussian vectors,
        but uses finite differences at the level of the map~$G$.
        More precisely,
        for independent perturbations~\(\xin{j}_n\sim\normal(0,\I_d)\),
        their update is
        \begin{equation*}
            u_{n+1} = u_n - \frac{\eta}{J\sigma}
            \sum_{j=1}^J
            \left\langle \nabla\ell\bigl(G(u_n)\bigr),
            G\bigl(u_n+\sigma\xin{j}_n\bigr)-G(u_n)\right\rangle
            \,\xin{j}_n \, .
        \end{equation*}
        When~$G$ is affine, the finite-difference term in this update satisfies
        \begin{equation}
            \label{eq:finite_difference_identity}
            \frac{1}{J\sigma}\sum_{j=1}^J
            \Bigl\langle \nabla\ell\bigl(G(u_n)\bigr),
            G\bigl(u_n+\sigma\xin{j}_n\bigr)-G(u_n)\Bigr\rangle \, \xin{j}_n
            = \frac{1}{J}\sum_{j=1}^J
            \bigl(\xin{j}_n \otimes\xin{j}_n \bigr)\nabla\Phi(u_n) \, ,
        \end{equation}
        which converges to \(\nabla\Phi(u_n)\) as \(J\to\infty\).
        More generally, for smooth objectives,
        equation~\eqref{eq:finite_difference_identity} holds only approximately,
        with an error that is small when the perturbation scale~\(\sigma\) is small.
        This paper directly inspired the method proposed in~\cite{pavliotis21derivative}
        and studied further in the present work.
        In~\cite{pavliotis21derivative}, this idea was adapted to sampling by adding Brownian noise,
        and to the continuous-time setting by replacing the independent perturbations with fast Ornstein--Uhlenbeck processes.

    \item
        Later, Roy and coauthors~\cite{MR4411512} proposed a discrete-time scheme for sampling probability measures of the form~\eqref{eq:sampling_problem},
        named~\emph{Zeroth order Langevin Monte Carlo} (ZO-LMC),
        also based on the identity~\eqref{eq:stein_identity}.
        Their method assumes access not directly to the potential,
        but only to an unbiased estimator of it,
        and it does not exploit any composite structure of the potential.
        When the potential estimator has zero variance,
        ZO-LMC coincides with the discrete-time method we study in the present paper in the multiscale limit~\(\varepsilon\to 0\).
        The setting of our study, however, is quite different:
        while the authors of~\cite{MR4411512} focus on the log-concave setting,
        and provide estimates on the computational budget required to achieve a given accuracy in Wasserstein distance,
        we do not assume convexity of the potential,
        and derive error estimates in different metrics classically used in the numerical analysis of stochastic differential equations.
\end{itemize}

Aside from these works,
the method we analyze bears some resemblance to other approaches proposed in the literature.
We mention, in particular, the method proposed in~\cite{akyildiz} for maximum marginal likelihood estimation,
as well as the heterogeneous multiscale method (HMM)~\cite{MR2165382}.
Furthermore, there is a strong connection between the gradient approximation~\eqref{eq:finite_difference_identity} and sampling methods derived from the ensemble Kalman filter~\cite{MR4059375,MR4123680}.
We do not explore this connection in detail in this paper,
but refer the interested reader to~\cite{2018arXiv180508034H,pavliotis21derivative} for discussions on this topic.

\paragraph{Our contributions}%
\label{par:our_contributions}
The main contributions of this paper are the following:
\begin{itemize}
    \item
        In contrast to~\cite{pavliotis21derivative},
        we do not restrict attention to potentials \(\Phi\) arising from a specific Bayesian inverse-problem structure.
        In the spirit of~\cite{2018arXiv180508034H}, we focus on the general composite representation \(\Phi=\ell\circ G\) on \(\torus^d\),
        which also includes a general potential by taking \(\ell\) to be the identity.
        Bayesian inverse problems motivate this composite formulation.

    \item
        We strengthen the strong convergence results of~\cite{pavliotis21derivative}
        by tracking explicitly the dependence of the error on the number of particles,
        and we obtain a sharper estimate for the convergence of the discrete-time scheme.
        We also relate the multiscale method to the
        overdamped Langevin dynamics associated with a regularized potential,
        which motivates its use when the potential exhibits highly oscillatory behavior,
        a setting also considered in~\cite{DunbarEtal2022}.

    \item
        We prove the existence of an invariant distribution for the multiscale dynamics~\eqref{eq:the_multiscale_dynamics}
        and its discretization~\eqref{eq:the_multiscale_dynamics_discrete},
        both introduced in~\cref{sec:Mathematical setting},
        and quantify the error between this invariant distribution and the target distribution~\(\mu_\Phi\).

    \item
        We derive an estimate quantifying the error between
        ergodic averages computed from the multiscale dynamics and
        the corresponding averages with respect to the target distribution.
        In addition, we show uniform-in-time weak convergence of the slow process to the overdamped Langevin dynamics associated with~\(\Phi\),
        in the limit where the parameters of the method are refined.
        Together, these estimates provide a rigorous justification for the use of the method to approximate the target distribution~\(\mu_\Phi\).
\end{itemize}

\paragraph{Plan of the paper}%
\ifsiamart
    The rest of the paper is organized as follows.
    In~\cref{sec:Mathematical setting}, we describe the multiscale method and its time discretization.
    In~\cref{section:main_results}, we state the main results on invariant distributions, strong convergence, ergodic averages, and uniform-in-time weak convergence.
    The continuous-time proofs are given in~\cref{sec:proofs_continuous}, while the discrete-time proofs are given in~\sref{sec:proofs_discrete}.
    The supplementary material also contains the technical auxiliary results.
    \Cref{sec:conclusion} is reserved for conclusions and perspectives.
\else
The rest of the paper is organized as follows.
In~\cref{sec:Mathematical setting}, we describe the multiscale method and its time discretization.
In~\cref{section:main_results}, we state the main results of this work,
which concern existence and uniqueness of the invariant distribution,
strong convergence,
convergence of ergodic averages,
and uniform-in-time weak convergence.
The proofs are given separately for the continuous-time and discrete-time dynamics in~\cref{sec:proofs_continuous,sec:proofs_discrete},
respectively.
\Cref{sec:conclusion} is reserved for conclusions and perspectives.
The appendices collect the auxiliary results used in these proofs.
\fi

\paragraph{Statement on the use of artificial intelligence}%
\label{sub:Use of artificial intelligence}

In accordance with the recommendations included in the Leiden Declaration on Artificial Intelligence and Mathematics~\cite{alper_2026_20302944},
we disclose the use of automated tools as transparently as possible.
AI tools were not used until the final stages of the preparation of this manuscript.
In particular,
near-final complete proofs of the main results were written without them.
However,
we used AI in the final stages of writing,
where it was useful for the following purposes:
\begin{itemize}
    \item
        A good and near-complete first draft of the section~\cref{sub:notation} with the notations used in this paper was generated by AI.

    \item
        Use of AI was useful in~\sref{sec:results_drift}.
        We provided the statements of the results we needed to an AI agent,
        and it was able to write first drafts of the corresponding routine proofs.
        The proofs written were usually much too verbose and detailed,
        so we had to edit them to make them more concise and readable.

    \item
        Use of AI was also very useful to find references to related works.
        In particular,
        this is how we found the paper~\cite{MR4411512} on zeroth-order Langevin dynamics,
        which is closely related to the method we study in this paper.

    \item
        The main results in this work were proved without the use of AI,
        but AI tools were very useful to detect and fix small typos,
        and to homogenize notation.

    \item
        We asked AI to criticize the manuscript from the viewpoint of a referee.
        Specifically, we went through four iterations (two with free models offered through OpenCode,
        and two using paid models)
        of the following process:
        ask AI to provide an assessment of the manuscript is the form of a referee report,
        then address the issues raised by AI that we agreed with.
\end{itemize}

\paragraph{Acknowledgements}%
This work was partially supported by the French National Research Agency (ANR) through the IPSO (ANR-23-CE40-0027) project.
We also thank Grigorios Pavliotis, Sebastian Reich and Andrew Stuart for helpful discussions.

\section{The numerical method}%
\label{sec:Mathematical setting}

We now specify the numerical method and the mathematical setting used in the
analysis. We first describe the target measure and the multiscale dynamics,
together with its time discretization, in~\cref{sub:description_of_the_method}.
We then summarize the results of~\cite{pavliotis21derivative} in~\cref{sub:summary_previous_results},
and explain the motivation for retaining the time-scale parameter~$\varepsilon$ in~\cref{sub:why_epsilon}.
Finally, we collect the notation used throughout the paper in~\cref{sub:notation}.

\subsection{Description of the method}
\label{sub:description_of_the_method}

For simplicity, we formulate and analyze the method on the torus~\(\torus^d \coloneq \real^d / \Z^d\).
The same construction can in principle be applied on~$\real^d$,
but its analysis would be considerably more technical,
and so we do not pursue this direction in the present work.
We seek to generate approximate samples from the target measure~$\mu_\Phi$
on~$\torus^d$:
\begin{equation*}
    \mu_\Phi(\d x) = \frac{1}{Z_\Phi} \e^{-\Phi(x)} \, \d x \, ,
    \qquad
    Z_\Phi = \int_{\torus^d} \e^{-\Phi(x)} \, \d x \, .
\end{equation*}
Here~$\Phi = \ell \circ G$,
where $\ell \colon \real^K \to \real$ is a function that is easy to differentiate,
whereas $G \colon \torus^d \to \real^K$ is not.
We assume this composite structure for the potential~$\Phi$,
also considered in~\cite{2018arXiv180508034H},
in order to cover both arbitrary scalar potentials, obtained by taking
$K=1$, $G=\Phi$, and~$\ell$ the identity on~$\real$,
as well as Bayesian inverse problems,
as illustrated in the following example.
\begin{example}
    [Bayesian inverse problem]
    \label{example:bayesian_inverse_problem}
    Consider the finite-dimensional inverse problem consisting of estimating an unknown parameter~\(\theta \in \real^d\)
    from a partial and noisy observation
    \[
        y_0 = G_0(\theta) + \eta \, ,
    \]
    where~$\eta\sim\normal(0,\Gamma_0)$ is additive Gaussian noise
    and $G_0\colon\real^d\to\real^{K_0}$ is the forward operator.
    Suppose that the prior distribution is Gaussian,~$\theta \sim \normal(m, \Sigma)$.
    Then, applying Bayes' theorem,
    the posterior distribution of~\(\theta\) given~\(y_0\) has density
    given up to normalization by
    \[
        \pi_*(\theta) \propto \e^{-\frac{1}{2} \lvert y_0 - G_0(\theta) \rvert_{\Gamma_0}^2} \, \e^{- \frac{1}{2}\lvert \theta - m \rvert_{\Sigma}^2} \, ,
    \]
    which is the product of the likelihood and the prior.
    Here, for a positive definite matrix~\(A\),
    we used the notation~\(
    \abs{u}_A^2 = u^{\t} A^{-1} u.
    \)
    The posterior density can be rewritten in the form~\(\e^{-\ell \circ G}\),
    up to a multiplicative constant,
    by defining
    \[
        G(\theta) =
        \begin{pmatrix}
            G_0(\theta) \\
            \theta
        \end{pmatrix} \,
        \qquad
        \ell(z) = \frac{1}{2} \lvert z - y \rvert_{\Gamma}^2 \,,
        \qquad
        y = \begin{pmatrix}
            y_0 \\
            m
        \end{pmatrix} \, ,
        \qquad
        \Gamma =
        \begin{pmatrix}
            \Gamma_0 & 0 \\
            0 & \Sigma
        \end{pmatrix} \, .
    \]
    This posterior illustrates the composite structure of the posterior density, with~\(G\) being the forward operator and~\(\ell\) the loss function.
    Although it is naturally posed on~\(\real^d\),
    for simplicity the results we present below assume a periodic forward model on~\(\torus^d\).
    Thus, the Gaussian-prior example motivates the formulation but is not directly covered by the theorems proved here.
\end{example}

Throughout the rest of the paper,
we make the following implicit assumption on~$\ell$.
\begin{assumption}[Standing assumption]
    The function~$\ell \colon \real^K \to \real$ is smooth.
\end{assumption}

Before describing the multiscale dynamics and its time discretization,
we introduce the necessary notation.
We denote by~$(\Omega, \mathcal F, \proba)$ a probability space on which all the random variables and processes are defined.
On this probability space,
we fix $d$-dimensional independent Brownian motions~$\wslow$ and~$\wfast{j}$, for $j \in \{1, 2, \dotsc\}$.
We also fix~i.i.d.\ samples \(\xin{1}_0, \xin{2}_0, \dotsc \) with common distribution~\(\normal(0, \I_d)\),
independent of the Brownian motions,
as well as a sample~\(\theta_0 \sim \rho_0\in \mathcal P(\torus^d)\),
independent of all the Brownian motions and of the \(\xin{j}_0\)'s.
Here~$\rho_0$ is an arbitrary initial distribution for the slow variable~$\theta$,
which plays no role in the long-time behavior of the dynamics.

Having introduced the necessary notation,
we now describe the multiscale dynamics and its time discretization.
The main purpose of this paper is to analyze the following general multiscale dynamics on \(\torus^d \times \real^{dJ}\),
which is designed to generate approximate samples from the distribution on \(\torus^d\) with density proportional to~\(\e^{-\Phi}\):
\begin{equation}
    \tag{DFOL}
    \label{eq:the_multiscale_dynamics}
    \boxed{
        \begin{aligned}
            \d \theta_t &= -\frac{1}{J\sigma}\sum_{j=1}^J
            \Bigl\langle \nabla \ell\bigl(G(\theta_t)\bigr), G\bigl(\theta_t + \sigma \xin{j}_t\bigr) - G(\theta_t) \Bigr\rangle \, \xin{j}_t \, \d t + \sqrt{2} \, \d \wslow_t \, , \\
            \d \xin{j}_t &= - \frac{1}{\varepsilon} \xin{j}_t \, \d t + \sqrt{\frac{2}{\varepsilon}} \d \wfast{j}_t \,  \qquad j=1,\dots,J \, .
        \end{aligned}
    }
\end{equation}
In this dynamics, the parameters~$J \in \nat_{>0}$, $\varepsilon>0$, and~$\sigma>0$
denote, respectively, the number of auxiliary fast processes, the fast time scale,
and the perturbation scale.
We write for brevity~$\allxi_t = (\xin{1}_t, \dots, \xin{J}_t)$,
and denote the slow drift in~\eqref{eq:the_multiscale_dynamics} by
\begin{equation}
    \driftall(\theta, \allxi) \coloneq \frac{1}{J} \sum_{j=1}^{J}
    \driftone\bigl(\theta, \xin{j}\bigr) \, ,
    \qquad
    \driftone(\theta, \xi) \coloneq
    - \frac{1}{\sigma}
    \Bigl\langle \nabla \ell\bigl(G(\theta)\bigr),
    G\bigl(\theta + \sigma \xi\bigr) - G(\theta) \Bigr\rangle \, \xi \, .
    \label{eq:slow_drift_definition}
\end{equation}
For the multiscale SDE~\eqref{eq:the_multiscale_dynamics},
the drift coefficient in the~\(\theta\) variable is not a globally Lipschitz continuous function of $(\theta,\allxi)$,
but existence of a unique strong solution still holds by classical results based on a Lyapunov condition provided that~$G$ is Lipschitz continuous;
see, for example,~\cite[Theorem~3.5]{khasminskii2012stochastic}.
For the Lyapunov function,
it suffices, for example, to take~$V(\theta, \allxi) = 1 + \frac{1}{J} \sum_{j=1}^{J} \lvert \xin{j} \rvert^2$.
We note that when the map~$G$ is affine,
the slow drift in~\eqref{eq:the_multiscale_dynamics} is independent of~$\sigma$
and equals its zero-perturbation limiting drift, denoted by~$\driftlim$:
\[
    \driftlim(\theta, \allxi) = - \frac{1}{J} \sum_{j=1}^{J}
    \bigl( \xin{j} \otimes \xin{j} \bigr) \, \nabla \Phi(\theta) \, ,
\]
which motivates the name \emph{derivative-free overdamped Langevin dynamics} for~\eqref{eq:the_multiscale_dynamics}.
We also wish to understand the convergence of its time discretization with step~$h > 0$,
initialized at~$\dtheta{0} = \theta_0$ and~$\dxin{j}{0} = \xin{j}_0$ for $j \in \range{1}{J}$.
For the discrete scheme, we use independent standard Gaussian variables
$(\xslow{n})_{n\in\N}$ and $(\xfast{j}{n})_{j\in\range{1}{J},\,n\in\N}$,
independent of one another.
\begin{equation}
    \tag{EM-DFOL}
    \label{eq:the_multiscale_dynamics_discrete}
    \boxed{
        \begin{aligned}
            \dtheta{n+1} &= \dtheta{n} -\frac{h}{J\sigma}\sum_{j=1}^J \Bigl\langle \nabla \ell\bigl(G(\dtheta{n})\bigr), G\bigl(\dtheta{n} + \sigma \dxin{j}{n}\bigr) - G(\dtheta{n}) \Bigr\rangle \, \dxin{j}{n} + \sqrt{2h} \, \xslow{n}  \, , \\
            \dxin{j}{n+1} &= \e^{- \frac{h}{\varepsilon}} \dxin{j}{n}
            + \sqrt{1 - \e^{- \frac{2h}{\varepsilon}}}
            \xfast{j}{n} \,  \qquad j=1,\dots,J \, ,
        \end{aligned}
    }
\end{equation}
We write~$\dallxi{n} = (\dxin{1}{n},\dotsc,\dxin{J}{n})$.
The slow equation is updated by Euler--Maruyama,
whereas the fast update is the exact transition of the Ornstein--Uhlenbeck process.
Such an Euler-slow/exact-fast scheme, already proposed in~\cite{pavliotis21derivative},
has also been considered in the paper~\cite{MR4376295},
where its asymptotic-preserving properties are studied.
See also~\cite{brehier:hal-05287175},
where an asymptotic-preserving scheme is proposed for multiscale dynamics whose fast component is Brownian motion on the torus.
In the following example,
we make the dynamics explicit for the two particular choices of~$\ell$.

\begin{example}
    We make the dynamics~\eqref{eq:the_multiscale_dynamics}
    explicit for the two choices of~$\ell$ that are most relevant here.
    \begin{itemize}
    \item
        For a general potential, take~$\ell$ to be the identity.
        Then the $\theta$-dynamics in~\eqref{eq:the_multiscale_dynamics} reduces to
        \begin{equation}
            \label{eq:usual_dfol}
            \d \theta_t = -\frac{1}{J\sigma}\sum_{j=1}^J \Bigl(\Phi\bigl(\theta_t + \sigma \xin{j}_t\bigr) - \Phi(\theta_t)\Bigr) \xin{j}_t \, \d t + \sqrt{2} \, \d \wslow_t \, ,
        \end{equation}
        This dynamical system was proposed in~\cite{pavliotis21derivative} and
        can be viewed as a continuous-time analogue of the
        \emph{zeroth-order perturbed gradient descent} method~\cite{JinEtal2018},
        generalized to the sampling setting.

    \item
        For the quadratic loss~$\ell(z) = \frac{1}{2}\lvert z-y \rvert_{\Gamma}^2$,
        which arises in Bayesian inverse problems with Gaussian noise,
        the~$\theta$-dynamics in~\eqref{eq:the_multiscale_dynamics} becomes
        \begin{equation}
            \label{eq:enkf_dfol}
            \d \theta_t = -\frac{1}{J\sigma}\sum_{j=1}^J \Bigl\langle G(\theta_t) - y, G(\theta_t + \sigma \xin{j}_t) - G(\theta_t) \Bigr\rangle_\Gamma \, \xin{j}_t \, \d t + \sqrt{2} \, \d \wslow_t \, .
        \end{equation}
        This dynamics,
        which was the focus of~\cite{pavliotis21derivative},
        can be viewed as a continuous-time analogue of the method proposed in~\cite{2018arXiv180508034H},
        generalized to the sampling setting.
        The drift coefficient in this case has the same structure as
        that of the ensemble Kalman inversion and sampling dynamics~\cite{MR3041539,MR4059375}.
    \end{itemize}
\end{example}

All theorem statements below concern the general composite DFOL dynamics~\eqref{eq:the_multiscale_dynamics}
and its discretization~\eqref{eq:the_multiscale_dynamics_discrete}.
In order to present the main results of this paper,
it will also be convenient to introduce the solution~$(\vartheta_t)_{t \geq 0}$ of the overdamped Langevin dynamics associated with the potential~$\Phi$,
driven by the same Brownian motion~$\wslow$ as in the slow dynamics of~\eqref{eq:the_multiscale_dynamics},
with initial condition~$\vartheta_0 = \theta_0$:
\begin{equation}
    \label{eq:overdamped_langevin_dynamics}
    \boxed{
    \d \vartheta_t = -\nabla \Phi(\vartheta_t) \, \d t + \sqrt{2} \, \d \wslow_t \, .
    }
    \tag{OL}
\end{equation}
Its Euler--Maruyama discretization with step~$h>0$,
initialized at \(\dvartheta{0}=\theta_0\) and
using the same Gaussian variables as~\eqref{eq:the_multiscale_dynamics_discrete},
is given by
\begin{equation}
    \label{eq:em_ol}
    \boxed{
    \dvartheta{n+1} = \dvartheta{n} - h \nabla\Phi(\dvartheta{n}) + \sqrt{2h} \, \xslow{n} \, .
    }
    \tag{EM-OL}
\end{equation}
We are now in a position to describe our contributions more precisely:
our goal is to obtain error estimates
in metrics that are specified in~\cref{section:main_results},
quantifying the difference between the slow dynamics in~\eqref{eq:the_multiscale_dynamics} and the overdamped Langevin dynamics~\eqref{eq:overdamped_langevin_dynamics},
as well as the difference between the time discretizations~\eqref{eq:the_multiscale_dynamics_discrete} and~\eqref{eq:em_ol}.

\subsection{Summary of previous results}
\label{sub:summary_previous_results}

We summarize here the results of~\cite{pavliotis21derivative} most relevant to the present work.
The method studied in that reference corresponds to the particular case where~$\ell(z) = \frac12\abs{z-y}_{\Gamma}^2$ in the setting above.
For that case,
the authors proved,
under appropriate smoothness assumptions on~$G$,
a strong error estimate between the multiscale dynamics~\eqref{eq:the_multiscale_dynamics} and the overdamped Langevin dynamics,
for a constant~$C_{p,T,J} > 0$ depending on~$p$, $T$ and~$J$:
\begin{equation*}
    \expect\left[\sup_{0 \leq t \leq T}
    \lvert \theta_t-\vartheta_t\rvert^p\right]
    \leq C_{p,T,J}\left(\varepsilon^{\frac p2}+\sigma^p\right) \, ,
    \qquad p\geq 1 \, .
\end{equation*}
For the Euler-slow/exact-fast discretization~\eqref{eq:the_multiscale_dynamics_discrete},
the corresponding estimate was established only for~$p=2$ and in a weaker metric:
\begin{equation*}
    \sup_{0 \leq n \leq T/h}
    \expect\Bigl[\lvert \dtheta{n}-\dvartheta{n}\rvert^2\Bigr]
    \leq C_{T,J}\left(h+\sigma^2
    +\varepsilon\log\left(1+\frac{1}{\sqrt{\varepsilon}}\right)\right) \, .
\end{equation*}
We improve these results in several directions, as detailed in~\cref{section:main_results}.

\subsection{Why we study the \texorpdfstring{method with $\varepsilon > 0$}{multiscale method}}
\label{sub:why_epsilon}
In practice, only discrete-time algorithms can be implemented on a computer,
and for the discrete-time algorithms we consider,
it is unclear whether choosing~$\varepsilon > 0$ yields any practical advantage
over sampling the auxiliary random variables~$(\xin{j}_n)_{j\in\range{1}{J}}$ independently at each iteration~$n$.
We therefore explain why we analyze the multiscale dynamics with~$\varepsilon>0$.
There are two main motivations for introducing the parameter~$\varepsilon$:
\begin{itemize}
\item
    The main reason is that, for fixed~$\varepsilon>0$,
    the discrete-time multiscale dynamics admits a clear continuous-time limit
    -- a multiscale stochastic differential equation (SDE) --
    as the time step tends to zero.
    This is very useful for the mathematical analysis.
    Indeed, as we shall see in the next sections,
    the discrete-time analysis closely follows the steps of the continuous-time analysis,
    with additional technicalities related to the discretization.
    The analysis of the continuous-time dynamics is also much simpler, both for technical reasons (notably because Itô's formula is available in continuous time)
    and because it allows us to employ techniques and results from the wider averaging and homogenization literature,
    which is mostly focused on continuous-time dynamics.

\item
    The second reason is that the multiscale dynamics with~$\varepsilon>0$ can be viewed as a variant
    of the \emph{heterogeneous multiscale method} (HMM) for multiscale SDEs~\cite{MR2165382},
    the analysis of which is of independent interest.
    This connection allows us to compare our results to the existing literature on HMM,
    and gain confidence in their optimality in terms of scaling with respect to the parameters.
    Conversely, some of the techniques we develop may be of interest for the analysis of HMM and related methods.
    For instance, to our knowledge,
    there does not exist in the HMM literature any result on the bias and mean squared error of ergodic estimators,
    of the type we prove in~\cref{sub:invariant_distribution_continuous_proof}.
\end{itemize}

We emphasize that, in the analysis presented in this paper,
we never make use of the estimates obtained for the continuous-time dynamics in the analysis of the discrete-time dynamics.
Such an approach would be feasible but not straightforward,
as it would require to control the error between the discrete-time dynamics and its continuous-time limit uniformly in the parameters~$\varepsilon,\sigma$,
which presents a similar level of difficulty as the direct analysis of the discrete-time dynamics.
Rather, the role of the estimates obtained for the continuous-time dynamics is to provide intuition and guidance for the analysis of the discrete-time dynamics.
Furthermore, for readers familiar with the averaging and homogenization literature,
the analysis of the continuous-time method may be easier to follow,
whilst still exhibiting the main ideas of the analysis.

\subsection{Notation}
\label{sub:notation}

All random variables considered in this article are defined on the probability space \((\Omega,\mathcal F,\proba)\).
For any metric space \(\mathcal X\), we denote by~\(\mathcal B(\mathcal X)\) its Borel \(\sigma\)-algebra,
and by \(\mathcal P(\mathcal X)\) the set of probability measures on the measurable space \((\mathcal X,\mathcal B(\mathcal X))\).
For any Borel measure~\(\mu\) on \(\mathcal X\) and any \(\mu\)-integrable function \(f\colon\mathcal X\to\real\),
we write
\[
    \mu(f)\coloneq\int_{\mathcal X}f(x)\,\mu(\d x) \, .
\]

\paragraph{Vectors and matrices}%
For two column vectors \(u, v \in \real^d\),
we sometimes write \(u \otimes v = u v^{\t}\).
For two matrices~\(A,B\in \real^{d\times d}\),
the matrix operation \(A:B = {\rm trace}{\left(A^{\textsf{T}} B\right)}\) denotes the Frobenius inner product.
For a scalar field~\(\phi\colon \real^d\to \real\),
its gradient is denoted by \(\nabla\phi(\theta)\).
More generally, for a vector field~\(v\colon \real^d\to \real^m\),
\(\nabla v(\theta)\) denotes its Jacobian matrix, with entry on row~\(i\) and column~\(j\) given by~\(\partial v_i / \partial\theta_j\).
If \(\phi\) admits derivatives up to order \(p\),
the notation~\(\nabla^p\phi(\theta)\) refers to the tensor of order \(p\) containing all partial derivatives of order \(p\).
The Laplacian of a vector field~$v$ is computed component-wise.
For a multilinear map~\(A \colon \real^{m_1} \times \dots \times \real^{m_k} \to \real^n\),
we denote by~\(|A|\) its operator norm for the Euclidean norm.

For a \(3\)-rd order tensor \(T\in \real^{d\times d\times d}\) and a matrix \(A\in \real^{d\times d}\),
we use \(T:A \in \real^d\) for the contraction of the tensor \(T\) with the matrix \(A\) along the last two indices of \(T\).
More generally, we use \(\cdot\) to denote the full contraction of tensors of the same order.

\paragraph{Functional spaces}%
For an integer~$k \geq 0$ and metric spaces~$\mathcal X$ and~$\mathcal Y$,
we denote by~$\cont^k(\mathcal X, \mathcal Y)$ the space of $k$-times
continuously differentiable functions from~$\mathcal X$ to~$\mathcal Y$.
When the codomain is~$\real$, it is omitted:
$\cont^k(\mathcal X) \coloneq \cont^k(\mathcal X, \real)$.
For~$k = 0$ we simply write~$\cont(\mathcal X)$.
The associated norm is
\[
    \lVert f \rVert_{\cont^k}
    \coloneq \max_{|\alpha| \leq k} \lVert \partial^\alpha f \rVert_\infty \, .
\]
For a time interval~$I$, the notation~$u\in\cont^{1,m}(I\times\torus^d)$ means that~$\partial_tu$ and all spatial derivatives~$\partial_\theta^\alpha u$ with~$|\alpha|\leq m$ are continuous on~$I\times\torus^d$.
This notation does not require mixed time--space derivatives.
For~$k \geq 0$ and~$\alpha \in (0,1)$,
the notation $\cont^{k+\alpha}(\mathcal X)$ refers to the H\"older space of functions in~$\cont^k(\mathcal X)$
whose $k$-th order derivatives are $\alpha$-H\"older continuous.
We equip this space with the norm
\[
    \lVert f \rVert_{\cont^{k+\alpha}}
    \coloneq \lVert f \rVert_{\cont^k}
    + \max_{|\beta| = k}
    \sup_{x \neq y} \frac{|\partial^{\beta} f(x) - \partial^{\beta} f(y)|}{\operatorname{dist}_{\mathcal X}(x,y)^\alpha} \, ,
\]
where~$\operatorname{dist}_{\mathcal X}$ is the Euclidean distance or the periodic distance on~$\torus^d$, as appropriate.
For~$k \geq 0$,
$\sobolev^{k,\infty}(\mathcal X)$ denotes the Sobolev space of functions whose
weak derivatives up to order~$k$ are essentially bounded,
with norm
\[
    \lVert f \rVert_{\sobolev^{k,\infty}}
    \coloneq \sum_{|\alpha| \leq k} \lVert \partial^\alpha f \rVert_{L^\infty} \, .
\]
We use the shorthand~$\lVert f \rVert_{k,\infty} \coloneq \lVert f \rVert_{\sobolev^{k,\infty}}$.

For a measure~$\pi$,
we define the weighted Sobolev space~$H^i(\pi)$ as
the subspace of \(L^2(\pi)\) of functions whose derivatives up to order \(i\) are in \(L^2(\pi)\).
The associated norm is given by
\[
    \norm{u}_i^2 = \sum_{|\alpha| \leq i} \norm*{\partial^{\alpha} u}^2 \, ,
\]
where we use standard multi-index notation and~$\norm{\dummy}$ denotes the norm
of~\(L^2(\pi)\).
We also define the mean-zero spaces~\(H^{i}_0(\pi) = H^i(\pi) \cap L^2_0(\pi)\),
where \(L^2_0(\pi) = \{u \in L^2(\pi): \int u \, \d \pi = 0\}\),
and we recall that a probability measure~\(\pi\) is said to satisfy the Poincaré inequality with constant \(R\) if
\begin{equation}
    \label{eq:poincare}
    \forall f \in H^1_0(\pi), \qquad
    \norm*{f}^2 \leq \frac{1}{R} \norm{\grad f}^2 \, .
\end{equation}
For a function \(V \colon \real^d \to \real\),
the notation \(L^{\infty}_V\) refers to the Banach space of functions such that
\(\lVert \varphi \rVert_{L^{\infty}_V} = \left\lVert \frac{\varphi}{V} \right\rVert_{\infty} < +\infty\).
Finally, for~$m \geq 0$,
we define the weight function~$\weight{m} \colon \real^d \to \real$
by~$\weight{m}(\xi) \coloneq 1 + |\xi|^m$.
Thus~$L^\infty_{\weight{m}}$ denotes the Banach space of measurable
functions~$\varphi \colon \real^d \to \real$ with finite weighted norm
\begin{equation}
    \label{eq:def_weighted_L_infty_space}
    \lVert \varphi \rVert_{L^\infty_{\weight{m}}}
    \coloneq \sup_{\xi \in \real^d}
    \frac{|\varphi(\xi)|}{\weight{m}(\xi)} < +\infty \, .
\end{equation}
We also denote by~$L^\infty_{\weight{m},0} \subset L^\infty_{\weight{m}}$ the closed
subspace of functions with mean zero under~$\gamma_d$,
where~$\gamma_d$ denotes the standard Gaussian distribution on \(\real^d\)
with mean zero and covariance matrix~\(\I_d\).
For~$k\in\N$ and~$m\geq0$, we define the weighted Sobolev space~$\sobolev^{k,\infty}_{\weight{m}}$
as the space of functions~$\varphi\colon\real^d\to\real$ such that
\begin{equation}
    \label{eq:weighted_sobolev_norm}
    \norm{\varphi}_{\sobolev^{k,\infty}_{\weight{m}}}
    \coloneq\sum_{|\alpha|\leq k}\norm{D^\alpha\varphi}_{L^\infty_{\weight{m}}}<+\infty \, .
\end{equation}

\paragraph{Torus}%
We identify the torus~$\torus^d=\real^d/\Z^d$ with the unit cube \([0,1)^d\) with periodic boundary conditions.
For any norm~$\abs{\dummy}$ on~$\real^d$,
the associated periodic distance is
\begin{equation*}
    \abs*{\dummy}_\torus = \inf_{k\in \Z^d}\abs*{\dummy - k} \, .
\end{equation*}

\paragraph{Asymptotic notation}%
We write~$a \lesssim b$ to mean~$a \leq C b$ for some constant~$C > 0$
independent of the parameters~$(\sigma, \varepsilon, h, J)$.
Subscripts on~$\lesssim$ indicate additional parameters on which the
constant may depend;
for instance,~$a \lesssim_{p,d} b$ means~$a \leq C(p,d)\,b$.
The corresponding Landau big-O notation is denoted~$\bigo(\placeholder)$.

\section{Main results}
\label{section:main_results}

In this section,
we present and comment on the main results of this work.
We first state the existence and uniqueness of an invariant distribution in~\cref{sub:existence_uniqueness_invariant}.
In each of the three subsequent subsections, we present results in both continuous and discrete time:
strong convergence in~\cref{sub:strong_convergence_of_the_numerical_scheme},
estimates on the invariant distribution and ergodic averages in~\cref{sub:invariant_distribution},
and uniform-in-time weak convergence in~\cref{sec:weak_error}.
We only give the statements here;
the proofs in continuous and discrete time are provided in~\cref{sec:proofs_continuous} and~\sref{sec:proofs_discrete}, respectively.

The regularity assumptions used in the principal results are summarized in~\cref{table:regularity_assumptions}.
\begin{table}[H]
\centering
\begin{tabular}{p{0.66\linewidth}p{0.13\linewidth}p{0.13\linewidth}}
Result & $G$ & Observable \\ \hline
Strong convergence, continuous time, fixed $\sigma$ (\cref{theorem:robustness}) & $\cont^4$ & -- \\
Strong convergence, discrete time, fixed $\sigma$ (\cref{theorem:strong_convergence_fixed_sigma_discrete}) & $\cont^5$ & -- \\
Invariant distribution and ergodic averages, continuous time (\cref{theorem:invariant_continuous}) & $\cont^{3+\beta}$ & $\sobolev^{2,\infty}$ \\
Invariant distribution and ergodic averages, discrete time (\cref{theorem:bound_invariant_discrete}) & $\cont^{3+\alpha+\beta}$ & $\cont^{2+\alpha}$ \\
Uniform-in-time weak convergence, continuous time (\cref{thm:uniform_continuous_weak_convergence}) & $\cont^4$ & $\cont^3$ \\
Uniform-in-time weak convergence, discrete time (\cref{thm:uniform_discrete_weak_convergence}) & $\cont^6$ & $\cont^5$
\end{tabular}
\caption{Regularity assumptions on the forward model~$G$ and, where relevant, on the observable~$f$, required by the principal results of the paper.}
\label{table:regularity_assumptions}
\end{table}

\subsection{Existence and uniqueness of an invariant distribution for the multiscale dynamics}
\label{sub:existence_uniqueness_invariant}

There are several classical approaches to prove the existence and uniqueness of an invariant measure.
In our setting, since the state space is compact in the $\theta$ variable,
since the noise is nondegenerate,
and since~$\allxi$ is an Ornstein--Uhlenbeck process in~$\real^{dJ}$,
the result follows from off-the-shelf results.

\begin{lemma}
    [Existence and uniqueness of an invariant measure]
    \label{lemma:existence_uniqueness_invariant}
    For any~$(\varepsilon, \sigma) \in \real_{>0} \times \real_{>0}$,
    there exists a unique invariant probability distribution \(\fullinvariant \in \mathcal P(\torus^d \times \real^{dJ})\)
    for the multiscale dynamics~\eqref{eq:the_multiscale_dynamics}.
    Furthermore, for any~$(\varepsilon, \sigma, h) \in \real_{>0} \times \real_{>0} \times \real_{>0}$,
    the discrete-time dynamics~\eqref{eq:the_multiscale_dynamics_discrete} also admits a unique invariant probability distribution \(\fullinvariantdiscrete \in \mathcal P(\torus^d \times \real^{dJ})\).
\end{lemma}

\ifsiamart
The proof of the discrete-time assertion is given in~\sref{sub:proof_existence_uniqueness_invariant_discrete}.
\fi

\begin{remark}
    Since the fast dynamics is autonomous and admits~$\fastgaussianall$ as its unique invariant distribution,
    the fast marginal of both invariant distributions in the lemma is~$\fastgaussianall$.
\end{remark}

\subsection{Strong convergence}
\label{sub:strong_convergence_of_the_numerical_scheme}

In this section, we present strong convergence results that improve upon those in~\cite[Theorem 2.1]{pavliotis21derivative} in several ways:
we precisely quantify the dependence of the error on the number of auxiliary processes~\(J\),
we extend the error estimate for the discrete-time method from~$L^2$ to~$L^p$,
and we improve its dependence on~$\varepsilon$.
We present two types of strong convergence results.
\begin{itemize}
    \item
        The first type, in~\cref{theorem:robustness,theorem:strong_convergence_fixed_sigma_discrete},
        compares the multiscale dynamics to a smoothed dynamics with drift~\(F_\sigma^\infty\) defined below in~\eqref{eq:asymptotic_drift}, for a fixed~\(\sigma\).

    \item
        The second type,
        in~\cref{corollary:strong_convergence_continuous,corollary:strong_convergence_discrete},
        builds on the first type and adds an error estimate in~\(\sigma\) (due to the difference between \(F_\sigma^\infty\) and the asymptotic drift \(-\nabla\Phi\))
        to compare the multiscale dynamics with the overdamped Langevin dynamics~\eqref{eq:overdamped_langevin_dynamics}.
\end{itemize}
In both cases, results are first stated in continuous time, and then in discrete time.
We use the slow-drift notation introduced in~\eqref{eq:slow_drift_definition}.
We begin with the continuous-time result for a fixed~\(\sigma\).
\begin{theorem}
  [Strong convergence for a fixed \(\sigma\)]
  \label{theorem:robustness}
  Let \(p\geq 2\), $0<\varepsilon\leq1$, and assume that \(G \in \cont^4\left(\torus^d,\real^K\right)\).
  Consider the strong solution \((\theta_t, \allxi_t)_{t \geq 0}\) to~\eqref{eq:the_multiscale_dynamics} with initial condition~\((\theta_0, \allxi_0)\).
  Consider also the strong solution \((\vvartheta_t)_{t \geq 0}\) to the following stochastic differential equation with initial condition~\(\vvartheta_0 = \theta_0\):
  \begin{equation*}
      \vvartheta_t = \theta_0
      +\int_{0}^{t} F^\infty_\sigma(\vvartheta_s) \, \d s
      + \sqrt{2} \, \wslow_t \, ,
  \end{equation*}
  where $(\wslow_t)_{t \geq 0}$ is the same Brownian motion as in~\eqref{eq:the_multiscale_dynamics} and the drift \(F^\infty_\sigma\) is given by
  \begin{equation}
      \label{eq:asymptotic_drift}
      \driftinf(\theta)
      \coloneq \int_{\real^d} \driftone(\theta, \xi) \, \d \gamma_{d}(\xi) \, .
  \end{equation}
  Then there exists a constant~\(C_{T,p,d} > 0\), depending on fixed derivative bounds on~$G$ and~$\ell$ but independent of~\((\rho_0,\sigma,\varepsilon,J)\), such that
  \[
      \expect\left[ \sup_{t \in [0, T]} \left|\theta_t-\vvartheta_t\right|^p \right]
      \leq C_{T,p,d}\frac{\varepsilon^{\frac{p}{2}}}{J^{\frac{p}{2}}} \, .
  \]
\end{theorem}
\begin{proof}
    See~\cref{subsub:strong_convergence_continuous}.
\end{proof}
\begin{remark}
  This error estimate is quite satisfying,
  as it recovers the usual~\(\varepsilon^{\frac{p}{2}}\) scaling of averaging estimates for fixed~\(J\),
  see~\cite{MR2382139},
  and the usual Monte Carlo scaling as~\(J^{-\frac{p}{2}}\) for fixed~\(\varepsilon\).
\end{remark}

As a main consequence of~\cref{theorem:robustness},
we deduce the following strong convergence result,
which quantifies the distance between the multiscale dynamics and the overdamped Langevin dynamics~\eqref{eq:overdamped_langevin_dynamics}.
\begin{corollary}
  [Strong convergence in continuous time]
  \label{corollary:strong_convergence_continuous}
  Let~$(\vartheta_t)_{t\geq 0}$ be the solution to the overdamped Langevin dynamics~\eqref{eq:overdamped_langevin_dynamics},
  with the same initial condition~$\vartheta_0 = \theta_0$ as for the slow process in~\eqref{eq:the_multiscale_dynamics}.
  Then, under the assumptions of~\cref{theorem:robustness},
  for all \(T\geq 0\) and for all \(p\geq 1\), there exists a constant \(C_{T,p,d} > 0\) independent of~\((\sigma, \varepsilon,J)\) such that
  \begin{equation}
      \expect\left[\sup_{0\leq t\leq
      T}\left|\theta_t-\vartheta_t\right|^p\right]\leq
      C_{T,p,d}\left(\sigma^{2p} +\frac{\varepsilon^{\frac{p}{2}}}{J^{\frac{p}{2}}}\right) \, .
  \end{equation}
\end{corollary}
\begin{proof}
    See~\cref{subsub:strong_convergence_small_sigma}.
\end{proof}

\begin{remark}
    We note that the regularity required on~$G$ in the strong convergence result in~\cref{corollary:strong_convergence_continuous} is higher than
    the regularity required in~\cite[Theorem 2.1]{pavliotis21derivative}.
    This regularity requirement arises because, in the proof of~\cref{theorem:robustness},
    the asymptotic drift~$F_\sigma^\infty$ does not enjoy a product structure;
    this function cannot be expressed as the product of a function of~$\theta$ and a function of~$\xi$.
    Specifically, the assumption allows us to control the right-hand side of~\eqref{eq:fourth_term_second_equation} uniformly in~$\sigma$ in the proof.
\end{remark}

To further illustrate the independent interest of~\cref{theorem:robustness},
we consider the setting where $\ell(y) = y$ and the function~$G \colon \torus^d \to \real$ is given by
\begin{equation}
    \label{eq:multiscale_forward_model}
    G(x) = \Phi_0(x) + \Phi_1\Bigl(\frac{x}{\delta}\Bigr) \, ,
\end{equation}
for a smooth function~$\Phi_0 \colon \torus^d \to \real$,
a smooth $1$-periodic function~$\Phi_1 \colon \real^d \to \real$ such that \(\int_{[0,1]^d} \Phi_1(x) \, \d x = 0\),
and for a small parameter \(\delta>0\) satisfying~$\delta^{-1}\in\nat_{>0}$.
It is well known; see, for example,~\cite{MR4569061},
that the overdamped Langevin dynamics~\eqref{eq:overdamped_langevin_dynamics} associated with the potential~$\Phi = \Phi_0 + \Phi_1(\placeholder/\delta)$
converges weakly in~$\cont\bigl([0, T],\torus^d\bigr)$,
in the limit $\delta \to 0$, to the following dynamics with multiplicative noise:
\[
    \d \vvartheta^0_t = - \mathcal  M \nabla \Phi_0(\vvartheta^0_t) \, \d t + \sqrt{2 \mathcal M} \, \d \wslow_t \, ,
\]
with an effective diffusion tensor~$\mathcal M$ that satisfies $\mathcal M \leq \I_d$ in the sense of symmetric matrices;
see~\cite[Proposition~2.5]{MR4569061} for a proof of this bound in the more general setting of an~$N$-scale potential.
Since~$\mathcal M \leq \I_d$, the presence of high-frequency oscillations in~$G$ leads to slower convergence of the dynamics towards the target distribution.
In contrast,
as we now prove,
the solution to the multiscale dynamics~\eqref{eq:the_multiscale_dynamics} behaves as~$\delta \to 0$ as if~$\Phi_1 \equiv 0$;
the high-frequency oscillations on top of~$\Phi_0$ are averaged out and ignored in the limit.
This property is often a desirable property in practice,
as rapid oscillations in the potential are viewed as noise.
A similar observation was made in~\cite{DunbarEtal2022} for the ensemble Kalman sampler,
an interacting particle system for sampling.
We prove this result as another simple consequence of~\cref{theorem:robustness}.
This result provides interesting information on the limiting dynamics obtained when~$\varepsilon/J \to 0$,
in the limit when~$\delta \to 0$.
\begin{corollary}
    \label{corollary:robustness_homogenization}
    Let~$(\theta_t, \allxi_t)_{t \geq 0}$ be the solution to the multiscale dynamics~\eqref{eq:the_multiscale_dynamics} with~$\ell \colon \real \to \real$ the identity function
    and function~$G \colon \torus^d \to \real$ given by~\eqref{eq:multiscale_forward_model}.
    For fixed~$\sigma>0$, set~$\Phi_{0,\sigma} \coloneq \Phi_0\star\normal(0,\sigma^2\I_d)$,
    and let~$(v^\sigma_t)_{t \geq 0}$ be the solution of
    \[
        v^\sigma_t
        = \theta_0 - \int_0^t \nabla\Phi_{0,\sigma}(v^\sigma_s) \, \d s
        + \sqrt{2} \, \wslow_t \, ,
    \]
    where~$(\wslow_t)_{t\geq 0}$ is the Brownian motion in~\eqref{eq:the_multiscale_dynamics},
    which drives the slow component of~\eqref{eq:the_multiscale_dynamics}.
    Let~$(\vartheta^0_t)_{t \geq 0}$ be the solution to the overdamped Langevin dynamics associated with the potential~$\Phi_0$,
    with the same initial condition~$\vartheta^0_0 = \theta_0$ as for the slow process in~\eqref{eq:the_multiscale_dynamics} and the same driving Brownian motion:
    \[
        \vartheta^0_t = \theta_0 - \int_{0}^{t} \nabla \Phi_0(\vartheta^0_s) \, \d s + \sqrt{2} \, \wslow_t \, .
    \]
    Then, for any~$T > 0$ and $p \geq 2$,
    it holds that
    \begin{equation}
        \lim_{\delta\to0}\lim_{\varepsilon/J\to0}
        \left(\expect\biggl[\sup_{0\leq t\leq T}
        \left|\theta_t-v^\sigma_t\right|^p\biggr]\right)^{\frac{1}{p}}=0 \, ,
        \label{eq:robustness_homogenization}
    \end{equation}
    where~$\delta^{-1}\in\nat_{>0}$ and~$0<\varepsilon\leq1$ along the limit.
    Furthermore, there exists a constant~$C_{T,p}>0$, independent of~$\sigma$, such that
    \[
        \left(\expect\biggl[\sup_{0\leq t\leq T}
        \left|v^\sigma_t-\vartheta^0_t\right|^p\biggr]\right)^{\frac{1}{p}}
        \leq C_{T,p}\sigma^2 \, .
    \]
\end{corollary}
\begin{proof}
    See~\cref{subsub:homogenization_proof}.
\end{proof}

We now turn to the discrete-time analogue of \cref{theorem:robustness,corollary:strong_convergence_continuous}.
As is often the case in numerical analysis, the discrete-time result requires a slightly higher regularity assumption on the forward model~$G$.
The specific place where this assumption is used is mentioned explicitly in the proof.

\begin{theorem}
    [Strong convergence for fixed~\(\sigma\) in discrete time]
    \label{theorem:strong_convergence_fixed_sigma_discrete}
    Assume that \(G\in \cont^5\left(\torus^d, \real^K\right)\).
    Let \(J\in \nat_{>0}\) and~\(\sigma, \varepsilon, h \in (0, 1]\).
    Consider the process \(\left(\dtheta{n},\dallxi{n}\right)_{n\in \N}\)
    defined by~\eqref{eq:the_multiscale_dynamics_discrete},
    and the process \((\dvvartheta{n})_{n\in \N}\) defined by
    \begin{equation}
        \label{eq:discrete_asymptotic_dynamics}
        \dvvartheta{n+1} = \dvvartheta{n} + h \driftinf(\dvvartheta{n}) + \sqrt{2h}\, \xslow{n} \, ,
    \end{equation}
    with the same initial state~$\dtheta{0}=\dvvartheta{0}$ and the same Gaussian variables~$(\xslow{n})_{n\in\N}$.
    Then for all~$T >0$ and $p \geq 2$, there exists a constant $C_{T,p,d} > 0$, 
    depending on fixed derivative bounds on~$G$ and~$\ell$ but independent of~$(\rho_0,\sigma,\varepsilon,h,J)$, such that
    \begin{equation}
        \label{eq:strong_convergence_fixed_sigma}
        \max_{0 \leq n \leq T/h} \expect \bigl[ \lvert \dtheta{n} - \dvvartheta{n} \rvert^p \bigr]
        \leq C_{T,p,d} \, \frac{\varepsilon^{\frac{p}{2}} + h^{\frac{p}{2}}}{J^{\frac{p}{2}}} \, .
    \end{equation}
\end{theorem}
\begin{proof}
    See~\sref{subsub:strong_convergence_discrete}.
\end{proof}

Just as for~\cref{corollary:strong_convergence_continuous},
we can deduce the following corollary from~\cref{theorem:strong_convergence_fixed_sigma_discrete} in the same way.
We omit the proof, as it follows exactly the same lines as the one of~\cref{corollary:strong_convergence_continuous}.
\begin{corollary}
  [Strong convergence in discrete time]
  \label{corollary:strong_convergence_discrete}
  Let \(G\in \cont^5\left(\torus^d, \real^K\right)\),
  \(J\in \nat_{>0}\), and \((\sigma, \varepsilon, h) \in (0, 1]\).
  Consider the process \(\left(\dtheta{n},\dallxi{n}\right)_{n\in \N}\)
  defined by~\eqref{eq:the_multiscale_dynamics_discrete}
  and the process \((\dvartheta{n})_{n\in \N}\) defined by~\eqref{eq:em_ol},
  with the same initial condition~$\dvartheta{0} = \dtheta{0}$ and the same Gaussian variables~$(\xslow{n})_{n\in\N}$.
  Then for all \(T\geq 0\) and \(p\geq 1\), there exists a constant
  \(C_{T,p,d} > 0\) independent of \((\rho_0,\sigma,\varepsilon,h,J)\) such that
  \begin{equation}
      \expect\left[\max_{0\leq n\leq
      \frac{T}{h}}\left|\dtheta{n}-\dvartheta{n}\right|^p\right]\leq
      C_{T,p,d}\left(\sigma^{2p} +\frac{\varepsilon^{\frac{p}{2}}+h^{\frac{p}{2}}}{J^{\frac{p}{2}}}\right) \, .
  \end{equation}
\end{corollary}

\subsection{Estimate on the invariant distribution and convergence of time averages}
\label{sub:invariant_distribution}

Having proved the existence of an invariant distribution in~\cref{sub:existence_uniqueness_invariant}, we address here the
issue of bounding the distance,
in an appropriate metric,
between this distribution and the target distribution.
To this end, we use the approach pioneered
in~\cite{MattinglyStuartTretyakov2010},
where the effects of numerical discretization on the invariant distribution,
the bias, and the mean-square error of ergodic averages are investigated.
The technique introduced in that paper is remarkably versatile,
and can be employed to study errors introduced by other types of approximation.

We note that a bound on the distance between the invariant distribution and the target distribution can be deduced from the uniform-in-time weak error estimates proved in~\cref{sec:weak_error_continuous_proof}.
However, the former can be proved under milder assumptions than the latter,
and using simpler mathematical techniques.
Indeed, while uniform-in-time weak error estimates require to study the solution of an appropriate parabolic PDE,
an estimate on the invariant distribution can be established solely by studying an elliptic Poisson equation.
The proofs of the results in this subsection are given in~\cref{sub:invariant_distribution_continuous_proof} and~\sref{sub:bound_invariant_discrete}.

We first present the continuous-time estimates.
Before stating the result, we introduce some notation.
For an observable~\(f \colon \torus^d \to \real\), we define the ergodic estimator
\begin{equation}
    \label{eq:ergodic_estimator}
    \widehat f_T \coloneq \frac{1}{T} \int_{0}^{T} f(\theta_t) \, \d t \, ,
\end{equation}
which provides an approximation of~\(\mu_{\Phi}(f)\).
Let also \(\fullinvariant\) denote the unique invariant distribution of the multiscale dynamics~\eqref{eq:the_multiscale_dynamics},
and let \(\fullinvarianttheta \in \mathcal P(\torus^d)\) be its marginal with respect to~\(\theta\).

\begin{theorem}
    \label{theorem:invariant_continuous}
    Consider the multiscale dynamics~\eqref{eq:the_multiscale_dynamics}
    with initial condition~$(\theta_0, \allxi_0) \sim \pi_0 \in \Pi(m_0, \fastgaussianall)$,
    for any probability measure $m_0 \in \mathcal P(\torus^d)$ and where~$\Pi(m_0, \fastgaussianall)$ denotes the set of couplings between~$m_0$ and~$\fastgaussianall$.
    Assume the forward model satisfies~\(G\in \cont^{3+\beta}(\torus^d,\R^K)\),
    for some \(\beta \in \{0, 1\}\),
    and let~$(\varepsilon, \sigma, J) \in (0, 1] \times (0, 1] \times \nat_{>0}$.
    We regard~$\beta$ as a parameter indicating whether~$G$ enjoys the additional regularity of being in~$\cont^4$.
    Then the following statements hold:
    \begin{itemize}
        \item
            There exists a constant~\(C > 0\) not depending on \(\varepsilon,\sigma,J\) such that for all~\(f\in \sobolev^{2,\infty}(\torus^d,\R)\), we have
            \begin{equation}
                \label{eq:bound_distributions}
            \abs[\Big]{\fullinvarianttheta(f) - \mu_\Phi(f)}
            \leq C\left(\sigma^2 + (1-\beta)\frac{\sigma}{\sqrt{J}}\right)  \lVert f \rVert_{\infty}
            + \frac{C\varepsilon}{\sqrt{J}} \,  \lVert f \rVert_{2,\infty} \, .
        \end{equation}

    \item
        There exists a (different) constant~\(C > 0\) not depending on \(\varepsilon,\sigma,J\) such that for all~\(f\in \sobolev^{2,\infty}(\torus^d,\R)\),
        the bias of the ergodic estimator~\eqref{eq:ergodic_estimator} satisfies for all~$T > 0$
        \begin{equation}
            \label{eq:bound_bias_continuous}
            \left\lvert \expect \Bigl[ \widehat f_T - \target(f) \Bigr] \right\rvert
            \leq
            C\left(\frac{1}{T} + \sigma^2 + (1-\beta)\frac{\sigma}{\sqrt{J}}\right)  \lVert f \rVert_{\infty}
            + \frac{C\varepsilon}{\sqrt{J}} \,  \lVert f \rVert_{2,\infty} \, .
        \end{equation}

    \item
        There exists a (different) constant~\(C > 0\) not depending on \(\varepsilon,\sigma,J\) such that for all~\(f\in \sobolev^{2,\infty}(\torus^d,\R)\),
        the mean square error of the ergodic estimator~\eqref{eq:ergodic_estimator} satisfies for all~$T \geq 1$
        \begin{align}
            \label{eq:bound_mse_continuous}
            \expect  \left\lvert \widehat f_T - \target(f)  \right\rvert^2
            &\leq
            \frac{C}{T} \norm{f}_{\infty}^2
            +
            C\left(\sigma^2 + (1-\beta)\frac{\sigma}{\sqrt{J}}\right)^2  \lVert f \rVert_{\infty}^2
            + \frac{C\varepsilon^2}{J} \,  \lVert f \rVert_{2,\infty}^2
            \, .
        \end{align}
\end{itemize}
    The constant~$C$ in these statements depends on fixed derivative bounds on~$G$ and~$\ell$; 
    this dependence is made explicit in the proof.
\end{theorem}
\begin{proof}
    See~\cref{sub:invariant_distribution_continuous_proof}.
\end{proof}
\begin{remark}
    Before presenting the proof of the theorem,
    a few remarks are in order:
    \begin{itemize}
        \item
            The first bias term on the right-hand side of~\eqref{eq:bound_bias_continuous}
            accounts for the mismatch between the initial distribution~$\pi_0$ and the invariant distribution~$\fullinvariant$ of the multiscale dynamics.

        \item
            For~$\beta=1$, the parameters~$\varepsilon$ and~$J$ enter these estimates only through~$\varepsilon/\sqrt{J}$.
            This indicates the time and ensemble averaging effects are, in some sense, cumulative.

        \item
            The limits as~$T \to \infty$ of the upper bounds in~\eqref{eq:bound_bias_continuous} and~\eqref{eq:bound_mse_continuous} are consistent with the bound~\eqref{eq:bound_distributions}.

        \item
            In the bound~\eqref{eq:bound_mse_continuous} on the mean square error of the ergodic estimator,
            the first term on the right-hand side, which scales as~$T^{-1}$, arises from the variance of the estimator,
            while the other terms arise from the bias of the estimator.
    \end{itemize}
\end{remark}

Similarly, in discrete time we define
\begin{equation}
    \label{eq:ergodic_estimator_discrete}
    \widehat f_N \coloneq \frac{1}{N} \sum_{n=0}^{N-1} f(\dtheta{n}) \, .
\end{equation}
Before stating the discrete-time counterpart of~\cref{theorem:invariant_continuous},
we introduce the notation~$\fullinvariantdiscrete$ 
for the unique invariant distribution of the discretized multiscale dynamics~\eqref{eq:the_multiscale_dynamics_discrete},
and the notation~$\fullinvariantdiscretetheta$ its marginal with respect to~$\theta$.
The following theorem then provides the discrete-time estimates.

\begin{theorem}
    \label{theorem:bound_invariant_discrete}
    Consider the discretized multiscale dynamics~\eqref{eq:the_multiscale_dynamics_discrete}
    with initial condition~$(\dtheta{0}, \dallxi{0}) \sim \pi_0$
    for a probability measure $\pi_0 \in \Pi(m_0, \fastgaussianall)$,
    with some $m_0 \in \mathcal P(\torus^d)$.
    Let~$J\in\nat_{>0}$ and~$\varepsilon,\sigma,h\in(0,1]$.
    Assume that the forward model satisfies~\(G\in \cont^{3+\alpha+\beta}(\torus^d,\R^K)\),
    for some~\(\alpha \in (0, 1)\) and~\(\beta \in \{0,1\}\).
    Then the following statements hold:
    \begin{itemize}
        \item
            There exists a constant~\(C > 0\) not depending on \(\varepsilon,\sigma,J,h\) such that for all~\(f\in \cont^{2+\alpha}(\torus^d,\R)\),
            we have
            \begin{equation}
                \label{eq:bound_distributions_discrete}
                \abs*{\fullinvariantdiscretetheta(f)-\mu_\Phi(f)} \leq C\left( h + \sigma^2 + (1-\beta)\frac{\sigma}{\sqrt{J}} +
                \frac{\varepsilon}{\sqrt{J}} \right)\norm{f}_{\cont^{2+\alpha}} \, .
            \end{equation}

        \item
            There exists a (different) constant~\(C > 0\) not depending on \(\varepsilon,\sigma,J,h\) such that for all~\(f\in \cont^{2+\alpha}(\torus^d)\),
            the bias of the ergodic estimator~\eqref{eq:ergodic_estimator_discrete} satisfies for all~$N \in \nat_{>0}$
            \begin{equation}
                \label{eq:bound_bias_discrete}
                \left\lvert \expect \Bigl[ \widehat f_N - \target(f) \Bigr] \right\rvert
                \leq C \left( \frac{1}{Nh} + h + \sigma^2 + (1-\beta)\frac{\sigma}{\sqrt{J}} +
                \frac{\varepsilon}{\sqrt{J}} \right) \norm{f}_{\cont^{2+\alpha}} \, .
            \end{equation}
        \item
            There exists a (different) constant~\(C > 0\) not depending on \(\varepsilon,\sigma,J,h\) such that for all~\(f\in \cont^{2+\alpha}(\torus^d)\),
            the mean square error of the ergodic estimator~\eqref{eq:ergodic_estimator_discrete} satisfies for all~$N \in \nat_{>0}$
            \begin{align}
                \label{eq:bound_mse_discrete}
                \expect \Bigl[ \Bigl( \widehat f_N - \target(f) \Bigr)^2 \Bigr]
                \leq \frac{C}{Nh} \norm{f}_{\infty}^2 + C\left( \frac{1}{Nh} + h + \sigma^2 + (1-\beta)\frac{\sigma}{\sqrt{J}} + \frac{\varepsilon}{\sqrt{J}} \right)^2
                \norm{f}_{\cont^{2+\alpha}}^2 \, .
            \end{align}
    \end{itemize}
\end{theorem}
\begin{proof}
    See~\sref{sub:bound_invariant_discrete}.
\end{proof}

Note that this recovers the same convergence rates in~$h$ and the total time~$Nh$ as in the classical results from~\cite{MattinglyStuartTretyakov2010},
showing that our result is consistent with the classical numerical analysis literature.

\iffalse
\begin{remark}
    [\textcolor{blue}{Parameter choices and computational work}]
    \label{remark:parameter_choice}
    {\color{blue}
    For a fixed model, dimension, and observable~$f$, the bound~\eqref{eq:bound_mse_discrete} gives sufficient parameter choices for a root-mean-square error of order~$\eta\in(0,1)$.
    When~$\beta=1$, take
    \[
        h=\eta \, ,\qquad \sigma=\sqrt\eta \, ,\qquad
        \varepsilon=\min\{1,\eta\sqrt J\} \, ,\qquad
        N=\lceil\eta^{-3}\rceil \, .
    \]
    Then~$Nh\geq\eta^{-2}$ and~$\varepsilon/\sqrt J\leq\eta$, so
    \[
        \left(\expect\bigl|\widehat f_N-\target(f)\bigr|^2\right)^{\frac12}
        \lesssim\eta \, .
    \]
    When~$\beta=0$, the same conclusion follows by taking~$\sigma=\min\{\sqrt\eta,\eta\sqrt J\}$ instead.
    Each step of~\eqref{eq:the_multiscale_dynamics_discrete} requires at most~$J+1$ evaluations of~$G$ when the baseline evaluation is shared,
    giving~$\bigo\bigl((J+1)\eta^{-3}\bigr)$ forward-model evaluations.
    Since~$\varepsilon$ can be reduced without a stability restriction from the exact fast update,
    these bounds alone do not recommend increasing~$J$ as accuracy is refined: fixed~$J$ already gives~$\bigo(\eta^{-3})$ evaluations.
    Larger ensembles may still be useful when~$\varepsilon$ is prescribed or evaluations can be parallelized.
    These are sufficient choices, not optimality claims; the constants depend on the model and dimension,
    and do not provide a dimension-explicit complexity or mixing comparison.}
\end{remark}
\fi

\begin{remark}
    \label{remark:assumptions_discrete_invariant}
    In contrast with the continuous-time setting,
    here we require that~$f \in \cont^{2+\alpha}(\torus^d)$
    instead of merely~$f \in \sobolev^{2,\infty}(\torus^d)$
    to obtain the desired bounds.
    This is necessary to ensure that the discretization error scales as~$\bigo(h)$ in the bounds~\eqref{eq:bound_distributions_discrete} and~\eqref{eq:bound_bias_discrete}.
    If $f \in \sobolev^{2,\infty}(\torus^d)$,
    then it is possible to prove similar bounds,
    but with a discretization error scaling instead as~$\bigo(h^{1-\gamma})$ for any~$\gamma \in (0,1)$.
    Of course, the constant implicit in the big-O notation would diverge as~$\gamma \to 0$.
\end{remark}

\subsection{Uniform-in-time weak convergence}
\label{sec:weak_error}

We finally provide uniform-in-time weak convergence results for the multiscale dynamics~\eqref{eq:the_multiscale_dynamics} and its discretization~\eqref{eq:the_multiscale_dynamics_discrete}.
While proving the results in~\cref{sub:invariant_distribution} required a detailed study of an appropriate Poisson equation,
the results in this subsection require a more involved analysis of the backward Kolmogorov equation associated with the overdamped Langevin dynamics~\eqref{eq:overdamped_langevin_dynamics}.
The details are given together with the proof in~\cref{sec:weak_error_continuous_proof}.
We begin with the continuous-time result.
\begin{theorem}
  \label{thm:uniform_continuous_weak_convergence}
  Let \(G\in \cont^4(\torus^d,\real^K)\), \(\varepsilon,\sigma\in(0,1]\), \(J\in \nat_{>0}\).
  Let $(\theta_t, \allxi_t)_{t\geq 0}$ denote the solution to the multiscale dynamics~\eqref{eq:the_multiscale_dynamics} with initial condition as described in~\cref{sec:Mathematical setting},
  and let $(\vartheta_t)_{t\geq 0}$ denote the solution to the overdamped Langevin dynamics~\eqref{eq:overdamped_langevin_dynamics} with matching initial condition~$\vartheta_0 \sim \rho_0$.
  Then, there exists a constant~\(C>0\) independent of \(\varepsilon,\sigma,J\) and of~$\rho_0$ such that
  for all \(f\in \cont^3(\torus^d)\),
  it holds that
  \begin{equation*}
      \forall t \geq 0,
      \qquad
      \Bigl\lvert \expect \bigl[ f(\theta_t) \bigr] - \expect \bigl[f(\vartheta_t)\bigr] \Bigr\rvert
      \leq C \left( \sigma^2 + \frac{\varepsilon}{\sqrt{J}} \right) \norm*{f}_{3,\infty} \, .
  \end{equation*}
  \noindent Here, the constant \(C>0\) may depend on the function~\(G\).
\end{theorem}
\begin{proof}
    See~\cref{sec:weak_error_continuous_proof}.
\end{proof}

\begin{remark}
    [Connection with the results of~\cref{sub:invariant_distribution}]
    Apart from the more restrictive regularity assumption,
    \cref{thm:uniform_continuous_weak_convergence} implies the first two items of~\cref{theorem:invariant_continuous} obtained in~\cref{sub:invariant_distribution}.
    Indeed, assume that for some function~$f\colon \torus^d \to \real$,
    there exists a constant~$C$ independent of the parameters such that,
    for all~$x \in \torus^d$,
    \[
        \forall t \geq 0, \qquad
        \bigl\lvert \expect_x\bigl[f(\theta_t) \bigr] - \expect_x\bigl[f(\vartheta_t)\bigr] \bigr\rvert
        \leq C\left(\sigma^2 + \frac{\varepsilon}{\sqrt{J}} \right) \, .
    \]
    Taking the limit~$t \to \infty$ in the above expression and using the ergodicity of the processes \((\theta_t)_{t\geq 0}\) and \((\vartheta_t)_{t\geq 0}\)
    immediately yields the following bound on the difference between the invariant distribution and the target distribution:
    \[
        \bigl\lvert \fullinvarianttheta(f) - \target(f) \bigr\rvert
        \leq C\left(\sigma^2 + \frac{\varepsilon}{\sqrt{J}} \right) \, ,
    \]
    for the same constant~$C$.
    Furthermore,
    an upper bound on the bias of ergodic averages can be obtained by decomposing the bias as follows,
    \[
        \expect {\left[ \frac{1}{T} \int_0^T f(\theta_t) \, \d t - \target(f) \right]}
        =  \frac{1}{T} \int_0^T \expect \Bigl[ f(\theta_t) - f(\vartheta_t) \Bigr] \, \d t  + \expect{\left[ \frac{1}{T} \int_0^T f(\vartheta_t) \, \d t - \target(f) \right]} \, ,
    \]
    and by using the classical approach described in~\cite[Section 4.1]{MattinglyStuartTretyakov2010} to prove that the second term scales as~$\bigo(1/T)$.
\end{remark}

The discrete-time version is as follows.

\begin{theorem}
  \label{thm:uniform_discrete_weak_convergence}
  Assume that \(G \in \cont^6(\torus^d, \real^K)\), let \(J\in \nat_{>0}\), and let~$\varepsilon,\sigma,h\in(0,1]$.
  Let $(\dtheta{n}, \dallxi{n})_{n\geq 0}$ denote the solution to the discretized multiscale dynamics~\eqref{eq:the_multiscale_dynamics_discrete}
  supplemented with an initial condition~$(\dtheta{0}, \dallxi{0}) \sim \pi_0$
  for a probability measure $\pi_0 \in \Pi(m_0, \fastgaussianall)$.
  Let $(\vartheta_t)_{t\geq 0}$ also denote the solution to the overdamped Langevin dynamics~\eqref{eq:overdamped_langevin_dynamics} with initial condition~$\vartheta_0 \sim m_0$.
  Then, there exists a constant \(C>0\) independent of~\(\varepsilon,\sigma,h,J\) and of~$m_0$ such that
  for all~\(f\in \cont^5(\torus^d)\),
  it holds that
  \begin{equation}
      \label{eq:uniform_discrete_weak_convergence}
      \forall N \geq 0,
      \qquad
      \Bigl\lvert \expect \bigl[ f(\dtheta{N}) \bigr] - \expect \bigl[f(\vartheta_{Nh})\bigr] \Bigr\rvert
      \leq
      C \left( h + \sigma^2 + \frac{\varepsilon}{\sqrt{J}} \right)  \norm*{f}_{5,\infty} \, .
  \end{equation}
\end{theorem}
\begin{proof}
    See~\sref{ssub:weak_discrete}.
\end{proof}

We note that the regularity assumption on~$G$ in~\cref{thm:uniform_discrete_weak_convergence} is stronger than that in~\cref{thm:uniform_continuous_weak_convergence},
and that the former theorem requires~$f \in \cont^5(\torus^d)$ instead of~$f \in \cont^3(\torus^d)$.
The extra regularity is useful in the proof to control remainders arising from Taylor expansions.
The assumptions of~\cref{thm:uniform_discrete_weak_convergence} are also stronger than those of~\cref{theorem:bound_invariant_discrete},
which highlights the value of using the dedicated approach inspired by~\cite{MattinglyStuartTretyakov2010}
when studying the invariant distribution and ergodic averages.

\newpage
\section{Proofs of the main results (continuous-time)}
\label{sec:proofs_continuous}

In this section, we present the proofs of the main results for the continuous-time multiscale dynamics~\eqref{eq:the_multiscale_dynamics}.
We first comment on the proof strategy,
then prove existence and uniqueness of an invariant measure for the dynamics in~\cref{sub:proof_existence_uniqueness_invariant}.
Next, in \cref{sub:strong_convergence_continuous_proof},
we provide the proof of the strong convergence results;
\Cref{sub:invariant_distribution_continuous_proof} then provides the proof of convergence for ergodic averages,
and \cref{sec:weak_error_continuous_proof} provides the proof of weak convergence.

\subsection*{Comments on the proof strategy}%
\label{sub:proof_strategy}

We use two different proof strategies in the following sections,
both of which rely on a triangle inequality.
\begin{itemize}
    \item 
        In the proof of the strong convergence statements stated in~\cref{sub:strong_convergence_of_the_numerical_scheme},
        we use as a pivot the averaged dynamics obtained by letting~$\varepsilon \to 0$ or~$J \to \infty$ in the multiscale dynamics.
        The advantage of this approach is that it exhibits the smoothing effect of the multiscale dynamics:
        when~$\varepsilon/J \to 0$, the process~\((\theta_t)_{t \geq 0}\) converges to the solution of the Langevin dynamics 
        \begin{equation}
            \label{eq:averaged_dynamics}
            \d \vvartheta_t = \driftinf(\vvartheta_t) \, \d t + \sqrt{2} \, \d \wslow_t \, ,
            \qquad \driftinf(\theta) 
            \coloneq
            \expect_{\xi \sim \gamma_d} \Bigl[ \driftone(\theta, \xi) \Bigr] \, .
        \end{equation}
        In particular, when~$\ell = \id$ and~$G = \Phi$,
        we have~$\driftinf(\theta) = - \nabla \Phi_{\sigma}(\theta)$,
        where~$\Phi_{\sigma} = \Phi \star \normal(0, \sigma^2 \I_d)$ is a smooth approximation of~$\Phi$.
        Understanding the relation between the multiscale dynamics and the averaged dynamics~\eqref{eq:averaged_dynamics} is of independent interest,
        but presents a higher level of difficulty.
        Indeed, the drift function~$\driftone$ does not enjoy the separable structure of~$\driftonelim(\theta,\xi) \coloneq - (\xi \otimes \xi) \nabla \Phi(\theta)$,
        and so neither does the solution to the Poisson equation~\eqref{eq:poisson_strong_conv} that we are led to consider in the proof of strong convergence.
    \item
        In contrast, in the proofs of the results in~\cref{sub:invariant_distribution} and~\cref{sec:weak_error},
        we do not introduce any averaged dynamics~\eqref{eq:averaged_dynamics} as a pivot.
        As a consequence, the statements only relate the multiscale dynamics~\eqref{eq:the_multiscale_dynamics} to the target Langevin dynamics~\eqref{eq:overdamped_langevin_dynamics},
        and the various contributions to the error are not analyzed separately.
        To this end,
        the drift term in the multiscale dynamics~\eqref{eq:the_multiscale_dynamics} is systematically decomposed as
        \[
            \driftall(\theta, \allxi) = \driftlim(\theta, \allxi) + \bigl( \driftall(\theta, \allxi) - \driftlim(\theta, \allxi) \bigr) \, ,
        \]
        with the second part being viewed as an error term.
        Here~$\driftlim(\theta, \allxi) = \frac{1}{J} \sum_{j=1}^{J} \driftonelim(\theta, \xin{j})$.
        The advantage of this approach is that it allows to bypass the technical difficulties mentioned above,
        since the drift function~$\driftlim$ enjoys a separable structure:
        it can be written as the product of a function of~\(\theta\) and a function of~\(\allxi\).
\end{itemize}
It would also be possible to use the averaged dynamics~\eqref{eq:averaged_dynamics} as a pivot
in the proofs of the results in~\cref{sub:invariant_distribution} and~\cref{sec:weak_error}.
However, since doing so would render the analysis considerably more technical,
we do not pursue this approach in the present work.
We comment on some of the changes that would be required to do so in~\cref{remark:averaged_dynamics_as_pivot}.

\subsection{Existence and uniqueness for the continuous-time dynamics}
\label{sub:proof_existence_uniqueness_invariant}

\begin{proof}[Proof of~\cref{lemma:existence_uniqueness_invariant}]
    Existence of an invariant measure for the continuous-time dynamics~\eqref{eq:the_multiscale_dynamics} follows, for example,
    from the Krylov--Bogoliubov method;
    for instance, we can apply~\cite[Theorem 3.7]{khasminskii2012stochastic} with the Lyapunov function
    \[
        V(\theta, \allxi) = 1 + \frac{1}{J} \sum_{j=1}^{J} \lvert \xin{j} \rvert^2 \, .
    \]
    Indeed, the generator~$\mathcal L$ of the continuous-time dynamics satisfies
    \[
        \mathcal L V = -\frac{2}{\varepsilon}V + \frac{2(d+1)}{\varepsilon} \, ,
    \]
    and the sublevel sets of~$V$ are compact because the slow variable belongs to the torus.
    Uniqueness requires additional arguments,
    for example uniformly nondegenerate noise and a finite mean recurrence time to some bounded open domain~\cite[Theorem 4.4]{khasminskii2012stochastic},
    both of which are satisfied in our setting.
\end{proof}

\begin{remark}
    For the continuous-time dynamics,
    it would also have been possible to rely on Harris' theorem;
    see~\cite{MR2857021} and also~\cite[Section~2.4]{MR3509213} for a formulation of the method in the continuous-time setting.
    The advantage of this approach is that it provides a quantitative estimate on the rate of convergence to the invariant distribution in a weighted total variation distance.
    Applying this result requires to check two conditions:
    a Lyapunov condition and a minorization condition.
    The Lyapunov condition is straightforward to check for the multiscale dynamics~\eqref{eq:the_multiscale_dynamics} using the same Lyapunov function as above.
    A non-quantitative version of the minorization condition can be proved using a controllability argument~\cite{MR2248986,MR1931266}.
    A quantitative minorization condition can also be obtained in our setting using Girsanov's theorem~\cite{MR3986068},
    but the convergence rate obtained in this manner exhibits a bad dependence on the parameter~\(\sigma\),
    as Novikov's condition can only be verified for a small time interval whose length depends on~\(\sigma\).
    Therefore, we do not pursue this approach here.
\end{remark}

\subsection{Strong convergence}%
\label{sub:strong_convergence_continuous_proof}

Throughout this section,
we write~\(a \lesssim b\) to mean~\(a \leq C b\) for some constant~\(C > 0\) independent of the relevant parameters~\((\sigma, \varepsilon, h, J)\).
Subscripts on~\(\lesssim\) indicate additional parameters on which the constant may depend,
but we do not always make these explicit.

\subsubsection{Strong convergence for fixed \texorpdfstring{$\sigma$}{sigma}}%
\label{subsub:strong_convergence_continuous}

Throughout this section, we denote by $(\wslow_t)_{t\geq 0}$,
and $(\wfast{j}_t)_{t \geq 0}$, for $j \in \{1, 2, \dotsc\}$, 
mutually independent $d$-dimensional Brownian motions.
We also fix i.i.d.\ samples \(\xin{1}_0, \xin{2}_0, \dotsc\) with common distribution~\(\normal(0, \I_d)\),
independent of the Brownian motions,
and \(\theta_0 \sim \rho_0\in \mathcal P(\torus^d)\),
independent of all the Brownian motions and of the \(\xin{j}_0\)'s,
where~$\rho_0$ is an arbitrary initial distribution on the torus.
It will also be convenient to decompose the generator of the multiscale dynamics as
\begin{equation}
    \label{eq:generator_decomposition}
    \fullgen = \frac{1}{\varepsilon} \genfast + \genslow \, \qquad
    \text{where} \qquad
    \genfast = \sum_{j=1}^{J} \left( - \xin{j} \cdot \nabla_{\xin{j}} + \Delta_{\xin{j}} \right) \quad \text{and} \quad
    \genslow = \driftall \cdot \nabla_{\theta} + \Delta_{\theta} \, .
\end{equation}
\begin{proof}[Proof of~\cref{theorem:robustness}]
  We follow the classical approach for proving averaging results for multiscale dynamics;
  see for example~\cite{MR2382139}.
  The main difference is that we track the dependence on the number~\(J\) of fast processes.
  For~\(t \in \real_{\geq 0}\), let \(n_t = \lfloor t / \varepsilon \rfloor\).
  We first write
  \begin{align}
      \notag
      \theta_t-\vvartheta_t
      &= \int_{0}^{t} \Bigl(\driftinf(\theta_s)-\driftinf(\vvartheta_s)\Bigr) \, \d s
      + \int_{0}^{n_t \varepsilon} \Bigl(\driftall(\theta_s, \allxi_s)-\driftinf(\theta_s)\Bigr) \, \d s \\
      &\qquad + \int_{n_t \varepsilon}^{t} \Bigl(\driftall(\theta_s, \allxi_s)-\driftinf(\theta_s)\Bigr) \, \d s
      \label{eq:init_integral}
  \end{align}
  At this point, it is not clear why we separate the second and third integrals.
  The benefit of this approach will be justified later in the proof, see~\eqref{eq:max_bound} and the discussion following.
  By definition of \(\driftinf\),
  for any fixed~\(\theta \in \torus^d\) 
  the function \(\xi \mapsto \driftone(\theta, \xi)-\driftinf(\theta)\) 
  has average 0 with respect to the measure~\(\normal(0, \I_d)\).
  Therefore, by~\sref{corollary:poisson_ou}, for all~\(\theta \in \torus^d\)
  the following Poisson equation admits a unique solution~$\solone(\theta,\placeholder) \in L^2_0(\gamma_d;\real^d)$,
  where~$\genfastone$ is the generator of the Ornstein--Uhlenbeck process in~$\real^d$, given by
  \[
      \genfastone = -\xi\cdot\nabla_\xi + \Delta_\xi \, ,
  \]
  and we recall that~\(\gamma_d\) denotes the Gaussian distribution~\(\normal(0, \I_d)\):
  \begin{equation}
      \label{eq:poisson_strong_conv}
      - \genfastone \solone(\theta, \xi) =
      \driftone(\theta, \xi) - \driftinf(\theta) \, ,
  \end{equation}
  Define \(\solall \colon \torus^d \times \real^{dJ} \to \real^d\) by
  \(
  \solall(\theta, \allxi) = \frac{1}{J}\sum_{j=1}^J \solone\bigl(\theta,\xin{j}\bigr)
  \)
  and note that \(- \genfast \solall(\theta, \allxi) = \driftall(\theta, \allxi)-\driftinf(\theta)\).
  Thus, by It\^o's formula,
  it holds for all \(u \geq 0\) that
  \begin{align*}
      \solall\bigl(\theta_u, \allxi_u\bigr) - \solall\bigl(\theta_0, \allxi_0\bigr)
      &= - \frac{1}{\varepsilon} \int_{0}^{u} \Bigl( \driftall(\theta_s, \allxi_s)-\driftinf(\theta_s) \Bigr) \, \d s
      + \frac{1}{J}\sum_{j=1}^{J}\int_{0}^{u} \genslow \solone\bigl(\theta_s, \xin{j}_s\bigr) \, \d s \\
      & \qquad
      + \frac{\sqrt{2}}{J} \sum_{j=1}^{J} \int_{0}^{u} \nabla_{\theta} \solone\bigl(\theta_s, \xin{j}_s\bigr) \cdot \d \wslow_s
      + \frac{\sqrt{2}}{\sqrt{\varepsilon} J} \sum_{j=1}^{J} \int_{0}^{u} \nabla_{\xi} \solone\bigl(\theta_s, \xin{j}_s\bigr) \cdot \d \wfast{j}_s \, ,
  \end{align*}
  Combining this equation evaluated at~\(u = n_t \varepsilon\) together with~\eqref{eq:init_integral},
  we deduce that
  \begin{align}
      \notag
      \theta_t-\vvartheta_t
      &=
      \int_{0}^{t} \Bigl(\driftinf(\theta_s)-\driftinf(\vvartheta_s)\Bigr) \, \d s
      - \frac{1}{J} \sum_{j=1}^{J} \int_{n_t \varepsilon}^{t} \genfastone \solone \bigl(\theta_s, \xin{j}_s\bigr) \, \d s
      + \varepsilon \Bigl( \solall\bigl(\theta_0, \allxi_0\bigr) - \solall\bigl(\theta_{n_t \varepsilon}, \allxi_{n_t \varepsilon}\bigr)  \Bigr)
   \\ &\qquad
   \notag
   + \frac{\varepsilon}{J}\sum_{j=1}^{J}\int_{0}^{n_t \varepsilon} \genslow \solone\bigl(\theta_s, \xin{j}_s\bigr) \, \d s
   + \frac{\sqrt{2}\varepsilon}{J} \sum_{j=1}^{J} \int_{0}^{n_t \varepsilon} \nabla_{\theta} \solone\bigl(\theta_s, \xin{j}_s\bigr) \cdot \d \wslow_s
   \\ &\qquad
   + \frac{\sqrt{2\varepsilon}}{J} \sum_{j=1}^{J} \int_{0}^{n_t \varepsilon} \nabla_{\xi} \solone\bigl(\theta_s, \xin{j}_s\bigr) \cdot \d \wfast{j}_s
   \eqcolon \sum_{i=1}^{6} \termi{i}_t \, .
   \label{eq:six_terms_robustness}
  \end{align}
  We now bound each of the terms on the right-hand side separately.

  \paragraph{Bound on the first term in~\eqref{eq:six_terms_robustness}}%
  By~\sref{lemma:lipschitz_asymptotic_drift},
  the averaged drift~\(\driftinf\) is Lipschitz continuous with Lipschitz constant~\(\lipasym\) independent of~\(\sigma\),
  so it holds immediately that
  \[
      \expect \biggl[ \sup_{s \in [0, t]} \bigl\lvert \termi{1}_s \bigr\rvert^p \biggr]
      \leq \lipasym^p \int_{0}^{t} \expect {\left[ \sup_{u \in [0, s]} \lvert \theta_u - \vvartheta_u \rvert^p \right]} \, \d s \, .
  \]

  \paragraph{Bound on the second term in~\eqref{eq:six_terms_robustness}}%
  For any \(s \in [0, t]\), the integration interval in~\(\termi{2}_s\) satisfies
  \([n_s \varepsilon, s] \subseteq [0, t]\),
  and \(s - n_s \varepsilon \leq \varepsilon\).
  Therefore, by Hölder's inequality,
  \begin{equation}
      \label{eq:second_term_first_equation}
      \sup_{s \in [0, t]} \bigl\lvert \termi{2}_s \bigr\rvert^p
      \leq
      \varepsilon^{p- 1} \int_{0}^{t}
      \biggl\lvert \frac{1}{J} \sum_{j=1}^{J} \genfastone \solone \bigl(\theta_s, \xin{j}_s\bigr) \biggr\rvert^{p} \, \d s \, .
  \end{equation}
  Note that, for a given \(\theta \in \torus^d\),
  the function \(\xi \mapsto \genfastone \solone(\theta, \xi)\) has average 0 with respect to the Gaussian measure~\(\normal(0, \I_d)\).
  Consequently, for a fixed \(\theta \in \torus^d\),
  the sum
  \[
      \frac{1}{J} \sum_{j=1}^{J} \genfastone \solone \bigl(\theta, \xin{j}_s\bigr)
  \]
  can be shown to converge to 0 in \(L^p(\Omega)\) by the Marcinkiewicz--Zygmund inequality.
  However, in equation~\eqref{eq:second_term_first_equation},
  the summands in the expectation on the right-hand side are identically distributed but not independent,
  which precludes the use of a direct Marcinkiewicz--Zygmund inequality.
  To circumvent this issue,
  introducing for convenience the function \(H^J(\theta, \allxi) = \frac{1}{J} \sum_{j=1}^{J} \genfastone \solone \bigl(\theta, \xin{j} \bigr)\),
  we note that
  \[
      \forall s \in [0, t],
      \qquad
      \expect \Bigl\lvert H^J(\theta_s, \allxi_s) \Bigr\rvert^{p}
      \leq
      \expect \sup_{\theta \in \torus^d} \Bigl\lvert H^J(\theta, \allxi_s) \Bigr\rvert^{p} \, .
  \]
  In order to recover the Monte Carlo convergence rate on the right-hand side,
  we need a functional version of the law of large numbers in \(L^p(\Omega)\).
  This is the content of~\sref{proposition:expect_supremum_torus},
  the assumptions of which are satisfied in view of~\sref{lemma:upper_bound_drift_and_gradient,lemma:lipschitz_asymptotic_drift}.
  Applying this result gives, for \(q \coloneq \max\{p, d+1\}\), that
  \[
      \forall s \in [0, t], \qquad
       \expect \sup_{\theta \in \torus^d} \Bigl\lvert H^J\bigl(\theta, \allxi_{s}\bigr) \Bigr\rvert^{p}
       \lesssim_{p,d}
       J^{-\frac{p}{2}}
       \biggl( \int_{\torus^d} \expect \Bigl\lvert \genfastone \solone \bigl(\theta, \xin{1}_{s} \bigr) \Bigr\rvert^{q}
           + \expect \Bigl\lvert  \nabla_{\theta} \genfastone \solone \bigl(\theta, \xin{1}_{s} \bigr) \Bigr\rvert^{q} \, \d \theta
       \biggr)^{\frac{p}{q}} \, .
  \]
  Since the expectation on the right-hand side is bounded independently of the parameters~\((J, \sigma, \varepsilon)\),
  taking the expectation in~\eqref{eq:second_term_first_equation} and integrating in~\(s\) gives
  \begin{equation*}
      \expect {\left[ \sup_{s \in [0, t]} \bigl\lvert \termi{2}_s \bigr\rvert^p \right]}
      \lesssim_{p,d,T}
      \varepsilon^{p-1} J^{-\frac{p}{2}} \, .
  \end{equation*}
  Since \(p \geq 2\) and \(\varepsilon \leq 1\),
  one has \(p - 1 \geq p/2\), so this bound will be sufficient for the final estimate.

  \paragraph{Bound on the third term in~\eqref{eq:six_terms_robustness}}%
  It is clear that
  \begin{equation}
      \label{eq:max_bound}
      \expect {\left[ \sup_{s \in [0, t]} \bigl\lvert \termi{3}_s \bigr\rvert^p \right]}
      \leq 2^p \varepsilon^{p} \expect {\left[ \sup_{s \in [0, t]} \Bigl\lvert \solall\bigl(\theta_{n_s \varepsilon}, \allxi_{n_s \varepsilon}\bigr) \Bigr\rvert^p \right]}
      = 2^p \varepsilon^{p} \expect {\left[ \max_{m \in \left\{0, \dotsc, n_t \right\}} \Bigl\lvert \solall\bigl(\theta_{m \varepsilon}, \allxi_{m \varepsilon}\bigr) \Bigr\rvert^p \right]} \, .
  \end{equation}
  The benefit of separating the second and third integrals on the right-hand side of~\eqref{eq:init_integral} now appears clearly:
  it enables to replace a supremum over a continuous parameter set \(s \in [0, t]\) by a discrete one.
  This is useful because it enables us to write, for any \(q \geq 1\),
  \begin{align}
      \expect \biggl[ \max_{m \in \left\{0, \dotsc, n_t \right\}} \Bigl\lvert \solall\bigl(\theta_{m \varepsilon}, \allxi_{m \varepsilon}\bigr) \Bigr\rvert^p \biggr]
      \notag
      &\leq \expect \biggl[ \max_{m \in \left\{0, \dotsc, n_t \right\}} \Bigl\lvert \solall\bigl(\theta_{m \varepsilon}, \allxi_{m \varepsilon}\bigr) \Bigr\rvert^{pq} \biggr]^{\frac{1}{q}} \\
      &\leq \left( \sum_{m = 0}^{n_t} \expect \Bigl\lvert \solall\bigl(\theta_{m \varepsilon}, \allxi_{m \varepsilon}\bigr) \Bigr\rvert^{pq} \right)^{\frac{1}{q}} \, .
      \label{eq:trick_max_sum}
  \end{align}
  This inequality follows from Jensen's inequality.
  To bound the right-hand side,
  we first note that
  \[
      \forall m \in \N,
      \qquad
      \expect \Bigl\lvert \solall\bigl(\theta_{m \varepsilon}, \allxi_{m \varepsilon}\bigr) \Bigr\rvert^{pq}
      \leq
       \expect \sup_{\theta \in \torus^d}\Bigl\lvert \solall\bigl(\theta, \allxi_{m \varepsilon}\bigr) \Bigr\rvert^{pq} \, .
  \]
  In order to recover the Monte Carlo convergence rate on the right-hand side,
  we use again~\sref{proposition:expect_supremum_torus},
  which gives, assuming that \(pq \geq d+1\) (this is not restrictive as we always can take \(q\) sufficiently large),
  that
  \[
       \expect \sup_{\theta \in \torus^d}\Bigl\lvert \solall\bigl(\theta, \allxi_{m \varepsilon}\bigr) \Bigr\rvert^{pq}
       \lesssim_{p,q,d}
       J^{-\frac{pq}{2}} \biggl( \int_{\torus^d} \expect \Bigl\lvert \solone \bigl(\theta, \xin{1}_{m \varepsilon} \bigr) \Bigr\rvert^{pq}
           + \expect \Bigl\lvert  \nabla_{\theta} \solone \bigl(\theta, \xin{1}_{m \varepsilon} \bigr) \Bigr\rvert^{pq}
        \, \d \theta \biggr) \, .
  \]
  Substituting this inequality in~\eqref{eq:trick_max_sum},
  and using the stationarity of the fast processes,
  we finally deduce that, for all \(t \in [0, T]\),
  \begin{align*}
      \expect \biggl[ \sup_{s \in [0, t]} \bigl\lvert \termi{3}_s \bigr\rvert^p \biggr]
      &\lesssim_{p,q,d}
      2^p \varepsilon^{p} (n_t+1)^{\frac{1}{q}} J^{-\frac{p}{2}}
      \biggl( \int_{\torus^d} \expect \Bigl\lvert \solone \bigl(\theta, \xin{1}_{0} \bigr) \Bigr\rvert^{pq}
           + \expect \Bigl\lvert  \nabla_{\theta} \solone \bigl(\theta, \xin{1}_{0} \bigr) \Bigr\rvert^{pq}
       \, \d \theta \biggr)^{\frac{1}{q}}
       \lesssim_{p,q,d,T} \varepsilon^{p - \frac{1}{q}} J^{-\frac{p}{2}} \, .
  \end{align*}
  Here we used, for $t\leq T$ and $0<\varepsilon\leq1$,
  it holds that
  \(
      \varepsilon^p(n_t+1)^{\frac{1}{q}}
      \leq (T+1)^{\frac{1}{q}}\varepsilon^{p-\frac{1}{q}} \, .
  \)
  Since \(q\) was arbitrary,
  we can ensure that \(p - \frac{1}{q} \geq \frac{p}{2}\) by taking~\(q\) sufficiently large.

  \paragraph{Bound on the fourth term in~\eqref{eq:six_terms_robustness}}%
  \label{par:Bounding the fourth term}
  This term may be bounded similarly to the second term.
  Using the expression of the generator~\(\genslow\)
  and the Cauchy--Schwarz inequality,
  we first note that
  \begin{align}
      \label{eq:fourth_term_first_equation}
      \expect {\left[ \sup_{s \in [0, t]} \bigl\lvert \termi{4}_s \bigr\rvert^p \right]}
      &\leq t^{p-1} \varepsilon^p \int_{0}^{n_t \varepsilon} \expect {\left\lvert \frac{1}{J} \sum_{j=1}^{J}\genslow \solone\bigl(\theta_s, \xin{j}_s\bigr) \right\rvert}^p \, \d s \, .
  \end{align}  
  By definition of the generator~\(\genslow\),
  we have
  \begin{align*}
      \expect {\left\lvert \frac{1}{J} \sum_{j=1}^{J}\genslow \solone\bigl(\theta_s, \xin{j}_s\bigr) \right\rvert}^p
      &=
      \expect {\left\lvert \driftall(\theta_s, \allxi_s) \cdot \left( \frac{1}{J} \sum_{j=1}^{J} \nabla_{\theta} \solone\bigl(\theta_s, \xin{j}_s\bigr) \right) 
      + \Delta_{\theta} \left( \frac{1}{J} \sum_{j=1}^{J} \solone\bigl(\theta_s, \xin{j}_s\bigr) \right) \right\rvert}^p \, \\
      &\leq 
      \left( \expect {\left\lvert \driftall(\theta_s, \allxi_s) \right\rvert}^{2p} \right)^{\frac{1}{2}}
      \left( \expect \biggl\lvert \frac{1}{J} \sum_{j=1}^{J} \nabla_{\theta} \solone\bigl(\theta_s, \xin{j}_s\bigr) \biggr\rvert^{2p} \right)^{\frac{1}{2}}
      + \expect {\left\lvert \frac{1}{J} \sum_{j=1}^{J} \Delta_{\theta} \solone\bigl(\theta_s, \xin{j}_s\bigr) \right\rvert}^p \, .
  \end{align*}
  By~\sref{lemma:upper_bound_drift_and_gradient},
  the first expectation in the last line is bounded from above by a constant independent of the parameters~\((\varepsilon, J, \sigma)\).
  In order to bound the other expectations,
  we will use~\sref{proposition:expect_supremum_torus} with the functions~$\nabla_{\theta} \solone$ and~$\Delta_{\theta} \solone$.
  To this end,
  we should first check the assumptions of the proposition,
  but since this is a bit technical, we postpone the details to the next paragraph of the proof.
  For now, taking for granted that the assumptions of~\sref{proposition:expect_supremum_torus} are satisfied,
  we use this result and stationarity of the fast processes to obtain,
  with $q = \max\{2p, d+1\}$,
  \begin{equation}
      \label{eq:fourth_term_second_equation}
      \expect {\left\lvert \frac{1}{J} \sum_{j=1}^{J}\genslow \solone\bigl(\theta_s, \xin{j}_s\bigr) \right\rvert}^p
       \lesssim_{p,d}
       J^{-\frac{p}{2}} \biggl( \int_{\torus^d} \expect \Bigl\lvert \nabla_{\theta} \solone \bigl(\theta, \xin{1}_{0} \bigr) \Bigr\rvert^{q}
           + \expect \Bigl\lvert  \nabla_{\theta}^2 \solone \bigl(\theta, \xin{1}_{0} \bigr) \Bigr\rvert^{q}
           + \expect \Bigl\lvert  \nabla_{\theta}^3 \solone \bigl(\theta, \xin{1}_{0} \bigr) \Bigr\rvert^{q}
       \, \d \theta \biggr)^{\frac{p}{q}} \, .
  \end{equation}
  All the derivatives of~$\solone$ appearing on the right-hand side are bounded independently of the parameters~\((\varepsilon, J, \sigma)\) 
  by~\sref{lemma:upper_bound_drift_and_gradient,remark:higher_order_derivatives_drift},
  together with~\eqref{eq:theta_deriv_solone} and the boundedness of the solution operator~$\genfastone^{-1}$ on~$L^\infty_{\weight{m},0}$.
  It then follows from~\eqref{eq:fourth_term_first_equation} that, for all \(t \in [0, T]\),
  \[
      \expect {\left[ \sup_{s \in [0, t]} \bigl\lvert \termi{4}_s \bigr\rvert^p \right]}
      \lesssim_{p,d,T} \frac{\varepsilon^{p}}{J^{\frac{p}{2}}} \, .
  \]

  \paragraph{Check of the assumptions of~\sref{proposition:expect_supremum_torus} for the functions~$\nabla_{\theta} \solone$ and~$\Delta_{\theta} \solone$}%
  Here we focus on the function~$\Delta_{\theta} \solone$,
  noting that the same arguments apply to the function~$\nabla_{\theta} \solone$.
  Three assumptions need to be verified:
  \begin{itemize}
      \item it holds for all~$\theta \in \torus^d$ that \(\expect \bigl[ \Delta_{\theta} \solone(\theta, \xin{1}_{0}) \bigr] = 0\);
      \item the function \(\theta \mapsto \Delta_{\theta} \solone(\theta, \xi)\) is continuously differentiable for all~\(\xi \in \real^d\);
      \item it holds that $\expect \bigl[ \sup_{\theta \in \torus^d} \bigl\lvert \nabla_{\theta} \Delta_{\theta} \solone(\theta, \xin{1}_{0}) \bigr\rvert \bigr] < \infty$.
  \end{itemize}
  All three items follow from the regularity of~$\theta \mapsto \solone(\theta, \placeholder)$ as a map from~$\torus^d$ into the weighted space~$L^\infty_{\weight{m},0}$.
  By assumption~$G \in \cont^4(\torus^d, \real^K)$ and $\ell$ is smooth,
  so the map $\theta \mapsto \driftone(\theta, \placeholder) - \driftinf(\theta)$ is $\cont^4$
  as a map from~$\torus^d$ into the weighted space of functions $L^\infty_{\weight{m},0}$
  defined in~\eqref{eq:def_weighted_L_infty_space}, 
  for a suitable weight function~$\weight{m}(x) = 1 + |x|^m$. 
  This follows from~\sref{lemma:upper_bound_drift_and_gradient,remark:higher_order_derivatives_drift}.
  Furthermore, by~\sref{corollary:poisson_ou}, the solution operator~$\genfastone^{-1}$ is a bounded linear map on~$L^\infty_{\weight{m},0}$.
  Therefore $\theta \mapsto \solone(\theta, \placeholder) = -\genfastone^{-1}\bigl(\driftone(\theta,\placeholder)-\driftinf(\theta)\bigr)$
  is also $\cont^4$ as a map $\torus^d \to L^\infty_{\weight{m},0}$,
  with~$\theta$-derivatives given by
  \begin{equation}
      \label{eq:theta_deriv_solone}
      \partial_{\theta}^k \solone(\theta, \placeholder)
      = -\genfastone^{-1} \Bigl( \partial_{\theta}^k \bigl(\driftone(\theta, \placeholder) - \driftinf(\theta)\bigr) \Bigr)
      \in L^\infty_{\weight{m}, 0} \, .
  \end{equation}
  The first assumption above then follows immediately from the definition of~$L^\infty_{\weight{m},0}$.
  The second assumption also follows since~$\solone$ is $\cont^4$ as a map from~$\torus^d$ into~$L^\infty_{\weight{m},0}$:
  indeed pointwise evaluation at some~$\xi \in \real^d$ is a continuous linear functional on the subspace~$L^\infty_{\weight{m},0} \cap \cont(\real^d) \subset L^\infty_{\weight{m},0}$,
  for which pointwise evaluation is well-defined,
  and the function~$\Delta_{\theta} \solone(\theta, \placeholder)$ is continuous for all~$\theta \in \torus^d$ by~\sref{corollary:poisson_ou}.
  And finally, the third assumption follows from bounds on the $\theta$-derivatives of~$\driftone$ in~\sref{lemma:upper_bound_drift_and_gradient,remark:higher_order_derivatives_drift},
  together with the boundedness of the solution operator~$\genfastone^{-1}$ on~$L^\infty_{\weight{m},0}$.

  \paragraph{Bound on the fifth term in~\eqref{eq:six_terms_robustness}}%
  By the Burkholder--Davis--Gundy (BDG) and Jensen inequalities,
  we have
  \begin{align*}
      \expect {\left[ \sup_{t \in [0, T]} \bigl\lvert \termi{5}_t \bigr\rvert^p \right]}
      &= \expect \sup_{t \in [0, T]}
      \left\lvert \frac{\sqrt{2}\varepsilon}{J} \sum_{j=1}^{J} \int_{0}^{n_t \varepsilon} \nabla_{\theta} \solone\bigl(\theta_s, \xin{j}_s\bigr) \cdot \d \wslow_s \right\rvert^p \\
      &\leq
      \expect \sup_{t \in [0, T]}
      \left\lvert \frac{\sqrt{2}\varepsilon}{J} \sum_{j=1}^{J} \int_{0}^{t} \nabla_{\theta} \solone\bigl(\theta_s, \xin{j}_s\bigr) \cdot \d \wslow_s \right\rvert^p \\
      % &\leq
      % C_{\rm BDG}(p)
      % \delta^{2p}
      % \expect \left( \int_{0}^{T}  \biggl\lvert \frac{1}{J}\sum_{j=1}^{J} \nabla_{\theta} \solone\bigl(\theta_s, \xin{j}_s\bigr) \biggr\rvert^2 \, \d s \right)^{\frac{p}{2}} \\
        &\leq
        C_{\rm BDG}(p)
        \varepsilon^{p}
        T^{\frac{p}{2}-1}
        \int_{0}^{T}  \expect \Bigl\lvert \frac{1}{J} \sum_{j=1}^{J}\nabla_{\theta} \solone\bigl(\theta_s, \xin{j}_s\bigr) \Bigr\rvert^p \, \d s
        \lesssim_{p,d,T} \frac{\varepsilon^{p}}{J^{\frac{p}{2}}} \, ,
  \end{align*}
  where we used again~\sref{proposition:expect_supremum_torus} in the last inequality.

  \paragraph{Bound on the sixth term in~\eqref{eq:six_terms_robustness}}%
  For application of the BDG inequality as stated in~\cite[Theorem~7.3]{MR2380366},
  it is useful to rewrite this term as
  \[
      \frac{\sqrt{2\varepsilon}}{J} \int_{0}^{n_t \varepsilon}
      \begin{pmatrix}
          \nabla_{\xi} \solone \Bigl(\theta_s, \xin{1}_s \Bigr)  \\
          \vdots \\
          \nabla_{\xi} \solone \Bigl(\theta_s, \xin{J}_s \Bigr)
      \end{pmatrix}^\t
      \begin{pmatrix}
          \d \wfast{1}_s \, \\
          \vdots \\
          \d \wfast{J}_s \,
      \end{pmatrix} \, .
  \]
  Thus, using the BDG inequality, we deduce that
  \begin{align*}
      \expect {\left[ \sup_{s \in [0, t]} \bigl\lvert \termi{6}_s \bigr\rvert^p \right]}
      &\leq C_{\rm BDG}(p) \frac{\varepsilon^{\frac{p}{2}}}{J^p} \expect {\left( \int_{0}^{n_t \varepsilon} \sum_{j=1}^{J} \Bigl\lvert \nabla_{\xi} \solone \Bigl(\theta_s, \xin{j}_s \Bigr) \Bigr\rvert^2 \, \d s \right)}^{\frac{p}{2}} \\
      &\leq C_{\rm BDG}(p) \frac{\varepsilon^{\frac{p}{2}} T^{\frac{p}{2}-1}}{J^{\frac{p}{2} + 1}}   \int_{0}^{n_t \varepsilon} \sum_{j=1}^{J} \expect\Bigl\lvert \nabla_{\xi} \solone \Bigl(\theta_s, \xin{j}_s \Bigr) \Bigr\rvert^p \, \d s \\
      &\leq C_{\rm BDG}(p) \frac{\varepsilon^{\frac{p}{2}} T^{\frac{p}{2}}}{J^{\frac{p}{2}}} \expect \sup_{\theta \in \torus^d}  \Bigl\lvert \nabla_{\xi} \solone \Bigl(\theta, \xin{1}_0 \Bigr) \Bigr\rvert^p \, .
  \end{align*}
  By~\sref{corollary:poisson_ou} and the expression of~$\driftone$,
  the expectation on the right-hand side is bounded independently of~$\sigma$.

  \paragraph{Conclusion of the proof}%
  Combining all the bounds,
  we obtain for all \(t \in [0, T]\)
  \[
      \expect\left[
          \sup_{s \in [0, t]} \left|\theta_s-\vvartheta_s\right|^p
      \right]
      \lesssim_{p,d,T}
      \int_{0}^{t} \expect {\left[ \sup_{r \in [0, s]} \lvert \theta_r - \vvartheta_r \rvert^p \right]} \, \d s
      + \frac{\varepsilon^{\frac{p}{2}}}{J^{\frac{p}{2}}} \, .
  \]
  The conclusion then follows from Gr\"onwall's lemma.
\end{proof}

\subsubsection{Strong convergence with \texorpdfstring{$\sigma \to 0$}{sigma small}}%
\label{subsub:strong_convergence_small_sigma}

\begin{proof}[Proof of~\cref{corollary:strong_convergence_continuous}]
    The proof relies on rewriting \(\theta_t - \vartheta_t = (\theta_t - \vvartheta_t) + (\vvartheta_t - \vartheta_t)\),
    where \((\vvartheta_t)_{t\geq 0}\) is the solution to
    \[
      \d \vvartheta_t = F^\infty_\sigma(\vvartheta_t) \, \d t + \sqrt{2} \, \d \wslow_t \, ,
    \]
    with~$\vvartheta_0 = \theta_0$, and using a triangle inequality.
    A bound on the~\(L^p(\Omega)\) norm of the difference \(\theta_t - \vvartheta_t\) is given by~\cref{theorem:robustness}.
    For the other term,
    subtracting~\eqref{eq:overdamped_langevin_dynamics} from this equation and integrating in time,
    we deduce that
    \begin{align*}
        \vvartheta_t - \vartheta_t
        &= \int_{0}^{t} \Bigl( F^\infty_\sigma(\vvartheta_s) + \nabla \Phi(\vartheta_s) \Bigr) \, \d s \, \\
        &= \int_{0}^{t} \Bigl( \nabla \Phi(\vartheta_s) - \nabla \Phi(\vvartheta_s) \Bigr) \, \d s
        +\int_{0}^{t} \Bigl( \nabla \Phi(\vvartheta_s) + F^\infty_\sigma(\vvartheta_s) \Bigr) \, \d s  \,.
    \end{align*}
    In particular, by Lipschitz continuity of~\(\nabla \Phi\),
    we deduce that
    \begin{align*}
        \sup_{t \in [0, T]} \lvert \vvartheta_t - \vartheta_t \rvert^p
        &\leq (2T)^{p-1} L^p \int_{0}^{T} \lvert \vartheta_s - \vvartheta_s \rvert^p  \, \d s
        + 2^{p-1} T^{p} \bigl\lVert \nabla \Phi + F^\infty_\sigma \bigr\rVert_{L^{\infty}}^p  \,.
    \end{align*}
    Thus, the key is to understand the scaling with respect to~\(\sigma\) of the norm~\(\bigl\lVert \nabla \Phi + F^\infty_\sigma \bigr\rVert_{L^{\infty}}\),
    which is given in~\sref{lemma:bound_sigma}.
    The conclusion then follows from Gr\"onwall's inequality.
\end{proof}

\subsubsection{Homogenization limit in a multiscale potential}
\label{subsub:homogenization_proof}

\begin{proof}[Proof of~\cref{corollary:robustness_homogenization}]
    Let~$(\vvartheta^\sigma_t)_{t\geq 0}$ denote the solution to the smoothed dynamics
    \[
        \d\vvartheta^\sigma_t = F^\infty_\sigma(\vvartheta^\sigma_t)\,\d t + \sqrt{2}\,\d\wslow_t \, ,
    \]
    with the same Brownian motion~$(\wslow_t)_{t \geq 0}$ as in the multiscale dynamics and with~$\vvartheta^\sigma_0 = \theta_0$.
    By~\cref{theorem:robustness}, for every fixed~$\delta$ and~$\sigma$,
    \begin{equation}
        \label{eq:bound_theta_delta_vvartheta_sigma}
        \lim_{\varepsilon/J\to0}
        \left(\expect{\left[\sup_{t\in[0,T]}\bigl|\theta_t-\vvartheta^\sigma_t\bigr|^p\right]}\right)^{\frac{1}{p}}=0 \, .
    \end{equation}
    The difference~$\Delta_t=\vvartheta^\sigma_t-v^\sigma_t$ satisfies
    \[
        \d\Delta_t = \bigl(\driftinf\bigl(\vvartheta^\sigma_t\bigr)
        + \nabla\Phi_{0,\sigma}(\vvartheta^\sigma_t)\bigr) \,\d t
        + \bigl(\nabla\Phi_{0,\sigma}(v^\sigma_t)
        - \nabla\Phi_{0,\sigma}(\vvartheta^\sigma_t)\bigr) \,\d t \, .
    \]
    Since~$\nabla\Phi_{0,\sigma}$ is Lipschitz continuous,
    it follows that
    \[
        |\Delta_t|
        \leq \int_0^t \left|\driftinf\bigl(\vvartheta^\sigma_s\bigr)
        + \nabla\Phi_{0,\sigma}(\vvartheta^\sigma_s)\right|\,\d s
        + L\int_0^t |\Delta_s|\,\d s \, .
    \]
    By Gr\"onwall's inequality,
    \begin{equation}
        \label{eq:bound_delta}
        |\Delta_t|
        \leq \int_0^t \left|\driftinf\bigl(\vvartheta^\sigma_s\bigr)
        + \nabla\Phi_{0,\sigma}(\vvartheta^\sigma_s)\right|\e^{L(t-s)}\,\d s
        \leq \left\lVert \driftinf + \nabla\Phi_{0,\sigma} \right\rVert_{\infty}
        \int_0^t \e^{L(t-s)}\,\d s \, .
    \end{equation}
    For the specific model~\eqref{eq:multiscale_forward_model} with~$\ell(y)=y$,
    the drift~$\driftinf$ defined in~\eqref{eq:asymptotic_drift}
    can be rewritten as
    \begin{align}
        \notag
        \driftinf(\theta)
        &= - \int_{\real^d} \Biggl(\nabla\Phi_0(\theta+\sigma z) + \frac{1}{\delta} \nabla\Phi_1\Bigl(\frac{\theta}{\delta} + \frac{\sigma}{\delta} z\Bigr)\Biggr) \, \d \gamma_d(z) \, .
    \end{align}
    By the definition of~$\Phi_{0,\sigma}$, the preceding identity gives
    \begin{equation}
        \label{eq:bound_drift}
        \driftinf(\theta) + \nabla\Phi_{0,\sigma}(\theta)
        = - \frac{1}{\delta} \bigl(\nabla \Phi_1 * \rho_r\bigr)\left(\frac{\theta}{\delta}\right) \, .
    \end{equation}
    where we set $r \coloneq \sigma/\delta$,
    and $\rho_r(x)$ is the density of~$\normal(0, r^2 \I_d)$.
    In order to bound the right-hand side,
    we introduce the pushforward distribution of~$\rho_r$ on~$\torus^d$,
    given by
    \[
        \rho_r^{\torus}(x) = \sum_{k \in \Z^d} \rho_r(x + k) \, .
    \]
    To bound~$\lVert \rho_r^{\torus} - 1 \rVert_{L^\infty(\torus^d)}$,
    it is useful to recall that~$\rho_r^{\torus}$ is the law at time~$r^2$
    of a Brownian particle on the torus~$\torus^d$.
    Therefore, by classical results on convergence to equilibrium~\cite{MR2857021},
    there exist~$C, \lambda > 0$ such that
    \[
        \forall r \geq 1, \qquad
        \bigl\lVert \rho^{\torus}_r - 1 \bigr\rVert_{\infty}
        \lesssim C \e^{-\lambda r^2} \, .
    \]
    This range suffices since~$r=\frac{\sigma}{\delta}\to\infty$ at fixed~$\sigma>0$.
    Hence,
    it follows that
    \begin{align*}
        \nabla \Phi_1 * \rho_r(x)
        &= \int_{\real^d} \nabla \Phi_1(y) \rho_r(x-y) \, \d y= \int_{\torus^d} \nabla \Phi_1(y) \rho^{\torus}_r(x-y) \, \d y \\
        &= \int_{\torus^d} \nabla \Phi_1(y) \, \d y + \int_{\torus^d} \nabla \Phi_1(y) \Bigl(\rho^{\torus}_r(x-y) - 1\Bigr) \, \d y 
        = 0 + \bigo \Bigl( \e^{- \lambda r^2} \Bigr) \, .
    \end{align*}
    Going back to~\eqref{eq:bound_drift}, 
    we deduce that
    \[
        \left\lVert \driftinf + \nabla\Phi_{0,\sigma} \right\rVert_{\infty}
        \lesssim \frac{\e^{-\lambda r^2}}{\delta} \, .
    \]
    Combining this estimate with~\eqref{eq:bound_delta} and~\eqref{eq:bound_theta_delta_vvartheta_sigma}
    proves the first limit in~\eqref{eq:robustness_homogenization}.
    Finally, subtracting the equations for~$v^\sigma$ and~$\vartheta^0$ gives

    \[
        \d\bigl(v^\sigma_t-\vartheta^0_t\bigr)
        = \Bigl(\nabla\Phi_0(\vartheta^0_t)-\nabla\Phi_0(v^\sigma_t)
        + \nabla\Phi_0(v^\sigma_t)-\nabla\Phi_{0,\sigma}(v^\sigma_t)\Bigr)\,\d t \, .
    \]
    The first difference is bounded by the Lipschitz constant of~$\nabla\Phi_0$.
    For the second one, the centeredness of~$\gamma_d$ and a second-order Taylor
    expansion give
    \[
        \left\lVert\nabla\Phi_0-\nabla\Phi_{0,\sigma}\right\rVert_\infty
        \leq \frac{\sigma^2}{2}\left\lVert\nabla^3\Phi_0\right\rVert_\infty
        \int_{\real^d}|z|^2\,\d\gamma_d(z)
        = \frac{d\sigma^2}{2}\left\lVert\nabla^3\Phi_0\right\rVert_\infty \, .
    \]
    Gr\"onwall's inequality and the common initial condition then give
    \[
        \left(\expect\left[\sup_{t\in[0,T]} \left|v^\sigma_t-\vartheta^0_t\right|^p\right]\right)^{\frac{1}{p}}
        \lesssim_{T,p} \sigma^2 \, ,
    \]
    which proves the final assertion.
\end{proof}

\subsection{Invariant distribution and ergodic averages}%
\label{sub:invariant_distribution_continuous_proof}

We recall that,
since \(\mu_\Phi(\d x) \propto \e^{- \Phi(x)} \d x\) admits a density with respect
to the Lebesgue measure that is uniformly bounded from below by a positive
constant, it satisfies the Poincaré inequality with some constant \(R[\mu_\Phi] > 0\):
\begin{equation}
    \label{eq:poincare2}
    \forall \varphi \in H^1(\torus^d;\mu_\Phi) \cap L^2_0(\torus^d;\mu_\Phi), \qquad
    \int_{\torus^d}\abs{\varphi(x)}^2\mu_\Phi(\d x) \leq \frac{1}{R[\mu_\Phi]}\int_{\torus^d}
    \abs*{\nabla \varphi(x)}^2\mu_\Phi(\d x) \, ,
\end{equation}
where we recall that $H^1(\torus^d;\mu_\Phi)$ denotes the weighted Sobolev space
\[
    H^1(\torus^d;\mu_\Phi) =
    \Bigl\{
        \varphi \in L^2(\mu_\Phi)\mid \nabla \varphi\in L^2(\torus^d;\mu_\Phi)
    \Bigr\} \, ,
\]

We begin by stating the well-posedness of the Poisson equation associated with the
infinitesimal generator of the Langevin dynamics,
which plays a key role in the rest of this section.
For completeness, we give a proof in~\sref{sec:proof_poisson} that reduces the periodic regularity statements to standard elliptic results on bounded domains.
\begin{lemma}
    \label{lem:poisson}
    Consider the Poisson equation
    \begin{equation}
        \label{eq:poisson}
        \forall x\in \torus^d, \qquad
        \nabla \Phi(x)\cdot \nabla \psi(x) - \Delta \psi(x) = f(x) \, .
    \end{equation}
    Then for any~$f \in L^2_0(\mu_\Phi)$,
    there exists a unique weak solution~$\psi \in H^1(\mu_\Phi) \cap L^2_0(\mu_\Phi)$ to this equation.
    Furthermore, for any $k\in \mathbb{N}$, the following statements hold:
    \begin{itemize}
        \item
            If \(\Phi\in \sobolev^{k+1,\infty}(\torus^d)\) and \(f\in \sobolev^{k,\infty}(\torus^d)\),
            then for all $\lambda \in [0, 1)$ it holds that $\psi \in \cont^{k+1+\lambda}(\torus^d)$
            and there exists~\(C>0\) not depending on \(f\) such that
            \begin{equation}
                \label{eq:bound_poisson}
                \norm*{\psi}_{\cont^{k+1+\lambda}} \leq C \norm*{f}_{\sobolev^{k,\infty}} \, .
            \end{equation}

        \item
            If \(\Phi\in \cont^{k+1+\alpha}(\torus^d)\) and \(f\in \cont^{k+\alpha}(\torus^d)\) for some~$\alpha \in (0, 1)$,
            then~\(\psi\in \cont^{k+2+\alpha}(\torus^d)\) and there exists \(C>0\) not depending on \(f\) such that
            \begin{equation}
                \label{eq:bound_poisson_holder}
                \norm*{\psi}_{\cont^{k+2+\alpha}} \leq C \norm*{f}_{\cont^{k+\alpha}} \, .
            \end{equation}
    \end{itemize}
\end{lemma}

\begin{remark}
    [Regularity for bounded right-hand sides]
    Theorem~4.1 in~\cite{MattinglyStuartTretyakov2010} states, in the elliptic case,
    that if~$f\in L^\infty(\torus^d)$ with~$\mu_\Phi(f)=0$,
    then the mean-zero solution~$\psi$ of~\eqref{eq:poisson} belongs to~$W^{2,\infty}(\torus^d)$.
    We believe that this conclusion may require an additional regularity assumption on the right-hand side,
    as suggested by the following example for the Laplacian on~$\torus^2$,
    adapted from the construction in the Mathematics Stack Exchange post~\cite{stack}.
    Choose a smooth cutoff~$\chi$ supported in a disk of radius less than~$\frac{1}{2}$
    centered at the origin in~$(-\frac{1}{2},\frac{1}{2})^2$, and equal to~1 near the origin,
    and set
    \[
        u(x,y)=\chi(x,y)\,xy\log r \, , \qquad r=\sqrt{x^2+y^2} \, ,
    \]
    with~$u(0,0)=0$, extending~$u$ by zero outside the disk.
    Near the origin, for~$r>0$,
    \[
        \Delta u=\frac{4xy}{r^2} \, , \qquad
        \partial_x\partial_y u=\log r+\bigo(1) \, .
    \]
    Since~$u\in C^1(\torus^2)$ and~$\nabla u=\bigo(r|\log r|)$ near the origin,
    the distributional Laplacian has no additional point mass there.
    Thus~$\Delta u\in L^\infty(\torus^2)$ but~$u\notin W^{2,\infty}(\torus^2)$.
    Subtracting its mean from~$u$ gives a mean-zero weak solution of~\eqref{eq:poisson}
    with~$\Phi=0$ and bounded mean-zero right-hand side~$f=-\Delta u$.
    For this reason, we use~\eqref{eq:bound_poisson} for bounded data
    and~\eqref{eq:bound_poisson_holder} when the right-hand side is H\"older continuous.
\end{remark}

The first item in~\cref{lem:poisson} will be useful to obtain estimates in continuous time,
while the second item will be useful to obtain estimates in discrete time.
Having established the well-posedness of the Poisson equation~\eqref{eq:poisson},
we move to the proof of~\cref{theorem:invariant_continuous},
which provides an estimate on the distance between the $\theta$-marginal of the invariant distribution of the multiscale dynamics~\eqref{eq:the_multiscale_dynamics}
and the target distribution~$\mu_{\Phi}$,
together with estimates for the bias and mean square error of time averages of observables along trajectories of the dynamics.
We recall that the generator of the multiscale dynamics~\eqref{eq:the_multiscale_dynamics} can be decomposed as in~\eqref{eq:generator_decomposition},
\[
    \mathcal L^J[\varepsilon,\sigma] \varphi
    =
    \genslow[\sigma]\varphi
    + \frac{1}{\varepsilon} \mathcal L_\xi^J\varphi \, ,
\]
where we emphasize the dependence on the parameters~\(\varepsilon\) and~\(\sigma\) in square brackets.
We also recall that $\driftall$ denotes the drift term for the $\theta$-equation in the multiscale dynamics~\eqref{eq:the_multiscale_dynamics},
and that $\driftlim$ is the pointwise limit of~$\driftall$ as~$\sigma \to 0$.
Finally, we use the symbol $\lesssim$ to denote inequalities up to a multiplicative constant independent of the relevant parameters,
which here are~\(\varepsilon,\sigma,J\) and the test function~\(f\),
and we use $\bigo(\cdot)$ to denote the corresponding Landau big-O notation.
Before providing the proof of~\cref{theorem:invariant_continuous},
we present the following auxiliary result,
which is used twice in the proof of~\cref{theorem:invariant_continuous}.
We note that~\cref{lemma:averaging} can be easily generalized to the setting where~$\chi$ is not a polynomial,
but this is not necessary for our purposes,
so we state it only in the polynomial case for simplicity.
\begin{lemma}
    [Averaging lemma]
    \label{lemma:averaging}
    Assume that~$\chi \colon \real^d \to \real^n$ is a vector-valued function such that
    each of its components is a polynomial of degree at most~$r$, for some~$r\in\N$, with mean zero under~$\fastgaussianone$,
    and assume also that $A \colon \torus^d \to \real^n$ belongs to~$\cont(\torus^d, \real^n)$.
    Then, under the same assumptions as in~\cref{theorem:invariant_continuous},
    there exists a constant~$C > 0$, depending only on~$d,n,r$ and fixed derivative bounds on~$G$ and~$\ell$ but independent of~$A,\chi,\varepsilon,\sigma,J,T$, such that
    the following estimates hold for~$(\varepsilon, \sigma, J, T) \in (0, 1] \times (0, 1] \times \nat_{>0} \times (0, +\infty)$.
    \begin{itemize}
        \item
            It holds that
            \begin{equation}
                \label{eq:averaging_lemma_first_estimate}
                \expect
                \left\lvert
                    \frac{1}{T}
                     \int_{0}^{T}  \frac{1}{J} \sum_{j=1}^{J} \chi{\left(\xin{j}_t\right)}  \cdot A(\theta_t)  \, \d t
                \right\rvert^2
                \leq
                \frac{C}{J}
                \norm{A}_{\infty}^2 \norm{\chi}_{L^2(\fastgaussianone)}^2 \, .
            \end{equation}
        \item
            If additionally~$A \in \cont^2(\torus^d, \real^n)$, then
            \begin{subequations}
            \begin{equation}
                \label{eq:averaging_lemma_second_estimate}
                \left\lvert
                    \frac{1}{T}
                    \expect {\left[ \int_{0}^{T}  \frac{1}{J} \sum_{j=1}^{J} \chi{\left(\xin{j}_t\right)}  \cdot A(\theta_t)  \, \d t \right]}
                \right\rvert
                \leq
                \frac{C\varepsilon}{\sqrt{J}} \, \norm{\chi}_{L^2(\fastgaussianone)}
                \left( \kappa \frac{\norm{A}_{\infty}}{T} + \norm{A}_{2,\infty} \right) \, ,
            \end{equation}
            where~$\kappa = 0$ if~$\pi_0 = \fullinvariant$ and $\kappa = 1$ otherwise,
            and
            \begin{align}
                \notag
                \expect
                \left\lvert
                    \frac{1}{T}
                    \int_{0}^{T}  \frac{1}{J} \sum_{j=1}^{J} \chi{\left(\xin{j}_t\right)}  \cdot A(\theta_t)  \, \d t
                \right\rvert^2
                &\leq
                \frac{C\varepsilon^2}{T^2 J} \norm{\chi}_{L^2(\fastgaussianone)}^2 \norm{A}_{\infty}^2
                + \frac{C\varepsilon^2}{J} \norm{\chi}_{L^2(\fastgaussianone)}^2 \norm{A}_{2,\infty}^2 \\
                &\qquad
                + \frac{C\varepsilon}{T J} \norm{\chi}_{L^2(\fastgaussianone)}^2 \norm{A}_{\infty}^2
                + \frac{C\varepsilon^2}{T J} \norm{\chi}_{L^2(\fastgaussianone)}^2 \norm{A}_{1,\infty}^2 \, .
                \label{eq:averaging_lemma_second_estimate_mean_square}
            \end{align}
            \end{subequations}
    \end{itemize}
\end{lemma}
\begin{proof}[Proof of~\cref{lemma:averaging}]
    Since the tensor-valued case can be reduced to the scalar case by working componentwise,
    in this proof we focus only on the scalar case~$n=1$.

    \paragraph{Proof of~\eqref{eq:averaging_lemma_first_estimate}}
    The random variables~$(\xin{j}_t)_{j=1}^{J}$ are i.i.d.\ with distribution~$\fastgaussianone$ for all~$t \geq 0$
    by stationarity of the fast process,
    so in particular~$\expect \bigl[ \chi\bigl(\xin{j}_t\bigr) \bigr] = \fastgaussianone(\chi) = 0$.
    By Jensen's inequality applied to the time integral, then Fubini's theorem,
    the bound~$\lvert A(\theta_t) \rvert \leq \norm{A}_{\infty}$,
    and independence together with the zero-mean property,
    \[
        \expect \left\lvert \frac{1}{T} \int_{0}^{T} \frac{1}{J} \sum_{j=1}^{J} \chi{\left(\xin{j}_t\right)} A(\theta_t) \, \d t \right\rvert^2
        \leq \frac{\norm{A}_{\infty}^2}{T} \int_{0}^{T}
        \expect \left[ \left\lvert \frac{1}{J} \sum_{j=1}^{J} \chi{\left(\xin{j}_t\right)} \right\rvert^2 \right] \d t
        = \frac{\norm{A}_{\infty}^2 \norm{\chi}_{L^2(\fastgaussianone)}^2}{J} \, .
    \]

    \paragraph{Proof of~\eqref{eq:averaging_lemma_second_estimate} and~\eqref{eq:averaging_lemma_second_estimate_mean_square}}
    We use a technique from the averaging literature for multiscale SDEs~\cite[Chapter~10]{MR2382139}.
    Since~$\chi \in L^2_0(\fastgaussianone)$ and the operator~$-\genfastone$
    is positive definite on~$L^2_0(\fastgaussianone)$ with spectral gap equal to~$1$,
    there exists a unique solution~$\varphi \in L^2_0(\fastgaussianone)$ to the Poisson equation
    \begin{equation}
        \label{eq:poisson_averaging}
        - \genfastone \varphi = \chi \, , \qquad \int_{\real^d} \varphi{\left(\xi\right)} \, \fastgaussianone(\d \xi) = 0 \, ,
    \end{equation}
    and this solution satisfies~$\norm{\varphi}_{L^2(\fastgaussianone)} \leq \norm{\chi}_{L^2(\fastgaussianone)}$.
    Furthermore, since $\chi$ is a polynomial by assumption and the eigenfunctions of~$-\genfastone$ are Hermite polynomials,
    the solution~$\varphi$ is also a polynomial, of the same degree.
    Consider the function~$h \colon \torus^d \times \real^{dJ} \to \real$ defined by
    \[
        h{\left(\theta, \allxi\right)} \coloneq \frac{1}{J} \sum_{j=1}^{J} \varphi{\left(\xin{j}\right)} \cdot A(\theta) \, .
    \]
    Since the function~$h$ is in~$\cont^2(\torus^d \times \real^{dJ})$,
    we can apply It\^o's formula to~$h{\left(\theta_t, \allxi_t\right)}$ and use~\eqref{eq:poisson_averaging} to obtain
    \begin{align*}
        h{\left(\theta_T, \allxi_T\right)} - h{\left(\theta_0, \allxi_0\right)}
        &= - \frac{1}{\varepsilon} \int_{0}^{T} \left( \frac{1}{J} \sum_{j=1}^{J} \chi{\left(\xin{j}_t\right)} \right) A(\theta_t) \, \d t \\
        &\qquad + \int_{0}^{T} \frac{1}{J} \sum_{j=1}^{J} \varphi{\left(\xin{j}_t\right)}
          \Bigl( \driftall(\theta_t, \allxi_t) \cdot \grad A(\theta_t) + \Delta A(\theta_t) \Bigr) \, \d t \\
        &\qquad + \frac{1}{J} \sqrt{\frac{2}{\varepsilon}} \sum_{j=1}^{J} \int_{0}^{T} A(\theta_t) \, \grad_\xi \varphi{\left(\xin{j}_t\right)} \cdot \d \fastw{j}_t
          + \frac{\sqrt{2}}{J} \sum_{j=1}^{J} \int_{0}^{T} \varphi{\left(\xin{j}_t\right)} \, \grad A(\theta_t) \cdot \d \sloww_t \, .
    \end{align*}
    Multiplying by~$\varepsilon$, dividing by~$T$, and rearranging,
    we obtain
    \begin{align}
        \label{eq:averaging_lemma_ito_pathwise}
        \frac{1}{T} \int_{0}^{T} \left( \frac{1}{J} \sum_{j=1}^{J} \chi{\left(\xin{j}_t\right)} \right) A(\theta_t) \, \d t
        &= \frac{\varepsilon}{T}   \Bigl( h{\left(\theta_0, \allxi_0\right)} - h{\left(\theta_T, \allxi_T\right)} \Bigr)  \\
        \notag
        &\qquad + \frac{\varepsilon}{T}   \int_{0}^{T} \frac{1}{J} \sum_{j=1}^{J} \varphi{\left(\xin{j}_t\right)} \bigl( \driftall \cdot \grad A + \Delta A \bigr)(\theta_t, \allxi_t) \, \d t \\
        &\qquad + \frac{\sqrt{2\varepsilon}}{TJ} \sum_{j=1}^{J} \int_{0}^{T} A(\theta_t) \, \grad \varphi{\left(\xin{j}_t\right)} \cdot \d \fastw{j}_t
        + \frac{\sqrt{2}\varepsilon}{T} \int_{0}^{T} \frac{1}{J} \sum_{j=1}^{J} \varphi{\left(\xin{j}_t\right)} \, \grad A(\theta_t) \cdot \d \sloww_t \, .
        \notag
    \end{align}
    Taking expectations in this equation,
    we have
    \begin{align*}
        \frac{1}{T} \expect {\left[ \int_{0}^{T} \left( \frac{1}{J} \sum_{j=1}^{J} \chi{\left(\xin{j}_t\right)} \right) A(\theta_t) \, \d t \right]}
        &= \frac{\varepsilon}{T} \expect \bigl[ h{\left(\theta_0, \allxi_0\right)} - h{\left(\theta_T, \allxi_T\right)} \bigr] \\
        &\quad + \frac{\varepsilon}{T} \expect {\left[ \int_{0}^{T} \frac{1}{J} \sum_{j=1}^{J} \varphi{\left(\xin{j}_t\right)}
        \bigl( \driftall \cdot \grad A + \Delta A \bigr)(\theta_t, \allxi_t) \, \d t \right]} \, .
    \end{align*}
    The first expectation on the right-hand side is zero for stationary initial conditions ($\kappa = 0$).
    Otherwise, since the particles~$\xin{j}$ are i.i.d.\ with distribution~$\fastgaussianone$
    and $\varphi \in L^2_0(\fastgaussianone)$,
    \[
        \expect {\left[ \left\lvert \frac{1}{J} \sum_{j=1}^{J} \varphi{\left(\xin{j}\right)} \right\rvert \right]}
        \leq \expect {\left[ \left\lvert \frac{1}{J} \sum_{j=1}^{J} \varphi{\left(\xin{j}\right)} \right\rvert^2 \right]}^{\!\frac{1}{2}}
        = \frac{\norm{\varphi}_{L^2(\fastgaussianone)}}{\sqrt{J}} \, ,
    \]
    where the last equality uses independence and the zero-mean property of~$\varphi$.
    Therefore,
    it holds in general that
    \[
        \left\lvert \frac{\varepsilon}{T} \expect \bigl[ h{\left(\theta_0, \allxi_0\right)} - h{\left(\theta_T, \allxi_T\right)} \bigr] \right\rvert
        \lesssim \frac{\varepsilon}{T} \cdot \frac{\norm{\varphi}_{L^2(\fastgaussianone)}}{\sqrt{J}} \cdot \norm{A}_{\infty} \, .
    \]
    For the integral term,
    we recall the following bounds,
    which follow, respectively, from centering and independence of the summands,
    and from stationarity of the fast variables together with the bound on~$\driftall$ in~\sref{lemma:upper_bound_drift_and_gradient}:
    \begin{equation}
        \label{eq:averaging_auxiliary_l2_bounds}
        \left\lVert \frac{1}{J}\sum_{j=1}^{J}\varphi(\xin{j}_t)\right\rVert_{L^2(\Omega)}
        = \frac{\norm{\varphi}_{L^2(\fastgaussianone)}}{\sqrt J} \, ,
        \qquad
        \sup_{t\geq0}\left\lVert\driftall(\theta_t,\allxi_t)\right\rVert_{L^2(\Omega)} \lesssim1 \, .
    \end{equation}
    Using these bounds together with the Cauchy--Schwarz inequality,
    we obtain
    \[
        \left\lvert \frac{\varepsilon}{T} \expect {\left[ \int_{0}^{T} \frac{1}{J} \sum_{j=1}^{J} \varphi{\left(\xin{j}_t\right)} \cdot
        \Bigl( \driftall \cdot \grad A + \Delta A \Bigr) \bigl(\theta_t, \allxi_t\bigr) \, \d t \right]} \right\rvert
        \lesssim \frac{\varepsilon}{\sqrt{J}} \, \norm{\varphi}_{L^2(\fastgaussianone)} \, \norm{A}_{2,\infty} \, .
    \]
    Combining both bounds and using~$\norm{\varphi}_{L^2(\fastgaussianone)} \leq \norm{\chi}_{L^2(\fastgaussianone)}$
    yields~\eqref{eq:averaging_lemma_second_estimate}.

    For~\eqref{eq:averaging_lemma_second_estimate_mean_square} we return to the pathwise It\^o identity~\eqref{eq:averaging_lemma_ito_pathwise}.
    Taking the square of both sides, using the elementary inequality~$(a + b + c + d)^2 \leq 4 (a^2 + b^2 + c^2 + d^2)$,
    using Jensen's inequality to move the square inside the time integrals,
    then taking expectations,
    using It\^o's isometry for the stochastic integrals,
    we obtain
    \begin{align*}
        \expect \left\lvert \frac{1}{T} \int_{0}^{T} \left( \frac{1}{J} \sum_{j=1}^{J} \chi{\left(\xin{j}_t\right)} \right) A(\theta_t) \, \d t \right\rvert^2
        &\lesssim
        \frac{\varepsilon^2}{T^2} \expect {\left[ \left\lvert h{\left(\theta_0, \allxi_0\right)} \right\rvert^2 + \left\lvert h{\left(\theta_T, \allxi_T\right)} \right\rvert^2 \right]} \\
        &\qquad + \frac{\varepsilon^2}{T} \int_{0}^{T} \expect {\left[ \left\lvert \frac{1}{J} \sum_{j=1}^{J} \varphi{\left(\xin{j}_t\right)} \cdot
        \Bigl( \driftall \cdot \grad A + \Delta A \Bigr) \bigl(\theta_t, \allxi_t\bigr) \right\rvert^2 \right]} \d t \\
        &\qquad + \frac{\varepsilon}{T^2 J^2} \sum_{j=1}^{J}
        \int_{0}^{T} \expect {\left[ \Bigl\lvert \nabla \varphi{\left(\xin{j}_t\right)} \Bigr\rvert^2 \cdot \Bigl\lvert A(\theta_t) \Bigr\rvert^2 \right]} \, \d t
        \\
        &\qquad + \frac{\varepsilon^2}{T^2} \int_{0}^{T} \expect {\left[ \left\lvert \frac{1}{J} \sum_{j=1}^{J} \varphi{\left(\xin{j}_t\right)} \nabla A(\theta_t) \right\rvert^2 \right]} \, \d t \, .
    \end{align*}
    All contributions except the part involving~$\driftall$ in the second line are controlled by the $L^1$--$L^\infty$ H\"older inequality
    together with stationarity and~\eqref{eq:averaging_auxiliary_l2_bounds}.
    For the remaining term,
    since~$\varphi$ is a polynomial of degree at most~$r$,
    equivalence of the~$L^2(\gamma_d)$ and~$L^4(\gamma_d)$ norms on this finite-dimensional polynomial space
    and the Marcinkiewicz--Zygmund inequality (see~\sref{lemma:mz_inequality} with~$p=4$) give
    \[
        \left\lVert \frac{1}{J}\sum_{j=1}^{J}\varphi(\xin{j}_t)\right\rVert_{L^4(\Omega)}
        \lesssim \frac{\norm{\varphi}_{L^4(\fastgaussianone)}}{\sqrt J}
        \lesssim_r \frac{\norm{\varphi}_{L^2(\fastgaussianone)}}{\sqrt J} \, .
    \]
    This estimate and the uniform $L^4(\Omega)$ bound on~$\driftall$,
    which follows from~\sref{lemma:upper_bound_drift_and_gradient} and stationarity of the fast variables,
    allow us by Cauchy--Schwarz to control the remaining contribution, leading to
    \begin{align*}
        \expect \left\lvert \frac{1}{T} \int_{0}^{T} \left( \frac{1}{J} \sum_{j=1}^{J} \chi{\left(\xin{j}_t\right)} \right) A(\theta_t) \, \d t \right\rvert^2
        &\lesssim
        \frac{\varepsilon^2}{T^2 J} \norm{\varphi}_{L^2(\fastgaussianone)}^2 \norm{A}_{\infty}^2
        + \frac{\varepsilon^2}{J} \norm{\varphi}_{L^2(\fastgaussianone)}^2 \norm{A}_{2,\infty}^2 \\
        &\qquad
        + \frac{\varepsilon}{TJ} \norm{\nabla \varphi}_{L^2(\fastgaussianone)}^2 \norm{A}_{\infty}^2
        + \frac{\varepsilon^2}{TJ} \norm{\varphi}_{L^2(\fastgaussianone)}^2 \norm{\grad A}_{\infty}^2 \, .
    \end{align*}
    The conclusion then follows because,
    since~$\varphi$ is a weak solution to the Poisson equation~\eqref{eq:poisson_averaging}
    and~$\fastgaussianone$ satisfies a Poincar\'e inequality with constant~$1$,
    we have
    \[
        \norm{\grad \varphi}_{L^2(\fastgaussianone)}^2
        = \bigl\langle {-\genfastone \varphi}, \varphi \bigr\rangle_{L^2(\fastgaussianone)}
        = \bigl\langle \chi, \varphi \bigr\rangle_{L^2(\fastgaussianone)}
        \leq \norm{\chi}_{L^2(\fastgaussianone)} \norm{\varphi}_{L^2(\fastgaussianone)}
        \leq \norm{\chi}_{L^2(\fastgaussianone)}^2 \, .
    \]
    which concludes the proof.
\end{proof}

We are now in a position to prove~\cref{theorem:invariant_continuous},
the central result of this section.
\begin{proof}
    [Proof of~\cref{theorem:invariant_continuous}]
    We decompose the proof into four steps for clarity.

    \paragraph{Step 0. Decomposition of the error}
    Let $f \in \sobolev^{2,\infty}(\torus^d)$ and let~$(\theta_t, \allxi_t)_{t \geq 0}$ denote a strong solution of the multiscale dynamics.
    By \cref{lem:poisson}, there exists a unique solution \(\psi \in \cont^{3+\lambda}(\torus^d,\real)\),
    for any~$\lambda \in [0, 1)$,
    to the Poisson equation
    \begin{equation}
        \label{eq:poisson_proof}
      - \genlimit \psi = \grad \Phi\cdot \grad \psi -\Delta \psi = f - \mu_\Phi(f) \, , \qquad \text{with }\mu_\Phi(\psi) = 0 \, .
    \end{equation}
    It holds by It\^o's formula that
    \begin{align}
        \notag
        \psi(\theta_T) - \psi(\theta_0)
        &= \int_{0}^{T} \genlimit \psi(\theta_t) \,  \d t
        + \int_{0}^{T} \Bigl(\genslow[\sigma=0]\psi{\left(\theta_t, \allxi_t\right)} - \genlimit \psi(\theta_t) \Bigr) \, \d t
        \\
        &\qquad
        + \int_{0}^{T} \Bigl(\genslow[\sigma]\psi{\left(\theta_t, \allxi_t\right)} - \genslow[\sigma=0]\psi{\left(\theta_t, \allxi_t\right)} \Bigr) \, \d t
        + \sqrt{2} \int_0^T \grad_{\theta} \psi(\theta_t) \cdot \d \sloww_t \, ,
        \label{eq:ito_bound_invariant}
    \end{align}
    where $\genslow[\sigma=0]$ denotes the operator
    \(
        \genslow[\sigma=0] \psi{\left(\theta,\allxi\right)} = - \cov (\allxi) : \bigl(\nabla \Phi(\theta) \otimes \nabla \psi(\theta)\bigr) + \Delta \psi(\theta) \, .
    \)
    Since the forward model satisfies ~$G \in \cont^{3+\beta}(\torus^d, \real^K)$,
    it holds by the second estimate in~\sref{lemma:drift_lim_small_sigma} that
    \begin{equation}
        \label{eq:taylor_drift}
        \driftall{\left(\theta, \allxi\right)}
        = \driftlim{\left(\theta, \allxi\right)}
        + \partial_{\sigma} \driftall{\left(\theta, \allxi\right)} \bigl\lvert_{\sigma = 0} \, \sigma
            + \sigma^2 \mathcal R^J_{\sigma}{\left(\theta,\allxi\right)} \, , \,
    \end{equation}
    where
    the explicit expression of~$\partial_{\sigma} \driftall{\left(\theta, \allxi\right)} \bigl\lvert_{\sigma = 0}$ is given by~\sref{eq:drift_deriv_1}
    and the remainder term~$\mathcal R^J_{\sigma}{\left(\theta,\allxi\right)}$ satisfies
    \begin{equation}
        \label{eq:bound_remainder_drift}
        \bigl\lvert \mathcal R^J_{\sigma}{\left(\theta, \allxi\right)} \bigr\rvert
        \lesssim \frac{1}{J} \sum_{j=1}^{J} \bigl\lvert \xin{j} \bigr\rvert^4  \, ,
    \end{equation}
    with a constant implicit in the~$\lesssim$ notation that depends on third derivatives of~$G$.
    Substituting the expressions of~$\genslow[\sigma]$ and~$\genslow[\sigma=0]$ in~\eqref{eq:ito_bound_invariant},
    and using the decomposition of~$\driftall$ in~\eqref{eq:taylor_drift},
    we obtain
    \begin{align*}
        \notag
        \psi(\theta_T) - \psi(\theta_0)
        &= \int_{0}^{T} \genlimit \psi(\theta_t) \,  \d t
        + \int_{0}^{T} \Bigl(\driftlim{\left(\theta_t, \allxi_t\right)} + \nabla \Phi(\theta_t) \Bigr) \cdot \grad \psi(\theta_t) \, \d t
        \\
        &\qquad
        + \sigma \int_{0}^{T} \partial_{\sigma} \driftall{\left(\theta_t, \allxi_t\right)} \bigl\lvert_{\sigma = 0} \, \cdot \, \grad \psi(\theta_t) \, \d t
        + \sigma^2 \int_{0}^{T} \mathcal R^J_{\sigma}{\left(\theta_t, \allxi_t\right)} \cdot \grad \psi(\theta_t) \, \d t
        + \sqrt{2} \int_0^T \grad_{\theta} \psi(\theta_t) \cdot \d \sloww_t \, .
    \end{align*}
    Since~$\psi$ is a solution to the Poisson equation~\eqref{eq:poisson_proof},
    we obtain by rearranging the terms in this previous equation
    and using that~$\driftlim{\left(\theta, \allxi\right)} = - \cov(\allxi) \nabla \Phi(\theta)$ that
    \begin{align}
        \notag
        \frac{1}{T} \int_{0}^{T} f(\theta_t)  \,  \d t - \mu_\Phi(f)
        &=
        \frac{1}{T} \Bigl( \psi(\theta_0) - \psi(\theta_T) \Bigr)
        - \frac{1}{T} \int_{0}^{T} \Bigl( \cov\bigl(\allxi_t\bigr) - \I_d   \Bigr) : \bigl( \nabla \Phi(\theta_t) \otimes \grad \psi(\theta_t) \bigr) \, \d t
            \\
        \notag
        &\qquad
        + \frac{\sigma}{T} \int_{0}^{T} \partial_{\sigma} \driftall{\left(\theta_t, \allxi_t\right)} \bigl\lvert_{\sigma = 0} \, \cdot \, \grad \psi(\theta_t) \, \d t
        + \frac{\sigma^2}{T} \int_{0}^{T} \mathcal R^J_{\sigma}{\left(\theta_t, \allxi_t\right)} \cdot \grad \psi(\theta_t) \, \d t
        \\
        &\qquad
        +  \frac{\sqrt{2}}{T} \int_0^T \grad_{\theta} \psi(\theta_t) \cdot \d \sloww_t \, .
        \label{eq:decomposition_bound_invariant}
    \end{align}
    The three results in the statement of the theorem can be obtained from this equation.

    \paragraph{Step 1. Proof of~\eqref{eq:bound_distributions}}
    The bound~\eqref{eq:bound_distributions} on the invariant distribution
    follows from the estimate~\eqref{eq:bound_bias_continuous} on the bias of the ergodic estimator,
    by considering the particular case of stationary initial condition~$(\theta_0, \allxi_0) \sim \fullinvariant$,
    in which case the left-hand side of~\eqref{eq:bound_bias_continuous} reduces to~$\abs{\fullinvarianttheta(f) - \mu_\Phi(f)}$,
    and then letting~$T \to \infty$.
    Thus, it suffices to establish the bound~\eqref{eq:bound_bias_continuous}.

    \paragraph{Step 2. Proof of~\eqref{eq:bound_bias_continuous}}
    For~$0<T<1$, the bias is bounded by~$2\norm f_\infty\leq2T^{-1}\norm f_\infty$, so we may assume~$T\geq1$.
    Taking expectations in~\eqref{eq:decomposition_bound_invariant}, we obtain
    \begin{align}
         \notag
         \expect {\left[\frac{1}{T} \int_{0}^{T} f(\theta_t)  \,  \d t - \mu_\Phi(f)\right]}
         &=
         \frac{1}{T} \expect \Bigl[ \psi(\theta_0) - \psi(\theta_T) \Bigr]
            - \frac{1}{T} \expect {\left[ \int_{0}^{T} \Bigl( \cov\bigl(\allxi_t\bigr) - \I_d   \Bigr) : \bigl( \nabla \Phi(\theta_t) \otimes \grad \psi(\theta_t) \bigr) \, \d t \right]}
                \\
         \notag
            &\qquad
            + \frac{\sigma}{T} \expect {\left[ \int_{0}^{T} \partial_{\sigma} \driftall{\left(\theta_t, \allxi_t\right)} \bigl\lvert_{\sigma = 0} \, \cdot \, \grad \psi(\theta_t) \, \d t  \right]}
            \\
         \notag
            &\qquad
            + \frac{\sigma^2}{T} \expect {\left[ \int_{0}^{T} \mathcal R^J_{\sigma}{\left(\theta_t, \allxi_t\right)} \cdot \grad \psi(\theta_t) \, \d t  \right]}
            \\
        &\eqcolon \frac{1}{T} \expect \Bigl[ \psi(\theta_0) - \psi(\theta_T) \Bigr]
        + \mathcal A_1(T) + \mathcal A_2(T) + \mathcal A_3(T) \, .
        \label{eq:main_equation_bound_bias}
    \end{align}
    The first term is bounded by~$\norm{f}_{\infty} T^{-1}$,
    up to a multiplicative constant,
    since~$\norm{\psi}_{\infty} \lesssim \norm{f}_{\infty}$ by~\cref{lem:poisson}.
    We next bound each of the three terms~$\mathcal A_1(T)$, $\mathcal A_2(T)$, and~$\mathcal A_3(T)$.
    To bound~$\mathcal A_1(T)$,
    we use the estimate~\eqref{eq:averaging_lemma_second_estimate} in~\cref{lemma:averaging} with functions~$\chi(\xi) = \xi \otimes \xi - \I_d$ and~$A(\theta) = \nabla \Phi(\theta) \otimes \grad \psi(\theta)$.
    The assumptions of this statement are satisfied since~$G \in \cont^{3}(\torus^d, \real^K)$ and~$f \in \sobolev^{2,\infty}(\torus^d)$,
    so~$\psi \in \cont^{3}(\torus^d)$ by~\cref{lem:poisson}.
    Using this estimate together with the bound on~$\norm{\psi}_{3,\infty}$ in~\eqref{eq:bound_poisson},
    we obtain
    \begin{equation}
        \label{eq:bound_invariant_bound_second_term}
        \left\lvert \mathcal A_1(T) \right\rvert
        \lesssim \frac{\varepsilon}{\sqrt{J}} \left(1 + \frac{1}{T} \right) \norm{\Phi}_{3,\infty} \norm{\psi}_{3,\infty}
        \lesssim \frac{\varepsilon}{\sqrt{J}} \left(1 + \frac{1}{T} \right) \norm{G}_{3,\infty} \norm{f}_{2, \infty} \, .
    \end{equation}
    We note that, in this case,
    the unique solution~$\varphi$ to the Poisson equation~\eqref{eq:poisson_averaging},
    with matrix-valued right-hand side given by the function~$\xi \mapsto \bigl(\xi \otimes \xi - \I_d\bigr)$,
    is explicit and given by $\varphi(\xi) = \frac{1}{2} \bigl(\xi \otimes \xi - \I_d\bigr)$.

    For the term~$\mathcal A_2(T)$, we notice from~\sref{eq:drift_deriv_1} that,
    for each fixed~$\theta$,
    the function~$\xi \mapsto \partial_{\sigma} \driftone{\left(\theta, \xi\right)} \bigl\vert_{\sigma = 0}$
    is a centered cubic polynomial in~$\xi$ with~$\theta$-dependent coefficients;
    there exists a function~$\Psi \colon \torus^d \to \real^{d \times d}$ such that
    \begin{equation}
        \label{eq:tensor_product_drift_derivative}
        \partial_{\sigma} \driftone{\left(\theta, \xi\right)} \bigl\vert_{\sigma = 0} = \bigl(\xi \otimes \xi \otimes \xi\bigr) : \Psi(\theta) \, ,
    \end{equation}
    where $:$ denotes contraction over the last two tensor indices.
    By~\sref{eq:drift_deriv_1},
    the coefficients in~$\Psi(\theta)$ depend on~$G$ and its derivatives up to order~2.
    The rapid relaxation of the fast variables suggests heuristically that
    the dependence between~$\theta_t$ and~$\allxi_t$ becomes negligible as~$\varepsilon/\sqrt{J} \to 0$ for positive~$t$,
    and hence that~$\mathcal A_2(T)$ should vanish by an averaging argument.
    Indeed, by the first statement in~\cref{lemma:averaging}
    together with Jensen's inequality and the bound~\eqref{eq:bound_poisson} on~$\psi$,
    we have immediately that
    \begin{subequations}
    \begin{equation}
        \label{eq:bound_invariant_second_term_bound_1}
        \bigl\lvert \mathcal A_2(T) \bigr\rvert
        \lesssim
        \frac{\sigma}{\sqrt{J}} \, \norm{\nabla \psi}_{\infty} \,
        \lesssim \frac{\sigma}{\sqrt{J}} \, \norm{f}_{\infty} \, .
    \end{equation}
    For the assumption of the second bullet point in~\cref{lemma:averaging} to be satisfied,
    we use that~$\Psi \in \cont^2(\torus^d, \real^{d \times d})$,
    which follows from~$G \in \cont^{4}(\torus^d, \real^K)$.
    It is also necessary that~$\nabla \psi \in \cont^2(\torus^d, \real^d)$,
    but this follows from~\cref{lem:poisson} since~$f \in \sobolev^{2,\infty}(\torus^d)$.
    Thus, for~$G \in \cont^{4}(\torus^d, \real^K)$,
    we have
    \begin{equation}
        \label{eq:bound_invariant_second_term_bound_2}
        \bigl\lvert \mathcal A_2(T) \bigr\rvert
        \lesssim
        \frac{\sigma \varepsilon}{\sqrt{J}} \,  \, \left( \frac{\norm{\nabla \psi}_{\infty}}{T} + \norm{\nabla \psi}_{2, \infty} \right) \,
        \lesssim \frac{\sigma \varepsilon}{\sqrt{J}} \,  \, \left( \frac{\norm{f}_{\infty}}{T} + \norm{f}_{2, \infty} \right) \,
        \lesssim \frac{\sigma \varepsilon}{\sqrt{J}} \left(\frac{1}{T} + 1\right)\, \norm{f}_{2, \infty} \, .
    \end{equation}
    \end{subequations}

    Finally, for the third term~$\mathcal A_3(T)$,
    we use the bound on the remainder term~$\mathcal R^J_{\sigma}$ in~\eqref{eq:bound_remainder_drift},
    together with the bound on~$\nabla \psi$ in~\eqref{eq:bound_poisson},
    to obtain
    \[
        \bigl\lvert \mathcal A_3(T) \bigr\rvert
        \lesssim
        \sigma^2 \expect {\left[ \frac{1}{T} \int_{0}^{T} \frac{1}{J} \sum_{j=1}^{J} \bigl\lvert \xin{j}_t \bigr\rvert^4  \, \d t \right]} \norm{\grad \psi}_{\infty}
        \lesssim
        \sigma^2  \norm{f}_{\infty}  \, .
    \]
    Combining this equation together with the bounds~\eqref{eq:bound_invariant_bound_second_term},~\eqref{eq:bound_invariant_second_term_bound_1}, and~\eqref{eq:bound_invariant_second_term_bound_2} in~\eqref{eq:main_equation_bound_bias},
    we finally deduce the desired bound~\eqref{eq:bound_bias_continuous}.

    \paragraph{Step 3. Bound on the mean square error of the ergodic estimator}%
    As mentioned, a bound on the mean square error of the estimator can also be deduced from~\eqref{eq:decomposition_bound_invariant}.
    Indeed, squaring both sides of this equation, taking expectations, using the It\^o isometry,
    and using~\eqref{eq:averaging_lemma_second_estimate_mean_square} in~\cref{lemma:averaging},
    we obtain for all~$T \geq 1$
    \begin{align*}
        \notag
        \expect \biggl\lvert \frac{1}{T} \int_{0}^{T} f(\theta_t) \, \d t - \mu_{\Phi}(f) \biggr\rvert^2
        &\lesssim
        \frac{1}{T^2}  \norm{\psi}_{\infty}^2
        +
        \left(
            \frac{\varepsilon^2}{T^2 J} \norm{\nabla \psi}_{\infty}^2
            + \frac{\varepsilon^2}{J} \norm{\nabla \psi}_{2,\infty}^2
            + \frac{\varepsilon}{TJ} \norm{\nabla \psi}_{\infty}^2
            + \frac{\varepsilon^2}{TJ} \norm{\nabla \psi}_{1,\infty}^2
        \right)
        \\
        &\qquad
        +
        \sigma^2
        \left(
            \frac{1 - \beta}{J} \norm{\nabla \psi}_{\infty}^2
            + \beta
            \left(
                \frac{\varepsilon^2}{T^2 J} \norm{\nabla \psi}_{\infty}^2
                + \frac{\varepsilon^2}{J} \norm{\nabla \psi}_{2,\infty}^2
                + \frac{\varepsilon}{TJ} \norm{\nabla \psi}_{\infty}^2
                + \frac{\varepsilon^2}{TJ} \norm{\nabla \psi}_{1,\infty}^2
            \right)
        \right) \\
        &\qquad
        + \sigma^4 \norm{\nabla \psi}_{\infty}^2
        + \frac{1}{T} \norm{\nabla \psi}_{\infty}^2 \, \\
        &\lesssim
        \frac{1}{T^2} \norm{f}_{\infty}^2
        + (1-\beta) \frac{\sigma^2}{J} \norm{f}_{\infty}^2
        + \sigma^4 \norm{f}_{\infty}^2
        + \frac{1}{T} \norm{f}_{\infty}^2
        + \frac{\varepsilon^2}{J} \norm{f}_{2,\infty}^2
        + \frac{\varepsilon}{TJ} \norm{f}_{\infty}^2 \, ,
    \end{align*}
    which, upon rearranging and using $\frac{\varepsilon}{T} \leq \frac{\varepsilon^2}{2} + \frac{1}{2T^2}$,
    is the desired inequality~\eqref{eq:bound_mse_continuous}.
\end{proof}

\begin{remark}
    \label{remark:averaged_dynamics_as_pivot}
    As discussed in the comments on the proof strategy above,
    it would also have been possible to use,
    in the proof of~\cref{theorem:invariant_continuous},
    the averaged dynamics~\eqref{eq:averaged_dynamics} as a pivot,
    just as we did for the proof of strong convergence.
    To obtain a bound on the distance between the $\theta$-marginal of the invariant distribution of the multiscale dynamics, on the one hand,
    and the invariant measure $\nu_\sigma$ of the averaged dynamics~\eqref{eq:averaged_dynamics} for fixed~$\sigma > 0$, on the other hand,
    we would need to consider the Poisson equation
    \begin{equation}
        \label{eq:slow_poisson_alternative}
        \driftinf \cdot \nabla \psi + \Delta \psi = f - \nu_{\sigma}(f) \, ,
    \end{equation}
    instead of~\eqref{eq:poisson}, as well as the Poisson equation
    \begin{equation}
        \label{eq:fast_poisson_alternative}
        - \mathcal L_\xi^J \phi = \driftall(\theta, \allxi) - \driftinf(\theta) \, ,
    \end{equation}
    instead of~\eqref{eq:poisson_averaging}.
    Here, $\nu_\sigma$ denotes the invariant measure of the averaged dynamics.
    In the special case $\ell=\id$ and~$G=\Phi$, 
    the averaged drift is $-\nabla\Phi_\sigma$, where $\Phi_\sigma = \Phi \star \normal(0,\sigma^2\I_d)$;
    hence $\nu_\sigma$ has density proportional to~$\e^{-\Phi_\sigma}$.
    While the proof would be conceptually similar to that presented above,
    it would be much more tedious.
    Indeed, a uniform-in-\(\sigma\) regularity estimate for the solution to~\eqref{eq:slow_poisson_alternative} would be required.
    Furthermore,
    more involved techniques based on~\sref{proposition:expect_supremum_torus} would be required to obtain the correct scaling in~$J$ for terms arising
    from the solution to~\eqref{eq:fast_poisson_alternative}, which,
    unlike the solution to~\eqref{eq:poisson_averaging},
    does not tensorize as a product of a function of~\(\theta\) and a function of~\(\allxi\).
    We do not pursue this alternative approach here.
\end{remark}

\subsection{Weak convergence}%
\label{sec:weak_error_continuous_proof}

As emphasized in~\cref{sub:invariant_distribution},
establishing the accuracy of the invariant distribution and of ergodic averages requires studying the elliptic Poisson equation~\eqref{eq:poisson}.
In this section, we establish uniform-in-time weak convergence for the continuous-time multiscale dynamics~\eqref{eq:the_multiscale_dynamics};
the proof for the discretized dynamics is given in~\sref{ssub:weak_discrete}.
This requires to understand the regularization properties of the parabolic equation corresponding to~\eqref{eq:poisson},
known as the backward Kolmogorov equation:
\begin{equation}
    \label{eq:backward_kolmogorov}
    \begin{aligned}
        \partial_t v &= \genlimit v \quad && \text{in $(0, \infty) \times \torus^d$} \, , \\
        v &= f \quad && \text{in $\{0\} \times \torus^d$} \, .
    \end{aligned}
\end{equation}
We recall that~$\genlimit = -\nabla \Phi \cdot \nabla + \Delta$ is the generator of the overdamped Langevin dynamics
\begin{equation}
    \label{eq:overdamped_langevin_weak}
    \d \vartheta_t = -\nabla \Phi(\vartheta_t) \, \d t + \sqrt{2} \, \d \wslow_t \, .
\end{equation}
For our purposes,
we need two key properties of the function~$v$,
aside from its existence and uniqueness as a classical solution:
its connection to the overdamped Langevin dynamics~\eqref{eq:overdamped_langevin_weak} through the celebrated Feynman--Kac formula,
and a pointwise control of its spatial derivatives with sufficiently fast decay in time.
This is the content of the following classical result,
whose proof is provided in \sref{appendix:proof_auxiliary_parabolic}.

\begin{theorem}
    \label{theorem:auxiliary_parabolic}
    Let~$M\geq2$ be an integer, and assume that~$\Phi\in\cont^{\max\{4,M+1\}}(\torus^d)$ and~$f\in\cont^M(\torus^d)$.
    Then there exists a unique classical solution~$v\in\cont^{1,M}([0,\infty)\times\torus^d)$
    to the parabolic equation~\eqref{eq:backward_kolmogorov}.
    This solution is given by the Feynman--Kac formula:
    \begin{equation}
        \label{eq:feynman_kac}
        \forall \theta\in \torus^d,\quad v(t,\theta)
        = \expect_\theta\bigl[ f(\vartheta_t)\bigr] \, ,
    \end{equation}
    where \((\vartheta_t)_{t\geq 0}\) is the solution to the overdamped Langevin dynamics~\eqref{eq:overdamped_langevin_weak}.

    Furthermore,
    there exists \(\lambda_{\Phi} > 0\) such that,
    for every integer~$0\leq m\leq M$, there exists a constant~\(C_m > 0\) satisfying
    \begin{equation}
        \label{eq:decay_derivatives_semigroup_statement}
        \forall s \geq 0, \qquad
        \lVert v(s, \placeholder) - \target(f) \rVert_{m, \infty}
        \leq C_m \, \e^{-\lambda_{\Phi} s} \lVert f \rVert_{m, \infty} \, .
    \end{equation}
\end{theorem}

Before providing the proof of~\cref{thm:uniform_continuous_weak_convergence},
we also present the following auxiliary result,
which generalizes \cref{lemma:averaging} to time-dependent test functions $A(t,\theta)$,
and will be used twice in the proof of~\cref{thm:uniform_continuous_weak_convergence}.
\begin{lemma}
    [Averaging lemma with time-dependent test functions]
    \label{lemma:averaging_weak}
    Let $\chi \colon \real^d \to \real^n$ be a polynomial with mean zero under $\fastgaussianone$,
    and let $A \in \cont^{1,2}([0,\infty) \times \torus^d, \real^n)$.
    Under the same assumptions as in \cref{thm:uniform_continuous_weak_convergence},
    there exists a constant $C > 0$,
    depending only on~$d,n$ and fixed derivative bounds on~$G$ and~$\ell$ but independent of~$\varepsilon,\sigma,J,\chi,A$, 
    such that
    \[
    \begin{aligned}
        \left\lvert \int_0^t \expect{\left[ \frac{1}{J} \sum_{j=1}^J \chi(\xin{j}_s) \cdot A(s,\theta_s) \right]} \d s \right\rvert
        &\leq \frac{C\varepsilon}{\sqrt{J}} \, \lVert\chi\rVert_{L^2(\fastgaussianone)} \\
        &\qquad \times \left( \lVert A(0,\placeholder)\rVert_\infty
        + \lVert A(t,\placeholder)\rVert_\infty
        + \int_0^t \bigl( \lVert \partial_s A(s,\placeholder)\rVert_\infty
        + \lVert A(s,\placeholder)\rVert_{2,\infty} \bigr) \d s \right) \, .
    \end{aligned}
    \]
\end{lemma}
\begin{proof}[Proof of~\cref{lemma:averaging_weak}]
    Let $\varphi \in L^2_0(\fastgaussianone;\real^n)$ be the unique solution to the Poisson equation
    \[
        -\genfastone \varphi = \chi \, , \qquad \int_{\real^d} \varphi(\xi) \, \fastgaussianone(\d\xi) = 0 \, .
    \]
    Since $\chi$ is a polynomial, $\varphi$ is also a polynomial of the same degree,
    and it satisfies $\lVert\varphi\rVert_{L^2(\fastgaussianone)} \leq \lVert\chi\rVert_{L^2(\fastgaussianone)}$.
    Consider the function
    \[
        h{\left(s,\theta,\allxi\right)} \coloneq \frac{1}{J}\sum_{j=1}^J \varphi{\left(\xin{j}\right)} \cdot A{\left(s,\theta\right)} \, .
    \]
   Applying It\^o's formula to $h(s,\theta_s,\allxi_s)$ we obtain
    \begin{align*}
        h{\left(t,\theta_t,\allxi_t\right)} - h{\left(0,\theta_0,\allxi_0\right)}
        &= \int_0^t \frac{1}{J}\sum_{j=1}^J \varphi{\left(\xin{j}_s\right)} \cdot
        \Bigl( \partial_s + \driftall{\left(\theta_s,\allxi_s\right)} \cdot \nabla_\theta + \Delta_\theta \Bigr) A{\left(s,\theta_s\right)} \, \d s \\
        &\qquad - \frac{1}{\varepsilon} \int_0^t \frac{1}{J}\sum_{j=1}^J \chi{\left(\xin{j}_s\right)} \cdot A{\left(s,\theta_s\right)} \, \d s
        + \sqrt{2} \int_0^t \frac{1}{J}\sum_{j=1}^J \left(\varphi{\left(\xin{j}_s\right)} \cdot \nabla_\theta A{\left(s,\theta_s\right)}\right) \cdot \d \sloww_s \\
        &\qquad + \sqrt{\frac{2}{\varepsilon}} \frac{1}{J}\sum_{j=1}^J \int_0^t \Bigl(A{\left(s,\theta_s\right)} \cdot \nabla_\xi \varphi{\left(\xin{j}_s\right)}\Bigr) \cdot \d \fastw{j}_s \, .
    \end{align*}
    Taking expectations, rearranging, and multiplying by $\varepsilon$ yields
    \begin{align*}
        \int_0^t \expect{\left[ \frac{1}{J}\sum_{j=1}^J \chi{\left(\xin{j}_s\right)} \cdot A{\left(s,\theta_s\right)} \right]} \d s
        &= \varepsilon \, \expect\Bigl[ h{\left(0,\theta_0,\allxi_0\right)} - h{\left(t,\theta_t,\allxi_t\right)} \Bigr] \\
        &\qquad + \varepsilon \int_0^t \expect{\left[ \frac{1}{J}\sum_{j=1}^J \varphi{\left(\xin{j}_s\right)} \cdot
        \Bigl( \partial_s + \driftall{\left(\theta_s,\allxi_s\right)} \cdot \nabla_\theta + \Delta_\theta \Bigr) A{\left(s,\theta_s\right)} \right]} \d s \, .
    \end{align*}
    It remains to bound each term on the right-hand side.
    Since $\varphi$ has mean zero and the fast variables $\xin{j}$ are i.i.d.\ with law $\fastgaussianone$,
    we have
    \[
        \expect{\left[ \biggl\lvert \frac{1}{J}\sum_{j=1}^J \varphi{\left(\xin{j}_s\right)} \biggr\rvert \right]}
        \leq \frac{\lVert\varphi\rVert_{L^2(\fastgaussianone)}}{\sqrt{J}} \, ,
    \]
    uniformly in $s$ by stationarity of the fast process.
    Consequently,
    \[
        \left\lvert \expect\Bigl[ h{\left(s,\theta_s,\allxi_s\right)} \Bigr] \right\rvert
        \leq \frac{\lVert\varphi\rVert_{L^2(\fastgaussianone)}}{\sqrt{J}} \, \lVert A(s,\placeholder)\rVert_\infty \, .
    \]
    For the integral term,
    we use the two auxiliary bounds in~\eqref{eq:averaging_auxiliary_l2_bounds}.
    Using these bounds together with the Cauchy--Schwarz inequality, we obtain
    \[
        \biggl\lvert \expect\biggl[ \frac{1}{J}\sum_{j=1}^J \varphi{\left(\xin{j}_s\right)} \cdot
        \Bigl( \partial_s + \driftall{\left(\theta_s,\allxi_s\right)} \cdot \nabla_\theta + \Delta_\theta \Bigr) A{\left(s,\theta_s\right)} \biggr] \biggr\rvert
        \lesssim \frac{\lVert\varphi\rVert_{L^2(\fastgaussianone)}}{\sqrt{J}}
        \Bigl( \lVert\partial_s A(s,\placeholder)\rVert_\infty + \lVert A(s,\placeholder)\rVert_{2,\infty} \Bigr) \, .
    \]
    The desired estimate follows by collecting these bounds and using $\lVert\varphi\rVert_{L^2} \leq \lVert\chi\rVert_{L^2}$.
\end{proof}

We are now in a position to prove~\cref{thm:uniform_continuous_weak_convergence},
the main result of this subsection.
The main idea of proof is classical;
see, for example, the textbook~\cite[Part VI]{MR1214374},
as well as the recent references~\cite{MR4237949,MR4939957} for techniques to obtain uniform-in-time estimates.

\begin{proof}[Proof of~\cref{thm:uniform_continuous_weak_convergence}]
Let \(f\in \cont^3(\torus^d)\) and fix some time $t\geq 0$.
By~\cref{theorem:auxiliary_parabolic} with~\(M=3\),
there exists a unique classical solution~\(v\in \cont^{1,3}([0,\infty) \times \torus^d)\)
to the initial-value backward Kolmogorov equation
\begin{equation}
  \label{eq:forward}
  \left\{
  \begin{aligned}
      \partial_r v &= \genlimit v \quad && \text{in $\torus^d \times (0,\infty)$} \, , \\
    v &= f \quad && \text{in $\torus^d \times \{0\}$} \, .
  \end{aligned}
  \right.
\end{equation}
For this proof,
it will be more convenient to work with the terminal-value formulation of this parabolic equation.
Setting \(u(s,\theta) \coloneq v(t-s,\theta)\) for \(s\in (-\infty, t]\),
we obtain by direct computation that \(u\) satisfies
\begin{equation}
  \label{eq:kbe}
  \left\{
  \begin{aligned}
      \partial_s u + \genlimit u &= 0 \quad && \text{in $\torus^d \times (-\infty, t]$} \, , \\
    u &= f \quad && \text{in $\torus^d \times \{t\}$} \, .
  \end{aligned}
  \right.
\end{equation}
By the Feynman--Kac formula,
see~\sref{lemma:semigroup_parabolic},
the initial-value solution has the representation
\(v(r,\theta) = \expect^\theta[f(\vartheta_r)]\),
so that by the Markov property
\[
    u(s,\theta) = \expect\bigl[f(\vartheta_t) \bigm| \vartheta_s = \theta\bigr] \, ,
\]
where $(\vartheta_s)_{s\geq 0}$ is the solution to the overdamped Langevin dynamics~\eqref{eq:overdamped_langevin_weak}.
We can decompose the difference between the semigroup using the solution of \eqref{eq:kbe}:
\begin{align*}
    \expect \bigl[f(\theta_t)\bigr] - \expect\bigl[f(\vartheta_t)\bigr]
    & = \expect\bigl[u(t,\theta_t)\bigr] - \expect\bigl[u(0,\theta_0)\bigr] \, .
    % \\
    % & = \int_0^t\frac{\intd}{\intd s}\expect\bigl[u(s,\theta_s)\bigr] \, \d s \, .
\end{align*}
By applying Itô's formula to the function \((s, \theta) \mapsto u(s,\theta)\),
we obtain
\begin{align*}
  u(t, \theta_t) - u(0, \theta_0)
  =& \int_0^t \left(\frac{\partial u}{\partial s}(s,\theta_s) + \fullgenparams u{\left(s,\theta_s,\allxi_s\right)}\right) \, \d s
  + \sqrt{2} \int_0^t \nabla_\theta u(s,\theta_s) \cdot \d \sloww_s
  \\
  =& \int_0^t\left(\frac{\partial u}{\partial s}(s,\theta_s) + \genlimit u{\left(s,\theta_s\right)}\right) \, \d s
  \\
    &+ \int_0^t\left(\mathcal L^J[\varepsilon,\sigma] - \genlimit\right)u{\left(s,\theta_s,\allxi_s\right)}\, \d s
  + \sqrt{2} \int_0^t \nabla_\theta u(s,\theta_s) \cdot \d \sloww_s \, ,
\end{align*}
where, by a slight abuse of notation, we wrote $\fullgenparams u{\left(t, \theta,\allxi\right)} = \driftall{\left(\theta,\allxi\right)} \cdot \nabla_\theta u(t, \theta) + \Delta_{\theta} u(t, \theta)$.
Substituting the expressions of $\fullgenparams$ and $\genlimit$,
we obtain
\[
    u(t, \theta_t) - u(0, \theta_0)
     = \int_0^t\Bigl(\driftall{\left(\theta_s,\allxi_s\right)}
      + \nabla_\theta\Phi(\theta_s)\Bigr)\cdot \nabla_\theta u(s,\theta_s)\, \d s
     + \sqrt{2} \int_0^t \nabla_\theta u(s,\theta_s) \cdot \d \sloww_s \, .
\]
Taking the expectation on both sides,
we deduce that
\begin{align}
    \notag
  \expect\bigl[f(\theta_t)\bigr] - \expect\bigl[f(\vartheta_t)\bigr]
  &= \int_0^t\expect\Bigl[ \Bigl(\driftall{\left(\theta_s, \allxi_s\right)} + \nabla \Phi(\theta_s) \Bigr) \cdot \nabla_\theta u(s,\theta_s)\Bigr] \, \d s  \\
  \notag
  &= \int_0^t\expect\Bigl[ \Bigl(    \driftall{\left(\theta_s, \allxi_s\right)} - \driftlim{\left(\theta_s, \allxi_s\right)} \Bigr) \cdot \nabla_\theta u(s,\theta_s)\Bigr] \, \d s  \\
  \label{eq:diff_measure_semigroup}
  &\qquad + \int_0^t\expect\Bigl[ \Bigl(\driftlim{\left(\theta_s, \allxi_s\right)} + \nabla \Phi(\theta_s) \Bigr) \cdot \nabla_\theta u(s,\theta_s)\Bigr] \, \d s \, .
\end{align}
The first term on the right-hand side corresponds to the gradient approximation error,
while the second term is an averaging error.
We now bound these two terms separately.

\paragraph{Bound on the gradient approximation error}
By~\sref{lemma:drift_lim_small_sigma},
it holds that
\begin{equation}
    \label{eq:drift_approximation_error_expansion}
    \driftall{\left(\theta_s, \allxi_s\right)} - \driftlim{\left(\theta_s, \allxi_s\right)}
    = \sigma \driftderiv{\left(\theta_s, \allxi_s\right)} \,
    + \sigma^2 \mathcal R^J_{\sigma}{\left(\theta_s, \allxi_s\right)} \, ,
\end{equation}
where the remainder term~$\mathcal R^J_{\sigma}$ satisfies
\begin{equation}
    \label{eq:bound_remainder_term_grad_approx}
    \forall \bigl(\theta, \allxi\bigr) \in \torus^d \times \R^{dJ},
    \qquad \bigl\lvert \mathcal R^J_{\sigma}(\theta, \allxi) \bigr\rvert \leq \frac{1}{6} \bigl\lVert (\nabla \ell)\circ G \bigr\rVert_{\infty} \bigl\lVert \nabla^3 G \bigr\rVert_{\infty} \frac{1}{J}\sum_{j=1}^J \abs*{\xi^j}^4  \, .
\end{equation}
% Since the process \((\xin{j}_t)_{t\geq 0}\) is a stationary Ornstein-Uhlenbeck process for all \(j\in \{1,\dots,J\}\),
% we have for all \(t\geq 0\):
% \[
%     \expect\left[\abs*{\xi^1_t}^3\right] = 2^{3/2}\frac{\Gamma\left(\frac{d+3}{2} \right)}{\Gamma\left(\frac{d}{2}\right)} \, .
% \]
Thus, we have
\begin{align}
    \notag
    \int_0^t\expect\Bigl[ \Bigl(    \driftall{\left(\theta_s, \allxi_s\right)} - \driftlim{\left(\theta_s, \allxi_s\right)} \Bigr) \cdot \nabla_\theta u(s,\theta_s)\Bigr] \, \d s
    &=
    \sigma \int_0^t\expect\Bigl[ \driftderiv{\left(\theta_s, \allxi_s\right)} \cdot \nabla_\theta u(s,\theta_s)\Bigr] \, \d s  \\
    &\qquad
    + \sigma^2 \int_0^t \expect\Bigl[ \mathcal R^J_{\sigma}{\left(\theta_s, \allxi_s\right)} \cdot \nabla_\theta u(s,\theta_s)\Bigr] \, \d s \, .
    \label{eq:decomposition_grad_approx_error}
\end{align}
By~H\"older's inequality and~\eqref{eq:bound_remainder_term_grad_approx},
the second expectation on the right-hand side satisfies
\begin{align*}
    \Bigl\lvert \expect\Bigl[ \mathcal R^J_{\sigma}{\left(\theta_s, \allxi_s\right)} \cdot \nabla_\theta u(s,\theta_s)\Bigr] \Bigr\rvert
    &\leq
    \expect \Bigl\lvert \mathcal R^J_{\sigma}{\left(\theta_s, \allxi_s\right)} \Bigr\rvert
    \bigl\lVert \nabla_{\theta} u(s, \placeholder) \bigr\rVert_{\infty} \\
    &\lesssim \bigl\lVert \nabla_\theta u(s, \placeholder) \bigr\rVert_\infty \, ,
\end{align*}
where we omitted the dependence on~$\ell, G$, since these are considered fixed data of the problem.
By the estimate~\eqref{eq:decay_derivatives_semigroup_statement} in~\cref{theorem:auxiliary_parabolic}, applied with~$m=1$, there exists a constant \(C_1 > 0\) and \(\lambda_{\Phi} > 0\) such that
\begin{equation}
    \label{eq:semigroup}
    \forall s \in [0, t],
    \qquad
    \bigl\lVert \nabla u(s, \placeholder) \bigr\rVert_{\infty} \leq C_1 \e^{-\lambda_{\Phi} (t-s)} \norm*{f}_{1,\infty} \, .
\end{equation}
Hence, we deduce that
\begin{equation}
    \label{eq:bound_sigma2_remainder}
    \biggl\lvert \int_0^t \expect\Bigl[ \mathcal R^J_{\sigma}{\left(\theta_s, \allxi_s\right)} \cdot \nabla_\theta u(s,\theta_s)\Bigr] \, \d s \biggr\rvert
    \lesssim
    C_1  \norm*{f}_{1,\infty}
    \int_0^t \e^{-\lambda_{\Phi} (t-s)} \, \d s
    \leq \frac{C_1}{\lambda_{\Phi}} \norm*{f}_{1,\infty} \, .
\end{equation}
It remains to bound the first term on the right-hand side of~\eqref{eq:decomposition_grad_approx_error}.
For the multiscale dynamics~\eqref{eq:the_multiscale_dynamics},
it holds for an appropriate matrix-valued function $\Psi$, see~\sref{eq:drift_deriv_1}, that
\[
    \forall (\theta, \allxi) \in \torus^d \times \R^{dJ}, \qquad
    \driftderiv(\theta, \allxi) =
    \frac{1}{J}\sum_{j=1}^J (\xin{j} \otimes \xin{j} \otimes \xin{j}) : \Psi(\theta) \, ,
\]
where we recall that $:$ denotes contraction over the last two tensor indices.
Furthermore $\Psi \in \cont^2(\torus^d, \R^{d\times d})$ if~$G \in \cont^4(\torus^d, \R^{K})$.
By applying \cref{lemma:averaging_weak} with $\chi(\xi) = \xi \otimes \xi \otimes \xi$
and $A(s,\theta) = \nabla u(s,\theta) \otimes \Psi(\theta)$,
we obtain
\[
    \left\lvert \int_0^t\expect\Bigl[ \driftderiv{\left(\theta_s, \allxi_s\right)} \cdot \nabla_\theta u(s,\theta_s)\Bigr] \, \d s \right\rvert
    \lesssim \frac{\varepsilon}{\sqrt{J}}
    \left( \lVert A(0,\placeholder)\rVert_\infty + \lVert A(t,\placeholder)\rVert_\infty
    + \int_0^t \bigl( \lVert \partial_s A(s,\placeholder)\rVert_\infty + \lVert A(s,\placeholder)\rVert_{2,\infty} \bigr) \d s \right) \, .
\]
By the exponential decay~\eqref{eq:semigroup} and the regularity of~$\Psi$,
we have
\[
    \lVert A(s,\placeholder)\rVert_\infty \lesssim \lVert \Psi \rVert_\infty \lVert \nabla u(s,\placeholder) \rVert_\infty
    \lesssim \e^{-\lambda_{\Phi} (t-s)} \norm*{f}_{1,\infty} \, ,
\]
and, using the backward equation~$\partial_s u = -\genlimit u$ together with the semigroup decay,
\[
    \lVert A(s,\placeholder)\rVert_{2,\infty} + \lVert \partial_s A(s,\placeholder)\rVert_\infty
    \lesssim \e^{-\lambda_{\Phi} (t-s)} \norm*{f}_{3,\infty} \, .
\]
It follows that
\[
    \Bigl\lvert \int_0^t\expect\Bigl[ \driftderiv{\left(\theta_s, \allxi_s\right)} \cdot \nabla_\theta u(s,\theta_s)\Bigr] \, \d s \Bigr\rvert
    \lesssim \frac{\varepsilon}{\sqrt{J}} \, \norm*{f}_{3,\infty} \, .
\]
Combining this estimate with the bound on the~$\sigma^2$ remainder term~\eqref{eq:bound_sigma2_remainder},
we obtain the following bound on the gradient approximation error:
\begin{equation}
    \label{eq:grad_approx_error}
    \biggl\lvert \int_0^t\expect\Bigl[ \Bigl(\driftall{\left(\theta_s, \allxi_s\right)} - \driftlim{\left(\theta_s, \allxi_s\right)} \Bigr) \cdot \nabla_\theta u(s,\theta_s)\Bigr] \, \d s \biggr\rvert
    \lesssim \sigma\frac{\varepsilon}{\sqrt{J}} \, \norm*{f}_{3,\infty} + \sigma^2 \norm*{f}_{1,\infty} \, .
\end{equation}

\paragraph{Bound on the averaging error}%
Let us consider the second term of \eqref{eq:diff_measure_semigroup}.
By the expression of~\(\driftlim\),
we rewrite this term as
\[
    \int_0^t \expect{\Bigl[ \left(\driftlim{\left(\theta_s,\allxi_s\right)}+\nabla \Phi(\theta_s) \right) \cdot \nabla_\theta u(s,\theta_s) \Bigr]} \,\d s
    = - \int_0^t \expect{\left[  \frac{1}{J}\sum_{j=1}^J \Bigl( \xin{j}_s \otimes \xin{j}_s - \I_d \Bigr)  : \Bigl( \nabla_\theta u(s,\theta_s) \otimes \nabla \Phi(\theta_s) \Bigr) \right]} \,\d s \, .
\]
Applying \cref{lemma:averaging_weak} with $\chi(\xi) = \xi\otimes\xi - \I_d$
and $A(s,\theta) = \nabla_\theta u(s,\theta) \otimes \nabla \Phi(\theta)$,
we obtain
\begin{align*}
    \biggl\lvert \int_0^t\expect \Bigl[ \Bigl(\driftlim{\left(\theta_s,\allxi_s\right)}+\nabla \Phi(\theta_s) \Bigr) \cdot \nabla_\theta u(s,\theta_s) \Bigr] \,\d s \biggr\rvert
    &\lesssim \frac{\varepsilon}{\sqrt{J}}
    \Bigl( \lVert A(0,\placeholder)\rVert_\infty + \lVert A(t,\placeholder)\rVert_\infty \\
    &\qquad + \int_0^t \bigl( \lVert \partial_s A(s,\placeholder)\rVert_\infty + \lVert A(s,\placeholder)\rVert_{2,\infty} \bigr) \d s \Bigr) \, .
\end{align*}
By the exponential decay~\eqref{eq:semigroup} and the definition~$A(s,\theta) = \nabla_\theta u(s,\theta) \otimes \nabla \Phi(\theta)$,
we have
\[
    \lVert A(s,\placeholder)\rVert_\infty \lesssim \lVert \nabla \Phi \rVert_\infty \lVert \nabla u(s,\placeholder) \rVert_\infty
    \lesssim \e^{-\lambda_{\Phi} (t-s)} \norm*{f}_{1,\infty} \, ,
\]
and, by the backward equation together with the semigroup decay,
\[
    \lVert A(s,\placeholder)\rVert_{2,\infty} + \lVert \partial_s A(s,\placeholder)\rVert_\infty
    \lesssim \e^{-\lambda_{\Phi} (t-s)} \norm*{f}_{3,\infty} \, .
\]
and the time integrals are bounded uniformly in~$t$.
Therefore,
\begin{equation}
    \label{eq:weak_averaging_error}
    \biggl\lvert \int_0^t\expect \Bigl[ \Bigl(\driftlim{\left(\theta_s,\allxi_s\right)}+\nabla \Phi(\theta_s) \Bigr) \cdot \nabla_\theta u(s,\theta_s) \Bigr] \,\d s \biggr\rvert
    \lesssim \frac{\varepsilon}{\sqrt{J}} \, \norm*{f}_{3,\infty} \, .
\end{equation}
\paragraph{Conclusion of the proof}
By~\eqref{eq:diff_measure_semigroup}, the difference is decomposed into a gradient approximation error and an averaging error, bounded respectively in~\eqref{eq:grad_approx_error} and~\eqref{eq:weak_averaging_error}.
Combining these estimates and applying Young's inequality to the product term,
\[
    \sigma\frac{\varepsilon}{\sqrt{J}} \leq 
    \frac{1}{2}\left(\sigma^2 + \frac{\varepsilon^2}{J}\right)
    \leq \frac{1}{2}\left(\sigma^2 + \frac{\varepsilon}{\sqrt{J}}\right) \, ,
\]
we obtain
\[
    \bigl\lvert \expect\bigl[f(\theta_t)\bigr] - \expect\bigl[f(\vartheta_t)\bigr] \bigr\rvert
    \lesssim \sigma^2 \norm*{f}_{3,\infty} + \frac{\varepsilon}{\sqrt{J}} \norm*{f}_{3,\infty} \, .
\]
Since the bounds hold uniformly in~\(t\geq 0\), this completes the proof of~\cref{thm:uniform_continuous_weak_convergence}.
\end{proof}

\newpage
\section{Proofs of the main results (discrete-time)}
\label{sec:proofs_discrete}

In this section, we present the proofs of the main results for the discrete-time multiscale dynamics~\eqref{eq:the_multiscale_dynamics_discrete}.
We first prove existence and uniqueness of an invariant measure for the dynamics in~\cref{sub:proof_existence_uniqueness_invariant_discrete}.
Next, in~\cref{subsub:strong_convergence_discrete},
we provide the proof of the strong convergence results;
\Cref{sub:bound_invariant_discrete} then provides the proof of convergence of the invariant distribution and ergodic averages,
and~\cref{ssub:weak_discrete} provides the proof of weak convergence.

\subsection{Existence and uniqueness for the discrete-time dynamics}
\label{sub:proof_existence_uniqueness_invariant_discrete}

\begin{proof}[Proof of the discrete-time assertion in~\cref{lemma:existence_uniqueness_invariant}]
    For the discrete-time dynamics~\eqref{eq:the_multiscale_dynamics_discrete},
    we use Harris' theorem as stated in~\cite[Theorem~1.2]{MR2857021},
    which yields as a side-product geometric ergodicity in a weighted total variation distance.
    Letting~$a=\e^{-\frac{h}{\varepsilon}}\in(0,1)$ and denoting by~$P_h$ the one-step transition operator,
    we have that
    \[
        P_hV = a^2 V + (1-a^2)(d+1) \, ,
    \]
    where $V(\theta, \allxi) = 1 + \frac{1}{J} \sum_{j=1}^{J} \lvert \xin{j} \rvert^2$ as in the continuous-time proof.
    Thus the Lyapunov condition required to apply Harris' theorem clearly holds.
    The theorem also requires a minorization condition.
    For sufficiently large~$R>0$, let~$K_R = \{x\in\torus^d\times\real^{dJ}: V(x)\leq R\}$.
    This set is compact and has positive measure.
    The one-step transition density~$p_h(x,y)$, with respect to the product of Lebesgue measure on~$\torus^d$ and~$\real^{dJ}$,
    is continuous and strictly positive.
    Hence there exists~$c_R>0$ such that
    \[
        p_h(x,y) \geq c_R \qquad \text{for all } x,y\in K_R \, .
    \]
    This density lower bound implies the minorization condition required by Harris' theorem.
    Existence and uniqueness of the discrete invariant measure therefore follow from Harris' theorem.
\end{proof}

\subsection{Strong convergence}%
\label{subsub:strong_convergence_discrete}

\begin{proof}[Proof of~\cref{theorem:strong_convergence_fixed_sigma_discrete}]
    Fix \(p \geq 2\).
    Throughout the proof, \(C\) denotes a constant independent of \((\varepsilon,\sigma, h, J)\),
    whose value is irrelevant in the context and may change from occurrence to occurrence.
    By analogy with the proof in the continuous-time setting,
    we introduce
    \[
        n_\ell
        = \left\lfloor \frac{\ell}{\nu} \right\rfloor \, ,
        \qquad
        \nu = \left\lceil \frac{\varepsilon}{h} \right\rceil \, .
    \]
    Note that \(\ell - n_\ell \cdot \nu \leq \nu - 1\) and \(\ell = n_\ell \cdot \nu\) if \(\varepsilon \leq h\).
    It holds for any \(k \in \N\) that
    \begin{align}
        \notag
        \dtheta{k}-\dvvartheta{k}
        &= h \sum_{\ell=0}^{k-1} \Bigl( \driftinf(\dtheta{\ell}) - \driftinf(\dvvartheta{\ell}) \Bigr)
        + h \sum_{\ell=0}^{n_k \cdot \nu-1} \Bigl( \driftall\bigl(\dtheta{\ell},\dallxi{\ell}\bigr) - \driftinf\bigl(\dtheta{\ell}\bigr) \Bigr) \\
        \label{eq:main_equation_strong_conv}
        &\qquad+ h \sum_{\ell=n_k \cdot \nu}^{k-1} \Bigl( \driftall\bigl(\dtheta{\ell},\dallxi{\ell}\bigr) - \driftinf\bigl(\dtheta{\ell}\bigr) \Bigr) \, .
    \end{align}

    \paragraph{Discrete Poisson equation}
    By analogy with standard methods of proof for averaging results in continuous time,
    see for example~\cite[Chapter~17]{MR2382139},
    we consider the discrete Poisson equation
    \begin{equation}
        \label{eq:discrete_poisson}
        - \left(\frac{\e^{\frac{h}{\varepsilon} \genfastone} - \id}{h}\right) \solone(\theta, \xi) =
        \driftone(\theta, \xi) - \driftinf(\theta) \, .
    \end{equation}
    For convenience, we introduce the notation~$\Pi \driftone(\theta, \xi) = \driftone(\theta, \xi) - \driftinf(\theta)$ for the right-hand side of this equation.
    By~\cref{lemma:upper_bound_drift_and_gradient},
    the right-hand side of this equation belongs to the weighted space~\(L^\infty_{\weight{2},0}\) defined in~\eqref{eq:def_weighted_L_infty_space}.
    Therefore, by~\cref{lemma:discrete_poisson_eq} applied componentwise,
    for any fixed \(\theta \in \torus^d\) there exists a unique solution \(\solone(\theta,\placeholder)\colon\real^d\to\real^d\), with each component in~\(L^{\infty}_{\weight{2},0}\), to the Poisson equation~\eqref{eq:discrete_poisson},
    given by the Neumann series
    \[
        \solone(\theta, \xi)
        = h \sum_{\ell=0}^{\infty}  \e^{\ell \frac{h}{\varepsilon} \genfastone} (\Pi \driftone) \bigl(\theta, \xi\bigr) \, .
    \]
    This solution of course depends on the parameters~\((\varepsilon,\sigma,h)\),
    but for simplicity we omit these in the notation.
    For~\(j \in \range{1}{J}\) and any~\(k \in \N\),
    we now consider the decomposition
    \begin{align}
        \notag
        & \solone(\dtheta{k}, \dxin{j}{k}) -  \solone(\dtheta{0}, \dxin{j}{0})  \\
        \notag
        &\qquad = \sum_{\ell=0}^{k-1} \Bigl( \solone{\left(\dtheta{\ell+1}, \dxin{j}{\ell+1}\right)} - \solone{\left(\dtheta{\ell}, \dxin{j}{\ell}\right)} \Bigr) \\
        \label{eq:decomposition_a_b}
        &\qquad = \sum_{\ell=0}^{k-1} \Bigl( \solone{\left(\dtheta{\ell}, \dxin{j}{\ell+1}\right)} - \solone{\left(\dtheta{\ell}, \dxin{j}{\ell}\right)} \Bigr)
        + \sum_{\ell=0}^{k-1} \Bigl( \solone{\left(\dtheta{\ell+1}, \dxin{j}{\ell+1}\right)} - \solone{\left(\dtheta{\ell}, \dxin{j}{\ell+1}\right)} \Bigr)
        \eqcolon A_{k} + B_{k} \, .
    \end{align}
    % The discrete-time processes \((A_\ell)_{\ell \in \nat}\) and \((B_\ell)_{\ell \in \nat}\) are adapted to the filtration \((\mathcal F_n)_{n \in \nat}\),
    % generated by the random variables \(\{\dtheta{0}, \dallxi{0}, \dotsc, \dtheta{n}, \dallxi{n} \}\).
    % The Doob decomposition of the process \((A_n)_{n \in \nat}\) is given by
    We next decompose the process \((A_\ell)_{\ell \in \nat}\) as a sum of a predictable and a martingale process,
    and use~\eqref{eq:discrete_poisson} to rewrite the predictable part in terms of the drift functions:
    \begin{align*}
        A_k
        &= \sum_{\ell=0}^{k-1} \Bigl( \e^{\frac{h}{\varepsilon} \genfastone} \solone {\left(\dtheta{\ell}, \dxin{j}{\ell} \right)} - \solone {\left(\dtheta{\ell}, \dxin{j}{\ell} \right)} \Bigr)
        + \sum_{\ell=0}^{k-1} \Bigl( \solone {\left(\dtheta{\ell}, \dxin{j}{\ell+1}\right)} - \e^{\frac{h}{\varepsilon} \genfastone} \solone {\left(\dtheta{\ell}, \dxin{j}{\ell}\right)} \Bigr) \\
        &= - \sum_{\ell=0}^{k-1} h \Bigl( \driftone{\left(\dtheta{\ell}, \dxin{j}{\ell}\right)} - \driftinf(\dtheta{\ell}) \Bigr)
        + \sum_{\ell=0}^{k-1} \Bigl( \solone {\left(\dtheta{\ell}, \dxin{j}{\ell+1}\right)} - \e^{\frac{h}{\varepsilon} \genfastone} \solone {\left(\dtheta{\ell}, \dxin{j}{\ell}\right)} \Bigr) \, .
    \end{align*}
    Here the second term is a martingale with respect to \((\mathcal F_\ell)\),
    where \(\mathcal F_\ell\) is the \(\sigma\)-algebra generated by the random variables \(\{\dtheta{0}, \dallxi{0}, \dotsc, \dtheta{\ell}, \dallxi{\ell} \}\).
    Using this equation and~\eqref{eq:decomposition_a_b},
    we rewrite~\eqref{eq:main_equation_strong_conv} as
    \begin{align}
        \notag
        \dtheta{k}-\dvvartheta{k}
        &= h \sum_{\ell=0}^{k-1} \Bigl( \driftinf(\dtheta{\ell}) - \driftinf(\dvvartheta{\ell}) \Bigr)
        + h \sum_{\ell=n_k \cdot \nu}^{k-1} \Bigl( \driftall\bigl(\dtheta{\ell},\dallxi{\ell}\bigr) - \driftinf\bigl(\dtheta{\ell}\bigr) \Bigr)
        \\
        \notag
        & \qquad
        + \frac{1}{J} \sum_{j=1}^{J}
        \Bigl(
            \solone {\left(\dtheta{0}, \dxin{j}{0}\right)}
            - \solone {\left(\dtheta{n_k \cdot \nu}, \dxin{j}{n_k \cdot \nu}\right)}
        \Bigr) \\
        \notag
        &\qquad
        + \frac{1}{J} \sum_{j=1}^{J} \sum_{\ell=0}^{n_k \cdot \nu - 1}
        \Bigl(
            \solone{\left(\dtheta{\ell+1}, \dxin{j}{\ell+1}\right)}
        - \solone{\left(\dtheta{\ell}, \dxin{j}{\ell+1}\right)}
        \Bigr)
        \\
        \label{eq:main_equation_strong_convergence}
        &\qquad
        + \frac{1}{J} \sum_{j=1}^{J} \sum_{\ell=0}^{n_k \cdot \nu - 1} \Bigl( \solone {\left(\dtheta{\ell}, \dxin{j}{\ell+1}\right)} - \e^{\frac{h}{\varepsilon} \genfastone} \solone {\left(\dtheta{\ell}, \dxin{j}{\ell}\right)} \Bigr)
        \eqcolon \mathcal I^1_k + \mathcal I^2_k + \mathcal I^3_k + \mathcal I^4_k + \mathcal I^5_k \, .
    \end{align}
    Raising both sides to the power \(p \geq 2\),
    using the convexity of \(x \mapsto |x|^p\),
    taking the supremum over \(\range{0}{n}\) and taking the expectation,
    we obtain
    \[
        E_n := \expect {\left[ \sup_{\ell \in \range{0}{n}} \lvert \dtheta{\ell}-\dvvartheta{\ell} \rvert^p \right]}
        \leq 5^{p-1} \sum_{i=1}^{5} \expect {\left[ \sup_{\ell \in \range{0}{n}} \lvert \mathcal I^i_\ell \rvert^p \right]} \, .
    \]
    It remains to bound all the terms on the right-hand side.

    \paragraph{Bound on the first term in~\eqref{eq:main_equation_strong_convergence}}
    By Lipschitz continuity of~\(\driftinf\) given by~\cref{lemma:lipschitz_asymptotic_drift},
    the first term on the right-hand side of~\eqref{eq:main_equation_strong_convergence} is bounded as follows:
    \[
        \expect {\left[ \sup_{\ell \in \range{0}{n}} \left\lvert \mathcal I^1_\ell \right\rvert^p \right]}
        \leq n^{p-1} h^p \expect {\left[\sum_{\ell=0}^{n-1} \Bigl\lvert \driftinf(\dtheta{\ell})-\driftinf(\dvvartheta{\ell}) \Bigr\rvert^p \right]}
        \leq h L^p \sum_{\ell=0}^{n-1} \expect \Bigl[ \left\lvert \dtheta{\ell} - \dvvartheta{\ell} \right\rvert^p \Bigr]
        \lesssim h \sum_{\ell=0}^{n-1} E_\ell \, .
    \]

    \paragraph{Bound on the second term in~\eqref{eq:main_equation_strong_convergence}}%
    Since $\ell = n_\ell \cdot \nu$ for $\varepsilon \leq h$,
    this term is zero in this case,
    so we can assume without loss of generality that $\varepsilon > h$ in the rest of this paragraph.
    For any \(\ell \in \range{0}{k}\), the range \([n_\ell \nu, \ell-1]\) has length at most \(\nu - 1\),
    so by convexity of \(|\placeholder|^p\),
    \begin{equation}
        \label{eq:second_term_first_equation_discrete}
         \sup_{\ell \in \range{0}{k}} \left\lvert \mathcal I^2_\ell \right\rvert^p
        \leq
        (\nu-1)^{p-1} h^p
        \sum_{\ell=0}^{k-1} \biggl\lvert \frac{1}{J} \sum_{j=1}^{J} \driftone {\bigl(\dtheta{\ell},\dxin{j}{\ell}\bigr)} - \driftinf\bigl(\dtheta{\ell}\bigr) \biggr\rvert^p \, .
    \end{equation}
    For any given \(\theta \in \torus^d\),
    the function \(\xi \mapsto \driftone {\bigl(\theta,\xi\bigr)} - \driftinf\bigl(\theta\bigr)\) has average 0 with respect to the Gaussian measure~\(\normal(0, \I_d)\).
    Using the same technique as in the continuous-time setting,
    we write
    \[
        \expect {\left\lvert \frac{1}{J} \sum_{j=1}^{J} \driftone {\bigl(\dtheta{\ell},\dxin{j}{\ell}\bigr)} - \driftinf\bigl(\dtheta{\ell}\bigr) \right\rvert}^p
        \leq
        \expect \sup_{\theta \in \torus^d} \left\lvert \frac{1}{J} \sum_{j=1}^{J}  \driftone {\bigl(\theta,\dxin{j}{\ell}\bigr)} - \driftinf\bigl(\theta\bigr) \right\rvert^p \, .
    \]
    Denoting \(\Pi \driftone(\theta, \xi) \coloneq \driftone(\theta, \xi) - \driftinf(\theta)\)
    and letting \(q \coloneq \max\{d+1, p\}\),
    we apply~\cref{proposition:expect_supremum_torus} to obtain
    \[
        \expect \sup_{\theta \in \torus^d} \left\lvert \frac{1}{J} \sum_{j=1}^{J}  \driftone {\bigl(\theta,\dxin{j}{\ell}\bigr)} - \driftinf\bigl(\theta\bigr) \right\rvert^p
        \lesssim_{p,d}
        J^{-\frac{p}{2}}
        \biggl(
            \int_{\torus^d} \expect \Bigl\lvert \Pi \driftone {\bigl(\theta,\dxin{1}{0}\bigr)} \Bigr\rvert^{q}
            + \expect \Bigl\lvert  \nabla_{\theta} \Pi \driftone {\bigl(\theta,\dxin{1}{0}\bigr)} \Bigr\rvert^{q} \, \d \theta
        \biggr)^{\frac{p}{q}} \, ,
    \]
    where we used stationarity of the fast processes to replace \(\dxin{1}{\ell}\) by \(\dxin{1}{0}\).
    Since the right-hand side is bounded independently of the parameters~\((J, \sigma, \varepsilon)\),
    taking the expectation in~\eqref{eq:second_term_first_equation_discrete} we obtain
    \begin{equation*}
        \expect {\left[ \sup_{\ell \in \range{0}{n}} \bigl\lvert \mathcal I^2_\ell \bigr\rvert^p \right]}
        \lesssim_{p,d,T}
        \varepsilon^{p-1} J^{-\frac{p}{2}} \, .
    \end{equation*}
    where we used that \((\nu-1)^{p-1} h^p  k \lesssim (\varepsilon/h)^{p-1} h^p  (T/h) = T \varepsilon^{p-1}\).

    \paragraph{Bound on the third term in~\eqref{eq:main_equation_strong_convergence}}
    First note that
    \begin{align*}
        \max_{\ell \in \range{0}{k}} \left\lvert \mathcal I^3_\ell \right\rvert^p
        &\leq 2^{p-1} \biggl\lvert \frac{1}{J} \sum_{j=1}^{J} \solone(\dtheta{0}, \dxin{j}{0}) \biggr\rvert^p
        + 2^{p-1} \max_{\ell \in \range{0}{k}} \biggl\lvert \frac{1}{J} \sum_{j=1}^{J} \solone {\left( \dtheta{\nu \cdot n_\ell}, \dxin{j}{\nu \cdot n_\ell} \right)} \biggr\rvert^p \\
        &\leq 2^{p} \max_{\ell \in \range{0}{k}} \biggl\lvert \frac{1}{J} \sum_{j=1}^{J} \solone {\left( \dtheta{\nu \cdot n_\ell}, \dxin{j}{\nu \cdot n_\ell} \right)} \biggr\rvert^p \\
        &= 2^{p} \max_{\ell \in \range{0}{n_k}} \biggl\lvert \frac{1}{J} \sum_{j=1}^{J} \solone {\left( \dtheta{\ell \nu}, \dxin{j}{\ell \nu} \right)} \biggr\rvert^p \, .
    \end{align*}
    To proceed, note that for any \(q \geq 1\),
    \begin{align*}
        \expect {\left[ \max_{\ell \in \range{0}{n_k}} \Bigl\lvert \solall {\bigl(\dtheta{\ell \nu}, \dallxi{\ell \nu}\bigr)} \Bigr\rvert^p \right]}
        \notag
          &\leq \expect {\left[ \max_{\ell \in \range{0}{n_k}} \Bigl\lvert \solall {\bigl(\dtheta{\ell \nu}, \dallxi{\ell \nu}\bigr)} \Bigr\rvert^{pq} \right]}^{\frac{1}{q}} \\
          &\leq \left( \sum_{\ell = 0}^{n_k} \expect \Bigl\lvert \solall {\bigl(\dtheta{\ell \nu}, \dallxi{\ell \nu}\bigr)} \Bigr\rvert^{pq} \right)^{\frac{1}{q}}
          \leq \left( (n_k+1) \expect \sup_{\theta \in \torus^d} \Bigl\lvert \solall {\bigl(\theta, \dallxi{0}\bigr)} \Bigr\rvert^{pq} \right)^{\frac{1}{q}} \, .
    \end{align*}
    In order to recover the Monte Carlo convergence rate on the right-hand side,
    we use again~\cref{proposition:expect_supremum_torus},
    which gives, assuming that \(pq \geq d+1\) (this is not restrictive as we always can take \(q\) sufficiently large),
    that
    \[
        \expect \sup_{\theta \in \torus^d}\Bigl\lvert \solall\bigl(\theta, \dallxi{\ell \nu}\bigr) \Bigr\rvert^{pq}
        \lesssim_{p,q,d}
        J^{-\frac{pq}{2}} \biggl( \int_{\torus^d} \expect \Bigl\lvert \solone \bigl(\theta, \dxin{1}{0} \bigr) \Bigr\rvert^{pq}
            + \expect \Bigl\lvert  \nabla_{\theta} \solone \bigl(\theta, \dxin{1}{0} \bigr) \Bigr\rvert^{pq}
        \, \d \theta \biggr) \, ,
    \]
    Substituting this inequality in the previous display,
    and using the stationarity of the fast processes,
    we finally deduce that the following inequality holds for all \(k \in \left[0, \frac{T}{h}\right]\):
    \begin{align*}
        \expect {\left[ \max_{\ell \in \range{0}{k}} \bigl\lvert \mathcal I^3_\ell \bigr\rvert^p \right]}
          & \lesssim_{p,q,d}
          (n_k+1)^{\frac{1}{q}} J^{-\frac{p}{2}}
          \biggl( \int_{\torus^d} \expect \Bigl\lvert \solone \bigl(\theta, \dxin{1}{0} \bigr) \Bigr\rvert^{pq}
              + \expect \Bigl\lvert  \nabla_{\theta} \solone \bigl(\theta, \dxin{1}{0} \bigr) \Bigr\rvert^{pq}
          \, \d \theta \biggr)^{\frac{1}{q}} \, .
    \end{align*}
    By~\cref{lemma:ou_semigroup} and~\cref{lemma:discrete_poisson_eq},
    recalling that \(\weight{m}(\xi) = 1 + |\xi|^m\),
    we have
    \begin{equation}
        \label{eq:bound_discrete_poisson}
        \left\lVert \left( \frac{\e^{\frac{h}{\varepsilon} \genfastone} - \id}{h} \right)^{-1} \right\rVert_{L^{\infty}_{\weight{m},0}}
        \lesssim
        \frac{h}{1 - \e^{- \frac{h}{\varepsilon}}}
        \leq
        \varepsilon + h \, ,
    \end{equation}
    where we used in the last inequality the elementary bound~\eqref{eq:elem_1},
    which gives that \(\frac{x}{1-\e^{-x}} \leq 1 + x\) for all \(x > 0\).
    Since \(\solone(\theta, \placeholder)\) is the solution to the discrete Poisson equation~\eqref{eq:discrete_poisson},
    the bound~\eqref{eq:bound_discrete_poisson} together with~\cref{lemma:upper_bound_drift_and_gradient} gives that for all \(\theta \in \torus^d\),
    \[
        \Bigl\lVert \solone \bigl(\theta, \placeholder \bigr) \Bigr\rVert_{L^{\infty}_{\weight{2}}}
        \lesssim
        (\varepsilon + h)
        \Bigl\lVert \Pi \driftone \bigl(\theta, \placeholder \bigr) \Bigr\rVert_{L^{\infty}_{\weight{2}}} \, ,
    \]
    and the same bound holds for the \(\theta\)-derivatives of \(\solone\) up to order 5,
    with on the right-hand side the corresponding derivatives of \(\Pi \driftone\).
    This can be proved by following the same reasoning as in the proof of the corresponding derivative identity
    \ifsiamart in the main paper\else\fi
    in~\eqref{eq:theta_deriv_solone}.
    This, together with the inequality \(n_k+1\leq 1+T\min\left\{\frac1h,\frac1\varepsilon\right\}
    \leq(1+T)\min\left\{\frac1h,\frac1\varepsilon\right\}\), since~$h,\varepsilon\leq1$, leads to
    \[
        \expect {\left[ \max_{\ell \in \range{0}{k}} \bigl\lvert \mathcal I^3_\ell \bigr\rvert^p \right]}
        \lesssim_{p,d,T} \min \left\{ \frac{1}{h}, \frac{1}{\varepsilon} \right\}^{\frac{1}{q}} J^{-\frac{p}{2}} (\varepsilon + h)^p
        \lesssim_{p,d,T} J^{-\frac{p}{2}} \left(\varepsilon^{p - \frac{1}{q}} + h^{p - \frac{1}{q}}\right) \, .
    \]
    Since \(q\) was arbitrary,
    we can ensure that \(p - \frac{1}{q} \geq \frac{p}{2}\) by taking~\(q\) sufficiently large.

    \paragraph{Bound on the fourth term in~\eqref{eq:main_equation_strong_convergence}}
    This term is the one that requires the most effort to bound,
    compared to the proof in continuous time.
    First let
    \[
        \sum_{\ell=0}^{k-1} V_\ell \coloneq
        \sum_{\ell=0}^{k-1} \Bigl( \solall{\left(\dtheta{\ell+1}, \dallxi{\ell+1}\right)} - \solall{\left(\dtheta{\ell}, \dallxi{\ell+1}\right)} \Bigr)
    \]
    and set~$a_\ell \coloneq \driftall(\dtheta{\ell},\dallxi{\ell})$.
    Notice that
    \[
        \dtheta{\ell+1}-\dtheta{\ell}=h a_\ell+\sqrt{2h}\,x_\ell \, .
    \]
    In particular, the drift in this increment is evaluated at~$\dallxi{\ell}$,
    whereas the derivatives of the corrector in~$V_\ell$ are evaluated at~$\dallxi{\ell+1}$.
    Taylor expansion with respect to the slow variable therefore gives
    \begin{align}
        \notag
         V_\ell
         =\,& \sqrt{2h}\nabla_{\theta} \solall(\dtheta{\ell}, \dallxi{\ell+1})x_\ell
         + h\Bigl(\nabla_{\theta}\solall{\left(\dtheta{\ell}, \dallxi{\ell+1}\right)} \cdot a_\ell + \nabla_{\theta}^2 \solall{\left(\dtheta{\ell}, \dallxi{\ell+1}\right)}:{\left(x_\ell\otimes x_\ell\right)}\Bigr) \\
        \notag
            &+ \frac{h^{\frac{3}{2}}}{\sqrt 2} \nabla_{\theta}^2\solall{\left(\dtheta{\ell}, \dallxi{\ell+1}\right)}:\Bigl(a_\ell\otimes x_\ell+x_\ell\otimes a_\ell\Bigr)  \\
           \label{eq:difficult_terms_strong}
            &+ \frac{h^2}{2}\nabla_{\theta}^2\solall(\dtheta{\ell}, \dallxi{\ell+1}):\bigl(a_\ell\otimes a_\ell\bigr)+ R_\ell \, ,
    \end{align}
    where the remainder term satisfies, by Taylor's theorem,
    \[
         \lvert R_\ell \rvert
         \lesssim \max_{\theta \in \torus^d} \bigl\lVert \nabla_{\theta}^3 \solall{\left(\theta, \dallxi{\ell+1}\right)} \bigr\rVert
         \left\lvert \dtheta{\ell+1} - \dtheta{\ell} \right\rvert^3 \, .
    \]
    Consider the decomposition \(V_\ell = \mathcal V_\ell + \mathcal W_\ell + R_\ell\),
    where
    \begin{align*}
        \mathcal V_\ell &=  h \left(\nabla_{\theta}\solall(\dtheta{\ell}, \dallxi{\ell+1})\cdot a_\ell
        + \Delta_{\theta}\solall(\dtheta{\ell}, \dallxi{\ell+1})\right)
        + \frac{h^2}{2}\nabla_{\theta}^2\solall{\left(\dtheta{\ell}, \dallxi{\ell+1}\right)}:\bigl(a_\ell\otimes a_\ell\bigr) \, , \\
        \mathcal W_\ell &=  \sqrt{2h}\nabla_{\theta}\solall{\left(\dtheta{\ell}, \dallxi{\ell+1}\right)}x_\ell
        + h \nabla_{\theta}^2\solall{\left(\dtheta{\ell}, \dallxi{\ell+1}\right)}: {\left( x_\ell\otimes x_\ell - \I_d \right)} \\
                     &\qquad + \sqrt{2} h^{\frac{3}{2}}\nabla_{\theta}^2 \solall{\left(\dtheta{\ell}, \dallxi{\ell+1}\right)}:\bigl(a_\ell\otimes x_\ell\bigr)
                 \eqcolon \sqrt{h} \mathcal W_{\ell,1} + h \mathcal W_{\ell,2} + h^{\frac{3}{2}} \mathcal W_{\ell,3} \, .
    \end{align*}
    Before bounding these terms,
    we note that by~\cref{proposition:expect_supremum_torus},
    the discrete resolvent estimate~\eqref{eq:bound_discrete_poisson},
    and stationarity of the fast processes,
    it holds for every~$q\geq2$ and~$r\in\{1,2,3\}$,
    \begin{equation}
        \label{eq:corrector_moment_strong_discrete}
        \left(
        \expect\left[
        \sup_{\theta\in\torus^d}
        \left|\nabla_{\theta}^r\solall(\theta,\dallxi{\ell+1})\right|^q
        \right]
        \right)^{\frac1q}
        \lesssim_{q,d}\frac{\varepsilon+h}{\sqrt J} \, .
    \end{equation}
    Indeed, $\solone(\theta,\placeholder)$ and all its~$\theta$-derivatives are centered under~$\gamma_d$.
    On the other hand,~\cref{lemma:upper_bound_drift_and_gradient} gives
    \begin{equation}
        \label{eq:drift_moment_strong_discrete}
        \sup_{\ell\geq0}\expect|a_\ell|^q\lesssim_{q,d}1 \, .
    \end{equation}
    Although~$\dallxi{\ell}$ and~$\dallxi{\ell+1}$ are correlated,
    Cauchy--Schwarz applied with sufficiently high moments in~\cref{eq:corrector_moment_strong_discrete,eq:drift_moment_strong_discrete}
    therefore shows that, for the fixed powers~$m\geq0$ used below,
    \begin{equation}
        \label{eq:mixed_moment_strong_discrete}
        \expect\left[
        \sup_{\theta\in\torus^d}
        \left|\nabla_{\theta}^r\solall(\theta,\dallxi{\ell+1})\right|^q
        (1+|a_\ell|^{mq})
        \right]
        \lesssim_{q,d,m}\frac{(\varepsilon+h)^q}{J^{\frac q2}} \, .
    \end{equation}

    It remains to bound the contributions of the three terms in the decomposition of \(V_\ell\).
    \begin{itemize}
        \item
            For the first term,
            we use convexity of \(x \mapsto |x|^p\) and
            the mixed-moment estimate~\eqref{eq:mixed_moment_strong_discrete} to obtain
            \begin{align}
                \notag
                \expect\left[\max_{m \in \range{0}{k}}  \left\lvert \sum_{\ell=0}^{m-1} \mathcal V_{\ell} \right\rvert^p\right]
                &\leq k^{p-1} \sum_{\ell=0}^{k-1} \expect \Bigl[ \lvert \mathcal V_\ell \rvert^p \Bigr] \\
                &\lesssim_{p,d,T} \frac{(\varepsilon+h)^p}{J^{\frac{p}{2}}} \, .
                \label{eq:bound_predictable_strong_discrete}
            \end{align}

        \item
            To bound the second term,
            let~$\mathcal F_\ell$ be the sigma-algebra generated by the initial condition
            and all the slow and fast Gaussian innovations with indices strictly smaller than~$\ell$.
            Let~$\mathcal G_\ell$ be obtained by adjoining to~$\mathcal F_\ell$ the fast Gaussian innovations at step~$\ell$.
            Then~$a_\ell$ is~$\mathcal F_\ell$-measurable,
            $\dallxi{\ell+1}$ is~$\mathcal G_\ell$-measurable,
            and~$x_\ell$ is independent of~$\mathcal G_\ell$.
            It follows that
            \[
                \expect[\mathcal W_\ell\mid\mathcal G_\ell]=0 \, ,
                \qquad
                \expect[\mathcal W_\ell\mid\mathcal F_\ell]=0 \, .
            \]
            Consequently, the partial sums of~$\mathcal W_\ell$ form a martingale with respect to~$(\mathcal F_\ell)$.
            Therefore, by the discrete Burkholder--Davis--Gundy inequality
            and Jensen's inequality,
            it holds for all \(p \geq 2\) that
            \begin{align*}
                \expect {\left[\sup_{\ell \in \range{0}{k}} \left\lvert \sum_{r=0}^{\ell-1} \mathcal W_r  \right\rvert^p \right]}
                &\leq C_{\rm BDG}(p) \expect {\left[ \left(\sum_{\ell=0}^{k-1} |\mathcal W_\ell|^2 \right)^{\frac{p}{2}} \right]}
                \leq C_{\rm BDG} k^{\frac{p}{2}-1} \sum_{\ell=0}^{k-1} \expect \Bigl[ |\mathcal W_\ell|^{p} \Bigr]\\
                &\lesssim_{p} k^{\frac{p}{2}-1}  \sum_{\ell=0}^{k-1}
                \biggl(
                    h^{\frac{p}{2}} \expect \Bigl[ |\mathcal W_{\ell,1}|^{p} \Bigr]
                    + h^{p}  \expect \Bigl[ |\mathcal W_{\ell,2}|^{p} \Bigr]
                    + h^{\frac{3p}{2}} \expect \Bigl[ |\mathcal W_{\ell,3}|^{p} \Bigr]
                \biggr)
                \, .
            \end{align*}
            The three terms are treated in exactly the same manner,
            so focus only on the third one.
            By the tower property of conditional expectation and the independence of~$x_\ell$ from~$\mathcal G_\ell$,
            it holds that
            \begin{align*}
                \expect  \left\lvert \nabla_{\theta}^2 \solall{\left(\dtheta{\ell}, \dallxi{\ell+1}\right)}:\bigl(a_\ell\otimes x_\ell\bigr)  \right\rvert^p
                &= \expect {\left[ \expect {\left[ \left\lvert \nabla_{\theta}^2 \solall{\left(\dtheta{\ell}, \dallxi{\ell+1}\right)}:\bigl(a_\ell\otimes x_\ell\bigr)  \right\rvert^p \, \middle| \, \mathcal G_\ell \right]}  \right]} \\
                &\lesssim_d \expect \Bigl[ \left\lvert \nabla_{\theta}^2 \solall{\left(\dtheta{\ell}, \dallxi{\ell+1}\right)} \right\rvert^p |a_\ell|^p \Bigr] \\
                &\lesssim_d \left( \expect  \left\lvert \nabla_{\theta}^2 \solall{\left(\dtheta{\ell}, \dallxi{\ell+1}\right)} \right\rvert^{2p} \right)^{\frac{1}{2}}
                \left( \expect |a_\ell|^{2p} \right)^{\frac{1}{2}} \, .
            \end{align*}
            The second factor on the right-hand side is bounded from above uniformly for \((\sigma, \varepsilon, J) \in (0, 1] \times (0, 1] \times \nat_{>0}\),
            while~\eqref{eq:corrector_moment_strong_discrete}, with~$q=2p$ and~$r=2$, gives
            \[
                 \expect  \left\lvert \nabla_{\theta}^2 \solall{\left(\dtheta{\ell}, \dallxi{\ell+1}\right)} \right\rvert^{2p}
                 \lesssim_{p,d} \frac{\varepsilon^{2p} + h^{2p}}{J^{p}} \, .
            \]
            It follows that
            \[
                \expect {\left[\sup_{\ell \in \range{0}{k}} \left\lvert \sum_{r=0}^{\ell-1} \mathcal W_r  \right\rvert^p \right]}
                \lesssim_{p,d} k^{\frac{p}{2}} h^{\frac{p}{2}} \frac{\varepsilon^{p} + h^{p}}{J^{\frac{p}{2}}} \, .
            \]

        \item
            Finally,
            for the third term,
            we have
            \begin{align*}
                \expect\left[\max_{m \in \range{0}{k}} \left\lvert \sum_{r=0}^{m-1} R_r \right\rvert^p\right]
                &\leq k^{p-1} \sum_{r=0}^{k-1} \expect \biggl[\max_{\theta \in \torus^d} \bigl\lVert \nabla_{\theta}^3 \solall{\left(\theta, \dallxi{r+1}\right)} \bigr\rVert^p \left\lvert \dtheta{r+1} - \dtheta{r} \right\rvert^{3p}  \biggr]\\
                &\leq k^{p-1} \sum_{r=0}^{k-1}
                \left( \expect  \max_{\theta \in \torus^d} \bigl\lVert \nabla_{\theta}^3 \solall{\left(\theta, \dallxi{r+1}\right)} \bigr\rVert^{2p}  \right)^{\frac{1}{2}}
                \left( \expect  \left\lvert \dtheta{r+1} - \dtheta{r} \right\rvert^{6p} \right)^{\frac{1}{2}} \, .
            \end{align*}
            It is here that the assumption that $G \in \cont^5(\torus^d, \real^K)$ is required,
            so that~\cref{proposition:expect_supremum_torus} can be employed with~$\nabla^3_\theta \solall$ and the resulting expression can be bounded independently of~$\sigma$.
            Using the same arguments as above, 
            we obtain that
            \begin{align*}
                \left( \expect  \max_{\theta \in \torus^d} \bigl\lVert \nabla_{\theta}^3 \solall{\left(\theta, \dallxi{r+1}\right)} \bigr\rVert^{2p}  \right)^{\frac{1}{2}}
                &\lesssim_{p,d}
                 \frac{\varepsilon^{p} + h^{p}}{J^{\frac{p}{2}}} \,, \\
                \left( \expect  \left\lvert \dtheta{r+1} - \dtheta{r} \right\rvert^{6p} \right)^{\frac{1}{2}}
                &\lesssim
                h^{\frac{3p}{2}} \, .
            \end{align*}
    \end{itemize}
    It therefore follows that
    \[
        \expect {\left[ \max_{\ell \in \range{0}{k}} \bigl\lvert \mathcal I^4_\ell \bigr\rvert^p \right]}
        \lesssim_{p,d,T} \frac{\varepsilon^{p} + h^{p}}{J^{\frac{p}{2}}} \, .
    \]

    \paragraph{Bound on the fifth term in~\eqref{eq:main_equation_strong_convergence}}
    Define the full martingale with respect to $(\mathcal F_r)$ by
    \[
        M_r=\frac{1}{J}\sum_{j=1}^J\sum_{k=0}^{r-1}
        \Bigl(\solone\bigl(\dtheta{k},\dxin{j}{k+1}\bigr)
        -\e^{\frac{h}{\varepsilon}\genfastone}\solone\bigl(\dtheta{k},\dxin{j}{k}\bigr)\Bigr) \, .
    \]
    Since $\mathcal I^5_\ell=M_{\nu n_\ell}$ and $\nu n_\ell\leq\ell$,
    \[
        \max_{0\leq\ell\leq n}|\mathcal I^5_\ell|
        \leq\max_{0\leq r\leq n}|M_r| \, .
    \]
    Applying the discrete Burkholder--Davis--Gundy inequality to $M$
    and Jensen's inequality gives, for all $p\geq2$,
    \begin{align*}
        \expect {\left[ \sup_{\ell \in \range{0}{n}} \left\lvert \mathcal I^5_\ell \right\rvert^p \right]}
        &\leq C_{\rm BDG}(p) \expect {\left(\sum_{\ell=0}^{n-1}
        \left| \frac{1}{J}\sum_{j=1}^{J} \Bigl( \solone{\left(\dtheta{\ell}, \dxin{j}{\ell+1}\right)} - \e^{\frac{h}{\varepsilon} \genfastone} \solone{\left(\dtheta{\ell}, \dxin{j}{\ell}\right)} \Bigr) \right\rvert^2 \right)}^{\frac{p}{2}}  \\
        &\lesssim_{p} n^{\frac{p}{2}-1} \sum_{\ell=0}^{n-1} \expect
          \left| \frac{1}{J}\sum_{j=1}^{J} \Bigl( \solone{\left(\dtheta{\ell}, \dxin{j}{\ell+1}\right)} - \e^{\frac{h}{\varepsilon} \genfastone} \solone{\left(\dtheta{\ell}, \dxin{j}{\ell}\right)} \Bigr) \right|^p  \\
        &\lesssim_{p,T} h^{-\frac{p}{2}}  \expect {\left[ \sup_{\theta \in \torus^d} \left| \frac{1}{J}\sum_{j=1}^{J} \Bigl( \solone{\left(\theta, \dxin{j}{\ell+1}\right)} - \e^{\frac{h}{\varepsilon} \genfastone} \solone{\left(\theta, \dxin{j}{\ell}\right)} \Bigr) \right|^p \right]} \, ,
    \end{align*}
    where we used stationarity of the fast processes and \(n^{p/2} \leq (T/h)^{p/2}\) in the last inequality.
    Let us introduce
    \[
        A(\theta, \zeta, \xi) = \solone(\theta, \xi) - \e^{\frac{h}{\varepsilon} \genfastone} \solone(\theta, \zeta) \, .
    \]
    Let \(\mu\) denote the Gaussian distribution on \(\real^d \times \real^d\) defined by \((\dxin{1}{\ell}, \dxin{1}{\ell+1}) \sim \mu\).
    The random vectors \((\dxin{j}{\ell}, \dxin{j}{\ell+1})_{j \in \range{1}{J}}\) are independent and identically distributed with common distribution~\(\mu\).
    Furthermore \(\expect_{(\zeta,\xi) \sim \mu} A(\theta, \zeta, \xi) = 0\) for all~\(\theta \in \torus^d\),
    and so by~\cref{proposition:expect_supremum_torus} we deduce that
    with $q=\max\{p,d+1\}$,
    \begin{align*}
        \expect {\left[ \sup_{\ell \in \range{0}{n}} \left\lvert \mathcal I^5_\ell \right\rvert^p \right]}
        &\lesssim_{p,T}
        h^{-\frac{p}{2}}
        \expect {\left[ \sup_{\theta \in \torus^d} \left| \frac{1}{J} \sum_{j=1}^{J} A\bigl(\theta, \dxin{j}{\ell}, \dxin{j}{\ell+1}\bigr) \right|^p \right]} \\
        &\lesssim_{p,T} h^{-\frac{p}{2}} J^{-\frac{p}{2}} \biggl( \int_{\torus^d} \expect \Bigl\lvert A\bigl(\theta, \dxin{j}{\ell}, \dxin{j}{\ell+1}\bigr) \Bigr\rvert^{q}
            + \expect \Bigl\lvert \nabla_{\theta} A\bigl(\theta, \dxin{j}{\ell}, \dxin{j}{\ell+1}\bigr) \Bigr\rvert^{q}
        \, \d \theta \biggr)^{\frac{p}{q}} \, .
    \end{align*}
    This gives us the desired decay with respect to~\(J\).
    To conclude, we consider the terms in the integrand,
    focusing on the second as both terms are treated in exactly the same manner.
    Fix an arbitrary \(\theta \in \torus^d\) and \(j \in \range{1}{J}\),
    and notice that
    \[
        \expect \Bigl\lvert \nabla_{\theta} A\bigl(\theta, \dxin{j}{\ell}, \dxin{j}{\ell+1}\bigr) \Bigr\rvert^{q}
        = \expect \Bigl\lvert \nabla_{\theta} \solone(\theta, \dxin{j}{\ell+1}) - \e^{\frac{h}{\varepsilon} \genfastone} \nabla_{\theta} \solone(\theta, \dxin{j}{\ell}) \Bigr\rvert^q
    \]
    It is helpful to distinguish two cases:
    \begin{itemize}
        \item
            If \(h \geq \varepsilon\),
            then a sufficiently good bound is obtained by noting that
            \begin{equation}
                \label{eq:case_easy}
                \expect \Bigl\lvert \nabla_{\theta} \solone(\theta, \dxin{j}{\ell+1}) - \e^{\frac{h}{\varepsilon} \genfastone} \nabla_{\theta} \solone(\theta, \dxin{j}{\ell}) \Bigr\rvert^q
                \leq
                2^{q-1} \expect \Bigl\lvert \nabla_{\theta} \solone(\theta, \dxin{j}{\ell+1}) \Bigr\rvert^q +
                2^{q-1} \expect \Bigl\lvert \e^{\frac{h}{\varepsilon} \genfastone} \nabla_{\theta} \solone(\theta, \dxin{j}{\ell}) \Bigr\rvert^q \, .
            \end{equation}
            By definition of the semigroup \(t \mapsto \e^{t \genfastone}\) and
            by Jensen's inequality for conditional expectation,
            the second term on the right-hand side satisfies
            \begin{align*}
                \expect \Bigl\lvert \e^{\frac{h}{\varepsilon} \genfastone} \nabla_{\theta} \solone(\theta, \dxin{j}{\ell}) \Bigr\rvert^q
                &=
                \expect {\left[ \Bigl\lvert \expect \Bigl[ \nabla_{\theta} \solone(\theta, \dxin{j}{\ell+1}) \,\Bigm|\, \dxin{j}{\ell} \Bigr] \Bigr\rvert^q \right]} \\
                &\leq
                \expect {\left[ \expect {\left[ \Bigl\lvert \nabla_{\theta} \solone(\theta, \dxin{j}{\ell+1}) \Bigr\rvert^q\,\Bigm|\, \dxin{j}{\ell} \right]} \right]}
                =
                \expect  \bigl\lvert \nabla_{\theta} \solone(\theta, \dxin{j}{\ell+1}) \Bigr\rvert^q \, .
                % \expect \left[ \Bigl\lvert \nabla_{\theta} \solone(\theta, \dxin{j}{\ell+1}) \Bigl\lvert^q\,\Bigm|\, \dxin{j}{\ell} \Bigr] \right]
            \end{align*}
            The right-hand side may be bounded from above uniformly by \((\varepsilon + h)^q\),
            up to a constant factor independent of the parameters, by \eqref{eq:bound_discrete_poisson}.

        \item
            In the case \(h < \varepsilon\),
            we wish to show that the left-hand side of~\eqref{eq:case_easy} scales as \(\varepsilon^{\frac{q}{2}} h^{\frac{q}{2}}\).
            Since the fast operator does not depend on~$\theta$, differentiating~\eqref{eq:discrete_poisson} gives
            \[
                -\frac{\e^{\frac{h}{\varepsilon}\genfastone}-\id}{h}\nabla_\theta\solone
                =\nabla_\theta\Pi\driftone \, ,
            \]
            where we recall that~$\Pi\driftone(\theta,\xi)=\driftone(\theta,\xi)-\driftinf(\theta)$.
            Differentiating~\eqref{eq:drift_1_integral}, using~$G\in\cont^3(\torus^d,\real^K)$ and~$\sigma\leq1$, shows that
            both~$\Pi\driftone(\theta,\placeholder)$ and~$\nabla_\theta\Pi\driftone(\theta,\placeholder)$ have~$\sobolev^{1,\infty}_{\weight{2}}$ norms bounded uniformly over~$\theta\in\torus^d$ and~$\sigma\in(0,1]$.
            For each fixed~$\theta$, both have mean zero with respect to~$\fastgaussianone(\d\xi)$.
            The weighted norm is defined in~\eqref{eq:weighted_sobolev_norm} and applied componentwise in~$\xi$.
            Multiplying the Poisson equations by~$\varepsilon$ and applying~\cref{lemma:discrete_poisson_eq} componentwise with time step~$\frac{h}{\varepsilon}\in(0,1)$ therefore yields
            \[
                \sup_{\theta\in\torus^d}\Bigl(
                \|\nabla_\xi\solone(\theta,\placeholder)\|_{L^\infty_{\weight{2}}}
                +\|\nabla_\xi\nabla_\theta\solone(\theta,\placeholder)\|_{L^\infty_{\weight{2}}}
                \Bigr)\leq C\varepsilon \, .
            \]
            Here $C$ is independent of $(\sigma,\varepsilon,h,J)$.
            The mean value theorem gives a constant $L$ independent of these parameters such that
            \[
                \forall (\theta, \xi, \zeta) \in \torus^d \times \real^d \times \real^d, \qquad
                \bigl\lvert \nabla_{\theta} \solone(\theta, \xi) - \nabla_{\theta} \solone(\theta, \zeta) \bigr\rvert
                \leq L \varepsilon \Bigl(1 + |\xi|^2 + |\zeta|^2\Bigr) \bigl\lvert \xi - \zeta \bigr\rvert \, .
            \]
            The same bound holds without $\nabla_\theta$.
            % It follows that, for independent random variables \(Y, Z \sim \normal(0, I_d)\) and \(\alpha \in (0, 1]\),
            % \begin{align*}
            %     \forall  (\theta, \xi) \in \torus^d \times \real^d, \qquad
            %     \biggl\lvert \nabla_{\theta} \solone(\theta, \xi + \alpha Y)  - \expect \Bigl[ \nabla_{\theta} \solone(\theta, \xi + \alpha Z) \Bigr] \biggr\rvert^q
            %     &\lesssim
            %     \expect \Bigl[ L \left( 1 + \lvert \xi + \alpha Y \rvert^q + \lvert \xi + \alpha Z \rvert^q \right) \lvert \alpha Z + \xi - \xi \rvert^q  \Bigr] \\
            %     &\lesssim_{d,L} \delta^{2q} \left(  \alpha^q + \lvert \xi - \xi \rvert^q \right) \bigl( 1 + \lvert \xi \rvert^q + \lvert \xi \rvert^q \bigr).
            % \end{align*}
            Thus, since
            \(
                \dxin{1}{\ell+1} = \alpha \dxin{1}{\ell} + \sqrt{1 - \alpha^2} \xfast{1}{\ell}
            \)
            with
            \(
                \alpha \coloneq \e^{- \frac{h}{\varepsilon}}
            \)
            and \(\beta = \sqrt{1 - \alpha^2}\),
            using an independent standard Gaussian copy denoted by~$\xfast{2}{\ell}$,
            we deduce that
            \begin{align*}
                \expect \Bigl\lvert \nabla_{\theta} \solone(\theta, \dxin{1}{\ell+1}) - \e^{\frac{h}{\varepsilon} \genfastone} \nabla_{\theta} \solone(\theta, \dxin{1}{\ell}) \Bigr\rvert^q
                &=
                \expect \Bigl\lvert \nabla_{\theta} \solone(\theta, \dxin{1}{\ell+1}) - \expect {\left[ \nabla_{\theta} \solone(\theta, \dxin{1}{\ell+1}) \,\middle|\, \dxin{1}{\ell} \right]} \Bigr\rvert^q \\
                &=
                \expect \Bigl\lvert \expect \Bigl[ \nabla_{\theta} \solone\Bigl(\theta, \alpha \dxin{1}{\ell} + \beta \xfast{1}{\ell} \Bigr) - \nabla_{\theta} \solone\Bigl(\theta, \alpha \dxin{1}{\ell} + \beta \xfast{2}{\ell}\Bigr) \,\Bigm|\, \dxin{1}{\ell}, \xfast{1}{\ell} \Bigr] \Bigr\rvert^q \\
                &\leq
                L^q \beta^q  \varepsilon^{q}
                \expect \Bigl\lvert \expect \Bigl[ \Bigl(1 + \left\lvert \alpha \dxin{1}{\ell} + \beta \xfast{1}{\ell} \right\rvert^2 + \left\lvert \alpha \dxin{1}{\ell} + \beta \xfast{2}{\ell} \right\rvert^2 \Bigr) \bigl\lvert \xfast{1}{\ell} - \xfast{2}{\ell} \bigr\rvert \,\Bigm|\, \dxin{1}{\ell}, \xfast{1}{\ell} \Bigr] \Bigr\rvert^q \\
                &\lesssim_{d,L,q}
                \beta^q  \varepsilon^{q}
                \expect \Bigl[ \Bigl(1 + \lvert \dxin{1}{\ell} \rvert^{2q} + \lvert\xfast{1}{\ell}\rvert^{2q} + \lvert\xfast{2}{\ell}\rvert^{2q}\Bigr) \bigl\lvert \xfast{1}{\ell} - \xfast{2}{\ell} \bigr\rvert^q  \Bigr] \\
                &\lesssim_{d,L,q} \varepsilon^{q} \left( \frac{h}{\varepsilon} \right)^{\frac{q}{2}}
                \lesssim_{d,L,q} \varepsilon^{\frac{q}{2}} h^{\frac{q}{2}} \, .
            \end{align*}
            where we used in the last inequality the stationarity of the fast processes together with the elementary inequality \(1 - \e^{-x} \leq x\),
            which holds for all \(x \in (0, \infty)\).
    \end{itemize}
    Combining the estimates, we deduce that
    \[
        \expect {\left[ \sup_{\ell \in \range{0}{n}} \left\lvert \mathcal I^5_\ell \right\rvert^p \right]}
        \lesssim_{p,d,T}
        \frac{\varepsilon^{\frac{p}{2}} + h^{\frac{p}{2}}}{J^{\frac{p}{2}}}
    \]

    \paragraph{Conclusion of the proof}
    Gathering the bounds on all terms,
    we obtain from~\eqref{eq:main_equation_strong_convergence} that
    \[
        E_n \lesssim  h \sum_{\ell=0}^{n-1} E_\ell +  \left(\frac{\varepsilon^{\frac{p}{2}} + h^{\frac{p}{2}}}{J^{\frac{p}{2}}} \right) \, .
    \]
    Applying the discrete Gr\"onwall lemma, we obtain that
    \[
        E_N \lesssim \left(\frac{\varepsilon^{\frac{p}{2}} + h^{\frac{p}{2}}}{J^{\frac{p}{2}}} \right) \bigl(\exp \left( N h \right)-1\bigr)
        = \left(\e^T-1\right) \left(\frac{\varepsilon^{\frac{p}{2}} + h^{\frac{p}{2}}}{J^{\frac{p}{2}}} \right) \, .
    \]
    which concludes the proof of~\cref{theorem:strong_convergence_fixed_sigma_discrete} for \(p\geq 2\).
    By Cauchy--Schwarz inequality,
    the same inequality holds for \(p\in [1,2]\),
    possibly with a different constant prefactor.
\end{proof}

\subsection{Invariant distribution and ergodic averages}
\label{sub:bound_invariant_discrete}

Let us consider the discrete-time process \(\bigl(\dtheta{n}, \dallxi{n}\bigr)_{n\geq 0}\) obtained by discretization~\eqref{eq:the_multiscale_dynamics_discrete}
with step size \(h\):
\begin{equation}
    \label{eq:discretized_multiscale_dynamics}
    \left\{
    \begin{aligned}
        \dtheta{n+1} &= \dtheta{n} + h \driftall\bigl(\dtheta{n}, \dallxi{n}\bigr)
        + \sqrt{2h} \xslow{n} \, ,  \\
        \dxin{j}{n+1} &= \e^{-\frac{h}{\varepsilon}}\dxin{j}{n} +
        \sqrt{1 - \e^{-\frac{2h}{\varepsilon}}}\xfast{j}{n} \, , \qquad \dxin{j}{0} \stackrel{\rm i.i.d.}{\sim} \normal(0,\I_d) \, ,
        \qquad j \in \range{1}{J} \, ,
    \end{aligned}
    \right.
\end{equation}
The Markov transition operator associated with this process is denoted by $\markovtransition$.
This operator acts on any bounded measurable function \(\varphi \colon \torus^d\times \R^{dJ}\to \R\) in the following way:
\[
    \markovtransition \varphi{\left(\theta,\allxi\right)} =
    \expect \Bigl[ \varphi\bigl(\dtheta{1}, \dallxi{1}\bigr) \Bigm| \bigl(\dtheta{0}, \dallxi{0}\bigr) = \bigl(\theta, \allxi\bigr) \Bigr] \, .
\]

\begin{lemma}
    [Discrete averaging lemma]
    \label{lemma:averaging_discrete}
    Assume that~$\chi \colon \real^d \to \real^n$ is a vector-valued polynomial of degree at most~$r$, for some~$r\in\N$, with mean zero under~$\fastgaussianone$,
    and also that~$A \in \cont^0(\torus^d, \real^n)$.
    Then, under the same assumptions as in~\cref{theorem:bound_invariant_discrete},
    there exists a constant~$C > 0$, depending only on~$d,n,r$ and fixed derivative bounds on~$G$ and~$\ell$ but independent of~$A,\chi,\varepsilon,\sigma,h,J,N$, such that
    the following estimates hold.
    \begin{itemize}
        \item
            It holds that
            \begin{equation}
                \label{eq:averaging_lemma_discrete_coarse}
                \left\lvert
                    \frac{1}{N} \sum_{n=0}^{N-1} \expect \left[
                        \frac{1}{J} \sum_{j=1}^{J} \chi\bigl(\dxin{j}{n}\bigr) \cdot A\bigl(\dtheta{n}\bigr)
                    \right]
                \right\rvert
                \leq \frac{C}{\sqrt{J}} \norm{A}_{\infty} \norm{\chi}_{L^2(\fastgaussianone)} \, .
            \end{equation}

        \item
            If additionally~$A \in \cont^{2+\alpha}(\torus^d, \real^n)$,
            then we have the improved estimates
            \begin{subequations}
            \begin{equation}
                \label{eq:averaging_lemma_discrete}
                \left\lvert
                    \frac{1}{N} \sum_{n=0}^{N-1} \expect \left[
                        \frac{1}{J} \sum_{j=1}^{J} \chi\bigl(\dxin{j}{n}\bigr) \cdot A\bigl(\dtheta{n}\bigr)
                    \right]
                \right\rvert
                \leq
                 \frac{C(h+\varepsilon)}{\sqrt{J}} \left(\frac{\kappa}{Nh} \norm{A}_{\infty} + \norm{A}_{\cont^{2+\alpha}} \right)\norm{\chi}_{L^2(\fastgaussianone)} \, ,
            \end{equation}
            where~$\kappa = 0$ for stationary initial conditions and~$\kappa = 1$ otherwise.
            Moreover,
            \begin{align}
                \notag
                &\expect\left|\frac1N\sum_{n=0}^{N-1}
                \frac1J\sum_{j=1}^J\chi(\dxin{j}{n})\cdot A(\dtheta{n})\right|^2 \\
                &\qquad \qquad\leq C\norm{\chi}_{L^2(\fastgaussianone)}^2\Biggl(
                    \frac{(h+\varepsilon)^2}{J(Nh)^2}\norm A_\infty^2
                    +\frac{(h+\varepsilon)^2}{J}\norm A_{\cont^{2+\alpha}}^2
                +\frac{h+\varepsilon}{JNh}\norm A_\infty^2
            +\frac{(h+\varepsilon)^2}{JNh}\norm A_{1,\infty}^2\Biggr) \, .
            \label{eq:averaging_lemma_discrete_mean_square}
            \end{align}
            \end{subequations}
    \end{itemize}
\end{lemma}
\begin{proof}[Proof of~\cref{lemma:averaging_discrete}]
    Similarly to the proof of the continuous-time averaging lemma~\cref{lemma:averaging},
    we present the proof only for the scalar case~$n = 1$ for simplicity.

    \paragraph{Proof of~\eqref{eq:averaging_lemma_discrete_coarse}}
    The first estimate~\eqref{eq:averaging_lemma_discrete_coarse} can be proved in the same way as the first estimate in~\cref{lemma:averaging}.
    We omit the details for brevity.

    \paragraph{Proof of~\eqref{eq:averaging_lemma_discrete}}
    Let~$\varphi \colon \real^d \to \real$ be the unique solution in~$L^2_0(\gamma_d)$ to the discrete Poisson equation
    \begin{equation}
        \label{eq:poisson_discrete_averaging}
        - \left(\frac{\e^{\frac{h}{\varepsilon}\genfastone} - \id}{h}\right) \varphi = \chi \, ,
        \qquad \int_{\real^d} \varphi \, \d\fastgaussianone = 0 \, .
    \end{equation}
    Since the eigenfunctions of the generator~$\genfastone$ are Hermite polynomials,
    the solution~$\varphi$ can be written explicitly in terms of the Hermite expansion of~$\chi$.
    In particular, the function~$\varphi$ is a polynomial of the same degree as~$\chi$.
    By the results in~\cref{sub:decay_ou},
    we deduce that
    \begin{equation}
        \label{eq:bound_varphi_l2_discrete}
        \lVert \varphi \rVert_{L^{2}(\fastgaussianone)}
        \lesssim \frac{h}{1 - \e^{-\frac{h}{\varepsilon}}} \, \norm{\chi}_{L^{2}(\fastgaussianone)}
        \lesssim (h+\varepsilon) \norm{\chi}_{L^{2}(\fastgaussianone)} \, ,
    \end{equation}
    where we used the elementary inequality~\eqref{eq:elem_1} in the last line.
    Strictly speaking, the results in this appendix are stated in weighted~$L^{\infty}$,
    but the same approach yields the~$L^2$ bound given by~\eqref{eq:bound_varphi_l2_discrete}.
    For convenience we introduce~$H^J{\left(\theta, \allxi\right)} =  \frac{1}{J} \sum_{j=1}^{J} \varphi(\xin{j}) \cdot A(\theta) $.
    We have the decomposition
    \begin{align*}
        &\expect \Bigl[ H^J\bigl(\dtheta{N}, \dallxi{N}\bigr) - H^J\bigl(\dtheta{0},\dallxi{0}\bigr) \Bigr] \\
        % &\qquad = \sum_{n=0}^{N-1} \expect \Bigl[  H^J{\left(\dtheta{n+1}, \dallxi{n+1}\right)} - H^J\bigl(\dtheta{n}, \dallxi{n}\bigr)  \, \Bigr] \\
        &\qquad = \sum_{n=0}^{N-1} \expect \Bigl[ H^J{\left(\dtheta{n+1}, \dallxi{n+1}\right)} - H^J\bigl(\dtheta{n}, \dallxi{n+1}\bigr)  \, \Bigr]
        + \sum_{n=0}^{N-1} \expect \Bigl[ \expect \left[H^J\bigl(\dtheta{n}, \dallxi{n+1}\bigr) \Bigm| \bigl( \dtheta{n}, \dallxi{n} \bigr) \right] - H^J{\left(\dtheta{n}, \dallxi{n}\right)}  \, \Bigr] \, \\
        &\qquad = \sum_{n=0}^{N-1} \expect \Bigl[ H^J\bigl(\dtheta{n+1}, \dallxi{n+1}\bigr) - H^J\bigl(\dtheta{n}, \dallxi{n+1}\bigr)  \, \Bigr] \,
        - h \sum_{n=0}^{N-1} \expect \Biggl[ \frac{1}{J} \sum_{j=1}^{J} \chi\bigl(\dxin{j}{n}\bigr) \cdot A\bigl(\dtheta{n}\bigr)  \, \Biggr] \, ,
    \end{align*}
    It follows that
    \begin{align}
        \notag
            \frac{1}{N} \sum_{n=0}^{N-1} \expect\biggl[ \frac{1}{J} \sum_{j=1}^{J} \chi\bigl(\dxin{j}{n}\bigr) \cdot A\bigl(\dtheta{n}\bigr) \biggr]
        &= -\frac{1}{Nh} \expect \Bigl[ H^J\bigl(\dtheta{N}, \dallxi{N}\bigr) - H^J\bigl(\dtheta{0},\dallxi{0}\bigr) \Bigr] \\
        &\qquad + \frac{1}{Nh} \sum_{n=0}^{N-1} \expect \Bigl[ H^J\bigl(\dtheta{n+1}, \dallxi{n+1}\bigr) - H^J\bigl(\dtheta{n}, \dallxi{n+1}\bigr)  \, \Bigr] \, .
        \label{eq:two_terms_averaging_discrete}
    \end{align}
    The first term on the right-hand side is zero for stationary initial conditions,
    and can be bounded by using~\eqref{eq:bound_varphi_l2_discrete} in the general case as follows:
    \begin{align}
        \notag
        \left\lvert \frac{1}{Nh} \expect \Bigl[ H^J\bigl(\dtheta{N}, \dallxi{N}\bigr) - H^J\bigl(\dtheta{0},\dallxi{0}\bigr) \Bigr] \right\rvert
        &\leq \frac{2}{Nh} \norm{A}_\infty \expect \left\lvert \frac{1}{J} \sum_{j=1}^{J} \varphi(\dxin{j}{0}) \right\rvert
        \\
        &\lesssim \frac{h+\varepsilon}{Nh\sqrt{J}} \norm{A}_\infty \norm{\chi}_{L^{2}(\fastgaussianone)} \, ,
        \label{eq:bound_term_1_lemma_averaging_discrete}
    \end{align}
    where we used the bound~\eqref{eq:bound_varphi_l2_discrete} in the last line.
    The second term on the right-hand side of~\eqref{eq:two_terms_averaging_discrete} can be controlled by using a Taylor expansion:
    for any~$(\theta, \allxi, \widetildeallxi, x) \in \torus^d \times \R^{dJ} \times \R^{dJ} \times \R^d$,
    where~$\widetildeallxi=(\widetilde\xi^1,\dotsc,\widetilde\xi^J)$,
    and any~$h > 0$, it holds that
    \begin{align*}
        &H^J\Bigl(\theta + h\driftall\bigl(\theta,\allxi\bigr)+\sqrt{2h}\,x, \widetildeallxi \Bigr) - H^J\bigl(\theta, \widetildeallxi\bigr)  \\
        &\qquad = \nabla_\theta H^J\bigl(\theta,\widetildeallxi\bigr)\cdot \left(h \driftall(\theta,\allxi) + \sqrt{2h}\,x\right)
        + \frac{1}{2} \nabla_{\theta}^2 H^J{\left(\theta,\widetildeallxi\right)} : \left[h \driftall{\left(\theta,\allxi\right)} + \sqrt{2h}\,x\right]^{\otimes 2}
        + \mathcal R\bigl(\theta,\allxi,\widetildeallxi,x,h\bigr) \, ,
    \end{align*}
    where~$\mathcal R\bigl(\theta,\allxi,\widetildeallxi,x,h\bigr)$ is the remainder term in the Taylor expansion.
    From this expansion, it follows that
    \begin{align}
        \notag
        &H^J\Bigl(\theta + h\driftall\bigl(\theta,\allxi\bigr)+\sqrt{2h}\,x, \widetildeallxi \Bigr) - H^J\bigl(\theta, \widetildeallxi\bigr)  \\
        \notag
        &\qquad = h \Bigl( \nabla_\theta H^J\bigl(\theta,\widetildeallxi\bigr)\cdot  \driftall\bigl(\theta,\allxi\bigr) + \nabla_{\theta}^2 H^J\bigl(\theta,\widetildeallxi\bigr) : (x \otimes x) \Bigr)
        + \sqrt{h} \nabla_{\theta} H^J\bigl(\theta,\widetildeallxi\bigr) \cdot \sqrt{2}\,x
        \\
        &\qquad\qquad
        + \frac{\sqrt{2}h^{\frac{3}{2}}}{2} \nabla_{\theta}^2 H^J\bigl(\theta,\widetildeallxi\bigr) : \Bigl(\driftall\bigl(\theta,\allxi\bigr) \otimes x + x \otimes \driftall\bigl(\theta,\allxi\bigr)\Bigr)
        \notag \\
        &\qquad\qquad
        + \frac{h^2}{2}\nabla_\theta^2 H^J\bigl(\theta,\widetildeallxi\bigr):\driftall\bigl(\theta,\allxi\bigr)^{\otimes2}
        + \mathcal R\bigl(\theta,\allxi,\widetildeallxi,x,h\bigr) \, .
        \label{eq:taylor_expansion_H_discrete}
    \end{align}
    The remainder term \(\mathcal R\) is given by
    \[
        \mathcal R\bigl(\theta,\allxi,\widetildeallxi,x,h\bigr)
        = \int_0^1 (1-t) \left( \nabla_{\theta}^2 H^J\Bigl(\theta + t\Bigl(h\driftall(\theta,\allxi)+\sqrt{2h}\,x\Bigr),\widetildeallxi\Bigr) - \nabla_{\theta}^2 H^J(\theta,\widetildeallxi) \right)
         : \Bigl[h\driftall(\theta,\allxi)+\sqrt{2h}\,x\Bigr]^{\otimes 2} \, \d t \, .
    \]
    This form of the remainder is more tedious than the mean-value or integral forms,
    but using it is necessary to avoid requiring boundedness of third-order derivatives of~$A$,
    just as in the proof of Itô's formula for~$\cont^2$ functions.
    It follows from this expression that the remainder satisfies the bound
    \begin{equation}
        \label{eq:bound_remainder_varphi_discrete}
        \bigl\lvert \mathcal R(\theta,\allxi,\widetildeallxi,x,h) \bigr\rvert
        \lesssim
        \omega \Bigl(\bigl\lvert h\driftall(\theta,\allxi)+\sqrt{2h}\,x \bigr\rvert \Bigr)
        \Bigl\lvert h\driftall(\theta,\allxi)+\sqrt{2h}\,x \Bigr\rvert^2 \cdot \biggl\lvert \frac{1}{J} \sum_{j=1}^{J} \varphi\Bigl(\widetilde\xi^j\Bigr) \biggr\rvert\, ,
    \end{equation}
    where $\omega$ is the modulus of continuity of~$\nabla^2 A$.
    Since $A \in \cont^{2+\alpha}(\torus^d)$ by assumption,
    the modulus of continuity~$\omega$ satisfies
    \[
        \omega(s)  \lesssim  s^{\alpha} \lVert A \rVert_{\cont^{2+\alpha}} \qquad \text{for all } s \geq 0 \, .
    \]
    Let~$\mathcal F_n$ be the natural filtration of the chain and set~$\mathcal G_n=\sigma\bigl(\mathcal F_n,\dallxi{n+1}\bigr)$.
    The slow Gaussian increment~$x=\xslow{n}$ is independent of~$\mathcal G_n$.
    Applying~\eqref{eq:taylor_expansion_H_discrete} with~$\theta=\dtheta{n}$, $\allxi=\dallxi{n}$, and~$\widetildeallxi=\dallxi{n+1}$,
    and then taking the conditional expectation given~$\mathcal G_n$,
    we have
    \begin{align}
        \notag
        &\expect\Bigl[H^J\bigl(\dtheta{n+1},\dallxi{n+1}\bigr)-H^J\bigl(\dtheta{n},\dallxi{n+1}\bigr)\Bigm|\mathcal G_n\Bigr] \\
        \notag
        &\qquad = h\Bigl(\driftall\bigl(\dtheta{n},\dallxi{n}\bigr)\cdot\nabla A(\dtheta{n})+\laplacian A(\dtheta{n})\Bigr)
        \cdot\left(\frac1J\sum_{j=1}^J\varphi\bigl(\dxin{j}{n+1}\bigr)\right) \\
        &\qquad\quad +\frac{h^2}{2}\nabla_\theta^2H^J\bigl(\dtheta{n},\dallxi{n+1}\bigr):\driftall\bigl(\dtheta{n},\dallxi{n}\bigr)^{\otimes2}
        +\expect\Bigl[\mathcal R\bigl(\dtheta{n},\dallxi{n},\dallxi{n+1},x,h\bigr)\Bigm|\mathcal G_n\Bigr] \, .
        \label{eq:conditional_increment_H_discrete}
    \end{align}
    By the bound on~$\driftall$ in~\eqref{eq:bound_drift_bip_a}, the Gaussian moments of~$x$,
    and the stationarity of the fast variables,
    \[
        \left\lVert h\driftall(\dtheta{n},\dallxi{n})+\sqrt{2h}\,x\right\rVert_{L^{2(2+\alpha)}(\Omega)}^{2+\alpha}
        \lesssim h^{1+\frac{\alpha}{2}} \, ,
        \qquad
        \left\lVert\frac1J\sum_{j=1}^J\varphi(\dxin{j}{n+1})\right\rVert_{L^2(\Omega)}
        =\frac{\norm{\varphi}_{L^2(\fastgaussianone)}}{\sqrt J} \, .
    \]
    Taking expectations in~\eqref{eq:conditional_increment_H_discrete} and using these estimates together with~\eqref{eq:bound_remainder_varphi_discrete}
    and the Cauchy--Schwarz inequality, with~$h\leq1$, we deduce that
    \begin{align}
        \notag
        \left\lvert \expect {\Bigl[ H^J\bigl(\dtheta{n+1}, \dallxi{n+1}\bigr) - H^J\bigl(\dtheta{n}, \dallxi{n+1} \bigr) \Bigr]} \right\rvert
        &\lesssim h
        \left(1+\left( \expect {\left\lvert \driftall{\left(\dtheta{n}, \dallxi{n}\right)} \right\rvert^4} \right)^{\frac{1}{2}}\right) \norm{A}_{2,\infty}
        \left( \expect {\biggl\lvert \frac{1}{J} \sum_{j=1}^{J} \varphi{\left(\dxin{j}{n+1}\right)} \biggr\rvert^2} \right)^{\frac{1}{2}}
        \\
        \notag
        &\qquad
        + \expect \left[ \Bigl\lvert h \driftall{\left(\dtheta{n}, \dallxi{n}\right)} + \sqrt{2h}\,x \Bigr\rvert^{2+\alpha} \cdot \biggl\lvert \frac{1}{J} \sum_{j=1}^{J} \varphi{\left(\dxin{j}{n+1}\right)} \biggr\rvert \right]
        \norm{A}_{\cont^{2+\alpha}}
        \\
        &\lesssim
        \frac{h}{\sqrt{J}} \norm{A}_{2,\infty} \norm{\varphi}_{L^2(\fastgaussianone)}
        + \frac{h^{1+\frac{\alpha}{2}}}{\sqrt{J}} \norm{A}_{\cont^{2+\alpha}} \norm{\varphi}_{L^2(\fastgaussianone)}
        \, .
        \label{eq:bound_term_2_lemma_averaging_discrete}
    \end{align}
    Since~$\norm{\varphi}_{L^2(\fastgaussianone)}$ is bounded by~\eqref{eq:bound_varphi_l2_discrete},
    and the second term on the right-hand side of~\eqref{eq:two_terms_averaging_discrete} involves
    the summation~$\frac{1}{Nh}\sum_{n=0}^{N-1}$,
    which introduces a factor~$1/h$ that cancels the factor~$h$ in the one-step bound~\eqref{eq:bound_term_2_lemma_averaging_discrete},
    the desired estimate follows by combining~\eqref{eq:bound_term_1_lemma_averaging_discrete} and~\eqref{eq:bound_term_2_lemma_averaging_discrete}.

    \paragraph{Proof of~\eqref{eq:averaging_lemma_discrete_mean_square}}
    We follow the proof of~\eqref{eq:averaging_lemma_second_estimate_mean_square} in continuous time.
    Set
    \[
        H_n=H^J\bigl(\dtheta{n},\dallxi{n}\bigr) \, ,
        \qquad
        \widetilde H_{n+1}=H^J\bigl(\dtheta{n},\dallxi{n+1}\bigr) \, .
    \]
    The Poisson equation~\eqref{eq:poisson_discrete_averaging} gives the pathwise decomposition
    \begin{equation}
        \label{eq:averaging_decomposition_discrete_mse}
        \frac1N\sum_{n=0}^{N-1}\frac1J\sum_{j=1}^J
        \chi(\dxin{j}{n})\cdot A(\dtheta{n})
        =\frac{H_0-H_N}{Nh}
        +\frac1{Nh}\sum_{n=0}^{N-1}(H_{n+1}-\widetilde H_{n+1})
        +\frac1{Nh}\sum_{n=0}^{N-1}\eta^{\xi}_{n+1} \, ,
    \end{equation}
    where
    \[
        \eta^{\xi}_{n+1}
        =\widetilde H_{n+1}
        -\expect\bigl[\widetilde H_{n+1}\mid\mathcal F_n\bigr] \,.
    \]
    As in the continuous proof,
    using the Marcinkiewicz--Zygmund inequality~(\cref{lemma:mz_inequality} with~$p=4$)
    and equivalence of the~$L^2(\fastgaussianone)$ and~$L^4(\fastgaussianone)$ norms on the finite-dimensional space of polynomials of degree at most~$r$,
    we have
    \[
        \left\lVert\frac1J\sum_{j=1}^J\varphi(\dxin{j}{n+1})\right\rVert_{L^4(\Omega)}
        \lesssim\frac{\norm{\varphi}_{L^4(\fastgaussianone)}}{\sqrt J}
        \lesssim_r\frac{\norm{\varphi}_{L^2(\fastgaussianone)}}{\sqrt J} \, .
    \]
    To bound~$H_{n+1}-\widetilde H_{n+1}$, we apply~\eqref{eq:taylor_expansion_H_discrete}
    with~$\theta=\dtheta{n}$, $\allxi=\dallxi{n}$, $\widetildeallxi=\dallxi{n+1}$, and~$x=\xslow{n}$:
    \begin{equation}
        \label{eq:slow_increment_decomposition_discrete}
        H_{n+1}-\widetilde H_{n+1}
        =\sqrt{2h}\,\nabla_\theta H^J\bigl(\dtheta{n},\dallxi{n+1}\bigr)\cdot\xslow{n}+R_{n+1} \, .
    \end{equation}
    Here~$R_{n+1}$ collects all the other terms of that expansion.
    Using the preceding~$L^4(\Omega)$ bound, the uniform moments of the drift and Gaussian increments,
    and~\eqref{eq:bound_remainder_varphi_discrete},
    we obtain
    \[
        \norm{R_{n+1}}_{L^2(\Omega)}
        \leq C\Bigl(h+h^{\frac{3}{2}}+h^2+h^{1+\alpha/2}\Bigr)
        J^{-\frac{1}{2}}
        \norm A_{\cont^{2+\alpha}}
        \norm\varphi_{L^2(\fastgaussianone)} \, .
    \]
    Since~$h\leq1$, applying~\eqref{eq:bound_varphi_l2_discrete} gives
    \[
        \norm{R_{n+1}}_{L^2(\Omega)}
        \lesssim\frac{h(h+\varepsilon)}{\sqrt J}\norm A_{\cont^{2+\alpha}}\norm\chi_{L^2(\fastgaussianone)} \, .
    \]
    We now bound the contribution of the first term on the right-hand side of~\eqref{eq:slow_increment_decomposition_discrete} to the time average.
    This term has conditional mean zero given~$\mathcal G_n$,
    since~$\xslow{n}$ is independent of~$\mathcal G_n$.
    By orthogonality of the increments,
    \begin{align*}
        \expect\left|\frac{\sqrt{2h}}{Nh}\sum_{n=0}^{N-1}\nabla_\theta H^J\bigl(\dtheta{n},\dallxi{n+1}\bigr)\cdot\xslow{n}\right|^2
        &=\frac{2}{N^2h}\sum_{n=0}^{N-1}\expect\left|\nabla_\theta H^J\bigl(\dtheta{n},\dallxi{n+1}\bigr)\right|^2 \\
        &\leq\frac{2}{JNh}\norm{\nabla A}_\infty^2\norm{\varphi}_{L^2(\fastgaussianone)}^2
        \lesssim \frac{(h+\varepsilon)^2}{JNh}
        \norm A_{1,\infty}^2\norm\chi_{L^2(\fastgaussianone)}^2 \, .
    \end{align*}
    Here we used independence of the fast particles, the zero mean of~$\varphi$,
    and~\eqref{eq:bound_varphi_l2_discrete}.
    For the fast increment~$\eta^\xi_{n+1}$, let~$Q=\e^{\frac{h}{\varepsilon}\genfastone}$ denote the one-step transition operator of a fast particle.
    Then~$\expect[\varphi(\dxin{j}{n+1})\mid\mathcal F_n]=(Q\varphi)(\dxin{j}{n})$.
    Since the fast updates are independent conditional on~$\mathcal F_n$,
    bounding~$A$ by~$\norm A_\infty$ and using stationarity gives
    \[
        \expect|\eta^\xi_{n+1}|^2
        \leq\frac{\norm A_\infty^2}{J^2}\sum_{j=1}^J\expect\left|\varphi\bigl(\dxin{j}{n+1}\bigr)-(Q\varphi)\bigl(\dxin{j}{n}\bigr)\right|^2
        =\frac{\norm A_\infty^2}{J}\left(\norm\varphi_{L^2(\fastgaussianone)}^2-\norm{Q\varphi}_{L^2(\fastgaussianone)}^2\right) \, .
    \]
    Write~$\chi=\sum_{m=1}^r\chi_m$, where~$\chi_m$ is the sum of all multivariate Hermite terms of total degree~$m$.
    These components are orthogonal in~$L^2(\fastgaussianone)$ and satisfy~$-\genfastone\chi_m=m\chi_m$.
    By~\eqref{eq:poisson_discrete_averaging},
    \[
        \varphi=\sum_{m=1}^r\frac{h}{1-a_m}\chi_m \, ,
        \qquad a_m=\e^{-\frac{mh}{\varepsilon}} \, .
    \]
    Since
    \(
        Q\varphi=\sum_{m=1}^r\frac{ha_m}{1-a_m}\chi_m,
    \)
    it follows that
    \begin{align*}
        \norm\varphi_{L^2(\fastgaussianone)}^2-\norm{Q\varphi}_{L^2(\fastgaussianone)}^2
        &=\sum_{m=1}^r\frac{h^2(1-a_m^2)}{(1-a_m)^2}\norm{\chi_m}_{L^2(\fastgaussianone)}^2 \\
        &\leq\frac{2h^2}{1-\e^{-\frac{h}{\varepsilon}}}\norm\chi_{L^2(\fastgaussianone)}^2
        \lesssim h(h+\varepsilon)\norm\chi_{L^2(\fastgaussianone)}^2 \, ,
    \end{align*}
    where we used the upper bound in~\eqref{eq:elem_1} with~$x=\frac{h}{\varepsilon}$ in the last inequality.
    Combining these estimates yields
    \[
        \expect|\eta^\xi_{n+1}|^2
        \lesssim \frac{h(h+\varepsilon)}J\norm A_\infty^2\norm\chi_{L^2(\fastgaussianone)}^2 \, .
    \]
    Since~$\expect[\eta^\xi_{n+1}\mid\mathcal F_n]=0$, it follows that
    \[
        \expect\left|\frac1{Nh}\sum_{n=0}^{N-1}\eta^\xi_{n+1}\right|^2
        =\frac1{N^2h^2}\sum_{n=0}^{N-1}\expect|\eta^\xi_{n+1}|^2
        \lesssim\frac{h+\varepsilon}{JNh}\norm A_\infty^2\norm\chi_{L^2(\fastgaussianone)}^2 \, .
    \]
    Applying the triangle inequality in~$L^2(\Omega)$ to the three terms on the right-hand side of~\eqref{eq:averaging_decomposition_discrete_mse},
    using~\eqref{eq:bound_varphi_l2_discrete} for the boundary term and the estimates above for the other two terms, we obtain
    \begin{align*}
        &\expect\left|\frac1N\sum_{n=0}^{N-1}\frac1J\sum_{j=1}^J
        \chi(\dxin{j}{n})\cdot A(\dtheta{n})\right|^2 \\
        &\qquad\qquad\lesssim\norm\chi_{L^2(\fastgaussianone)}^2\Biggl(
        \frac{(h+\varepsilon)^2}{J(Nh)^2}\norm A_\infty^2
        +\frac{(h+\varepsilon)^2}{J}\norm A_{\cont^{2+\alpha}}^2
        +\frac{h+\varepsilon}{JNh}\norm A_\infty^2
        +\frac{(h+\varepsilon)^2}{JNh}\norm A_{1,\infty}^2\Biggr) \, .
    \end{align*}
    This proves~\eqref{eq:averaging_lemma_discrete_mean_square}.
\end{proof}

\begin{proof}
    [Proof of~\cref{theorem:bound_invariant_discrete}]
    As before,
    we use the symbol $\lesssim$ to denote inequalities up to a multiplicative constant independent of the relevant parameters~\(\varepsilon,\sigma,J,h\) and of the test function~\(f\),
    and we use $\bigo(\cdot)$ to denote the corresponding Landau big-O notation.
    Just as in the continuous-time case,
    the bound on the invariant distribution~\eqref{eq:bound_distributions_discrete}
    can be deduced as a corollary of the bound on the bias of the ergodic estimator~\eqref{eq:ergodic_estimator_discrete}
    by considering stationary initial conditions.
    Therefore, in the rest of the proof we focus on the bias and mean square error of the ergodic estimator~\eqref{eq:ergodic_estimator_discrete}.
    In this proof, we write~$\kappa=0$ when~$\pi_0=\fullinvariantdiscrete$
    and~$\kappa=1$ otherwise.
    By \cref{lem:poisson}, there exists a unique solution \(\psi\in \cont^{4+\alpha}(\torus^d,\real)\) to the Poisson equation
    \begin{equation}
        \label{eq:poisson_discrete}
         -\genlimit \psi = f-\mu_\Phi(f) \, ,
        \qquad \text{with }\mu_\Phi(\psi) = 0 \, .
    \end{equation}
    The idea of the proof is similar to that of the continuous-time case.
    We first write
    \begin{align*}
        \psi(\dtheta{N}) - \psi(\dtheta{0})
        &= \sum_{n=0}^{N-1} \bigl( \psi(\dtheta{n+1}) - \psi(\dtheta{n}) \bigr) \, \\
        % &= \sum_{n=0}^{N-1} \Bigl( \markovtransition \psi{\left(\dtheta{n}, \dallxi{n}\right)} - \psi(\dtheta{n}) \Bigr) + \sum_{n=0}^{N-1} \Bigl( \psi(\dtheta{n+1}) - \markovtransition \psi{\left(\dtheta{n}, \dallxi{n}\right)} \Bigr) \, , \\
          &= \sum_{n=0}^{N-1} h \genlimit \psi(\dtheta{n}) + \sum_{n=0}^{N-1} \Bigl( \markovtransition \psi{\left(\dtheta{n}, \dallxi{n}\right)} - \psi(\dtheta{n}) - h \genlimit \psi(\dtheta{n}) \Bigr)  \\
          &\qquad \qquad + \sum_{n=0}^{N-1} \Bigl( \psi(\dtheta{n+1}) - \markovtransition \psi{\left(\dtheta{n}, \dallxi{n}\right)} \Bigr) \, .
    \end{align*}
    Rearranging terms and dividing by \(Nh\),
    we obtain
    \begin{align}
        \notag
        \frac{1}{N} \sum_{n=0}^{N-1} f(\dtheta{n}) - \mu_\Phi(f)
        &= \frac{\psi(\dtheta{0}) - \psi(\dtheta{N})}{Nh}  \\
        \notag
        &\qquad + \frac{1}{N} \sum_{n=0}^{N-1} \biggl[\biggl( \frac{\markovtransition - \id}{h} \biggr) \psi{\left(\dtheta{n}, \dallxi{n}\right)} - \genlimit \psi(\dtheta{n})\biggr]  \\
        \label{eq:variance_term_discrete}
        &\qquad + \frac{1}{Nh} \sum_{n=0}^{N-1} \biggl( \psi(\dtheta{n+1}) - \markovtransition \psi{\left(\dtheta{n}, \dallxi{n}\right)} \biggr) \, .
    \end{align}
    We first focus on the bias of the ergodic estimator~\eqref{eq:ergodic_estimator_discrete}.
    Taking the expectation of both sides,
    we obtain
    \begin{align*}
        \expect \biggl[ \frac{1}{N} \sum_{n=0}^{N-1} f(\dtheta{n}) \biggr] - \mu_\Phi(f)
        &= \frac{1}{Nh} \expect \bigl[ \psi(\dtheta{0})  - \psi(\dtheta{N}) \bigr]
        + \frac{1}{N} \sum_{n=0}^{N-1} \expect \left[ \biggl( \frac{\markovtransition - \id}{h} \biggr) \psi{\left(\dtheta{n}, \dallxi{n}\right)} - \genlimit \psi(\dtheta{n}) \right] \, .
    \end{align*}
    Further decomposing the second term on the right-hand side,
    we obtain
    \begin{subequations}
        \begin{align}
            \notag
            \expect \biggl[ \frac{1}{N} \sum_{n=0}^{N-1} f(\dtheta{n}) \biggr] - \mu_\Phi(f)
        &=
        \frac{1}{Nh} \expect \bigl[ \psi(\dtheta{0})  - \psi(\dtheta{N}) \bigr] \\
        \label{eq:discretization_error}
        &\qquad + \frac{1}{N} \sum_{n=0}^{N-1} \expect \biggl[ \biggl( \frac{\markovtransition - \id}{h} - \fullgenparams \biggr) \psi\bigl(\dtheta{n}, \dallxi{n}\bigr) \biggr]  \\
        \label{eq:small_sigma_error}
         &\qquad + \frac{1}{N} \sum_{n=0}^{N-1} \expect\Bigl[\Bigl( \fullgenparams \psi - \genslow[\sigma=0] \psi \Bigr)\bigl(\dtheta{n}, \dallxi{n}\bigr)\Bigr] \\
        \label{eq:averaging_error}
         &\qquad + \frac{1}{N} \sum_{n=0}^{N-1} \expect\Bigl[\genslow[\sigma=0] \psi{\left(\dtheta{n}, \dallxi{n}\right)} - \genlimit \psi(\dtheta{n})\Bigr] \, .
        \end{align}
    \end{subequations}
    The second term~\eqref{eq:discretization_error} on the right-hand side corresponds to the discretization error,
    the third term~\eqref{eq:small_sigma_error} corresponds to the gradient approximation error,
    and the last term~\eqref{eq:averaging_error} corresponds to the averaging error.
    In the rest of the proof, we first bound these three terms separately.
    Our strategy for the first two is to bound the functions appearing in the expectations in the weighted~$L^{\infty}_{\weightall{p}}$ norm,
    with some~$p \geq 0$ and~$\weightall{p}$ defined by
    \begin{equation}
        \label{eq:definition_Kp}
         \weightall{p}\bigl(\theta, \allxi\bigr) = \frac{1}{J} \sum_{j=1}^{J} \weight{p}(\xin{j}) \, .
    \end{equation}
    Bounds in expectation then follow from the fact~$\dallxi{n}$ is distributed according to~$\fastgaussianall$ for all~$n \geq 0$.
    For the last term~\eqref{eq:averaging_error},
    this approach does not work directly since the argument of the expectation is not necessarily small in a weighted~$L^{\infty}$ norm,
    so we rely on the averaging argument from~\cref{lemma:averaging_discrete} instead.

     \paragraph{Bound on the discretization error~\eqref{eq:discretization_error}}
     We have by definition of the Markov transition operator that,
     for any fixed~\((\theta, \allxi) \in \torus^d \times \R^{dJ}\) and for~\(x \sim \normal(0,\I_d)\),
     \[
         \left(\markovtransition - \mathrm{id}\right) \psi{\left(\theta, \allxi\right)}
         = \expect\Bigl[\psi\bigl(\theta + h\driftall{\left(\theta,\allxi\right)}+\sqrt{2h}\,x \bigr) - \psi(\theta)\Bigr] \, .
     \]
     By a Taylor expansion up to order \(3\), we obtain
     \begin{align*}
         \psi\Bigl(\theta + h\driftall{\left(\theta,\allxi\right)}+\sqrt{2h}\,x\Bigr)
     &= \psi(\theta) + \nabla \psi(\theta)\cdot \left(h\driftall{\left(\theta,\allxi\right)}+\sqrt{2h}\,x\right)
     \\
     &\qquad + \frac{1}{2}
     \nabla^2 \psi(\theta) \cdot \Bigl[h\driftall{\left(\theta,\allxi\right)}+\sqrt{2h}\,x\Bigr]^{\otimes 2}
     + \frac{1}{6}\nabla^3 \psi(\theta) \cdot \Bigl[h\driftall{\left(\theta,\allxi\right)}+\sqrt{2h}\,x\Bigr]^{\otimes 3}
     \\
     &\qquad
     +\frac{1}{6}\int_0^1 (1-t)^3 \nabla^4 \psi\left(\theta + t\bigl(h\driftall{\left(\theta,\allxi\right)}+\sqrt{2h}\,x\bigr)\right) \cdot \Bigl[h\driftall{\left(\theta,\allxi\right)}+\sqrt{2h}\,x\Bigr]^{\otimes 4} \, \d t \, .
     \end{align*}
     Taking the expectation (with respect to~$x$, since~$\theta$ and~$\allxi$ are fixed), we have
     \begin{equation}
         \label{eq:bound_intermediate_discrete_error_ergodic_bias}
         \bigl( P^J[\varepsilon,\sigma,h] - \id \bigr) \psi \bigl(\theta, \allxi\bigr)  =
         h \mathcal L^J[\varepsilon,\sigma] \psi\bigl(\theta, \allxi\bigr) + h^{2}\expect\left[ \mathcal R{\left(\theta,\allxi,x,\sigma,h,J\right)}\right] \, ,
     \end{equation}
     where $\mathcal R{\left(\theta,\allxi,x,\sigma,h,J\right)}$ is a remainder term satisfying
     for all $(\theta, \allxi) \in \torus^d \times \R^{dJ}$,
     \[
         \expect \Bigl\lvert \mathcal R{\left(\theta,\allxi,x,\sigma,h,J\right)} \Bigr\rvert
         \lesssim \norm{\psi}_{4,\infty} \left( 1 + \frac{1}{J} \sum_{j=1}^{J} \bigl\lvert \xin{j} \bigr\rvert^2  \right)^4
         \lesssim \norm{f}_{\cont^{2+\alpha}} \left( 1 + \frac{1}{J} \sum_{j=1}^{J} \bigl\lvert \xin{j} \bigr\rvert^8  \right) \, ,
     \]
     where we used \cref{lemma:upper_bound_drift_and_gradient} and \cref{lem:poisson}.
     It follows that, for all~\(n \geq 0\),
     \begin{equation}
         \label{eq:bound_term_1}
         \expect
         \biggl[
             \biggl( \frac{\markovtransition - \id}{h} - \fullgenparams \biggr) \psi\bigl(\dtheta{n}, \dallxi{n}\bigr)
         \biggr] \,
         \lesssim
         h \norm{f}_{\cont^{2+\alpha}} \,.
     \end{equation}

     \paragraph{Bound on the gradient approximation error~\eqref{eq:small_sigma_error}}
    Similarly to the continuous-time case,
    since~$G \in \cont^{3+\alpha}(\torus^d)$ by assumption,
    we have by~\cref{lemma:drift_lim_small_sigma} the expansion
    \[
        \driftall\bigl(\theta,\allxi\bigr) = \driftlim\bigl(\theta,\allxi\bigr) + \sigma \driftderiv\bigl(\theta,\allxi\bigr) + \sigma^2 \mathcal R_{\sigma}^J{\left(\theta,\allxi\right)} \, ,
    \]
    with~$\driftderiv$ defined in~\eqref{eq:drift_deriv_1}.
    By~\cref{lemma:drift_lim_small_sigma}, the remainder term satisfies
    \[
        \lVert \mathcal R_{\sigma}^J \rVert_{L^{\infty}_{\weightall{4}}}
        \lesssim \norm{\nabla^3 G}_{\infty} \, .
    \]
    Substituting this decomposition into the expression of the generator yields
    \begin{align*}
        \fullgenparams \psi{\left(\theta, \allxi\right)} - \fullgenlimit \psi{\left(\theta, \allxi\right)}
        &= \Bigl( \driftall\bigl(\theta,\allxi\bigr) - \driftlim\bigl(\theta,\allxi\bigr) \Bigr) \cdot \nabla \psi(\theta)  \\
        &= \sigma \driftderiv\bigl(\theta,\allxi\bigr) \cdot \nabla \psi(\theta) + \sigma^2 \mathcal R_{\sigma}^J{\left(\theta,\allxi\right)} \cdot \nabla \psi(\theta) \, ,
    \end{align*}
    Therefore,
    it holds that
    \begin{align}
        \frac{1}{N} \sum_{n=0}^{N-1} \expect\Bigl[\Bigl( \fullgenparams \psi - \genslow[\sigma=0] \psi \Bigr)\bigl(\dtheta{n}, \dallxi{n}\bigr)\Bigr]
        &=
        \frac{\sigma}{N} \sum_{n=0}^{N-1} \expect\Bigl[\driftderiv\bigl(\dtheta{n},\dallxi{n}\bigr) \cdot \nabla \psi(\dtheta{n})\Bigr]
        + \bigo\bigl( \sigma^2 \norm{f}_{\infty} \bigr) \, .
        \label{eq:small_sigma_expansion}
    \end{align}
    The first term on the right-hand side is controlled by an averaging argument,
    exactly as in the continuous-time case.
    From the explicit expression~\eqref{eq:drift_deriv_1},
    for each fixed~$\theta$ the function~$\xi \mapsto \left.\partial_\sigma \driftone(\theta,\xi)\right|_{\sigma=0}$
    is a centered cubic polynomial in~$\xi$ with~$\theta$-dependent coefficients depending on~$G$ and its derivatives up to order~$2$.
    By applying~\cref{lemma:averaging_discrete} with the cubic polynomial~$\chi(\xi)=\xi\otimes\xi\otimes\xi$,
    together with~\cref{lem:poisson} which provides the required bounds on~$\nabla\psi$,
    we obtain
    \begin{equation}
        \label{eq:bound_small_sigma_discrete}
        \Biggl\lvert
        \frac{\sigma}{N} \sum_{n=0}^{N-1} \expect\Bigl[\driftderiv\bigl(\dtheta{n},\dallxi{n}\bigr) \cdot \nabla \psi(\dtheta{n})\Bigr]
        \Biggr\rvert
        \lesssim \frac{\sigma}{\sqrt{J}} \, \norm{f}_{\infty} \, .
    \end{equation}
    If~$G \in \cont^{4+\alpha}(\torus^d)$ so that the second estimate of~\cref{lemma:averaging_discrete}
    also applies, then the improved bound
    \begin{align}
        \notag
        \Biggl\lvert
        \frac{\sigma}{N} \sum_{n=0}^{N-1} \expect\Bigl[\driftderiv\bigl(\dtheta{n},\dallxi{n}\bigr) \cdot \nabla \psi(\dtheta{n})\Bigr]
        \Biggr\rvert
         &\lesssim \frac{\sigma (h+\varepsilon)}{\sqrt{J}} \left( \frac{\kappa}{Nh} \norm{\nabla \psi}_{\infty} + \norm{\nabla \psi}_{\cont^{2+\alpha}} \right)
        \\
         &\lesssim \frac{\sigma (h+\varepsilon)}{\sqrt{J}} \left( \frac{\kappa}{Nh} \norm{f}_{\infty} + \norm{f}_{2,\infty} \right)
        \label{eq:bound_small_sigma_discrete_sharp}
    \end{align}
     holds, where~$\kappa = 0$ for stationary initial conditions and~$\kappa = 1$ otherwise.

     \paragraph{Bound on the averaging error~\eqref{eq:averaging_error}}
    The following derivation is valid provided that $G \in \cont^{3+\alpha}(\torus^d)$.
    We introduce, just as in the continuous-time case,
    the function \(\chi^J\) given by
    \[
        \chi^J\bigl(\theta,\allxi\bigr) = \bigl(\cov(\allxi) - \I_d\bigr) : \Bigl( \nabla \Phi(\theta) \otimes \nabla \psi(\theta) \Bigr) \, ,
    \]
    which appears in the expectation in~\eqref{eq:averaging_error}, up to the sign.
    Using the identity~$\cov(\allxi) = \frac{1}{J}\sum_{j=1}^{J} \xin{j} \otimes \xin{j}$,
    we rewrite~$\chi^J$ as
    \[
        \chi^J(\theta,\allxi) = \frac{1}{J} \sum_{j=1}^{J} \bigl( \xin{j} \otimes \xin{j} - \I_d \bigr) : \bigl( \nabla \Phi(\theta) \otimes \nabla \psi(\theta) \bigr) \, .
    \]
    The quadratic polynomial~$\xi \mapsto \xi \otimes \xi - \I_d$ has mean zero under~$\fastgaussianone$,
    so by~\cref{lemma:averaging_discrete} applied with~$\chi(\xi) = \xi \otimes \xi - \I_d$
    and~$A(\theta) = \nabla \Phi(\theta) \otimes \nabla \psi(\theta)$,
    together with~\cref{lem:poisson} which provides the required bounds on~$\nabla\psi$,
    we obtain
    \begin{align}
        \biggl\lvert \frac{1}{N} \sum_{n=0}^{N-1} \expect \Bigl[ \chi^J\bigl(\dtheta{n}, \dallxi{n}\bigr) \Bigr] \biggr\rvert
        &\lesssim \frac{h + \varepsilon}{Nh\sqrt{J}} \norm{f}_{\infty}
        + \frac{h+\varepsilon}{\sqrt{J}} \norm{f}_{2,\infty} \, ,
        \label{eq:final_bound_averaging_discrete}
    \end{align}
    where the last line uses~$\norm{A}_{\infty} \lesssim \norm{f}_{\infty}$
    and~$\norm{A}_{\cont^{2+\alpha}} \lesssim \norm{f}_{2,\infty}$.
    This gives the desired bound on the averaging error~\eqref{eq:averaging_error}.

     \paragraph{Conclusion of the bias estimate}
    Combining the bound~\(\lesssim (Nh)^{-1}\norm{f}_{\infty}\) on the first term in the bias decomposition,
    the discretization error~\eqref{eq:bound_term_1},
    the small-\(\sigma\) expansion~\eqref{eq:small_sigma_expansion} together with~\eqref{eq:bound_small_sigma_discrete},
    and the averaging error~\eqref{eq:final_bound_averaging_discrete},
    we obtain for~\(\beta = 0\)
    \[
        \Bigl\lvert \expect[\widehat f_N] - \mu_\Phi(f) \Bigr\rvert
        \lesssim \frac{1}{Nh}\norm{f}_{\infty}
        + h \norm{f}_{\cont^{2+\alpha}}
        + \frac{\sigma}{\sqrt{J}}\norm{f}_{\infty}
        + \sigma^2 \norm{f}_{\infty}
        + \frac{h+\varepsilon}{Nh\sqrt{J}}\norm{f}_{\infty}
        + \frac{h+\varepsilon}{\sqrt{J}}\norm{f}_{2,\infty} \, .
    \]
    This yields the bound~\eqref{eq:bound_bias_discrete} for~\(\beta = 0\).
    For~\(\beta = 1\), the refined small-\(\sigma\) estimate~\eqref{eq:bound_small_sigma_discrete_sharp}
    gives the bound
    \begin{align*}
        \bigl\lvert \expect[\widehat f_N] - \mu_\Phi(f) \bigr\rvert
        &\lesssim \frac{1}{Nh}\norm{f}_{\infty}
        + h \norm{f}_{\cont^{2+\alpha}}
        + \sigma^2 \norm{f}_{\infty}
        + \frac{\sigma(h+\varepsilon)}{\sqrt{J}} \Bigl( \frac{\kappa}{Nh}\norm{f}_{\infty} + \norm{f}_{2,\infty} \Bigr)
        \\
        &\qquad
        + \frac{h+\varepsilon}{Nh\sqrt{J}}\norm{f}_{\infty}
        + \frac{h+\varepsilon}{\sqrt{J}}\norm{f}_{2,\infty} \, ,
    \end{align*}
    This concludes the proof of~\eqref{eq:bound_bias_discrete}.

     \paragraph{Bound on the mean square error~\eqref{eq:bound_mse_discrete}}
    Let~$\mathcal F_n$ be the natural filtration of the chain and set
    \[
        \Delta M_{n+1}
        = \psi(\dtheta{n+1})
        - \markovtransition\psi(\dtheta{n},\dallxi{n}) \, .
    \]
    The last term in~\eqref{eq:variance_term_discrete} is then~$(Nh)^{-1}\sum_{n=0}^{N-1}\Delta M_{n+1}$,
    and $(\Delta M_n)_{n\geq1}$ is a martingale-difference sequence.
    Since~$\psi$ is Lipschitz continuous and the drift has moments bounded uniformly in the parameters,
    \[
        \expect|\Delta M_{n+1}|^2
        \leq \expect|\psi(\dtheta{n+1})-\psi(\dtheta{n})|^2
        \lesssim h\norm{\psi}_{1,\infty}^2 \, .
    \]
    Orthogonality of the increments therefore gives
    \begin{equation}
        \label{eq:bound_martingale_mse_discrete}
        \expect\left|\frac1{Nh}\sum_{n=0}^{N-1}\Delta M_{n+1}\right|^2
        =\frac1{N^2h^2}\sum_{n=0}^{N-1}\expect|\Delta M_{n+1}|^2
        \lesssim \frac1{Nh}\norm{f}_{\infty}^2 \, .
    \end{equation}
    The additional discretization contribution is bounded in mean square by
    $Ch^2\norm{f}_{\cont^{2+\alpha}}^2$,
    using~\eqref{eq:bound_intermediate_discrete_error_ergodic_bias} and the Gaussian moment bounds.
    The remaining terms are treated as in the continuous-time proof,
    using~\eqref{eq:averaging_lemma_discrete_mean_square} in place of~\eqref{eq:averaging_lemma_second_estimate_mean_square},
    which leads to the desired bound~\eqref{eq:bound_mse_discrete} on the mean square error.
    % The boundary contribution is bounded by~$C(Nh)^{-2}\norm{f}_\infty^2$.
    % The additional discretization contribution is bounded in mean square
    % by~$Ch^2\norm{f}_{\cont^{2+\alpha}}^2$,
    % by~\eqref{eq:bound_intermediate_discrete_error_ergodic_bias}
    % and the Gaussian moment bounds.
    % Combining these estimates in~\eqref{eq:variance_term_discrete}
    % gives~\eqref{eq:bound_mse_discrete}.
\end{proof}

\subsection{Weak convergence}%
\label{ssub:weak_discrete}

\begin{lemma}[Discrete averaging lemma]
    \label{lemma:averaging_discrete_weak}
    Let $B\colon \real^d \to \real^n$ be a polynomial of degree at most~$r$, for some~$r\in\N$, with mean zero under~$\fastgaussianone$,
    and assume that $g\in\cont^{1,4}([0,Nh]\times\torus^d,\real^n)$.
    Under the same assumptions as in \cref{thm:uniform_discrete_weak_convergence},
    there exists a constant $C > 0$,
    depending only on~$d,n,r$ and fixed derivative bounds on~$G$ and~$\ell$ 
    but independent of~$\varepsilon,\sigma,J,h,N,B,g$, such that
    \begin{align*}
        & \biggl\lvert h \sum_{n=0}^{N-1} \expect \biggl[ \frac{1}{J} \sum_{j=1}^J B(\dxin{j}{n}) \cdot g(nh+h, \dtheta{n}) \biggr] \biggr\rvert \\
        &\qquad \qquad 
        \leq \frac{C(h+\varepsilon)}{\sqrt{J}} \, \lVert B\rVert_{L^2(\fastgaussianone)}
        \left( \sup_{s\in[0,Nh]} \lVert g(s,\placeholder)\rVert_\infty + h \sum_{n=0}^{N-1} \biggl( \sup_{r\in[nh,nh+h]}  \lVert \partial_t g(r,\placeholder)\rVert_\infty + \lVert g(nh+h,\placeholder)\rVert_{4,\infty} \biggr)  \right) \, .
    \end{align*}
\end{lemma}
\begin{proof}[Proof of~\cref{lemma:averaging_discrete_weak}]
    Let $\mathscr B\colon\real^d\to\real^n$ denote the mean-zero solution in~$L^2(\fastgaussianone)$ to the discrete Poisson equation
    \[
        - \left( \frac{\e^{\frac{h}{\varepsilon} \genfastone{}} - \id}{h} \right) \mathscr B(\xi) = B(\xi) \, ,
    \]
    so that $\expect \bigl[ \mathscr B\bigl(\dxin{j}{n+1}\bigr) \bigm| \dxin{j}{n} \bigr]
    = \mathscr B\bigl(\dxin{j}{n}\bigr) - h B\bigl(\dxin{j}{n}\bigr)$.
    By~\cref{lemma:discrete_poisson_eq},
    appropriately adapted to the~$L^2(\fastgaussianone)$ setting,
    and the elementary inequality~\eqref{eq:elem_1},
    we have
    \begin{equation}
        \label{eq:bound_B_discrete}
        \lVert \mathscr B \rVert_{L^2(\fastgaussianone)}
        \lesssim \frac{h}{1 - \e^{-\frac{h}{\varepsilon}}}\lVert B\rVert_{L^2(\fastgaussianone)}
        \lesssim (h + \varepsilon)\lVert B\rVert_{L^2(\fastgaussianone)} \, .
    \end{equation}
    Let $B^J(\allxi) = \frac{1}{J} \sum_{j=1}^J B(\xin{j})$ and $\mathscr B^J(\allxi) = \frac{1}{J} \sum_{j=1}^J \mathscr B(\xin{j})$.
    It holds that
    \begin{align*}
        &\expect \Bigl[ g(Nh, \dtheta{N}) \cdot \mathscr B^J\bigl(\dallxi{N}\bigr) \Bigr] - \expect \Bigl[ g(0, \dtheta{0}) \cdot \mathscr B^J\bigl(\dallxi{0}\bigr) \Bigr] \\
        &\qquad = \sum_{n=0}^{N-1} \Bigl( \expect \Bigl[ g(nh+h, \dtheta{n+1}) \cdot \mathscr B^J\bigl(\dallxi{n+1}\bigr) \Bigr] - \expect \Bigl[ g(nh, \dtheta{n}) \cdot \mathscr B^J(\dallxi{n}) \Bigr] \Bigr) \\
        &\qquad = \sum_{n=0}^{N-1} \expect \Bigl[ \expect
        \Bigl[  g(nh+h, \dtheta{n+1}) \cdot \mathscr B^J\bigl(\dallxi{n+1}\bigr) - g(nh, \dtheta{n}) \cdot \mathscr B^J\bigl(\dallxi{n}\bigr) \Bigm| \bigl( \dtheta{n}, \dallxi{n} \bigr) \Bigr] \Bigr] \\
        &\qquad = \sum_{n=0}^{N-1} \expect \Bigl[ g(nh+h, \dtheta{n+1}) \cdot \Bigl( \mathscr B^J\bigl(\dallxi{n}\bigr) - h B^J\bigl(\dallxi{n}\bigr) \Bigr) - g(nh, \dtheta{n}) \cdot \mathscr B^J(\dallxi{n}) \Bigr] \, ,
    \end{align*}
    where we used the conditional independence of $g(nh+h, \dtheta{n+1})$ and $\mathscr B^J(\dallxi{n+1})$ given $(\dtheta{n}, \dallxi{n})$ together with the Poisson property in the third equality.
    Rearranging, we obtain the following equation:
    \begin{align*}
        \notag
        h \sum_{n=0}^{N-1} \expect \Bigl[  g(nh+h, \dtheta{n+1}) \cdot B^J\bigl(\dallxi{n}\bigr) \Bigr]
        &= \expect \Bigl[ g(0, \dtheta{0}) \cdot \mathscr B^J\bigl(\dallxi{0}\bigr) \Bigr] - \expect \Bigl[ g(Nh, \dtheta{N}) \cdot \mathscr B^J\bigl(\dallxi{N}\bigr) \Bigr] \\
        &\qquad + \sum_{n=0}^{N-1} \expect \Bigl[ \Bigl( g(nh+h, \dtheta{n+1}) - g(nh, \dtheta{n}) \Bigr) \cdot \mathscr B^J\bigl(\dallxi{n}\bigr) \Bigr] \, .
    \end{align*}
    The left-hand side of the above equation is almost equal to the quantity we want to bound.
    Adding and subtracting the difference on both sides,
    we obtain
    \begin{align}
        \notag
        h \sum_{n=0}^{N-1} \expect \Bigl[  g(nh+h, \dtheta{n}) \cdot B^J\bigl(\dallxi{n}\bigr) \Bigr]
        &= \expect \Bigl[ g(0, \dtheta{0}) \cdot \mathscr B^J\bigl(\dallxi{0}\bigr) \Bigr] - \expect \Bigl[ g(Nh, \dtheta{N}) \cdot \mathscr B^J\bigl(\dallxi{N}\bigr) \Bigr] \\
        \notag
        &\qquad + \sum_{n=0}^{N-1} \expect \Bigl[ \Bigl( g(nh+h, \dtheta{n+1}) - g(nh, \dtheta{n}) \Bigr) \cdot  \mathscr B^J\bigl(\dallxi{n}\bigr)  \Bigr] \\
        &\qquad - h \sum_{n=0}^{N-1} \expect \Bigl[ \Bigl( g(nh+h, \dtheta{n+1}) - g(nh + h, \dtheta{n}) \Bigr) \cdot   B^J\bigl(\dallxi{n} \bigr)  \Bigr] \, .
        \label{eq:averaging_decomposition_discrete_weak_initial}
    \end{align}

    By a Cauchy--Schwarz inequality,
    the fact that $\mathscr B^J$ is a sum of i.i.d.\ centered random variables,
    and the bound~\eqref{eq:bound_B_discrete} on $\mathscr B$,
    the first term satisfies the bound
    \[
        \left\lvert \expect \Bigl[ g(0, \dtheta{0}) \cdot \mathscr B^J\bigl(\dallxi{0}\bigr) \Bigr] - \expect \Bigl[ g(Nh, \dtheta{N}) \cdot \mathscr B^J\bigl(\dallxi{N}\bigr) \Bigr] \right\rvert
        \lesssim \frac{h+\varepsilon}{\sqrt{J}} \, \lVert B\rVert_{L^2(\fastgaussianone)}
        \lVert g \rVert_{L^{\infty}\left([0, Nh]\times \torus^d\right)} \, .
    \]

    To obtain a bound for the second term on the right-hand side of~\eqref{eq:averaging_decomposition_discrete_weak_initial}
    we first decompose the increment of~$g$~as
    \begin{equation}
        \label{eq:g_difference_decomposition}
        g(nh+h, \dtheta{n+1}) - g(nh, \dtheta{n}) =
        \bigl( g(nh+h, \dtheta{n}) - g(nh, \dtheta{n}) \bigr) + \bigl( g(nh+h, \dtheta{n+1}) - g(nh+h, \dtheta{n}) \bigr) \, .
    \end{equation}
    The first term on the right-hand side satisfies
    \begin{equation}
        \label{eq:bound_g_time_increment_1}
        \bigl\lvert g(nh+h, \dtheta{n}) - g(nh, \dtheta{n}) \bigr\rvert
        \leq \int_{nh}^{nh+h} \bigl\lvert \partial_t g(s, \dtheta{n}) \bigr\rvert \, \mathrm{d}s
        \leq h \bigl\lVert \partial_t g \bigr\rVert_{L^{\infty}([nh, nh+h]\times \torus^d)} \, ,
    \end{equation}
    and for the second term on the right-hand side of~\eqref{eq:g_difference_decomposition},
    we have by a Taylor expansion of the function~$g(nh+h, \placeholder)$ around~$\dtheta{n}$ up to order~$4$,
    that
    \begin{equation}
        \label{eq:bound_g_time_increment_2}
        \expect \Bigl[  g(nh+h, \dtheta{n+1}) - g(nh+h, \dtheta{n})  \, \Bigm| \, \bigl(\dtheta{n}, \dallxi{n}\bigr) \Bigr]
        = h \fullgenparams g{\left(nh+h, \dtheta{n}, \dallxi{n}\right)} + \mathfrak R^J_{\sigma,h}\bigl(nh+h, \dtheta{n}, \dallxi{n}\bigr) \, ,
    \end{equation}
    where the remainder term satisfies, by~\cref{lemma:upper_bound_drift_and_gradient},
    \[
        \Bigl\lvert \mathfrak R^J_{\sigma, h} \bigl(nh+h, \dtheta{n}, \dallxi{n}\bigr) \Bigr\rvert
        \lesssim
        h^2 \lVert g(nh+h, \placeholder) \rVert_{4,\infty} \biggl( 1 + \frac{1}{J} \sum_{j=1}^{J} \lvert \dxin{j}{n} \rvert^8 \biggr)  \, .
    \]
    Therefore, combining~\eqref{eq:bound_g_time_increment_1} and~\eqref{eq:bound_g_time_increment_2},
    we obtain
    \begin{align*}
        &\left\lvert \expect \Bigl[  g(nh+h, \dtheta{n+1}) - g(nh, \dtheta{n})  \, \Bigm| \, \bigl(\dtheta{n}, \dallxi{n}\bigr) \Bigr] \right\rvert         \\
        &\qquad \qquad \leq
        h\sup_{r \in [nh, nh+h]} \left\lVert \partial_t g(r, \placeholder) \right\rVert_\infty
        + h \left\lvert \fullgenparams g{\left(nh+h, \dtheta{n}, \dallxi{n}\right)} \right\rvert
        + \left\lvert \mathfrak R^J_{\sigma, h} \bigl(nh+h, \dtheta{n}, \dallxi{n}\bigr) \right\rvert \, \\
        &\qquad \qquad \leq 
        h\sup_{r \in [nh, nh+h]} \left\lVert \partial_t g(r, \placeholder) \right\rVert_\infty
        + h\biggl( 1 + \frac{1}{J} \sum_{j=1}^{J} \lvert \dxin{j}{n} \rvert^2 \biggr) \left\lVert g(nh+h, \placeholder) \right\rVert_{2,\infty} \,  \\
        &\qquad \qquad \qquad \qquad +
        h^2 \lVert g(nh+h, \placeholder) \rVert_{4,\infty} \biggl( 1 + \frac{1}{J} \sum_{j=1}^{J} \lvert \dxin{j}{n} \rvert^8 \biggr)  \, .
    \end{align*}
    Together with the law of large numbers in~$L^2(\Omega)$ for $\mathscr B^J$ and the bound on $\mathscr B$ given in~\eqref{eq:bound_B_discrete},
    this leads to the following bound for the second term on the right-hand side of \eqref{eq:averaging_decomposition_discrete_weak_initial},
    \begin{align*}
        &\left\lvert \sum_{n=0}^{N-1} \expect \Bigl[ \Bigl( g(nh+h, \dtheta{n+1}) - g(nh, \dtheta{n}) \Bigr) \cdot \mathscr B^J(\dallxi{n}) \Bigr] \right\rvert \\
        &\qquad \qquad = \left\lvert \sum_{n=0}^{N-1} \expect \biggl[ \expect \Bigl[ \Bigl( g(nh+h, \dtheta{n+1}) - g(nh, \dtheta{n}) \Bigr) \,\Bigm|\, \bigl(\dtheta{n}, \dallxi{n} \bigr) \Bigr] \cdot \mathscr B^J(\dallxi{n}) \biggr] \right\rvert \\
        &\qquad \qquad \leq  \sum_{n=0}^{N-1} \left( \expect  \biggl\lvert \expect \Bigl[ \Bigl( g(nh+h, \dtheta{n+1}) - g(nh, \dtheta{n}) \Bigr) \,\Bigm|\, \bigl(\dtheta{n}, \dallxi{n} \bigr) \Bigr] \biggr\rvert^2 \right)^{\frac{1}{2}} 
        \left( \expect \Bigl\lvert \mathscr B^J(\dallxi{n}) \Bigr\rvert^2 \right)^{\frac{1}{2}} \\
        &\qquad 
        \lesssim
        \frac{h(h+\varepsilon)}{\sqrt{J}}\lVert B\rVert_{L^2(\fastgaussianone)} \sum_{n=0}^{N-1}  \left( \sup_{r\in[nh,nh+h]}  \lVert \partial_t g(r,\placeholder)\rVert_\infty + \lVert g(nh+h,\placeholder)\rVert_{4,\infty}  \right) \, .
    \end{align*}

    The third term on the right-hand side of~\eqref{eq:averaging_decomposition_discrete_weak_initial} can be bounded in exactly the same manner:
    using the tower property of conditional expectation,
    we have that
    \begin{align*}
        &h \sum_{n=0}^{N-1} \expect \Bigl[ \Bigl( g(nh+h, \dtheta{n+1}) - g(nh + h, \dtheta{n}) \Bigr) \cdot   B^J\bigl(\dallxi{n} \bigr)  \Bigr] \\
        &\qquad
        = h \sum_{n=0}^{N-1} \expect \biggl[ \expect \Bigl[ g(nh+h, \dtheta{n+1}) - g(nh+h, \dtheta{n}) \, \Bigm| \, \bigl(\dtheta{n}, \dallxi{n}\bigr) \Bigr] \cdot B^J\bigl(\dallxi{n}\bigr)\biggr] \\
        &\qquad
        = h^2 \sum_{n=0}^{N-1} \expect \biggl[ \fullgenparams g{\left(nh+h, \dtheta{n}, \dallxi{n}\right)} \cdot B^J\bigl(\dallxi{n}\bigr) \biggr] 
        + h \sum_{n=0}^{N-1} \expect \biggl[ \mathfrak R^J_{\sigma,h}\bigl(nh+h, \dtheta{n}, \dallxi{n}\bigr) \cdot B^J\bigl(\dallxi{n}\bigr) \biggr] \, .
    \end{align*}
    By the Cauchy--Schwarz inequality,
    we deduce that 
    \begin{align*}
        &\left\lvert h \sum_{n=0}^{N-1} \expect \Bigl[ \Bigl( g(nh+h, \dtheta{n+1}) - g(nh + h, \dtheta{n}) \Bigr) \cdot B^J\bigl(\dallxi{n} \bigr) \Bigr] \right\rvert \\
        &\qquad \lesssim \frac{h^2}{\sqrt{J}}\lVert B\rVert_{L^2(\fastgaussianone)} \sum_{n=0}^{N-1} \lVert g(nh+h, \placeholder) \rVert_{2,\infty} \\
        &\qquad\qquad + \frac{h^3}{\sqrt{J}}\lVert B\rVert_{L^2(\fastgaussianone)} \sum_{n=0}^{N-1} \lVert g(nh+h, \placeholder) \rVert_{4,\infty} \, .
    \end{align*}
    Together with the bound on the first two terms on the right-hand side of~\eqref{eq:averaging_decomposition_discrete_weak_initial},
    this estimate leads to the desired bound,
    which concludes the proof.
\end{proof}

\begin{proof}
    [Proof of~\cref{thm:uniform_discrete_weak_convergence}]
Let \(f\in \cont^5(\torus^d)\) and fix $N \in \N$.
Without loss of generality,
we assume that~$\target(f) = 0$.
Throughout this proof,
the constants implicit in the~$\lesssim$ and~$\bigo$ notations may depend on the dimension \(d\) and on~$\lVert \ell \rVert_{\cont^6}$ and $\lVert G \rVert_{\cont^6}$,
but not on the relevant parameters~$\varepsilon, \sigma, J,h $,
nor on the distribution $m_0$, the function~$f$ or the time horizon~$N$.
By~\cref{theorem:auxiliary_parabolic} with~\(M=5\) and time reversal,
there exists a unique classical solution~$u\in\cont^{1,5}([0,Nh]\times\torus^d)$
to the terminal-value backward Kolmogorov equation:
\begin{equation}
    \label{eq:kbe_discrete}
    \left\{
        \begin{aligned}
            \frac{\partial u}{\partial t} (t,\theta) 
            + \genlimit u{\left(t,\theta\right)}  &= 0 \, ,\\
            u(Nh,\theta) &= f(\theta) \, .
        \end{aligned}
    \right.
\end{equation}
Equation~\eqref{eq:kbe_discrete} gives~$\partial_tu=-\genlimit u$ in~$\cont\bigl([0,Nh];\cont^3(\torus^d)\bigr)$.
Since~$\genlimit\colon\cont^3(\torus^d)\to\cont^1(\torus^d)$ is bounded,
we may differentiate~$\genlimit u$ in time to obtain~$\partial_{tt}u=-\genlimit\partial_tu=\genlimit^2u$ in~$\cont\bigl([0,Nh];\cont^1(\torus^d)\bigr)$.
It follows that
\begin{equation}
    \label{eq:bound_second_time_derivative_u}
    \bigl\lVert \partial_{tt} u(t,\placeholder) \bigr\rVert_{\infty}
    = \bigl\lVert \genlimit^2 u(t,\placeholder) \bigr\rVert_{\infty}
    \lesssim \bigl\lVert u(t,\placeholder) \bigr\rVert_{4,\infty} \, .
\end{equation}
Furthermore, since~$\target(f)=0$, the estimate~\eqref{eq:decay_derivatives_semigroup_statement} in~\cref{theorem:auxiliary_parabolic}, with~$M=5$, gives
\begin{equation}
    \label{eq:semigroup_discrete_ell}
    \forall s \in [0, Nh],
    \qquad
    \forall k \in \{0, \ldots, 5\},
    \qquad
    \bigl\lVert u(s, \placeholder) \bigr\rVert_{k,\infty} \lesssim \e^{-\lambda_{\Phi} (Nh-s)} \norm*{f}_{k,\infty} \, .
\end{equation}
By the Feynman--Kac formula~\eqref{eq:feynman_kac}, 
we have the following representation formula for the solution
\[
    u(t,\theta) = \expect\bigl[f(\vartheta_{Nh}) \bigm| \vartheta_t = \theta\bigr] \, , 
\]
where $(\vartheta_t)_{t\geq 0}$ is the solution to the overdamped Langevin dynamics~\eqref{eq:overdamped_langevin_weak}.
Following the classical approach in numerical analysis,
we use~\eqref{eq:kbe_discrete} to rewrite the difference between expectations in~\eqref{eq:uniform_discrete_weak_convergence} as a telescoping sum:
\begin{align*}
    \expect \bigl[f(\dtheta{N})\bigr] - \expect\bigl[f(\vartheta_{Nh})\bigr]
    & = \expect\bigl[u(Nh,\dtheta{N})\bigr] - \expect\bigl[u(0,\dtheta{0})\bigr] \, \\
    & = \sum_{n=0}^{N-1} \expect \Bigl[ u(nh+h, \dtheta{n+1}) - u(nh, \dtheta{n}) \Bigr] \\
    & = \sum_{n=0}^{N-1} \expect \Bigl[ \markovtransition u{\left(nh + h, \dtheta{n}, \dallxi{n}\right)} - u(nh, \dtheta{n}) \Bigr] \, .
\end{align*}
Drawing inspiration from the proof in continuous time,
we further decompose this term as
\begin{align*}
    \expect \bigl[f(\dtheta{N})\bigr] - \expect\bigl[f(\vartheta_{Nh})\bigr]
&= \sum_{n=0}^{N-1} \expect \Bigl[ \markovtransition u{\left(nh +  h, \dtheta{n}, \dallxi{n}\right)} - u(nh, \dtheta{n}) \\ 
&\qquad \qquad - h \fullgenparams u{\left(nh+h, \dtheta{n}, \dallxi{n}\right)} - h\frac{\partial u}{\partial t}(nh+h, \dtheta{n}) \Bigr]
\\
    &\qquad + h \sum_{n=0}^{N-1} \expect \Bigl[ \fullgenparams u{\left(nh + h, \dtheta{n}, \dallxi{n}\right)} - \fullgenlimit u{\left(nh + h, \dtheta{n}, \dallxi{n}\right)} \Bigr] 
\\
    &\qquad + h \sum_{n=0}^{N-1} \expect \Bigl[ \fullgenlimit u{\left(nh + h, \dtheta{n}, \dallxi{n}\right)} - \genlimit u{\left(nh + h, \dtheta{n}\right)} \Bigr]
\eqcolon \mathcal A_1^N + \mathcal A_2^N + \mathcal A_3^N \, .
\end{align*}
The first term on the right-hand side corresponds to the local truncation error of the Euler--Maruyama scheme,
while the other two terms correspond to the gradient approximation error and the averaging error, respectively.

\paragraph{Bound on the discretization error $\mathcal A_1^N$}
The approach for bounding this term is based on the standard weak convergence proof for the Euler--Maruyama scheme.
We refer to~\cite[Section 7.5.2.]{MR3097957},~\cite[Chapter 14]{MR1214374} and~\cite[Chapter 3]{MR2069903} for textbook references on this topic.
By a Taylor expansion of the function~\(u\),
following exactly the same steps as in the reasoning leading up to~\eqref{eq:bound_intermediate_discrete_error_ergodic_bias},
and by using the definition of the generator~\(\fullgenparams\),
we have for all fixed $(\theta, \allxi) \in \torus^d \times \R^{dJ}$,
\begin{align}
    \notag
    \markovtransition u{\left(nh+h, \theta, \allxi\right)} 
    &= \expect \Bigl[ u\bigl(nh+h, \theta + h \driftall{\left(\theta, \allxi\right)} + \sqrt{2h} Z\bigr) \Bigr] \\
    &= u\bigl(nh+h, \theta \bigr) + h \fullgenparams u{\left(nh+h, \theta, \allxi\right)}
    + \mathcal R^J_{\sigma,h}(n, \theta, \allxi) \, ,
    \label{eq:bound_intermediate_disc_error_weak_discrete}
\end{align}
where $Z\sim \normal(0, \I_d)$ and the remainder term satisfies,
by a Taylor expansion of the function~$u(nh+h, \cdot)$ around~$\theta$ up to order~$4$,
for all $(\theta, \allxi) \in \torus^d \times \R^{dJ}$,
\begin{equation}
\label{eq:bound_remainder_term_discretization_error_weak_discrete}
\Bigl\lvert \mathcal R^J_{\sigma,h}\bigl(n, \theta,\allxi\bigr) \Bigr\rvert
\lesssim h^2 \bigl\lVert u(nh+h, \placeholder) \bigr\rVert_{4,\infty} \, \left( 1 + \frac{1}{J} \sum_{j=1}^{J} \lvert \xin{j} \rvert^8 \right) \, .
\end{equation}
Thus, substituting~\eqref{eq:bound_intermediate_disc_error_weak_discrete} in the discretization error~$\mathcal A_1^N$, 
we rewrite this error as
\[
    \mathcal A_1^N = \sum_{n=0}^{N-1} \expect \biggl[ u(nh+h, \dtheta{n}) - u(nh, \dtheta{n}) - h\frac{\partial u}{\partial t}(nh+h, \dtheta{n}) \biggr]
    + \sum_{n=0}^{N-1} \expect \Bigl[ \mathcal R^{J}_{\sigma,h}\bigl(n, \dtheta{n}, \dallxi{n}\bigr) \Bigr] \, .
\]
Since $u$ is twice continuously differentiable in time, 
we have for all $n\in \{0, \ldots, N-1\}$,
\[
\Bigl\lvert u(nh+h, \dtheta{n}) - u(nh, \dtheta{n}) - h\frac{\partial u}{\partial t}(nh+h, \dtheta{n}) \Bigr\rvert
\leq \frac{h^2}{2} \bigl\lVert \partial_{tt} u \bigr\rVert_{L^{\infty}([nh, nh+h]\times \torus^d)} \, .
\]
It follows from this equation, together with~\eqref{eq:bound_remainder_term_discretization_error_weak_discrete} and the bound on the second time derivative of $u$ given by~\eqref{eq:bound_second_time_derivative_u}, 
that the discretization error can be bounded as follows:
\begin{align*}
    \lvert\mathcal A_1^N\rvert
     &\lesssim h^2 \sum_{n=0}^{N-1} \sup_{s\in [nh, nh+h]} \bigl\lVert u(s, \placeholder) \bigr\rVert_{4,\infty} \,
     \lesssim h^2 \sum_{n=0}^{N-1} \sup_{s\in [nh, nh+h]} \e^{-\lambda_{\Phi}(Nh-s)} \norm*{f}_{4,\infty}
    \\
    &\lesssim h^2 \sum_{n=0}^{N-1} \e^{-\lambda_{\Phi}(Nh-nh-h)} \norm*{f}_{4,\infty}
    = \frac{\e^{\lambda_\Phi Nh} - 1}{\e^{\lambda_\Phi h} - 1} 
    \e^{-\lambda_{\Phi} (Nh-h)}
    \, h^2 \norm*{f}_{4,\infty}
    \lesssim h \norm*{f}_{4,\infty} \, .
\end{align*}

\paragraph{Bound on the gradient approximation error $\mathcal A_2^N$}
Given the decomposition~\eqref{eq:drift_approximation_error_expansion} of the drift~$\driftall$ used previously
\ifsiamart in the main paper\fi,
it holds that
\[
\mathcal A_2^N = 
h \sigma \sum_{n=0}^{N-1} \expect \Bigl[ \driftderiv\bigl(\dtheta{n}, \dallxi{n}\bigr)  \cdot \nabla u(nh+h, \dtheta{n}) \Bigr]
+ h \sigma^2 \sum_{n=0}^{N-1} \expect \Bigl[ \mathcal R^J_{\sigma}\bigl(\dtheta{n}, \dallxi{n}\bigr) \cdot \nabla u(nh+h, \dtheta{n}) \Bigr] \, .
\]
where the remainder term satisfies the estimate~\eqref{eq:bound_remainder_term_grad_approx}.
Using the estimate~\eqref{eq:semigroup_discrete_ell} on the~$L^{\infty}$ decay of the derivatives of the semigroup,
we deduce that
\begin{equation}
    \label{eq:refer_alternative_proof}
    \mathcal A_2^N = 
    h \sigma \sum_{n=0}^{N-1} \expect \Bigl[ \driftderiv\bigl(\dtheta{n}, \dallxi{n}\bigr)  \cdot \nabla u(nh+h, \dtheta{n}) \Bigr]
    + \bigo\bigl(\sigma^2\norm f_{1,\infty}\bigr) \, .
\end{equation}
As in the previous sections,
the first term on the right-hand side,
which will be denoted by~$\widetilde {\mathcal A}_2^N$, can be bounded using an averaging argument,
just like the averaging error~$\mathcal A_3^N$.
We bound both terms in a unified manner in the next paragraph.

\paragraph{Bound on the averaging error $\widetilde {\mathcal A}_2^N + \mathcal A_3^N$}
We recall from~\eqref{eq:drift_deriv_1} that
the $\driftderiv$ term can generically be written as follows:
\[
\driftderiv(\theta, \allxi) = \frac{1}{J} \sum_{j=1}^{J}
(\xin{j} \otimes \xin{j} \otimes \xin{j}) : \Psi(\theta) \, ,
\]
for an appropriate matrix-valued function~$\Psi$ depending on the form of~$\ell$,
and involving up to the second derivatives of~$G$.
Therefore, 
we can rewrite the sum of the two error terms as
\begin{align*}
\widetilde {\mathcal A}_2^N + \mathcal A_3^N 
&= h\sigma \sum_{n=0}^{N-1} \expect \biggl[ \frac{1}{J} \sum_{j=1}^J \Bigl( \dxin{j}{n} \otimes \dxin{j}{n} \otimes \dxin{j}{n} \Bigr) \cdot \Bigl( \Psi(\dtheta{n}) \otimes \nabla u(nh+h, \dtheta{n}) \Bigr) \biggr] \\
&\qquad  
- h \sum_{n=0}^{N-1} \expect \biggl[  \frac{1}{J} \sum_{j=1}^J \Bigl( \dxin{j}{n} \otimes \dxin{j}{n} - \I_d \Bigr)
    : \Bigl( \nabla \Phi(\dtheta{n}) \otimes \nabla u(nh+h, \dtheta{n}) \Bigr)
\biggr] \, .
\end{align*}
Both terms can be rewritten in the form
\[
    \mathcal B^N = h \sum_{n=0}^{N-1} \expect \biggl[ \frac{1}{J} \sum_{j=1}^J B\bigl(\dxin{j}{n}\bigr) \cdot g(nh+h, \dtheta{n}) \biggr] \, ,
\]
where~$B$ is a tensor-valued function (of order two or three, but this does not matter for the reasoning)
of the fast variable~$\xi$ that is centered with respect to the Gaussian measure~$\fastgaussianone$.
The definition of~$g$ is different for the two error terms~$\widetilde {\mathcal A}_2^N$ and~$\mathcal A_3^N$.
However,
given that~$\Psi$ and~$\nabla \Phi$ both belong to $\cont^4(\torus^d)$ by assumption and that~$\sigma\leq1$,
and since~$u$ is a solution to the backward Kolmogorov equation~\eqref{eq:kbe_discrete},
it holds in both cases that
\begin{equation}
    \label{eq:bound_ell_discrete}
    \bigl\lVert g(s, \placeholder) \bigr\rVert_{4,\infty} \lesssim \bigl\lVert u(s, \placeholder) \bigr\rVert_{5,\infty} \, , 
    \qquad
    \bigl\lVert \partial_{t} g(s, \placeholder) \bigr\rVert_{\infty}
    \lesssim 
    \bigl\lVert u(s, \placeholder) \bigr\rVert_{5,\infty} \, .
\end{equation}
The spatial regularity of~$u$ gives~$\nabla u\in\cont([0,Nh];\cont^4)$,
while~\eqref{eq:kbe_discrete} gives~$\partial_t\nabla u=-\nabla\genlimit u\in\cont([0,Nh];\cont^2)$.
By~\cref{lemma:averaging_discrete_weak},
whose regularity assumption is therefore satisfied by~$g$, we obtain
\begin{align*}
    \bigl\lvert \mathcal B^N \bigr\rvert
    \lesssim 
    \frac{h + \varepsilon}{\sqrt{J}} 
        \left( \sup_{s\in[0,Nh]} \lVert g(s,\placeholder)\rVert_\infty + h \sum_{n=0}^{N-1} \biggl( \sup_{r\in[nh,nh+h]}  \lVert \partial_t g(r,\placeholder)\rVert_\infty + \lVert g(nh+h,\placeholder)\rVert_{4,\infty} \biggr)  \right)
        \lesssim \frac{h + \varepsilon}{\sqrt{J}} \lVert f \rVert_{5,\infty}
    \, ,
\end{align*}
and so it follows that
\[
    \bigl\lvert \widetilde {\mathcal A}_2^N + \mathcal A_3^N \bigr\rvert
    \lesssim \frac{h+\varepsilon}{\sqrt{J}} \, \norm*{f}_{5,\infty} \, .
\]
The stated estimate~\eqref{eq:uniform_discrete_weak_convergence} then follows by combining the bounds on the three error terms~$\mathcal A_1^N$, $\mathcal A_2^N$ and~$\mathcal A_3^N$,
which concludes the proof.
\end{proof}

\section{Conclusions and perspectives}%
\label{sec:conclusion}

In this paper,
we considered the problem of sampling from a generic probability measure~$\pi \propto \exp(-\ell\circ G) \, \d x$ supported on the torus~$\torus^d$,
and we mathematically analyzed a multiscale numerical method for this purpose that does not require the derivatives of~$G$.
Using techniques from averaging theory,
we obtained careful error estimates showing convergence of the method to the overdamped Langevin dynamics in different metrics of practical interest.
Unlike derivative-free methods based on the ensemble Kalman filter,
the method analyzed in the present paper can be refined to obtain an approximate sample from the target measure~$\pi$,
with an error that can be made arbitrarily small.
As potential avenues for future research,
we mention the following perspectives:

\begin{itemize}
    \item
        We focused for technical simplicity on the sampling problem for a probability measure~$\pi$ supported on the torus~$\torus^d$,
        and many of our proofs rely crucially on the compactness of the domain.
        An interesting perspective is to extend our main results to the unbounded setting~$\real^d$.
        We expect this extension to be possible,
        under suitable assumptions on the potential~$\Phi$,
        using some of the techniques developed in~\cite{pardoux2001poisson}.

    \item
        We focused on a multiscale dynamics without any preconditioner.
        As demonstrated in~\cite{pavliotis21derivative},
        a fixed preconditioner can be incorporated into the multiscale dynamics~\eqref{eq:the_multiscale_dynamics} and
        its Euler-slow/fast-exact discretization to improve the convergence of the method.
        It should be immediately possible to extend our main results to this setting,
        without any conceptual obstruction.
        In addition, it would be interesting to consider a preconditioner that is position-dependent.
        For instance, a natural idea that was suggested to us by Sebastian Reich would be to consider a preconditioner that approximates the inverse of the Hessian of the potential~$\Phi = \ell \circ G$,
        using a formula extending~\eqref{eq:stein_identity} to second derivatives.
\end{itemize}

\appendix
\section{Technical auxiliary results}

Before presenting the main results of this appendix in~\cref{sub:sup_ou,sub:lln_fun,sub:decay_ou},
we first recall a few elementary results that are used in the proofs.
We first recall the following elementary inequality: 
\begin{equation}
    \label{eq:elem_1}
    \forall x > 0, \qquad
    1 + \frac{x}{2}  \leq \frac{x}{1 - \e^{-x}} \leq 1 + x \, .
\end{equation}

Next, we recall the following result.
\begin{lemma}
    [Marcinkiewicz--Zygmund coupled with Jensen's inequality]
    \label{lemma:mz_inequality}
    % Does not work in all normed vector spaces!
    % Let \(\left(\mathcal X,\abs{\cdot}\right)\) be a normed vector space.
    Let $X_1, X_2, \dotsc$ be~$\real^d$-valued i.i.d.\ random variables with mean zero and finite $p$-th moment for some $p \geq 2$.
    There exists a function \(C_{\rm MZ}(p) > 0\) such that
    the following holds:
    \[
        \forall N \in \nat_{>0}, \qquad
        \expect{\left[\frac{1}{N^p}\Biggl\lvert \sum_{j = 1}^N X_j \Biggr\rvert^p\right]} \leq
        \frac{C_{\rm MZ}(p)}{N^{\frac{p}{2}}}\expect\Bigl[\abs*{X_1}^p \Bigr] \, .
    \]
\end{lemma}
\begin{proof}
    Fix $p \geq 2$ and~$N \in \nat_{>0}$.
    By the Marcinkiewicz--Zygmund inequality, there exists \(C_{\rm MZ}(p) > 0 \) such that
    \[
        \expect{\left[\Biggl\lvert \sum_{j=1}^N X_j \Biggr\rvert^p\right]} 
        \leq C_{\rm MZ}(p)\expect{\left[{\left(\sum_{j = 1}^N\abs{X_j}^2\right)}^{\frac{p}{2}}\right]} \, .
    \]
    By Jensen's inequality applied to the convex function \(x \mapsto x^{\frac{p}{2}}\) (or by H\"older's inequality),
    we then have
    \[
        \expect{\left[\left(\sum_{j = 1}^N\abs{X_j}^2\right)^{\frac{p}{2}}\right]} 
        = N^{\frac{p}{2}} \expect{\left[{\left(\frac{1}{N}\sum_{j=1}^N\abs{X_j}^2\right)}^{\frac{p}{2}}\right]}
        \leq N^{\frac{p}{2}} \expect{\left[\frac{1}{N}\sum_{j=1}^N {\lvert X_j \rvert}^p\right]} 
        = N^{\frac{p}{2}}\expect{\bigl[\lvert X_1 \rvert^p\bigr]} \, ,
    \]
    which leads to the desired inequality.
\end{proof}

Finally, we prove a simple result on the expectation of the maximum of a collection of identically distributed random variables.
\begin{lemma}
    \label{lemma:auxiliary_lemma_iid}
    Let \(N\in \nat_{>0}\) and let \(\left(Z_k\right)_{1\leq k\leq N}\) be a collection of
    identically distributed real-valued random variables.
    Then for all \(q \geq 1\)
    it holds that
    \begin{equation}
        \label{eq:coarse_bound}
       \expect\Bigl[\max_{1\leq k\leq N}|Z_k|\Bigr]\leq N^{\frac{1}{q}}
       \Bigl(\expect\bigl[|Z_1|^q\bigr]\Bigr)^{\frac{1}{q}} \, ,
     \end{equation}
     where the right-hand side may be \(+ \infty\).
\end{lemma}
\begin{proof}
    Applying Jensen's inequality to the
    convex function \(x\mapsto |x|^q\),
    we have
    \[
        \left(\expect\biggl[\max_{1\leq k\leq N}|Z_k|\biggr]\right)^q
        \leq \expect\biggl[\max_{1\leq k\leq N}|Z_k|^q\biggr] 
        \leq \expect{\left[\sum_{k=1}^N|Z_k|^q\right]}
        = N\expect\bigl[|Z_1|^q\bigr] \, .
    \]
    Taking the \(q\)-th root of each side,
    we obtain the statement.
\end{proof}

\subsection{Supremum of the Ornstein--Uhlenbeck process}%
\label{sub:sup_ou}
Using \cref{lemma:mz_inequality,lemma:auxiliary_lemma_iid}, 
we can prove the following result.
\begin{proposition}
  \label{corollary:stationary_empirical_mean}
  Fix~$\varepsilon > 0$ and consider \(J\) independent copies \((\xin{j}_t)_{t\geq 0}\) of the stationary Ornstein--Uhlenbeck process with time scale parameter \(\varepsilon\),
  each of which solves the SDE
  \begin{equation*}
      \d \xin{j}_t = -\frac{1}{\varepsilon}\xin{j}_t \, \d t
      +\sqrt{\frac{2}{\varepsilon}} \, \d \wn{j}_t \, ,
      \qquad \xin{j}_0\sim\normal(0,\I_d) \, , \qquad j \in \range{1}{J} \, ,
  \end{equation*}
  where \((\wn{j}_t)_{t\geq 0}\) are independent \(d\)-dimensional Brownian motions.
  Suppose that \(F\in \cont^2(\R^d, \real^k)\) satisfies the centering condition \(\expect\bigl[F(\xin{1}_0)\bigr] = 0\).
  Then, for all \(h,T\geq 0\) and for all \(p,q\geq 1\) such that \(pq \geq 2\),
  it holds that
  \begin{equation}
      \label{eq:statement_corollary}
      \expect\left[\max_{0\le n \le T/h} \left|\frac{1}{J}\sum_{j=1}^JF(\xin{j}_{nh})\right|^p\right]
      \leq  J^{- \frac{p}{2}} \left(1 + \min \left\{\frac{T}{h}, \frac{T}{\varepsilon} \right\} \right)^{\frac{1}{q}} \bigl\lvert C(pq) \mathcal M_{pq} \bigr\rvert^{\frac{1}{q}} \, ,
  \end{equation}
  where the maximum on the left-hand side is taken over integers \(n\),
  and~$C(pq)$ is a constant depending on~\(pq\) and arising from the use of the Marcinkiewicz--Zygmund and Burkholder--Davis--Gundy inequalities in the proof
  (the precise value of this constant is not important for our purposes in this paper).
  Finally \(\mathcal M_{pq} \in [0, \infty]\) is given by
  \[
      \mathcal M_{pq} =
      \expect\Bigl[
          \bigl\lvert F(X)\bigr\rvert^{pq}
          + \bigl\lvert - X \cdot \nabla F(X) + \laplacian F(X) \bigr\rvert^{pq}
          + 2^{\frac{pq}{2}} \bigl\lvert \nabla F(X) \bigr\rvert^{pq}
      \Bigr] \, ,
      \qquad
      X\sim \normal(0,\I_d) \, .
  \]
\end{proposition}

\begin{remark}
    \Cref{corollary:stationary_empirical_mean} should be understood as stating that
    the quantity
    \[
        \mathcal E \coloneq \expect{\left[\max_{0\le n \le T/h} \left|\frac{1}{J}\sum_{j=1}^JF(\xin{j}_{nh})\right|^p\right]}
    \]
    decreases at the usual Monte Carlo rate as~$J \to \infty$,
    with a prefactor that is controlled in terms of \(T / \min\{\varepsilon,h\}\).
    For \(J = 1\) and for fixed~$(\varepsilon,h)$,
    it is known from~\cite[Lemma B.1]{pavliotis21derivative} that~\(\mathcal E\) grows in fact logarithmically with \(T\),
    but such a sharp estimate is not necessary for our purposes.
\end{remark}

\begin{proof}
    We assume that~$k = 1$, 
    so that \(F\) is scalar-valued, but the proof generalizes easily to vector-valued functions.
    We prove two bounds independently,
    each corresponding to a different argument of the minimum in~\eqref{eq:statement_corollary},
    which together imply~\eqref{eq:statement_corollary}.

    \paragraph{Proof of the first bound}
    We can apply \cref{lemma:auxiliary_lemma_iid} to
    the collection \(\left(Z_n\right)_{0\leq n\leq T/h}\)
    with
    \[
        Z_n = \left|\frac{1}{J}\sum_{j=1}^JF(\xin{j}_{nh})\right|^p
    \]
    to obtain the inequality
    \[
        \expect{\left[\max_{0\le n \le T/h} \left|\frac{1}{J}\sum_{j=1}^JF(\xin{j}_{nh})\right|^p\right]}
        \leq \left( \frac{T}{h} + 1\right)^{\frac{1}{q}}
        \expect{\left[\left|\frac{1}{J}\sum_{j=1}^JF(\xin{j}_{0})\right|^{pq}\right]}^{\frac{1}{q}} \, .
    \]
    Since \(\expect\bigl[F(\xin{1}_0)\bigr] = 0\) by assumption,
    there exists by the Marcinkiewicz--Zygmund inequality
    \(C_{\rm MZ} = C_{\rm MZ}(pq) > 0\) not depending on \(J\) such that
    \begin{align*}
        \expect{\left[\left|\frac{1}{J}\sum_{j=1}^JF(\xin{j}_{0})\right|^{pq}\right]}
        &\leq \frac{C_{\rm MZ}}{J^{pq}} \expect{\left[\left(\sum_{j=1}^J\left|F(\xin{j}_0)\right|^2\right)^{\frac{pq}{2}}\right]} \\
        &\leq \frac{C_{\rm MZ}}{J^{\frac{pq}{2} + 1}}\expect{\left[\sum_{j=1}^J|F(\xin{j}_0)|^{pq}\right]}
        = \frac{C_{\rm MZ}}{J^{\frac{pq}{2}}}\expect \Bigl[\bigl\lvert F(\xin{1}_0) \bigr\rvert^{pq}\Bigr] \, ,
    \end{align*}
    where we applied Jensen's inequality in the second line,
    noting that \(pq\geq 2\).
    Hence, we obtain that
    \begin{equation}
        \label{eq:first_bound}
        \expect{\left[\max_{0\le n \le
        T/h}\left|\frac{1}{J}\sum_{j=1}^JF(\xin{j}_{nh})\right|^p\right]}
        \leq J^{-\frac{p}{2}} \bigl(C_{\rm MZ}(pq) \mathcal M_{pq}\bigr)^{\frac{1}{q}} \left( \frac{T}{h} + 1\right)^{\frac{1}{q}} \, .
    \end{equation}

    \paragraph{Proof of the second bound}
    We use that
    \[
        \sup_{0\leq t\leq T}\left|\frac{1}{J}\sum_{j=1}^JF( \xin{j}_t)\right|^p
        \leq \max_{0\leq n \leq T/\varepsilon} Y_n \, ,
        \qquad Y_n \coloneq \sup_{t \in [n \varepsilon,  (n+1) \varepsilon]}\left|\frac{1}{J}\sum_{j=1}^JF(\xin{j}_t)\right|^p \, .
    \]
    Since the processes \((\xin{j}_t)_{t\geq 0}\) are stationary,
    the random variables \((Y_n)\) are identically distributed,
    and so we can apply \cref{lemma:auxiliary_lemma_iid} to obtain
    \[
        \expect\left[\sup_{0\leq t\leq T}\left|\frac{1}{J}\sum_{j=1}^JF(\xin{j}_t)\right|^p\right]
        \leq \expect\biggl[\max_{0\leq n \leq T/\varepsilon} Y_n\biggr]
        \leq
        \left( \frac{T}{\varepsilon} + 1\right)^{\frac{1}{q}} \expect \bigl[ \lvert Y_0 \rvert^q \bigr]^{\frac{1}{q}} \, .
    \]
    It remains to bound the expectation on the right-hand side.
    By It\^o's lemma,
    it holds that
    \begin{equation}
        \frac{1}{J}\sum_{j=1}^JF(\xin{j}_t)
        = \frac{1}{J}\sum_{j=1}^J F(\xin{j}_0) +\frac{1}{J \varepsilon}\sum_{j=1}^J\int_0^t
        \mathcal L_{\xi}F(\xin{j}_s) \, \d s
        + \frac{\sqrt 2}{J \sqrt{\varepsilon}}\sum_{j=1}^J\int_0^t \nabla F(\xin{j}_s) \cdot \d \wn{j}_s \, ,
        \label{eq:ito_F}
    \end{equation}
    where \(\mathcal L_{\xi}\) is the infinitesimal generator of the
    standard Ornstein--Uhlenbeck process:
    \(
        \mathcal L_{\xi} = - \xi \cdot \nabla + \laplacian.
    \)
    Raising both sides of~\eqref{eq:ito_F} to the power \(pq\),
    taking the supremum, taking the expectation,
    and using Jensen's inequality,
    we obtain
    \begin{equation}
        \begin{aligned}
            \frac{1}{3^{pq - 1}} \expect {\left[ \sup_{t \in [0, \varepsilon]}\left|\frac{1}{J}\sum_{j=1}^JF(\xin{j}_t)\right|^{pq} \right]}
    &\leq \expect{\left[\left|\frac{1}{J} \sum_{j=1}^J F(
    \xin{j}_0)\right|^{pq}\right]} +
    \frac{1}{\varepsilon} \int_0^{\varepsilon} 
    \expect {\left[\left|\frac{1}{J}\sum_{j=1}^J \mathcal L_{ \xi}F( \xin{j}_s)\right|^{pq}\right]} \, \d s\\
    &\qquad
    + \frac{2^{\frac{pq}{2}}}{\varepsilon^{\frac{pq}{2}}}
    \expect{\left[\sup_{t \in [0, \varepsilon]} \left|\frac{1}{J}\sum_{j=1}^J\int_0^t \, \nabla F(\xin{j}_s) \cdot \d \wn{j}_s \right|^{pq}\right]}
        \end{aligned}
        \label{eq:ub_ito_F}
    \end{equation}
    Since the \(J\) processes \((\xin{j}_t)_{t\geq 0}\) are independent and
    identically distributed and
    \(
    \expect\Bigl[F(\xin{1}_t)\Bigr]
    = \expect\Bigl[\mathcal L_{\xi}F(\xin{1}_t)\Bigr]
    = 0
    \)
    for all~\(t\geq 0\) by stationarity,
    we have by Marcinkiewicz--Zygmund inequality
    \[
        \expect{\left[\left|\frac{1}{J} \sum_{j=1}^J F( \xin{j}_0)\right|^{pq}\right]} 
        \leq \frac{C_{\rm MZ}(pq)}{J^{\frac{pq}{2}}}\expect\left[\left|F(\xin{1}_0)\right|^{pq}\right] \, ,
        \qquad \expect{\left[\left|\frac{1}{J} \sum_{j=1}^J \mathcal L_{ \xi}F( \xin{j}_0)\right|^{pq}\right]}
        \leq \frac{C_{\rm MZ}(pq)}{J^{\frac{pq}{2}}}\expect{\left[\left|\mathcal L_{\xi}F(\xi_0^1)\right|^{pq}\right]} \, .
    \]
    In order to bound the third term on the right-hand side of \eqref{eq:ub_ito_F},
    we note that the argument of the absolute value is a martingale with quadratic variation process
    \[
        \left\langle \frac{1}{J}\sum_{j=1}^J\int_0^t \, \nabla F(\xin{j}_s) \cdot \d \wn{j}_s \right\rangle_t
        = \frac{1}{J^2} \sum_{j=1}^{J} \int_{0}^{t} \lvert \nabla F(\xin{j}_s) \rvert^2 \d s \, .
    \]
    Therefore, by the Burkholder--Davis--Gundy inequality,
    there is \(C_{\rm BDG} = C_{\rm BDG}(pq)\) such that
    \begin{align*}
        \frac{1}{\varepsilon^{\frac{pq}{2}}} \expect{\left[\sup_{t \in [0, \varepsilon]} \left|\frac{1}{J}\sum_{j=1}^J\int_0^t \, \nabla F(\xin{j}_s) \cdot \d \wn{j}_s \right|^{pq}\right]}
          &\leq \frac{C_{\rm BDG}}{\varepsilon^{\frac{pq}{2}}}
          \expect{\left[\left|\frac{1}{J^2}\sum_{j=1}^{J} \int_{0}^{\varepsilon} \lvert \nabla F(\xin{j}_s) \rvert^2 \, \d s \right|^{\frac{pq}{2}}\right]} \\
          &\leq \frac{C_{\rm BDG}}{\varepsilon J^{\frac{pq}{2} + 1}}
          \expect{\left[ \sum_{j=1}^{J} \int_{0}^{\varepsilon} \bigl\lvert \nabla F(\xin{j}_s) \bigr\rvert^{pq} \, \d s \right]}
          = \frac{C_{\rm BDG}}{J^{\frac{pq}{2}}}
          \expect{\left[ \bigl\lvert \nabla F(\xin{1}_0) \bigr\rvert^{pq}  \right]} \, .
    \end{align*}
    We have thus shown,
    denoting by~$C = 3^{pq-1}\max\{C_{\rm MZ}(pq), C_{\rm BDG}(pq)\}$,
    that
    \begin{equation}
        \label{eq:second_bound}
        \expect\left[\sup_{0\leq t\leq
        T}\left|\frac{1}{J}\sum_{j=1}^JF(\xin{j}_t)\right|^p\right] \leq J^{-\frac{p}{2}} \bigl(C \mathcal M_{pq}\bigr)^{\frac{1}{q}} \left(1 + \frac{T}{\varepsilon}\right)^{\frac{1}{q}} \, .
    \end{equation}

    \paragraph{Conclusion of the proof}
    Combining~\eqref{eq:first_bound} and~\eqref{eq:second_bound},
    we obtain the claimed inequality~\eqref{eq:statement_corollary}.
\end{proof}

\subsection{A law of large numbers for functions}
\label{sub:lln_fun}

Suppose that \(X \sim \mu\) for some probability distribution~\(\mu \in \mathcal P(\real^d)\)
and that \(f \colon \torus^d \times \real^d \to \real\) satisfies
\[
    \forall \theta \in \torus^d, \qquad
    \expect \left[ f(\theta, X) \right] = 0 \, .
\]
Let \(X_{1}, X_{2}, \dotsc\) denote i.i.d.\ samples with common law~\(\mu\).
By the Marcinkiewicz--Zygmund inequality,
for all \(\theta \in \torus^d\) and all \(p \geq 2\),
it holds that
\begin{equation}
    \label{eq:mz_pointwise}
    \forall J \in \nat_{>0}, \qquad
    \expect \left[ \biggl\lvert \frac{1}{J} \sum_{j=1}^{J} f\bigl(\theta, X_j\bigr) \biggr\rvert^p \right]
    \leq \frac{C_{\rm MZ}(p)}{J^{\frac{p}{2}}} \expect \Bigl\lvert f\bigl(\theta, X_1 \bigr) \Bigr\rvert^p \, .
\end{equation}
Thus, provided that the right-hand side is finite,
convergence at the classical Monte Carlo rate holds pointwise in~\(\theta\).
The proposition following addresses the question of whether the random function
\[
    \theta \mapsto \frac{1}{J} \sum_{j=1}^{J} f\bigl(\theta, X_j \bigr)
\]
can be shown to converge to zero in the functional supremum norm.
In other words,
can we prove a functional counterpart of~\eqref{eq:mz_pointwise}?
This is the content of the following result.

\begin{proposition}
    [Expectation of the supremum over the torus of an empirical mean]
    \label{proposition:expect_supremum_torus}
    Consider a probability distribution \(\mu\in \mathcal P\left(\R^d\right)\) and 
    let \((X_j)_{j \in \nat_{>0}}\) be a sequence of i.i.d.\ samples with common distribution~\(\mu\).
    Assume that~\(f \colon \torus^d \times \real^d \to \real^K\) is a function satisfying the following conditions
    \begin{itemize}
        \item For all \(\theta \in \torus^d\), it holds that \(\expect \bigl[ f(\theta, X_1) \bigr] = 0\).
        \item For all~$\xi \in \real^d$, the function \(\theta \mapsto f(\theta, \xi)\) belongs to \(\cont^1(\torus^d, \real^K)\).
        \item It holds that \(\expect\bigl[\sup_{\theta \in \torus^d} \lvert \nabla_{\theta} f(\theta, X_1) \rvert\bigr] < \infty\).
    \end{itemize}
    Then for all \(p\geq 2\),
    there exists \(C = C(d, K, p) > 0\) such that for all \(J\in \nat_{>0}\), it holds that
    \begin{equation*}
        \expect\left[\sup_{\theta\in \torus^d} \biggl\lvert \frac{1}{J}\sum_{j = 1}^J f(\theta,X_j) \biggr\rvert^p\right]
        \leq C J^{-\frac{p}{2}} \biggl( \int_{\torus^d} \expect \Bigl\lvert f\bigl(\theta, X_{1}\bigr) \Bigr\rvert^{q}
            + \expect \Bigl\lvert  \nabla_{\theta} f\bigl(\theta, X_{1}\bigr) \Bigr\rvert^{q}
        \, \d \theta \biggr)^{\frac{p}{q}} \, ,
    \end{equation*}
    where \(q := \max\{p, d+1\}\).
    (Note that the right-hand side of this inequality may be infinite,
    in which case this result does not provide any useful information.)
\end{proposition}

\begin{proof}
    We assume for simplicity that \(f\) is scalar-valued,
    but the proof generalizes easily to vector-valued functions.
    It will be convenient to introduce the random function
    \[
        f^J(\theta; \omega) =
        \frac{1}{J} \sum_{j=1}^{J} f\bigl(\theta, X_j(\omega)\bigr) \, .
    \]
    By assumption, for every fixed \(\xi \in \real^d\),
    the function \(\theta \mapsto f(\theta, \xi)\) belongs to \(\cont^1(\torus^d)\),
    thus also to~\(W^{1,p}(\torus^d)\) for all~\(p \in [1, \infty]\).
    By the Sobolev embedding theorem,
    there exists a constant \(C_{\rm Sob}(d)\) such that,
    almost surely,
    \[
        \sup_{\theta \in \torus^d} \bigl\lvert f^J(\theta) \bigr\rvert
        \leq
        C_{\rm Sob}
        \left( \int_{\torus^d}
            \bigl\lvert f^J(\theta) \bigr\rvert^{d+1}
            + \bigl\lvert \nabla_{\theta} f^J(\theta) \bigr\rvert^{d+1}
        \, \d \theta \right)^{\frac{1}{d+1}} \, .
    \]
    By Jensen's inequality, it holds that
    \begin{align*}
        \expect \sup_{\theta \in \torus^d} \bigl\lvert f^J(\theta) \bigr\rvert^p
        \leq 
        \left( \expect {\left[  \sup_{\theta \in \torus^d} \bigl\lvert f^J(\theta) \bigr\rvert^{q} \right] } \right)^{\frac{p}{q}}
        &\leq
        C_{\rm Sob}^{p} 
        \left( \expect {\left[ \left( \int_{\torus^d} \bigl\lvert f^J(\theta) \bigr\rvert^{d+1} + \bigl\lvert \nabla_{\theta} f^J(\theta) \bigr\rvert^{d+1} \, \d \theta \right)^{\frac{q}{d+1}} \right]} \right)^{\frac{p}{q}} \\
        &\leq
        C_{\rm Sob}^{p} 2^{\frac{p}{q}\bigl(\frac{q}{d+1} - 1\bigr)}
        \left(  \int_{\torus^d} \expect\bigl\lvert f^J(\theta) \bigr\rvert^{q} + \expect\bigl\lvert \nabla_{\theta} f^J(\theta) \bigr\rvert^{q} \, \d \theta \right)^{\frac{p}{q}} \, .
    \end{align*}
    Notice that, since \(\expect {\left[ f\bigl(\theta, X_{1}\bigr) \right]} = 0\) for all \(\theta \in \torus^d\),
    it holds also that \(\expect {\left[ \nabla_{\theta} f\bigl(\theta, X_{1}\bigr) \right]} = 0\) for all \(\theta \in \torus^d\).
    Indeed, differentiating under the expectation is justified by the dominated convergence theorem:
    by the mean value theorem, the difference quotients of~\(f\) in~\(\theta\) are bounded pointwise by
    \(\sup_{\theta \in \torus^d} \lvert \nabla_{\theta} f(\theta, X_1) \rvert\),
    which has finite expectation by assumption.
    Therefore, using the Marcinkiewicz--Zygmund inequality (\cref{lemma:mz_inequality}),
    we obtain
    \[
        \expect \sup_{\theta \in \torus^d} \biggl\lvert \frac{1}{J} \sum_{j=1}^{J} f\bigl(\theta, X_j\bigr) \biggr\rvert^p
        \leq 2^{\frac{p}{q}\bigl(\frac{q}{d+1} - 1\bigr)}
        C_{\rm Sob}^{p} C_{\rm MZ}(q)^{\frac{p}{q}} J^{-\frac{p}{2}}
        \biggl( \int_{\torus^d} \expect \biggl\lvert f\bigl(\theta, X_{1}\bigr) \biggr\rvert^{q}
            + \expect \biggl\lvert  \nabla_{\theta} f\bigl(\theta, X_{1}\bigr) \biggr\rvert^{q}
        \, \d \theta \biggr)^{\frac{p}{q}} \, ,
    \]
    which concludes the proof.
\end{proof}

\subsection{Decay estimates for the Ornstein--Uhlenbeck semigroup}%
\label{sub:decay_ou}

In this section,
we recall useful properties of the Ornstein--Uhlenbeck semigroup.
Most of the content of this section is either standard or follows from standard results,
but we include it for the reader's convenience.
Consider the Ornstein--Uhlenbeck process in \(\real^d\) started from~\(x\),
which solves the SDE
\begin{equation}
    \label{eq:ou}
    \d \xi^x_t = - \xi^x_t \, \d t + \sqrt{2} \, \d W_t \, , \qquad \xi_0^x = x \, .
\end{equation}
The operator semigroup associated with this equation is denoted by \(t \mapsto \e^{t \mathcal L_{\xi}}\).
For \(t \geq 0\), the operator \(\e^{t \mathcal L_{\xi}}\) acts on functions \(\varphi\) as
\[
    \e^{t \mathcal L_{\xi}} \varphi(x) \coloneq \expect \bigl[ \varphi(\xi^{x}_t) \bigr] \, .
\]
It is well-known that this defines a bounded linear operator in appropriate function spaces~\cite{MR2313847},
in particular in the weighted \(L^\infty\) spaces considered below in \cref{lemma:ou_semigroup}.
The generator of this semigroup is denoted by~\(\mathcal L_{\xi}\)
and given on \(\cont^{\infty}_{\rm c}(\real^d)\) by
\begin{equation}
    \label{eq:generator_ou}
    \mathcal L_{\xi} = -\xi\cdot\nabla_\xi +\Delta_\xi \, .
\end{equation}

\begin{lemma}
    [Weighted \(L^{\infty}\) decay for the Ornstein-Uhlenbeck semi-group]
  \label{lemma:ou_semigroup}
    For~$m \geq 1$, using the spaces~$L^\infty_{\weight{m}}$
    and~$L^\infty_{\weight{m},0}$
    \ifsiamart defined in the main paper in~\eqref{eq:def_weighted_L_infty_space}\else defined in~\eqref{eq:def_weighted_L_infty_space}\fi,
    there exists \(C_{\mathcal L_\xi} > 0\) such that
    \begin{equation}
        \label{eq:convergence_to_equilibrium}
        \forall \varphi \in L^\infty_{\weight{m},0}, \qquad
        \norm*{\e^{t\mathcal L_\xi}\varphi}_{L^\infty_{\weight{m}}}
        \leq
        C_{\mathcal L_\xi} \e^{-t}
        \norm{\varphi}_{L^\infty_{\weight{m}}} \, .
    \end{equation}
\end{lemma}

\begin{proof}
    Fix \(x \in \real^d\) and let \(\gamma_{\infty} = \normal(0, \I_d)\) denote the standard centered normal distribution in~\(\real^d\).
    Since the law at time~\(t\) of the Ornstein--Uhlenbeck process initialized at a deterministic position~\(x\) is given by the normal distribution~\(\gamma^{x}_t \coloneq \normal\bigl(\e^{-t} x, (1 - \e^{-2t}) \I_d \bigr)\),
    it holds that
    \begin{equation*}
        \e^{t\mathcal L_\xi} \varphi(x) = \int_{\real^d} \varphi(y) \, \gamma^x_t (\d y) \, .
    \end{equation*}
    The function \(\varphi\) has mean zero with respect to~\(\gamma_{\infty}\) by assumption,
    and so we deduce that
    \[
        \bigl\lvert \e^{t\mathcal L_\xi} \varphi(x) \bigr\rvert
        =
        \Bigl\lvert
            \int_{\real^d} \varphi(y) \, (\gamma^x_t - \gamma_{\infty}) (\d y)
        \Bigr\rvert
        \leq
        d_m(\gamma^x_t, \gamma_{\infty})  \norm{\varphi}_{L^\infty_{\weight{m}}} \, ,
    \]
    where we introduced the weighted total variation distance for probability measures \(\mu, \nu \in \mathcal P(\real^d)\):
    \[
        d_{m} (\mu, \nu)
        = \int_{\real^d} \bigl(1 + |y|^m\bigr) \, |\mu - \nu| (\d y) \, .
    \]
    A standard argument based on Harris' theorem~\cite{MR2857021} can be employed to prove the convergence of \(\gamma^x_t\) to \(\gamma_{\infty}\) in the~\(d_m\) distance at an exponential rate~\(\e^{- \lambda t}\),
    for an appropriate constant~\(\lambda > 0\) that is usually difficult to control quantitatively.
    For the reader's convenience,
    here we prove convergence in~$d_m$ in a self-contained manner, with the exponential rate~\(\e^{-t}\).
    Such a rate of convergence (without the multiplicative constant~$\lambda$)
    is expected, but we were not able to find a reference in the literature where it is explicitly stated and proved.
    We refer to~\cite[Chapter 2]{MR2459454} and~\cite{MR2313847}
    for proofs of convergence with this rate in the Wasserstein and \(L^p\) metrics,
    respectively.
    We also mention the works~\cite{2018arXiv181008693D,carrillo2022ensemble},
    where the total variation distance between Gaussian distributions is studied.
    Writing
    \[
        \gamma^x_t = \frac{1}{Z_t} \exp \left( -\frac{|y - \mu_t|^2}{2 \sigma_t^2}\right) \, ,
        \qquad
        \gamma_{\infty} = \frac{1}{Z_{\infty}} \exp \left( -\frac{|y|^2}{2}\right) \, ,
    \]
    where \(\mu_t = \e^{-t} x\) and \(\sigma_t^2 = 1 - \e^{-2t}\),
    and using explicit formulas for the moments of a Gaussian probability distribution,
    we have that
    \begin{align*}
        \forall t  \leq  \log \left( 2 + |x| \right), \qquad
        d_m(\gamma^x_t, \gamma_{\infty})
        &\leq
        \int_{\real^d} \bigl(1 + |y|^m\bigr) \, (\gamma^x_t + \gamma_{\infty}) (\d y) \\
        &\lesssim_{m} 1 + |\e^{-t} x|^m
        = \e^{-mt} (\e^{mt} + |x|^m)
        \lesssim_{m} \e^{-mt} \bigl(1  + |x|^m \bigr) \, ,
    \end{align*}
    where, in the last inequality, we used that \(\e^{t} \leq 2 + |x|\) for the values of \(t\) considered.
    To obtain a bound for values of~\(t\) such that \(t \geq \log \left( 2 + |x| \right)\),
    we write
    \begin{align}
        \notag
        d_m(\gamma^x_t, \gamma_{\infty})
        &\leq \frac{1}{Z_t} \int_{\real^d} \bigl(1 + |y|^m\bigr)
        \left\lvert
        \exp \left( -\frac{|y - \mu_t|^2}{2 \sigma_t^2}\right)
        - \exp \left( -\frac{|y|^2}{2}\right)
        \right\rvert
         \, \d y \\
        &\qquad \qquad +
        \left\lvert \frac{1}{Z_t} - \frac{1}{Z_{\infty}} \right\rvert
        \int_{\real^d} \bigl(1 + |y|^m\bigr) \exp \left( -\frac{|y|^2}{2}\right)  \, \d y \,
        \eqcolon \mathcal A_t + \mathcal B_t \, .
        \label{eq:two_terms_total_var}
    \end{align}
    For the first term in this equation,
    we use the elementary inequality~\cite[Equation (A.7)]{MR5046225},
    \[
        \forall  (a, b) \in \real \times \real, \qquad
        \left\lvert \e^{-a^2} - \e^{-b^2} \right\rvert
        \leq 2 \abs{b - a} \left( \e^{-\frac{a^2}{2}} + \e^{-\frac{b^2}{2}} \right) \, ,
    \]
    to deduce that
    \begin{align*}
        \forall t \geq \log(2 + |x|), \qquad
        \mathcal A_t
        &\leq \sqrt{2}
        \int_{\real^d} \bigl(1 + |y|^m\bigr)
             \left\lvert
                 \frac{|y - \mu_t|}{\sigma_t} - |y|
            \right\rvert
            \left(  \exp \left( -\frac{|y - \mu_t|^2}{4 \sigma_t^2}\right) + \exp \left( -\frac{|y|^2}{4}\right)   \right)  \, \d y \\
        &\lesssim
        \e^{-t}
            \int_{\real^d} \bigl(1 + |y|^m\bigr)
             \bigl( |y| + |x|  \bigr)
            \left(  \exp \left( -\frac{|y - \mu_t|^2}{4 \sigma_t^2}\right) + \exp \left( -\frac{|y|^2}{4}\right)   \right)  \, \d y \\
        &\lesssim_m \e^{-t} \bigl( 1 + |x| \bigr) \, .
    \end{align*}
    Here we used that, 
    by the reverse triangle inequality,
    \[
        \forall t \geq \log(2 + |x|), \qquad
        \left\lvert
            \frac{|y - \mu_t|}{\sigma_t} - |y|
        \right\rvert
        \leq 
        \frac{|\mu_t|}{\sigma_t} + \frac{|\sigma_t - 1|}{\sigma_t} |y|
        \lesssim \e^{-t} \bigl(|x| + |y|\bigr) \, ,
    \]
    and that \(\lvert \mu_t \rvert \leq 1\) for \(t \geq \log(2 + |x|)\) to bound the integral including the first exponential term.
    For the second term in~\eqref{eq:two_terms_total_var},
    we use that the normalization constants \(Z_t, Z_{\infty}\) are bounded away from 0 by a strictly positive constant for \(t \geq \log(2)\) to deduce that
    \[
        \forall t \geq \log(2 + |x|), \qquad
        \mathcal B_t
        \lesssim_{m} \left\lvert \frac{1}{Z_t} - \frac{1}{Z_{\infty}} \right\rvert
        \leq \frac{\lvert Z_t - Z_{\infty} \rvert}{Z_t Z_{\infty}}
        \lesssim \e^{-2t} \, .
    \]
    Combining these bounds,
    we obtain that \(d_m(\gamma^x_t, \gamma_{\infty}) \lesssim_m \e^{-t} (1 + |x|^{m})\) for all \(t \geq 0\),
    from which~\eqref{eq:convergence_to_equilibrium} follows.
\end{proof}

In the following lemma,
we use the weighted Sobolev spaces~$\sobolev^{k,\infty}_{\weight{m}}$
\ifsiamart defined in the main paper in~\eqref{eq:weighted_sobolev_norm}\else defined in~\eqref{eq:weighted_sobolev_norm}\fi.
\begin{lemma}
    \label{lemma:ou_semigroup_der}
    With the same notations as in the previous lemma,
    for any multi-index~\(\alpha\),
    there is \(C_{\mathcal L_{\xi}, \alpha}\) such that,
    for any function~\(\varphi \in \sobolev^{|\alpha|,\infty}_{\weight{m}}\),
    it holds
    \begin{equation}
        \label{eq:decay_derivatives}
        \left\lVert D^{\alpha} \e^{t \mathcal L_{\xi}} \varphi \right\rVert_{L^{\infty}_{\weight{m}}}
        \leq C_{\mathcal L_{\xi}, \alpha} \e^{- |\alpha| t}
        \left\lVert D^{\alpha} \varphi \right\rVert_{L^{\infty}_{\weight{m}}} \, .
    \end{equation}
    % Furthermore, 
    % if~$\varphi \in L^{\infty}_{K_{m},0}$,
    % then there exists a unique solution $\Phi \in L^{\infty}_{K_{m},0}$ to the Poisson equation 
    % \begin{equation}
    %     \label{eq:poisson_ou}
    %     -\mathcal L_{\xi}\Phi=\varphi \, ,
    % \end{equation}
    % which furthermore satisfies $\lVert D^{\alpha}\Phi\rVert_{L^{\infty}_{\weight{m}}} \lesssim_{\alpha,m} \lVert D^{\alpha}\varphi\rVert_{L^{\infty}_{\weight{m}}}$.
\end{lemma}
\begin{proof}
    We assume that the right-hand side of~\eqref{eq:decay_derivatives} is finite,
    otherwise the statement is obvious.
    By~\cite{MR4176390}, see also~\cite[Lemma~9.3.6]{MR2313847},
    it holds for all \(\varphi \in L^p(\gamma_{\infty})\) for some \(p \in [1, + \infty)\) that,
    \[
        \e^{t \mathcal L_{\xi}} \varphi(x) = (\varphi \star g_t) (\e^{-t} x) \, ,
        \qquad g_t \coloneq \normal\Bigl(0, (1 - \e^{-2t}) \I_d \Bigr) \, .
    \]
    Differentiating,
    we obtain for any~\(t > 0\)
    \[
        D^{\alpha} \e^{t \mathcal L_{\xi}} \varphi(x) = \e^{- \lvert \alpha \rvert t} (D^{\alpha} \varphi \star g_t) (\e^{-t} x) \, ,
    \]
    and so
    \[
        \frac{\lvert D^{\alpha} \e^{t \mathcal L_{\xi}} \varphi(x) \rvert}{1 + |x|^m} \leq  \e^{-\lvert \alpha \rvert t}
        \norm{D^{\alpha} \varphi \star g_t}_{L^{\infty}_{\weight{m}}} \frac{1 + \e^{-m t} |x|^m}{1 + |x|^m} \, .
    \]
    The estimate~\eqref{eq:decay_derivatives} then follows because convolution with \(g_t\) is a bounded linear operator on \(L^{\infty}_{\weight{m}}\),
    uniformly in~\(t\).
    Indeed, it holds that
    \[
        \frac{\lvert (D^{\alpha} \varphi \star g_t) (x) \rvert}{1 + |x|^m}
        \leq \int_{\real^d} \norm{D^{\alpha} \varphi}_{L^{\infty}_{\weight{m}}} \frac{1 + |x - y|^m}{1 + |x|^m} \, g_t(\d y) \, 
        \lesssim_{m} \norm{D^{\alpha} \varphi}_{L^{\infty}_{\weight{m}}} \, ,
    \]
    which concludes the proof.
    % The existence of a unique solution in~\(L^{\infty}_{\weight{m},0}\) to the Poisson equation~\eqref{eq:poisson_ou} 
    % follows from~\cref{lemma:ou_semigroup} and the following representation formula for the solution operator corresponding to~\eqref{eq:poisson_ou}:
    % \[
    %     \mathcal L_{\xi}^{-1} = - \int_{0}^{+\infty} \e^{t \mathcal L_{\xi}} \, \d t \, .
    % \]
    % See, for instance, \cite[Proposition 2.1]{MR3509213} for more details.
    % The estimate on the derivatives of the solution~$\Phi$ then follows from the decay of the derivatives of the semigroup.
\end{proof}

 \begin{remark}
     It is possible to show decay of the derivatives of \(\e^{t \mathcal L_{\xi}} \varphi\) for $t > 0$
     even when~\(\varphi\) is not differentiable,
     by rewriting the convolution as
     \begin{align*}
         \bigl\lvert D^{\alpha} \e^{t \mathcal L_{\xi}} \varphi(x) \bigr\rvert
         &= \Bigl\lvert \e^{- \lvert \alpha \rvert t} (\varphi \star D^{\alpha} g_t) (\e^{-t} x) \Bigr\rvert \\
         &\leq  \e^{- \lvert \alpha \rvert t} (\weight{m} \star |D^{\alpha} g_t|) (\e^{-t} x)  \norm{\varphi}_{L^{\infty}_{\weight{m}}} \\
         &\leq  \e^{- \lvert \alpha \rvert t} \norm{\weight{m} \star |D^{\alpha} g_t|}_{L^{\infty}_{\weight{m}}}   \norm{\varphi}_{L^{\infty}_{\weight{m}}} \weight{m}(\e^{-t} x) \, .
     \end{align*}
     Since $g_t(x) = \sigma_t^{-d} g_{\infty}(x/\sigma_t)$, with $\sigma_t \coloneq \sqrt{1 - \e^{-2t}}$,
     we have that
     \[
         D^{\alpha} g_t(x) = \frac{1}{\sigma_t^{d + |\alpha|}} D^{\alpha} g_{\infty} \Bigl( \frac{x}{\sigma_t} \Bigr) \, ,
     \]
     and so
     \begin{align*}
         \Bigl( \weight{m} \star |D^{\alpha} g_t| \Bigr) (x) 
         &= \frac{1}{\sigma_t^{d + |\alpha|}} \int_{\real^d} \bigl(1 + |x - y|^m\bigr) \Bigl\lvert D^{\alpha} g_{\infty} \Bigl( \frac{y}{\sigma_t} \Bigr) \Bigr\rvert \, \d y \\
         &= \frac{1}{\sigma_t^{|\alpha|}} \int_{\real^d} \bigl(1 + |x - \sigma_t z|^m\bigr) \Bigl\lvert D^{\alpha} g_{\infty} (z) \Bigr\rvert \, \d z 
         \lesssim_{m} \frac{1}{\sigma_t^{|\alpha|}} \weight{m}(x) \, .
     \end{align*}
     We conclude that
     \[
         \bigl\lVert D^{\alpha} \e^{t \mathcal L_{\xi}} \varphi \bigr\rVert_{L^{\infty}_{\weight{m}}}
         \lesssim_{m} \frac{\e^{-|\alpha|t}}{(1 - \e^{-2t})^{\frac{|\alpha|}{2}}} \norm{\varphi}_{L^{\infty}_{\weight{m}}} \, .
     \]
     This estimate will not,
     however,
     be useful for our purposes in this paper.
 \end{remark}

The following results follow from~\cref{lemma:ou_semigroup,lemma:ou_semigroup_der}.
\begin{corollary}
    \label{corollary:poisson_ou}
    If~\(\varphi \in L^{\infty}_{\weight{m},0}\) for some multi-index~\(\alpha\),
    then there exists a unique solution~\(\Phi \in L^{\infty}_{\weight{m},0}\) to the Poisson equation
    \begin{equation}
        \label{eq:poisson_ou}
        - \mathcal L_{\xi} \Phi = \varphi \, ,
    \end{equation}
    which furthermore satisfies
    \(
    \lVert  D^{\alpha} \Phi \rVert_{L^{\infty}_{\weight{m}}}  \lesssim_{m, \alpha} \lVert D^{\alpha} \varphi \rVert_{L^{\infty}_{\weight{m}}}
    \)
    if \(\varphi \in \sobolev^{|\alpha|, \infty}_{\weight{m}}\).
\end{corollary}
\begin{proof}
    Existence of a solution satisfying the bound follows from the representation formula for the solution operator corresponding to~\eqref{eq:poisson_ou}:
    \[
        \mathcal L_{\xi}^{-1} = - \int_{0}^{+\infty} \e^{t \mathcal L_{\xi}} \, \d t \, .
    \]
    The solution is necessarily unique because,
    if there were two solutions \(\Phi\) and \(\Psi\) in \(L^\infty_{\weight{m},0}(\real^d)\),
    then it would hold by definition of the generator that \(\e^{t \mathcal L_{\xi}} (\Phi - \Psi) = (\Phi - \Psi)\)
    contradicting~\cref{lemma:ou_semigroup}.
    We refer to~\cite[Section 7.4]{MR2597943} and~\cite[Proposition 2.1]{MR3509213} for more details.
\end{proof}

 \begin{corollary}
     [Discrete Poisson equation]
     \label{lemma:discrete_poisson_eq}
     Let \(\mathcal L_\xi\) be the Ornstein-Uhlenbeck generator defined in~\eqref{eq:generator_ou}.
     Fix a time step~\(h > 0\) and a multi-index~\(\alpha\) and
     suppose that \(\varphi\in  L^{\infty}_{\weight{m},0}\).
     Then the discrete Poisson equation
     \begin{equation}
         \label{eq:discrete_poisson2}
         - \frac{\left( \e^{h \mathcal L_{\xi}} - \id \right)}{h} \Phi = \varphi
     \end{equation}
     admits a unique solution in \(L^{\infty}_{\weight{m},0}\).
     Furthermore,
     if \(\varphi \in \sobolev^{|\alpha|, \infty}_{\weight{m}}\),
     then the solution satisfies
     \[
        \forall h \in (0, 1), \qquad
         \lVert  D^{\alpha} \Phi \rVert_{L^{\infty}_{\weight{m}}}  \lesssim_{m, \alpha} \lVert D^{\alpha} \varphi \rVert_{L^{\infty}_{\weight{m}}} \, ,
     \]
     where the right-hand side is allowed to be \(+ \infty\).
\end{corollary}

\begin{proof}
    By~\eqref{eq:convergence_to_equilibrium},
    the Neumann series
    \[
        \sum_{n=0}^{+\infty} \e^{nh \mathcal L_{\xi}}
    \]
    converges in the operator norm on~\(L^{\infty}_{\weight{m},0}\),
    so the operator \(\id - \e^{h \mathcal L_{\xi}}\) is invertible on this space with inverse given by the series.
    In particular,
    it holds that
    \[
        \left\lVert \left( \frac{\e^{h \mathcal L_{\xi}} - \id }{h} \right)^{-1} \right\rVert_{L^{\infty}_{\weight{m},0}}
        \leq \sum_{n=0}^{+\infty} h \lVert \e^{nh \mathcal L_{\xi}} \rVert_{L^{\infty}_{\weight{m},0}}
        \leq C_{\mathcal L_{\xi}} \sum_{n=0}^{+\infty} h \e^{-n h} = C_{\mathcal L_{\xi}} \left( \frac{h}{1 - \e^{-h}} \right) \, .
    \]
    Here both norms are operator norms on~\(L^{\infty}_{\weight{m},0}\).
    See~\cite[Corollary 2.26]{MR3509213} for additional details.
    Showing bounds on the derivatives can then be achieved in a similar manner as in~\cref{corollary:poisson_ou}.
\end{proof}

% We can therefore define the function mapping the source term \(f\) to the solution
% of the associated discrete Poisson equation (\ref{eq:discrete_poisson2}) as an
% endomorphism of \(\mathcal
% C^0\left(\torus^d,L^\infty_{K,0}\left(\R^d,\R^k\right)\right)\).

\section{Auxiliary results from PDE theory}

This section is devoted to the proof of the auxiliary results on elliptic and parabolic PDEs.
These proofs are presented in~\cref{sec:proof_poisson} and~\cref{appendix:proof_auxiliary_parabolic}, respectively. 

\subsection{Proof of the result on the elliptic Poisson equation\texorpdfstring{ (\cref{lem:poisson})}{}}
\label{sec:proof_poisson}

\begin{proof}
    [Proof of \cref{lem:poisson}]
    We decompose the proof into several steps for clarity.
    We assume that $k=0$ until the last step,
    where the idea of the proof for the general case is explained.
    Furthermore, in the first two steps, we assume only that~$\Phi \in W^{1,\infty}(\torus^d)$ and~$f \in L^p(\torus^d)$ for some~$2 \leq p < \infty$.
    The assumptions that $f \in L^\infty(\torus^d)$ or~$f \in \cont^{\alpha}(\torus^d)$ are used only in later steps.

    \paragraph{Step 1. Existence and uniqueness of a weak solution in $H^1(\mu_\Phi) \cap L^2_0(\mu_{\Phi})$}
    The function space
    \[
        \mathcal H \coloneq \Bigl\{\varphi \in H^1(\torus^d;\mu_\Phi) \mid \mu_\Phi(\varphi) = 0\Bigr\}
    \]
    is a Hilbert space when endowed with the inner product~$\mathcal A_{\Phi}$ defined by
    \[
        % \forall \varphi_1, \varphi_2 \in H_0^1(\torus^d;\mu_\Phi), \quad
        % This is implicit in the notation. I think it is better to reserve \forall for actual statements, and not for definitions.
        \mathcal A_\Phi(\varphi_1,\varphi_2) = \int_{\torus^d} \nabla \varphi_1(x) \cdot \nabla \varphi_2(x)\,\mu_\Phi(\d x) \, .
    \]
    Indeed, by the Poincaré inequality \eqref{eq:poincare2},
    the associated norm is equivalent to the standard norm of \( H^1(\torus^d;\mu_\Phi) \),
    so completeness follows.
    In addition, the linear functional \(\mathcal L_f \colon \mathcal H \to \R\) defined by
    \[
        % \forall \varphi \in H_0^1(\torus^d;\mu_\Phi), \quad
        \mathcal L_f(\varphi) = \int_{\torus^d} \Bigl( f(x) - \mu_{\Phi}(f) \Bigr) \varphi(x)\,\mu_\Phi(\d x)
    \]
    is continuous on \(\left(\mathcal H, \mathcal A_\Phi\right)\)
    by the Cauchy--Schwarz inequality and the Poincaré inequality~\eqref{eq:poincare2}.
    % as we have:
    % \[
    % \forall \psi \in H_0^1(\mu_\Phi), \qquad \abs*{\mathcal L_f(\psi)} \leq
    % \frac{1}{\sqrt{R[\mu_\Phi]}}\left( \int_{\torus^d} \abs*{f(x)}^2\,\mu_\Phi(\d x)
    % \right)^{\frac{1}{2}} \mathcal A_\Phi(\psi,\psi)^{\frac{1}{2}}.
    % \]
    % Therefore, the operator \(\mathcal L_f\) is continuous on \(\left(H_0^1(\torus^d;\mu_\Phi), \mathcal
    % A_\Phi(\placeholder,\placeholder)^{1/2}\right)\),
    % so
    Therefore, by the Riesz representation theorem,
    there exists a unique \(\psi \in \mathcal H \) such that
    \begin{equation}
        \label{eq:riesz}
        \forall \varphi \in \mathcal H, \qquad \int_{\torus^d} \nabla \psi(x) \cdot
        \nabla \varphi(x)\,\mu_\Phi(\d x) = \int_{\torus^d} \Bigl( f(x) - \mu_{\Phi}(f) \Bigr) \varphi(x)\,\mu_\Phi(\d x) \, .
    \end{equation}
    Since any $\varphi \in H^1(\torus^d;\mu_\Phi)$ can be written as the sum of a function in $\mathcal H$ and a constant,
    this equation also holds for all~\(\varphi \in H^1(\torus^d;\mu_\Phi)\) with a mean that is not necessarily equal to~0.
    % U: I think this is obvious since \mu_\Phi has a density bounded from below and above
    % We also have that \(\psi\) is in \(H^1(\torus^d)\) since, for \(\alpha = 0\) or \(1\),
    % as \(\Phi\) is continuous over the torus,
    % \[
    %     \int_{\torus^d} \abs*{D^\alpha \psi(x)}^2\,\d x = Z_\Phi\int_{\torus^d}
    %     e^{\Phi(x)}\abs*{D^\alpha \psi(x)}^2\,\mu_\Phi(\d x)\leq Z_\Phi
    %     e^{\norm{\Phi}_\infty}\int_{\torus^d}\abs*{D^\alpha \psi(x)}^2\,
    %     \mu_\Phi(\d x)< \infty.
    % \]

\paragraph{Step 2. Proof that the solution belongs to~$W^{2,p}(\torus^d)$}%
Since $\e^{-\Phi}$ is Lipschitz and bounded from above and below by positive constants,
it follows from~\eqref{eq:riesz}, using $\e^\Phi\varphi$ as a test function, that \(\psi \in H^1(\torus^d)\) and
\begin{equation}
    \label{eq:weak_form}
    \forall \varphi \in H^1(\torus^d), \qquad
    \int_{\torus^d} \nabla \psi(x) \cdot \nabla \varphi(x) \, \d x
    + \int_{\torus^d} \bigl(\nabla \Phi(x) \cdot \nabla \psi(x) \bigr) \, \varphi(x) \, \d x
    = \int_{\torus^d} \Bigl( f(x) - \mu_{\Phi}(f) \Bigr) \varphi(x)\, \d x \, .
\end{equation}
From now on, we identify $\Phi$, $\psi$, and $f$ with their periodic extensions to \(\R^d\).
For any test function \(\varphi \in \smoothcompact(\real^d) \), it holds by periodicity of~$\psi$ that
\begin{align*}
    \int_{\R^d} \nabla \psi(x) \cdot \nabla \varphi(x) \, \d x
    &= \sum_{\alpha \in \Z^d} \int_{[0, 1]^d} \nabla \psi(x) \cdot \nabla \varphi(x + \alpha) \, \d x \\
    &=  \int_{[0, 1]^d} \nabla \psi(x) \cdot \nabla \left(\sum_{\alpha \in \Z^d} \varphi(x + \alpha) \right) \, \d x \, , \,
\end{align*}
and similarly for the other terms in~\eqref{eq:weak_form}.
Furthermore, the sum on the second line (which is finite for any given~$x$, since~$\varphi$ is compactly supported)
defines a smooth function that is 1-periodic in each coordinate.
Therefore, it follows from~\eqref{eq:weak_form} that
\begin{equation}
    \label{eq:weak_form_rd}
    \forall \varphi \in \smoothcompact(\real^d), \qquad
    \int_{\real^d} \nabla \psi(x) \cdot \nabla \varphi(x) \, \d x
    + \int_{\real^d} \bigl(\nabla \Phi(x) \cdot \nabla \psi(x) \bigr) \, \varphi(x) \, \d x
    = \int_{\real^d} \Bigl( f(x) - \mu_{\Phi}(f) \Bigr) \varphi(x)\, \d x \, ,
\end{equation}
and so \(\psi\) satisfies, in the sense of distributions over~$\real^d$,
the equation
\[
    \nabla \Phi\cdot \nabla \psi - \Delta \psi = f - \mu_\Phi(f) \, .
\]
We note that the first term on the left-hand side belongs to~$L^2_{\mathrm{loc}}(\real^d)$,
since~$\nabla \psi \in L^2_{\mathrm{loc}}(\real^d)$ and~$\nabla \Phi$ is bounded and measurable by assumption,
and hence defines a distribution of order~0.
To bootstrap the regularity of \(\psi\),
we relate the Poisson equation with periodic boundary conditions
to a Dirichlet problem.
Consider a smooth bump function \(\chi \in \smoothcompact(\R^d)\) such that \(\chi\equiv 1\) on
\([0,1]^d\) and~$\chi \equiv 0$ outside \([-1,2]^d\).
\ifdetails
We can, for instance, define \(\chi\) as follows:
\begin{align}
    \label{eq:bump_function}
    \chi(x) = \prod_{i=1}^d g(x_i +1) \,g(2-x_i) \, ,
\end{align}
where
\[
  g(x_i) = \frac{h(x_i)}{h(x_i) + h(1-x_i)} \, ,
  \qquad
  h(x_i) =
  \begin{cases}
      \displaystyle \e^{-\frac{1}{x_i}}
      &\text{ if } x_i > 0 \, ,\\
      0 &\text{ if } x_i \leq 0 \, .
  \end{cases}
\]
\Cref{fig:bump} illustrates the function \(\chi\) in dimension \(d = 1\).
\begin{figure}[htbp]
    \begin{center}
        \includegraphics[width = 0.5\textwidth]{fig/bump.pdf}
        \caption{The bump function \(\chi\) given in~\eqref{eq:bump_function} in dimension \(d = 1\).}
        \label{fig:bump}
    \end{center}
\end{figure}
\fi
In the sense of distributions, we have
\begin{equation*}
  \nabla \Phi \cdot \nabla (\psi \chi) - \Delta (\psi \chi) =
  \chi \left(f - \mu_\Phi(f)\right) + \psi \nabla \Phi \cdot \nabla
  \chi - \psi \Delta \chi - 2 \nabla \chi \cdot \nabla \psi \, .
\end{equation*}
Let~$\openball{R}$ denote the open ball of radius~$R$ centered at~0 in~$\real^d$,
and fix~$R = 3\sqrt{d}$ so that the support of~$\chi$ is contained in~$\openball{R}$.
It follows that \(u\coloneq\psi \chi\) is the unique weak solution to the Dirichlet
problem:
\begin{subequations}
\begin{align}
  \label{eq:dirichlet_problem_a}
      \nabla \Phi\cdot \nabla u - \Delta u &= \chi \left(f - \mu_\Phi(f)\right) + \psi \nabla \Phi \cdot \nabla \chi - \psi \Delta \chi - 2 \nabla \chi \cdot \nabla \psi
      \eqcolon F
                                           &\text{ on } \openball{R} \, ,
      \\ u &= 0 &\text{ on } \partial \openball{R} \, .
\end{align}
\end{subequations}
Since \(\psi\in H^1(\openball{R})\) and $\Phi \in W^{1,\infty}(\torus^d)$ by assumption,
the right-hand side~$F$ belongs to~$L^2(\openball{R})$.
We now prove that~$F$ belongs to~$L^{p}(\openball{R})$,
i.e.\ that it has the same integrability as~$f$,
by repeatedly employing the classical elliptic regularity result given in~\cite[Theorem~9.15]{MR1814364} with the Sobolev embedding theorem~\cite[Chapter~5]{MR2597943}.

By elliptic regularity, it holds that~\(\psi\chi\) belongs to \(H^2(\openball{R})\),
so~\(\psi\in H^2(\torus^d)\) and \(\nabla \psi\in H^1(\torus^d)\).
If~$d \leq 2$,
it then follows immediately from a Sobolev embedding that~$F$ indeed belongs to~$L^{p}(\openball{R})$.
For $d > 2$, the Sobolev embedding theorem implies only that $F \in L^{\min\{p,q_1\}}(\openball{R})$ for \(\frac{1}{q_1} = \frac{1}{2} - \frac{1}{d}\),
but elliptic regularity can be used again to deduce that~$\psi\chi \in W^{2,\min\{p,q_1\}}(\openball{R})$,
and, unless the exponent~$p$ has already been reached, another application of the Sobolev embedding theorem then further improves the regularity of~$F$ to~$L^{\min\{q_2,p\}}(\openball{R})$ for \(\frac{1}{q_2} = \frac{1}{2} - \frac{2}{d}\),
provided that $\frac{1}{2} - \frac{2}{d} > 0$.
When the Sobolev exponent reaches~$d$, the embedding~$W^{1,d} \hookrightarrow L^q$ for every finite~$q$ gives~$F\in L^p$; above~$d$, the same conclusion follows from the embedding into~$L^\infty$.
This procedure can be repeated to eventually deduce that $F \in L^{p}(\openball{R})$ and~$\psi \in W^{2,p}(\torus^d)$.

\paragraph{Step 3. Proof of the bound~\eqref{eq:bound_poisson} with $k=0$}
So far, we have used mainly the elliptic regularity estimate~\cite[Theorem 9.15]{MR1814364},
which provides existence of a solution in \(W^{2,p}(\torus^d)\),
but no estimate on the~$W^{2,p}(\torus^d)$-norm of the solution.
To obtain such an estimate, let
\[
    -\genlimit \colon W^{2,p}(\torus^d) \cap L^p_0(\torus^d) \to L^p_0(\torus^d) \, , \qquad
    -\genlimit \psi = \nabla \Phi \cdot \nabla \psi - \Delta \psi \, ,
\]
where \(L^p_0(\torus^d) \coloneq \bigl\{ f \in L^p(\torus^d) \mid \mu_\Phi(f) = 0 \bigr\}\).
By the existence and uniqueness established in \textbf{Steps~1--2},
this operator is bijective.
% for each \(f \in L^p_0(\torus^d)\) there exists a unique \(\psi \in W^{2,p}(\torus^d)\) with mean zero solving \(-\genlimit \psi = f\).
Moreover, the operator \(-\genlimit\) is linear and bounded:
\[
\|-\genlimit \psi\|_{L^p(\torus^d)} = \| \nabla \Phi \cdot \nabla \psi - \Delta \psi \|_{L^p(\torus^d)}
\le \|\nabla \Phi\|_\infty \|\nabla \psi\|_{L^p(\torus^d)} + \|\Delta \psi\|_{L^p(\torus^d)} \le C \|\psi\|_{W^{2,p}(\torus^d)} \, .
\]
Hence, the linear operator \(-\genlimit\) is bijective between the Banach spaces \(W^{2,p} \cap L^p_0 \) and \(L^p_0\).
Thus, by Banach's bounded inverse theorem (or open mapping theorem),
its inverse is also bounded.
Therefore, there exists a constant~\(C_p > 0\), depending only on \(\Phi\) and \(p\), such that the solution satisfies
\begin{equation}
    \label{eq:bound_W2p}
    \|\psi\|_{W^{2,p}(\torus^d)} \le C_p \, \| f \|_{L^p(\torus^d)} \, .
\end{equation}
Then for any $0 \leq \lambda < 1$,
the Sobolev embedding theorem, employed with a sufficiently large value of $p > d$,
leads to the bound~\eqref{eq:bound_poisson} with~$k=0$.

\paragraph{Step 4. Proof of the bound~\eqref{eq:bound_poisson_holder} with $k=0$}
To obtain the improved bound~\eqref{eq:bound_poisson_holder},
assume additionally that~$\Phi \in \cont^{1+\alpha}(\torus^d)$;
it suffices to apply the same reasoning as in the previous step to the operator
\[
    -\genlimit \colon \cont^{2+\alpha}(\torus^d) \cap L^2_0(\torus^d) \to \cont^{\alpha}(\torus^d) \cap L^2_0(\torus^d) \, .
\]
This operator is bounded and linear,
so it remains to show that it is surjective to conclude that its inverse is bounded by the bounded inverse theorem.
To this end, take any~$f \in \cont^{0+\alpha}(\torus^d) \cap L^2_0(\mu_{\Phi})$.
By the second step, the corresponding solution satisfies $\psi \in W^{2,q}(\torus^d)$ for all $q \in [2, \infty)$,
and application of the Morrey--Sobolev embedding theorem then implies that the right-hand side~$F$ of~\eqref{eq:dirichlet_problem_a} belongs to $\cont^{0+\alpha}(\openball{R})$.
But then \(\chi \psi\in \cont^{2+\alpha}(\openball{R})\), by the Schauder regularity estimate~\cite[Theorem 6.14]{MR1814364},
and so~$\psi \in \cont^{2+\alpha}(\torus^d) \cap L^2_0(\mu_{\Phi})$.
Note that, in this case, the Poisson equation \eqref{eq:poisson} is satisfied in the classical sense.

\paragraph{Step 5. Generalization to higher derivatives}
It remains to prove the bounds~\eqref{eq:bound_poisson} and~\eqref{eq:bound_poisson_holder} for the case~$k > 0$.
We focus on the case~\(k = 1\) for simplicity,
but the general case can be treated by induction.
Assume that \(f\in W^{1,\infty}(\torus^d)\) and \(\Phi\in W^{2,\infty}(\torus^d)\).
For any~$\varphi \in \smoothcompact(\real^d)$ and any~$i \in \{1, \ldots, d\}$,
it holds by testing~\eqref{eq:weak_form_rd} with~$\partial_i \varphi$,
and by using integration by parts,
that
\[
    \int_{\real^d} \nabla (\partial_i \psi)(x) \cdot \nabla \varphi(x) \, \d x
    + \int_{\real^d} \bigl(\nabla \Phi(x) \cdot \nabla (\partial_i \psi)(x) \bigr) \, \varphi(x) \, \d x
    = \int_{\real^d} \bigl(\partial_{i}f(x) - \nabla \partial_{i}\Phi(x) \cdot \nabla \psi(x) \bigr) \varphi(x)\, \d x \, ,
\]
% By density, this equation also holds for all~\(\varphi \in H^1(\mu_\Phi)\),
and so~$u_i\coloneq\partial_i\psi$ is a weak solution to the Poisson equation
\[
    \nabla \Phi \cdot \nabla u_i -\Delta u_i =
    \partial_{i}f - \nabla \partial_{i}\Phi\cdot \nabla \psi \, .
\]
Since~$\nabla\psi$ is bounded by~\eqref{eq:bound_poisson} with~$k=0$,
the right-hand side belongs to~$L^\infty(\torus^d)$, with norm bounded by~$C\|f\|_{W^{1,\infty}}$.
Applying the same estimate to~$u_i-\mu_\Phi(u_i)$,
and using~$|\mu_\Phi(u_i)|\leq\|\nabla\psi\|_\infty\leq C\|f\|_\infty$,
gives~\eqref{eq:bound_poisson} for~$k=1$.
Repeating this argument proves the higher-order bounds by induction;
for~\eqref{eq:bound_poisson_holder}, we use \textbf{Step~4} and H\"older norms instead.
We omit the details for brevity.
\end{proof}

\subsection{Proof of the result on the backward Kolmogorov equation\texorpdfstring{ (\cref{theorem:auxiliary_parabolic})}{}}%
\label{appendix:proof_auxiliary_parabolic}

The proof is based on several auxiliary results,
which we present and prove in this section.
We first recall without proof the following classical result:
we refer to~\cite[Theorem 1.6.2 and Remark 1.6.3]{MR1840644} for the details.
In that reference, the setting where the equation is posed on~$\real^d$ with a dissipative drift term is considered,
but the proof applies mutatis mutandis to the case of the torus~$\torus^d$.
In fact, it is likely that this result can be proved under weaker regularity assumptions on~$\Phi$ because the torus is compact,
but for simplicity we do not pursue this direction in this paper.

\begin{lemma}
    [Stochastic representation of the solution to the backward Kolmogorov equation]
    \label{lemma:semigroup_parabolic}
    If \(\Phi\in \cont^4\left(\torus^d\right)\) and \(f\in \cont^2(\torus^d)\),
    then there exists a unique classical solution \(u\in \cont^{1,2}([0, +\infty) \times \torus^d)\) to the parabolic problem
    \begin{equation}
        \left\{
            \begin{aligned}
                \label{eq:parabolic_equation}
                \partial_t u(t,\theta) &= -\nabla \Phi(\theta)\cdot \nabla u(t,\theta) + \Delta u(t,\theta) \, , && t>0, \theta\in \torus^d \, ,\\
                u(0,\theta) &= f(\theta) \, , && \theta\in \torus^d \, .
            \end{aligned}
        \right.
    \end{equation}
    and this classical solution is given by the formula
    \begin{equation}
        \label{eq:representation}
        \forall \theta \in \torus^d,
        \qquad u(t,\theta ) 
        = \expect_{\theta }\bigl[ f(\vartheta_t)\bigr] \, ,
    \end{equation}
    where \((\vartheta_t)_{t\geq 0}\) is the solution of the stochastic differential equation~\eqref{eq:overdamped_langevin_weak} and
    \(\expect_{\theta }\) denotes the expectation with respect to the law of the process \((\vartheta_t)_{t\geq 0}\) with initial condition \(\vartheta_0 = \theta \).
\end{lemma}

In order to state the following results,
we introduce the following notation
for any bounded measurable function~$f \in \mathcal B(\torus^d)$:
\[
    P_t f(\theta) = \expect_{\theta }\bigl[ f(\vartheta_t)\bigr] \, .
\]
Since the stochastic differential equation~\eqref{eq:overdamped_langevin_weak} is well-posed for~$\Phi \in \cont^2(\torus^d)$,
see~\cite[Theorem 3.4]{khasminskii2012stochastic},
the function~$P_t f$ is well-defined under conditions that are less restrictive than those of \cref{lemma:semigroup_parabolic}.
We proceed by presenting a result that gives a bound on the derivatives of~$P_t f$
in terms of the derivatives of~$f$.
\begin{lemma}
  \label{lemma:gradient_semigroup}
  Assume that \(\Phi\in \cont^{k+1}(\torus^d)\) and \(f\in \cont^k(\torus^d)\),
  for some~$k \in \{1, 2, \dots\}$.
  Then for any~$t \geq 0$, it holds that~$P_t f \in \cont^k(\torus^d)$, 
  and for any multi-index \(\alpha\) such that \(\abs{\alpha}\leq k\),
  there exist constants \(C_{\abs{\alpha}}, \lambda_{\lvert \alpha \rvert} > 0\) 
  depending on~$\Phi$ and~$|\alpha|$ but not \(f\) such that
  \begin{equation}
      \label{eq:gradient_semigroup}
      \forall t \geq 0, \qquad
      \norm[\big]{\partial^\alpha_\theta P_t f}_\infty \leq C_{\abs{\alpha}}
      \exp\left(\lambda_{|\alpha|} t\right)\norm*{f}_{\abs{\alpha},\infty} \, .
\end{equation}
\end{lemma}

\begin{proof}
    Since the strategy of proof is standard,
    we only give a sketch of the proof.
    We refer to~\cite[Lemma 11]{darshan2024sticky} or~\cite[Lemma A.2]{kopec2015weak} for detailed proofs of these statements in more general settings.

    Under the present assumptions, the semigroup~$P_t$ is well defined through the stochastic representation.
    Moreover, when~$\Phi$ and~$f$ are smooth enough for~\cref{lemma:semigroup_parabolic} to apply,
    this representation coincides with the unique classical solution to~\eqref{eq:parabolic_equation}.
    Let~$(\vartheta^z_t)_{t\geq 0}$ be the solution to the overdamped Langevin SDE~\eqref{eq:overdamped_langevin_weak} with initial condition \(\vartheta^z_0 = z\in \torus^d\). 
    For differences and derivatives, we view these processes in~$\real^d$, with periodic coefficients and a common Brownian motion.
    It is possible to prove,
    see~\cite[Section~1.3]{MR1840644},
    that if~$\Phi \in \cont^{k+1}(\torus^d)$,
    then for any fixed~$t$ the solution~$z \mapsto \vartheta^z_t \in L^2(\Omega, \torus^d)$ is~$k$ times mean-square differentiable with respect to~$z$.
    Furthermore, the derivatives solve the SDEs obtained by formally differentiating the original SDE with respect to the initial condition.
    For example, the gradient~$\nabla_{z} \vartheta^z_t$ is the solution of
    \[
        \d \nabla_z \vartheta^z_t = -\nabla^2 \Phi(\vartheta^z_t) \nabla_z \vartheta^z_t \,\d t \, , 
        \qquad \nabla_z \vartheta^z_0 = \I_d \, ,
    \]
    and so by Gr\"onwall's lemma, we have that
    \begin{equation}
        \label{eq:gronwall_derivative}
        \expect\Bigl[\eucnorm*{\nabla_z \vartheta^z_t}\Bigr] \leq \exp\Bigl(t\norm*{\nabla^2 \Phi}_{\infty}\Bigr) \, .
    \end{equation}
    Since $\vartheta^z_t$ is mean-square differentiable with respect to~$z$,
    for any function~$g \in \cont^1(\torus^d)$,
    the function~$z \mapsto g(\vartheta^z_t)$ is also mean-square differentiable and 
    it holds that
    \begin{equation}
        \label{eq:chain_rule_derivative_mean_square}
        \partial_{z_i} g(\vartheta^z_t) = \nabla g(\vartheta^z_t) \cdot \partial_{z_i} \vartheta^z_t \, 
    \end{equation}
    We present a self-contained proof of this result at the end of this proof for completeness.
    It follows from the representation formula~\eqref{eq:representation},
    together with equations~\eqref{eq:gronwall_derivative} and~\eqref{eq:chain_rule_derivative_mean_square},
    that
    \begin{align}
        \label{eq:example_bound}
        \lvert \partial_{z_i} P_t f(z) \rvert 
        = \Bigl\lvert \frac{\partial}{\partial z_i} \expect \bigl[f(\vartheta^z_t)\bigr]  \Bigr\rvert
        = \Bigl\lvert \expect \bigl[\nabla f(\vartheta^z_t) \cdot \partial_{z_i} \vartheta^z_t\bigr] \Bigr\rvert
        \leq \expect \Bigl[  \eucnorm*{\partial_{z_i} \vartheta^z_t} \Bigr] \norm*{\nabla f}_\infty
        \leq \exp\Bigl(t\norm*{\nabla^2 \Phi}_\infty\Bigr) \norm*{\nabla f}_\infty \, .
    \end{align}
    Exchanging the derivative and the expectation is justified by the mean-square differentiability of~$z \mapsto f(\vartheta^z_t)$.
    Notice that the proof of~\eqref{eq:example_bound} relied
    on the regularity conditions~$\Phi \in \cont^2(\torus^d)$,
    to ensure that~$\vartheta^z_t$ is mean-square differentiable with respect to~$z$,
    and~$f \in \cont^1(\torus^d)$, to derive~\eqref{eq:chain_rule_derivative_mean_square}.
    The procedure to prove~\eqref{eq:gradient_semigroup} with greater values of~$|\alpha|$ is similar:
    if~$\Phi \in \cont^{|\alpha|+1}(\torus^d)$, 
    then~$\vartheta^z_t$ is~$|\alpha|$ times mean-square differentiable with respect to~$z$,
    and a bound similar to~\eqref{eq:gronwall_derivative} can be derived for the higher-order derivatives of~$\vartheta^z_t$ with respect to~$z$ 
    by using the SDEs satisfied by these derivatives.
    Next, if~$f \in \cont^{|\alpha|}(\torus^d)$, 
    then a formula similar to~\eqref{eq:chain_rule_derivative_mean_square} can be proved for higher-order derivatives,
    which can be combined with the bounds on the derivatives of~$\vartheta^z_t$ to obtain the desired bound.

    \paragraph{Proof of~\eqref{eq:chain_rule_derivative_mean_square}}
    From Taylor's formula, for all \(h\in \R^d\), 
    we have for any~$g \in \cont^1(\torus^d)$ that
    \[
        g(\vartheta^{z+h}_t) - g(\vartheta^z_t) = \nabla g(\vartheta^{z}_t) \cdot (\vartheta^{z+h}_t - \vartheta^z_t) \,
        + \int_0^1 \Bigl( \nabla g\bigl(\vartheta^z_t + s(\vartheta^{z+h}_t - \vartheta^z_t)\bigr) - \nabla g(\vartheta^z_t) \Bigr) \cdot (\vartheta^{z+h}_t - \vartheta^z_t) \,\d s \, ,
    \]
    and so, introducing the modulus of continuity~$w_{\nabla g} \colon \R_{\geq0} \to \R_{\geq0}$ of \(\nabla g\),
    we deduce that
    \[
        \Bigl\lvert g(\vartheta^{z+h}_t) - g(\vartheta^z_t) - \nabla g(\vartheta^{z}_t) \cdot (\vartheta^{z+h}_t - \vartheta^z_t) \Bigr\rvert \,
        \leq 
        \int_0^1  w_{\nabla g}\bigl(s\lvert \vartheta^{z+h}_t - \vartheta^z_t \rvert\bigr) \, \d s \, \bigl\lvert \vartheta^{z+h}_t - \vartheta^z_t \bigr\rvert \, .
    \]
    This implies the following inequality:
    \begin{align*}
        \frac{1}{|h|^2}
        \expect \Bigl[ \bigl\lvert g(\vartheta^{z+h}_t) - g(\vartheta^z_t) - \nabla g(\vartheta^{z}_t) \cdot (\nabla \vartheta^z_t h) \bigr\rvert^2 \Bigr] \,
        &\leq 
        \frac{2}{|h|^2}
        \expect \Bigl[ \bigl\lvert \nabla g(\vartheta^z_t) \cdot (\vartheta^{z+h}_t - \vartheta^z_t) - \nabla g(\vartheta^{z}_t) \cdot (\nabla \vartheta^z_t h)  \bigr\rvert^2  \Bigr] \\
        &\qquad 
        + \frac{2}{|h|^2} \int_0^1  \expect \Bigl[  w_{\nabla g}\bigl(s\lvert \vartheta^{z+h}_t - \vartheta^z_t \rvert\bigr)^2  \bigl\lvert \vartheta^{z+h}_t - \vartheta^z_t \bigr\rvert^2 \Bigr] \, \d s \,     \, .
    \end{align*}
    The first term on the right-hand side tends to zero as~$h \to 0$ since~$\vartheta^z_t$ is mean-square differentiable with respect to~$z$ and~$\nabla g$ is bounded,
    and it remains to prove that
    \[
        \forall s \in [0,1], \qquad
        \expect \Biggl[  w_{\nabla g}\bigl(s\lvert \vartheta^{z+h}_t - \vartheta^z_t \rvert\bigr)^2  \frac{\bigl\lvert \vartheta^{z+h}_t - \vartheta^z_t \bigr\rvert^2}{|h|^2} \Biggr] 
        \xrightarrow[h \to 0]{} 0 \, .
    \]
    By mean-square differentiability of~$\vartheta^z_t$ with respect to~$z$, 
    we have that $\vartheta^{z+h}_t - \vartheta^z_t = \nabla \vartheta^z_t h + o\bigl(|h|\bigr)$ in $L^2(\Omega)$ as~$h \to 0$, 
    and so, since~$w_{\nabla g}$ is uniformly bounded,
    it is sufficient to prove that
    \[
        \forall s \in [0,1], \qquad
        \expect \Biggl[  w_{\nabla g}\bigl(s\lvert \vartheta^{z+h}_t - \vartheta^z_t \rvert\bigr)^2 \bigl\lvert \nabla \vartheta^z_t \bigr\rvert^2 \Biggr]
        \xrightarrow[h \to 0]{} 0 \, ,
    \]
    but this is clear (using, for example, the Vitali convergence theorem) 
    since~$w_{\nabla g}\bigl(s\lvert \vartheta^{z+h}_t - \vartheta^z_t \rvert\bigr)$ converges to zero in probability.
\end{proof}

It is well known that the exponential growth with respect to the time~$t$ in~\cref{lemma:gradient_semigroup} is suboptimal.
In order to obtain a better bound, 
we first state the following regularization result,
which can be proved by following exactly the same approach as in~\cite[Appendix A, Lemma 12]{darshan2024sticky}.
See also~\cite[Theorem 1.5.2]{MR1840644} and~\cite[Lemma A.2]{kopec2015weak}.
The main technical tool of the proof is the Bismut--Elworthy--Li (BEL) formula.
Unlike the formula~\eqref{eq:example_bound},
the BEL formula provides a stochastic representation formula for the derivatives of the solution to the parabolic problem in terms of just~$f$, and not its derivatives.

\begin{lemma}
  \label{lemma:bel_formula}
  Assume that \(\Phi\in \cont^{k+1}(\torus^d,\R)\) for some~$k \in \{1, 2, \dots\}$.
  Then, for any multi-index \(\alpha\) such that~\(\abs{\alpha} \leq k\), 
  there exists a constant~\(\widetilde C_{\abs{\alpha}}>0\) such that for any~$f \in \mathcal B(\torus^d)$,
  it holds that
  \begin{equation}
    \label{eq:bel_formula}
    \forall t \in (0, 1], \qquad
    \norm[\big]{\partial^\alpha_\theta P_t f}_\infty \leq
    \widetilde C_{\abs{\alpha}} \norm*{f}_{\infty} t^{-\frac{|\alpha|}{2}} \, .
  \end{equation}
\end{lemma}

The last remaining auxiliary result, 
before we can prove the decay estimates presented in the statement of \cref{theorem:auxiliary_parabolic},
is a classical ergodicity result for the overdamped Langevin dynamics.
\begin{lemma}
  \label{lemma:ergodicity_langevin}
  Let \(\Phi\in \cont^2(\torus^d)\). 
  Then, there exist constants \(C_\Phi > 0\) and \(\lambda_\Phi > 0\) such that,
  for all \(f\in \mathcal B(\torus^d)\),
  \begin{equation}
      \forall t \geq 0, \qquad
      \norm*{P_t f- \mu_{\Phi}(f)}_\infty 
      \leq C_\Phi \e^{- \lambda_\Phi t} \norm*{f- \mu_{\Phi}(f)}_\infty \, 
       ,\qquad \mu_{\Phi}(f) \coloneq
      \frac
      {\displaystyle\int_{\torus^d} f(\theta) \e^{-\Phi(\theta)} \, \d \theta}
      {\displaystyle \int_{\torus^d} \e^{-\Phi(\theta)} \, \d \theta} \, .
  \end{equation}
\end{lemma}
\begin{proof}
    The result follows from~\cite[Theorem 1.2]{MR2857021},
    see also \cite[Equation (2.78)]{MR3509213}.
    Since the state space is compact,
    it suffices to check the minorization condition,
    i.e.\ that there exist~$t_0 > 0$ and a probability measure~$\nu$ on~$\torus^d$ such that, 
    for some~$\varepsilon > 0$,
    \[
        \forall \theta \in \torus^d, \qquad
        P_{t_0}(\theta, \placeholder) \geq \varepsilon \nu(\placeholder) \, .
    \]
    There are various approaches to check this condition.
    For example,
    we could use the Girsanov transformation to compare the diffusion process with a Brownian motion on the torus.
    For the Langevin process in~\eqref{eq:overdamped_langevin_weak},
    using It\^o's formula we can rewrite the Brownian integral in terms of endpoint values and time integrals:
    \[
        \sqrt{2}\int_{0}^t \nabla \Phi(\vartheta_s) \cdot \d W_s
        = \Phi(\vartheta_t) - \Phi(\vartheta_0)
        + \int_0^t \Bigl(\lvert\nabla\Phi(\vartheta_s)\rvert^2-\Delta\Phi(\vartheta_s)\Bigr) \, \d s \, ,
    \]
    showing that the Girsanov weight is bounded from above and below by positive constants over a fixed time interval, 
    which allows to deduce the minorization condition from the minorization condition for Brownian motion.
    Here, 
    to avoid the technicalities related to the Girsanov transformation on the torus,
    we present a short argument based on existing bounds on the transition density.

    Consider the overdamped Langevin dynamics on~$\real^d$,
    with $\Phi$ extended to a periodic function in~$\cont^2(\real^d)$.
    The bounds in~\cite[Theorem~1]{QianRussoZheng2003} give a Gaussian lower bound on the transition probability:
    there exist~$\alpha_1,K>0$ such that
    for any~$(x,y)\in\real^d\times\real^d$ and~$0<t\leq1$,
    \begin{equation}
        \label{eq:transition_density_lower_bound}
        K^{-1}t^{-\frac d2}\exp\Bigl(-\frac{\eucnorm*{y-x}^2}{\alpha_1t}\Bigr)\leq p_t(x,y) \, .
    \end{equation}
    For the overdamped Langevin dynamics on the torus~$\torus^d$,
    the transition density can be expressed as the sum of the transition densities of the overdamped Langevin dynamics on~$\real^d$ over all the periodic translations of the initial condition~$x$.
    In particular, the lower bound in~\eqref{eq:transition_density_lower_bound} is still valid for representatives~$x,y\in[0,1)^d$.
    At~$t=1$, this gives a uniform positive lower bound, so the minorization condition holds with~$\nu$ the uniform measure.
\end{proof}

We are now in a position to prove the exponential decay of the derivatives of the solution to the backward Kolmogorov equation,
by proceeding exactly as in the proof of~\cite[Proposition 7]{darshan2024sticky}.
\begin{lemma}
    \label{lemma:exponential_decay_derivatives}
    Assume that \(\Phi\in \cont^{k+1}(\torus^d)\),
    for some~$k \in \{1, 2, \dotsc \}$.
    Then, for any multi-index \(\alpha\) such that~\(1 \leq \abs{\alpha} \leq k\),
    there exists a constant~$c > 0$ depending on~$\Phi$ and~$\alpha$ such that for any~$f \in \cont^{|\alpha|}(\torus^d)$,
    \begin{equation}
        \label{eq:decay}
        \forall t \geq 0, \qquad
        \norm*{\partial^{\alpha} P_t f}_\infty 
        \leq c \e^{-\lambda_{\Phi} t} \norm*{f}_{|\alpha|,\infty} \, , 
    \end{equation}
    where~$\lambda_\Phi$ is the constant defined in \cref{lemma:ergodicity_langevin}.
\end{lemma}
\begin{proof}
    Fix a multi-index~$\alpha$ such that~$1 \leq |\alpha| \leq k$ and take $f \in \cont^{|\alpha|}(\torus^d)$.
    In view of~\cref{lemma:gradient_semigroup},
    the inequality~\eqref{eq:decay} holds for all~$t \in [0, 1]$ with constant~$C_{|\alpha|}\e^{\lambda_{|\alpha|}+\lambda_\Phi}$.
    Assuming now that $t \geq 1$,
    we write
    \[
        P_t f = P_1 P_{t-1} f = P_1 \bigl(P_{t-1} f - \mu_\Phi(f)\bigr) + \mu_\Phi(f) \, .
    \]
    It then follows from~\cref{lemma:ergodicity_langevin,lemma:bel_formula} that
    \begin{align*}
        \forall t \geq 1, \qquad
        \norm*{\partial^{\alpha} P_t f}_\infty 
        &= \norm*{\partial^{\alpha} P_1 \bigl(P_{t-1} f - \mu_\Phi(f)\bigr)}_\infty \\
        &\leq \widetilde C_{|\alpha|} \norm*{P_{t-1} f - \mu_\Phi(f)}_\infty \\
        & \leq \widetilde C_{|\alpha|} C_\Phi \e^{-\lambda_\Phi (t-1)} \norm*{f - \mu_\Phi(f)}_\infty \\
        &\leq 2 \widetilde C_{|\alpha|} C_\Phi \e^{\lambda_\Phi} \e^{-\lambda_\Phi t} \norm*{f}_{|\alpha|,\infty} \, .
    \end{align*}
    We conclude that~\eqref{eq:decay} holds for all~$t \geq 0$ with~$c = \max\{C_{|\alpha|} \e^{\lambda_{|\alpha|}+\lambda_\Phi}, 2 \widetilde C_{|\alpha|} C_\Phi \e^{\lambda_\Phi}\}$,
    which concludes the proof.
\end{proof}

Combining these lemmas proves~\cref{theorem:auxiliary_parabolic}, with time regularity following from~\eqref{eq:parabolic_equation}.

% \begin{remark}
%     As a consequence,
%     it holds that
%     \begin{equation}
%         \label{eq:exponential_decay_derivatives_time}
%         \norm*{\partial_t u(t,\placeholder)}_\infty \lesssim \e^{-\lambda_2 t} \norm*{f}_{2,\infty},
%         \qquad
%         \norm*{\partial_t \nabla_{\theta} u(t,\placeholder)}_\infty \lesssim C_1 \e^{-\lambda_1 t} \norm*{f}_{2,\infty} \, .
%     \end{equation}
% \end{remark}

\section{Further auxiliary results concerning the drift coefficient}%
\label{sec:results_drift}

In this section,
we collect a few elementary results concerning the drift coefficients~$\driftall$ that appear in the multiscale dynamics~\eqref{eq:the_multiscale_dynamics}.
The proofs are straightforward and rely mainly on Taylor expansions;
but we provide most of them for completeness.
We first recall the expressions of the drift coefficient:
\begin{equation}
    \label{eq:drift_1_recall}
    \driftall(\theta, \allxi) 
    \coloneq \frac{1}{J}\sum_{j = 1}^J \driftone\bigl(\theta, \xin{j}\bigr) \, ,
    \qquad
    \driftone(\theta, \xi) 
    \coloneq 
    - \frac{1}{\sigma} \Bigl\langle \nabla \ell\bigl(G(\theta)\bigr), G\bigl(\theta + \sigma \xi\bigr) - G(\theta) \Bigr\rangle \, \xi \, .
\end{equation}
We recall that~$\ell \colon \real^K \to \real$ is assumed to be smooth throughout the paper.
Thus, if~$G\in\cont^1(\torus^d, \real^K)$,
then the functions~$\driftall$ and~$\driftone$ converge pointwise as~$\sigma \to 0$ to
\[
    \driftlim(\theta, \allxi) \coloneq \frac{1}{J} \sum_{j=1}^J \driftonelim(\theta, \xin{j}) \, , \qquad
    \driftonelim(\theta, \xi) 
    \coloneq - (\xi \otimes \xi) \nabla \Phi(\theta) \, ,
\]
where $\Phi = \ell \circ G$.
In this case, the drift functions may be continuously extended to $\sigma = 0$.
For simplicity, we denote these continuous extensions again by $\driftone$ and $\driftall$.
Using the fundamental theorem of calculus,
we can rewrite the drift function~$\driftone$ as
\begin{equation}
    \label{eq:drift_1_integral}
    \driftone(\theta, \xi) = - \int_0^1 \Bigl\langle \nabla \ell\bigl(G(\theta)\bigr), \nabla G\bigl(\theta + t\sigma \xi\bigr) \, \xi \Bigr\rangle \, \xi \, \d t \, .
\end{equation}
The total drift $\driftall$ can be rewritten in a similar manner.
If~$G\in\cont^2(\torus^d, \real^K)$, 
then the function~$\driftone$ in~\eqref{eq:drift_1_integral} is continuously differentiable with respect to~$\sigma$ for all fixed $(\theta, \xi)$.
We shall denote the derivative of the total drift at~$\sigma=0$ by
\begin{equation}
    \label{eq:drift_deriv_1}
    \driftderiv(\theta, \allxi) 
    \coloneq \Bigl. \frac{\d}{\d\sigma} \driftall\bigl(\theta, \allxi\bigr) \Bigr|_{\sigma = 0} 
    =  - \frac{1}{2J} \sum_{j=1}^J
    \Bigl\langle \nabla \ell\bigl(G(\theta)\bigr), \nabla^2 G(\theta) : \bigl(\xin{j} \otimes \xin{j}\bigr) \Bigr\rangle \, \xin{j} \, .
\end{equation}
In this section, 
we first state a result that provides upper bounds on the drift coefficient and its gradient.
For a multi-linear operator~$A \colon \real^{m_1} \times \dots \times \real^{m_k} \to \real^n$,
we denote by~$|A|$ its operator norm for the Euclidean norm,
defined as
\[
    |A| \coloneq \sup_{x_1 \in \real^{m_1}\setminus\{0\}, \dots, x_k \in \real^{m_k}\setminus\{0\}} \frac{|A(x_1, \dots, x_k)|}{|x_1| \cdots |x_k|} \, .
\]
\begin{lemma}
    [Upper bound on the drift and its gradient]
    \label{lemma:upper_bound_drift_and_gradient}
    Let $J\in \nat_{>0}$ and $\sigma \geq 0$.
    If $G\in\cont^1(\torus^d, \real^K)$, then it holds that
    \begin{subequations}
        \begin{equation}
            \label{eq:bound_drift_bip_a}
            \forall (\theta,\allxi)\in\torus^d\times \real^{dJ}, \qquad
            \Bigl\lvert \driftall\bigl(\theta, \allxi\bigr) \Bigr\rvert
            \leq 
            \bigl\lVert (\nabla \ell) \circ G \bigr\rVert_{\infty}
            \lVert \nabla G \rVert_{\infty}
            \frac{1}{J}\sum_{j=1}^{J} \bigl\lvert \xin{j} \bigr\rvert^2 \, .
        \end{equation}
        Furthermore, if $G\in\cont^2(\torus^d, \real^K)$, 
        then it also holds for all $(\theta, \allxi) \in \torus^d \times \real^{dJ}$ that
        \begin{align}
            \Bigl\lvert \nabla_{\theta}\driftall\bigl(\theta, \allxi\bigr)\Bigr\rvert
            \leq 
            \Bigl(
                \bigl\lVert(\nabla^2 \ell) \circ G \bigr\rVert_{\infty}
                \|\nabla G\|_{\infty}^2
                +
                \bigl\lVert (\nabla \ell) \circ G \bigr\rVert_{\infty} \|\nabla^2G\|_{\infty}
            \Bigr)
            \frac{1}{J}\sum_{j=1}^{J}\abs*{\xin{j}}^2 \, ,
            \label{eq:bound_drift_bip_b}
        \end{align}
    \end{subequations}
    where $\|\nabla^2G\|_{\infty}$ is the $L^{\infty}$-norm of~$\theta \mapsto |\nabla^2G(\theta)|$.
    Here the second derivative $\nabla^2 G(\theta)$ is viewed as a bilinear map from~$\real^d \times \real^d$ to $\real^K$,
    and we recall that~$| \cdot |$ is the corresponding operator norm.
\end{lemma}
\begin{proof}
    Recall from~\eqref{eq:drift_1_integral} that
    \[
        \driftall\bigl(\theta,\allxi\bigr)
        =
        -\frac{1}{J}
        \sum_{j=1}^J
        \int_0^1
        \Bigl\langle
            \nabla \ell\bigl(G(\theta)\bigr), \nabla G\bigl(\theta+t\sigma\xin{j}\bigr) \, \xin{j}
        \Bigr\rangle
        \, \xin{j} \, \d t \, .
    \]
    The inequality~\eqref{eq:bound_drift_bip_a} then follows from the Cauchy--Schwarz inequality and elementary bounds.
    To obtain the second bound,
    we assume that $G\in\cont^2(\torus^d, \real^K)$ and differentiate the above expression with respect to~\(\theta\) under the integral sign
    to find the following expression for the Jacobian of $\driftall$:
    for any $a \in \real^d$,
    \begin{align*}
        \nabla_\theta \driftall\bigl(\theta,\allxi\bigr) a
        &=
        -\frac{1}{J} \sum_{j=1}^J \int_0^1
        \Bigl\langle
            \nabla^2 \ell\bigl(G(\theta)\bigr) \nabla G(\theta) a, \nabla G\bigl(\theta+t\sigma\xin{j}\bigr) \, \xin{j}
        \Bigr\rangle 
        \, \xin{j} \, \d t 
        \\
        &\qquad
        - \frac{1}{J} \sum_{j=1}^J \int_0^1
        \Bigl\langle
            \nabla \ell\bigl(G(\theta)\bigr), \nabla^2 G\bigl(\theta+t\sigma\xin{j}\bigr) : (\xin{j} \otimes a)
        \Bigr\rangle 
        \, \xin{j} \, \d t \, .
    \end{align*}
    This leads to the desired bound after taking norms and using elementary inequalities.
    % and using that $|AB| \leq |A| \cdot |B|$ for all matrices $A,B$.
\end{proof}
\begin{remark}
    \label{remark:higher_order_derivatives_drift}
    Using a similar approach as in the proof of~\cref{lemma:upper_bound_drift_and_gradient}, 
    it is possible to prove that if~$G \in \cont^{k}(\torus^d, \real^K)$ for some natural number~$k > 2$,
    then~$\driftall \in \cont^{k}(\torus^d \times \real^{dJ}, \real^d)$ with respect to the~$\theta$-variable.
    Furthermore,
    there exists a constant~$C_{k} > 0$ depending only on $k$ and derivative bounds on $G$ and $\ell$ up to order~$k$ such that,
    for all $m \in \{1, \dots, k-1\}$ and all $(\theta, \allxi) \in \torus^d \times \real^{dJ}$
    \begin{equation}
        \label{eq:bound_drift_high_order}
        \left\lvert \nabla_{\theta}^{m} \driftall\bigl(\theta, \allxi\bigr) \right\rvert
        \leq
        C_k \left( \frac{1}{J} \sum_{j=1}^J |\xin{j}|^2 \right) \, .
    \end{equation}
    We emphasize that the constant~$C_k$ does not depend on $\sigma$ or $J$,
    because we restricted~$m$ to be at most~$k-1$.
    We omit the details of the proof for brevity.
\end{remark}

We now state and prove a result concerning the limit of the drift coefficients as $\sigma \to 0$.
\begin{lemma}
    [Limit $\sigma \to 0$ of the drift]
    \label{lemma:drift_lim_small_sigma}
    Fix $J\in\nat_{>0}$.
    If $G\in\cont^2(\torus^d, \real^K)$,
    then it holds for all $(\theta, \allxi) \in \torus^d \times \real^{dJ}$ that
    \begin{subequations}
    \begin{equation}
        \label{eq:second_result_drift}
        \Bigl\lvert \driftall\bigl(\theta, \allxi\bigr) - \driftlim\bigl(\theta, \allxi\bigr) \Bigr\rvert
        \leq
        \frac{\sigma}{2} \, 
        \bigl\lVert ( \nabla \ell ) \circ G \bigr\rVert_{\infty} 
        \, 
        \lVert \nabla^2 G \rVert_{\infty} 
        \left( \frac{1}{J}\sum_{j=1}^J \bigl\lvert \xin{j} \bigr\rvert^3 \right) \, .
    \end{equation}
    Furthermore, if $G\in\cont^3(\torus^d, \real^K)$, then the following estimate holds for all $(\theta, \allxi) \in \torus^d \times \real^{dJ}$:
    \begin{equation}
        \label{eq:second_result_drift_b}
        \Bigl\lvert \driftall(\theta, \allxi) - \driftlim(\theta, \allxi) - \sigma \driftderiv(\theta, \allxi) \Bigr\rvert
        \leq
        \frac{\sigma^2}{6}\,
        \bigl\lVert (\nabla \ell) \circ G \bigr\rVert_{\infty}
        \| \nabla^3 G \|_{L^\infty(\torus^d)} \,
        \left( \frac{1}{J} \sum_{j=1}^J |\xin{j}|^4 \right) \, .
    \end{equation}
    \end{subequations}
\end{lemma}
\begin{proof}
    A similar statement was proved in~\cite[Appendix A]{pavliotis21derivative},
    but we include a sketch of the proof here for completeness.
    Fix $(\theta,\allxi)\in\torus^d\times\real^{dJ}$.
    For each $j\in\{1,\dots,J\}$, we apply the second-order Taylor expansion of $G$
    around $\theta$ with integral remainder:
    \[
        G(\theta+\sigma\xin{j})
        =
        G(\theta)
        +
        \sigma \nabla G(\theta) \xin{j}
        +
        \sigma^2
        \int_0^1 (1-s)\, \nabla^2 G(\theta+s\sigma\xin{j}) : (\xin{j} \otimes \xin{j}) \, \d s \, .
    \]
    It follows from this equation that
    \[
        \driftall\bigl(\theta,\allxi\bigr)
        -\driftlim\bigl(\theta,\allxi\bigr)
        =
        -\frac{\sigma}{J}
        \sum_{j=1}^J \xin{j} 
        \int_0^1 (1-s)\,  \Bigl\langle \nabla \ell\bigl(G(\theta)\bigr), \nabla^2 G\bigl(\theta+s\sigma\xin{j}\bigr) : \bigl(\xin{j} \otimes \xin{j}\bigr) \Bigr\rangle \, \d s \, ,
    \]
    This yields~\eqref{eq:second_result_drift} after applying the Cauchy--Schwarz inequality and noting that
    \[
        \Bigl\lvert \nabla^2 G\bigl(\theta+s\sigma\xin{j}\bigr) : \bigl(\xin{j} \otimes \xin{j}\bigr) \Bigr\rvert 
        \leq \lVert \nabla^2 G \rVert_{\infty} \bigl\lvert \xin{j} \bigr\rvert^2 \, .
    \]
    To obtain~\eqref{eq:second_result_drift_b},
    it suffices to push the Taylor expansion one order further:
    \[
        G(\theta+\sigma\xin{j})
        = G(\theta)
        + \sigma \nabla G(\theta) \xin{j}
        + \frac{\sigma^2}{2} \nabla^2G(\theta) : (\xin{j} \otimes \xin{j})
        + \frac{\sigma^3}{2} \int_0^1 (1-s)^2 \nabla^3 G\bigl(\theta+s\sigma\xin{j}\bigr) \cdot \bigl(\xin{j} \otimes \xin{j} \otimes \xin{j}\bigr) \, \d s \, .
    \]
    Substituting this expression into the definition of $\driftall$ and using the definition of $\driftderiv$ yields
    \begin{align}
        \notag
        &\driftall\bigl(\theta,\allxi\bigr)-\driftlim\bigl(\theta,\allxi\bigr)-\sigma \driftderiv\bigl(\theta,\allxi\bigr) \\
        &\qquad =
        -\frac{\sigma^2}{2J}
        \sum_{j=1}^J \xin{j} 
        \int_0^1 (1-s)^2\,  \Bigl\langle \nabla \ell\bigl(G(\theta)\bigr), \nabla^3 G\bigl(\theta+s\sigma\xin{j}\bigr) \cdot \bigl(\xin{j} \otimes \xin{j} \otimes \xin{j}\bigr) \Bigr\rangle \, \d s \, .
        \label{eq:aux_eq_drift}
    \end{align}
    This yields the desired bound.
\end{proof}

We now provide an estimate on the drift obtained in the averaging limit for fixed $\sigma > 0$.
For any fixed~$\theta \in \torus^d$, 
we rewrite the average drift from~\eqref{eq:asymptotic_drift} as an expectation:
\begin{align}
    \label{eq:asymptotic_drift_appendix}
    \driftinf(\theta) \coloneq \expect_{\xi \sim \fastgaussianone} \Bigl[ \driftone\bigl(\theta, \xi \bigr) \Bigr] \, , 
\end{align}
\begin{lemma}
    [Error estimate for the average drift]
    \label{lemma:bound_sigma}
    Fix $\sigma > 0$ and assume that $G\in\cont^2(\torus^d, \real^K)$.
    Then
    \[
        \bigl\lVert \driftinf + \nabla \Phi \bigr\rVert_{\infty} 
        \leq
        \frac{\sigma}{2} \, 
        \bigl\lVert (\nabla \ell) \circ G \bigr\rVert_{\infty}
        \lVert \nabla^2 G \rVert_{\infty} \, \expect_{\xi \sim \fastgaussianone} \bigl[ |\xi|^3 \bigr] \, .
    \]
    Furthermore, if $G\in\cont^3(\torus^d, \real^K)$, then the following estimate holds:
    \[
        \Bigl\lVert \driftinf + \nabla \Phi \Bigr\rVert_{\infty} 
        \leq
        \frac{\sigma^2}{6} \,
        \bigl\lVert (\nabla \ell) \circ G \bigr\rVert_{\infty}
        \| \nabla^3 G \|_{\infty} \, \expect_{\xi \sim \fastgaussianone} \bigl[ |\xi|^4 \bigr] \, .
    \]
\end{lemma}
\begin{proof}
    These statements can be proved by applying exactly the same arguments as in the proof of \cref{lemma:drift_lim_small_sigma}.
    Consider the equation~\eqref{eq:aux_eq_drift} that we already proved in the particular case~$J=1$.
    Taking the expectation with respect to $\xin{1} \sim \fastgaussianone$
    and noting that the third term on the left-hand side then vanishes
    yields the desired result.
\end{proof}

To conclude this section, 
we state a result concerning the derivative of the average drift $\driftinf$.
\begin{lemma}
    [Bound on the derivative of the average drift]
    \label{lemma:lipschitz_asymptotic_drift}
    Assume that $G\in\cont^2(\torus^d, \real^K)$.
    Then~$\driftinf \in \cont^1(\torus^d, \real^d)$ and it holds for all $\theta \in \torus^d$ that
    \[
        \lvert \nabla \driftinf(\theta) \rvert
        \leq 
        d
        \Bigl(
            \bigl\lVert(\nabla^2 \ell) \circ G \bigr\rVert_{\infty}
            \|\nabla G\|_{\infty}^2
            +
            \bigl\lVert (\nabla \ell) \circ G \bigr\rVert_{\infty} \|\nabla^2G\|_{\infty}
        \Bigr) \, .
    \]
\end{lemma}
\begin{proof}
    The result follows from~\cref{lemma:upper_bound_drift_and_gradient} for~$J=1$.
    Indeed, by dominated convergence and~\eqref{eq:bound_drift_bip_b},
    we can differentiate under the expectation in~\eqref{eq:asymptotic_drift_appendix} to find that
    \[
        \nabla \driftinf(\theta) 
        = \expect_{\xi\sim\fastgaussianone} \Bigl[ \nabla_\theta \driftone\bigl(\theta, \xi\bigr) \Bigr] \, ,
    \]
    and the desired bound then follows from~\eqref{eq:bound_drift_bip_b}.
\end{proof}

% \begin{remark}
%     \label{remark:higher_order_derivatives_drift_inf}
%     Just as in~\cref{remark:higher_order_derivatives_drift},
%     it is possible to prove that if~$G \in \cont^{k}(\torus^d, \real^K)$ for some~$k \in \{3, 4, \dotsc\}$,
%     then~$\driftinf \in \cont^{k}(\torus^d, \real^d)$ and 
%     there exists a constant~$C_k > 0$ depending only on $k$ and derivative bounds on $G$ and $\ell$ up to order~$k$ such that,
%     for all $m \in \{1, \dots, k-1\}$
%     \begin{equation}
%         \label{eq:bound_drift_inf_high_order}
%         \left\lvert \nabla_{\theta}^{m} \driftinf(\theta) \right\rvert_{\infty}
%         \leq
%         C_k \, .
%     \end{equation}
% \end{remark}

\printbibliography
\end{document}